\pdfoutput=1
\documentclass{scrbook}
\KOMAoptions{%
  fontsize=12pt,
	twoside=false,
	chapterprefix=true,
  index=totoc,
  paper=letter,
  DIV=11%
}
\DeclareTOCStyleEntry[indent=0pt]{tocline}{table}
\DeclareTOCStyleEntry[indent=0pt]{tocline}{figure}

\PassOptionsToPackage{dvipsnames,cmyk}{xcolor}

\usepackage{microtype}
\usepackage{amsmath,amsthm,amssymb,mathtools,mathrsfs,textcomp}
\usepackage[only,llbracket,rrbracket]{stmaryrd}
\usepackage{enumitem}
\usepackage{longtable,booktabs}
\usepackage{setspace}
\usepackage{rotating}
\usepackage{multirow}
\usepackage{makecell}

\usepackage{csquotes}
\usepackage[american]{babel}
\usepackage[giveninits=true,maxbibnames=99,sortcites=true,sorting=anyt,backend=biber,style=alphabetic]{biblatex}
\DeclareFieldFormat{postnote}{#1}
\DeclareFieldFormat{multipostnote}{#1}
\DefineBibliographyStrings{english}{phdthesis = {Ph\adddot D\adddotspace thesis}}

\usepackage{xpatch}

\setkomafont{sectioning}{\rmfamily\bfseries}
\setkomafont{title}{\rmfamily\bfseries\Large}
\setkomafont{caption}{\rmfamily\bfseries\small}
\setkomafont{captionlabel}{\rmfamily\bfseries\small}

\usepackage{chngcntr}
\counterwithout{footnote}{chapter}
\deffootnote[1em]{1em}{1em}{\textsuperscript{\thefootnotemark}}

\usepackage{hyphenat}
\makeatletter
\renewcommand*{\@pnumwidth}{2.2em}
\renewcommand*{\@tocrmarg}{3.2em}
\makeatother
\renewcommand{\contentsname}{Table of Contents}

\newcommand*{\listofappendices}{\listoftoc{loapp}}
\setuptoc{loapp}{totoc}

\PassOptionsToPackage{pdfusetitle,linktoc=all,hidelinks}{hyperref}
\usepackage[numbered]{bookmark}

\xapptocmd{\appendix}{
	\phantomsection %
	\addcontentsline{toc}{chapter}{Appendices}
}{}{\PatchFailed}

\xpatchcmd{\addchaptertocentry}{%
    \addtocentrydefault{chapter}{#1}{#2}%
  }{%
  	\ifstr{\chapapp}{Appendix}{
	    \ifstr{#1}{}
	      {\addtocentrydefault{chapter}{#1}{#2}} %
	      {\pdfbookmark[1]{\chapapp{} #1. #2}{app.#1}
        \addcontentsline{loapp}{chapter}{\chapapp{} #1. #2}} %
  	}{
	    \ifstr{#1}{}
	      {\addtocentrydefault{chapter}{#1}{#2}}
	      {\addtocentrydefault{chapter}{}{\chapapp{} #1. #2}}
  	}
  }{}{\PatchFailed}

\xpatchcmd{\addsectiontocentry}{%
    \addtocentrydefault{section}{#1}{#2}%
  }{%
  	\ifstr{\chapapp}{Appendix}{%
    \addcontentsline{loapp}{section}{\protect\numberline{#1} #2}%
    \makeatletter
    \bookmark[level=subsection,dest=\@currentHref]{#1 #2}
    \makeatother
  	}{
		\addtocentrydefault{section}{#1}{#2}%
  	}
  }{}{\PatchFailed}

\xpatchcmd{\addsubsectiontocentry}{%
    \addtocentrydefault{subsection}{#1}{#2}%
  }{%
  	\ifstr{\chapapp}{Appendix}{%
		\addcontentsline{loapp}{subsection}{\protect\numberline{#1} #2}%
    \makeatletter
    \bookmark[level=subsubsection,dest=\@currentHref]{#1 #2}
    \makeatother
  	}{
		\addtocentrydefault{subsection}{#1}{#2}%
  	}
  }{}{\PatchFailed}

\xpatchcmd{\addsubsubsectiontocentry}{%
    \addtocentrydefault{subsubsection}{#1}{#2}%
  }{%
  	\ifstr{\chapapp}{Appendix}{%
		\addcontentsline{loapp}{subsubsection}{\protect\numberline{#1} #2}%
  	}{
		\addtocentrydefault{subsubsection}{#1}{#2}%
  	}
  }{}{\PatchFailed}

\xapptocmd{\mainmatter}
{
	\onehalfspacing
}{}{\PatchFailed}
\xapptocmd{\backmatter}
{
	\singlespacing
}{}{\PatchFailed}

\usepackage{tikz-cd}
\usetikzlibrary{external,hobby,patterns,calc,arrows,decorations.pathmorphing,decorations.pathreplacing}
\definecolor{umdarkmaize}{cmyk}{0,.18,1,0}
\definecolor{umlightblue}{cmyk}{.72,.31,.03,.12}
\definecolor{umdarkblue}{cmyk}{1,.6,0,.6}
\definecolor{umlightgray}{cmyk}{.2,.14,.12,.4}
\colorlet{umdarkgray}{umlightgray!200}
\definecolor{umpoppy}{cmyk}{0,0.71,1,0.03}

\usepackage[toc,stylemods=longbooktabs]{glossaries-extra}
\makeglossaries
\newglossaryentry{finitefields}{%
  name={\ensuremath{\mathbf{F}_q}},
  description={finite field with $q$ elements},
  sort=Fq,
  nonumberlist
}
\newglossaryentry{naturalnumbers}{%
  name={\ensuremath{\mathbf{N}}},
  description={natural numbers $\{0,1,2,\ldots\}$},
  sort=N,
  nonumberlist
}
\newglossaryentry{integers}{%
  name={\ensuremath{\mathbf{Z}}},
  description={integers},
  sort=Z,
  nonumberlist
}
\newglossaryentry{rationalnumbers}{%
  name={\ensuremath{\mathbf{Q}}},
  description={rational numbers},
  sort=Q,
  nonumberlist
}
\newglossaryentry{realnumbers}{%
  name={\ensuremath{\mathbf{R}}},
  description={real numbers},
  sort=R,
  nonumberlist
}
\newglossaryentry{complexnumbers}{%
  name={\ensuremath{\mathbf{C}}},
  description={complex numbers},
  sort=C,
  nonumberlist
}
\newglossaryentry{twistingsheaf}{%
  name={\ensuremath{\mathcal{O}_{\mathbf{P}^n_k}(1)}},
  description={twisting sheaf of Serre},
  sort=O(1),
  nonumberlist
}
\newglossaryentry{divisorbundle}{%
  name={\ensuremath{\mathcal{O}_X(D)}},
  description={sheaf associated to a Cartier divisor},
  sort=OX(D),
  nonumberlist
}
\newglossaryentry{structuresheaf}{%
  name={\ensuremath{\mathcal{O}_X}},
  description={structure sheaf},
  sort={OX},
  nonumberlist
}
\newglossaryentry{sheafcohomology}{%
  name={\ensuremath{H^i(X,\mathscr{F})}},
  description={sheaf cohomology},
  sort={Hi(X,F)},
  nonumberlist
}
\newglossaryentry{singcohomology}{%
  name={\ensuremath{H^i_{\textup{sing}}(X,\mathbf{Z})}},
  description={singular cohomology with coefficients in $\mathbf{Z}$},
  sort={Hi sing(X,F)},
  nonumberlist
}
\newglossaryentry{dimofcohomology}{%
  name={\ensuremath{h^i(X,\mathscr{F})}},
  description={$\dim_k(H^i(X,\mathscr{F}))$ when $X$ is complete over a field $k$},
  sort={hj(X,F)},
  nonumberlist
}
\newglossaryentry{degree}{%
  name={\ensuremath{\deg L}},
  description={degree of a line bundle or divisor on a curve},
  sort=degL,
  nonumberlist
}
\newglossaryentry{projectivebundle}{%
  name={\ensuremath{\mathbf{P}(E)}},
  description={projective bundle of one-dimensional quotients},
  sort={P(E)},
  nonumberlist
}
\newglossaryentry{spectrumring}{%
  name={\ensuremath{\operatorname{Spec} R}},
  description={spectrum of a ring},
  sort={Spec R},
  nonumberlist
}
\newglossaryentry{spectrumsheaf}{%
  name={\ensuremath{\operatorname{\mathbf{Spec}}_X \mathscr{A}}},
  description={relative spectrum of a sheaf of $\mathcal{O}_X$-algebras},
  sort={SpecX A},
  nonumberlist
}
\newglossaryentry{rcirc}{%
  name={\ensuremath{(-)^\circ}},
  description={complement of the union of minimal primes in a ring},
  sort=-circ
}
\newglossaryentry{linearsystem}{%
  name={\ensuremath{\lvert V \rvert}},
  description={linear system},
  sort={ a|V|}
}
\newglossaryentry{completelinearsystem}{%
  name={\ensuremath{\lvert D \rvert}},
  description={complete linear system},
  sort={ a|D|3}
}
\newglossaryentry{projectivespace}{%
  name={\ensuremath{\mathbf{P}^n_k}},
  description={$n$-dimensional projective space over a field $k$},
  sort=P nk
}
\newglossaryentry{canonicaldivisor}{%
  name={\ensuremath{K_X}},
  description={canonical divisor},
  sort=KX
}
\newglossaryentry{canonicalsheaf}{%
  name={\ensuremath{\omega_X}},
  description={canonical bundle or sheaf},
  sort={ zgreek24omegax}
}
\newglossaryentry{dualizingcomplex}{%
  name={\ensuremath{\omega_X^\bullet}},
  description={(normalized) dualizing complex},
  sort={ zgreek24omegaxbullet}
}
\newglossaryentry{intersectionproduct}{%
  name={\ensuremath{(L^{\dim V} \cdot V)}},
  description={intersection product},
  sort={ a(LdimV.V)}
}
\newglossaryentry{seshadriconst}{%
  name={\ensuremath{\varepsilon(D;x)}},
  description={Seshadri constant},
  sort={ zgreek05epsilon(D;x)}
}
\newglossaryentry{frobsesh}{%
  name={\ensuremath{\varepsilon_F^\ell(D;x)}},
  description={$\ell$th Frobenius--Seshadri constant},
  sort={ zgreek05epsilonNFl(D;x)}
}
\newglossaryentry{fsigfrobsesh}{%
  name={\ensuremath{\varepsilon_{\textup{$F$-sig}}^\Delta(D;x)}},
  description={$F$-signature Seshadri constant},
  sort={ zgreek05epsilonsigD(D;x)}
}
\newglossaryentry{movingseshadri}{%
  name={\ensuremath{\varepsilon(\lVert D \rVert;x)}},
  description={moving Seshadri constant},
  sort={ zgreek05epsilon(||D||;x)}
}
\newglossaryentry{movingseshadrijets}{%
  name={\ensuremath{\varepsilon_{\textup{jet}}(\lVert D \rVert;x)}},
  description={jet separation description for $\varepsilon(\lVert D \rVert;x)$},
  sort={ zgreek05epsilonjet(||D||;x)}
}
\newglossaryentry{movingseshadrinakamaye}{%
  name={\ensuremath{\varepsilon_N(\lVert D \rVert;x)}},
  description={Nakamaye's description for $\varepsilon(\lVert D \rVert;x)$},
  sort={ zgreek05epsilonN(||D||;x)}
}
\newglossaryentry{sfx}{%
  name={\ensuremath{s(\mathscr{F};x)}},
  description={largest $\ell$ such that $\mathscr{F}$ separates $\ell$-jets at
  $x$},
  sort={s(F;x)}
}
\newglossaryentry{sfxd}{%
  name={\ensuremath{s(D;x)}},
  description={largest $\ell$ such that $\cO_X(D)$ separates $\ell$-jets at
  $x$},
  sort={s(D;x)}
}
\newglossaryentry{frobenius}{%
  name={\ensuremath{F^e}},
  description={$e$th iterate of the (absolute) Frobenius morphism},
  sort={Fe}
}
\newglossaryentry{frobeniuspower}{%
  name={\ensuremath{(-)^{[p^e]}}},
  description={$e$th Frobenius power of an ideal or module},
  sort={-[pe]}
}
\newglossaryentry{frobpushforward}{%
  name={\ensuremath{F^e_*R}},
  description={the ring $R$ with $R$-algebra structure given by $F^e$},
  sort={Fe*R}
}
\newglossaryentry{gradedfamily}{%
  name={\ensuremath{\mathfrak{a}_\bullet}},
  description={graded family of ideals},
  sort={a.}
}
\newglossaryentry{gradedfamilydivisor}{%
  name={\ensuremath{\mathfrak{a}_\bullet(D)}},
  description={graded family of ideals associated to a divisor},
  sort={a.(D)}
}
\newglossaryentry{triple}{%
  name={\ensuremath{(X,\Delta,\mathfrak{a}_\bullet^\lambda)}},
  description={log triple},
  sort={ a(X,Delta,a.lambda)}
}
\newglossaryentry{testideal}{%
  name={\ensuremath{\tau(X,\Delta,\mathfrak{a}^t)}},
  description={test ideal of a triple},
  sort={ zgreek19tau(X,Delta,at)}
}
\newglossaryentry{testidealofls}{%
  name={\ensuremath{\tau(X,\Delta,t \cdot \lvert D \rvert)}},
  description={test ideal of a Cartier divisor},
  sort={ zgreek19tau(X,Delta,t.a|D|)}
}
\newglossaryentry{asymptotictestideal}{%
  name={\ensuremath{\tau(X,\Delta,\mathfrak{a}_\bullet^\lambda)}},
  description={asymptotic test ideal of a triple},
  sort={ zgreek19tau(X,Delta,at.lambda)}
}
\newglossaryentry{asymtestidealofls}{%
  name={\ensuremath{\tau(X,\Delta,\lambda \cdot \lVert D \rVert)}},
  description={asymptotic test ideal of a $\QQ$-Cartier divisor},
  sort={ zgreek19tau(X,Delta,t.lambda||D||)}
}
\newglossaryentry{bundlepparts}{%
  name={\ensuremath{\mathscr{P}^\ell(L)}},
  description={bundle of principal parts},
  sort={Pl(L)}
}
\newglossaryentry{merofuncsheaf}{%
  name={\ensuremath{\mathscr{K}_X}},
  description={sheaf of total quotient rings},
  sort={KX}
}
\newglossaryentry{cartierdivisors}{%
  name={\ensuremath{\operatorname{Cart}_{\mathbf{k}}(X)}},
  description={group of $\mathbf{k}$-Cartier divisors},
  sort={Cartk(X)}
}
\newglossaryentry{weildivisors}{%
  name={\ensuremath{\operatorname{WDiv}_{\mathbf{k}}(X)}},
  description={group of $\mathbf{k}$-Weil divisors},
  sort={WDivk(X)}
}
\newglossaryentry{roundup}{%
  name={\ensuremath{\lceil D \rceil}},
  description={round-up of a divisor},
  sort={ a|D|1}
}
\newglossaryentry{rounddown}{%
  name={\ensuremath{\lfloor D \rfloor}},
  description={round-down of a divisor},
  sort={ a|D|2}
}
\newglossaryentry{fdual}{%
  name={\ensuremath{(-)^\vee}},
  description={dual of a sheaf},
  sort={-vee}
}
\newglossaryentry{baseideal}{%
  name={\ensuremath{\mathfrak{b}(\lvert V \rvert)}},
  description={base ideal of a linear system},
  sort={b(|V|)}
}
\newglossaryentry{basescheme}{%
  name={\ensuremath{\operatorname{Bs}(\lvert V \rvert)}},
  description={base scheme of a linear system},
  sort={Bs(|V|)}
}
\newglossaryentry{baselocus}{%
  name={\ensuremath{\operatorname{Bs}(\lvert V \rvert)_{\mathrm{red}}}},
  description={base locus of a linear system},
  sort={Bs(|V|)red}
}
\newglossaryentry{stablebaselocus}{%
  name={\ensuremath{\operatorname{\mathbf{B}}(D)}},
  description={stable base locus of a divisor},
  sort={B+(aD)}
}
\newglossaryentry{augmentedbaselocus}{%
  name={\ensuremath{\operatorname{\mathbf{B}_+}(D)}},
  description={augmented base locus of a divisor},
  sort={B+(D)}
}
\newglossaryentry{asymptoticcohfunction}{%
  name={\ensuremath{\widehat{h}^i(X,D)}},
  description={asymptotic cohomological function},
  sort={hk(X,D)}
}
\newglossaryentry{restrictedsections}{%
  name={\ensuremath{H^0(X \vert V,L)}},
  description={image of $H^0(X,L) \to H^0(V,L\rvert_V)$},
  sort={Hk(Z0(X|V,L)}
}
\newglossaryentry{dimrestrictedsections}{%
  name={\ensuremath{h^0(X \vert V,L)}},
  description={$\dim_k H^0(X \vert V,L)$},
  sort={hl(Z0(X|V,L)}
}
\newglossaryentry{volume}{%
  name={\ensuremath{\operatorname{vol}_X(D)}},
  description={volume of a line bundle or divisor},
  sort={volX(D)}
}
\newglossaryentry{restrictedvolume}{%
  name={\ensuremath{\operatorname{vol}_{X\vert V}(D)}},
  description={restricted volume of a line bundle or divisor},
  sort={volXa|V(D)}
}
\newglossaryentry{discrepancyofE}{%
  name={\ensuremath{a(E,X,\Delta,\mathfrak{a}^t)}},
  description={discrepancy of $E$ with respect to a triple},
  sort={a(E,X,D,at)}
}
\newglossaryentry{discrepancy}{%
  name={\ensuremath{\operatorname{discrep}(X,\Delta,\mathfrak{a}^t)}},
  description={disrcepancy of a triple},
  sort={discrep(X,D,at)}
}
\newglossaryentry{totaldiscrepancy}{%
  name={\ensuremath{\operatorname{totaldiscrep}(X,\Delta,\mathfrak{a}^t)}},
  description={total discrepancy of a triple},
  sort={totaldiscrep(X,D,at)}
}
\newglossaryentry{lct}{%
  name={\ensuremath{\operatorname{lct}_x((X,\Delta);\mathfrak{a})}},
  description={log canonical threshold of a pair with respect to an ideal},
  sort={lctx((X,D);a)}
}
\newglossaryentry{multiplierideal}{%
  name={\ensuremath{\mathcal{J}(X,\Delta,\mathfrak{a}^t)}},
  description={multiplier ideal of a triple},
  sort={J(X,D,at)}
}
\newglossaryentry{asymptoticmultiplierideal}{%
  name={\ensuremath{\mathcal{J}(X,\Delta,\mathfrak{a}_\bullet^\lambda)}},
  description={asymptotic multiplier ideal of a triple},
  sort={J(X,D,at.lambda)}
}
\newglossaryentry{multidealofls}{%
  name={\ensuremath{\mathcal{J}(X,\Delta,t \cdot \lvert D \rvert)}},
  description={multiplier ideal of a Cartier divisor},
  sort={J(X,D,t. |D|)}
}
\newglossaryentry{asymmultidealofls}{%
  name={\ensuremath{\mathcal{J}(X,\Delta,\lambda \cdot \lVert D \rVert)}},
  description={asymptotic multiplier ideal of a $\QQ$-Cartier divisor},
  sort={J(X,D,t.lambda||D||)}
}
\newglossaryentry{traceoffrobenius}{%
  name={\ensuremath{\mathrm{Tr}^e_{X,D}}},
  description={the trace of Frobenius},
  sort={TreX,D}
}
\newglossaryentry{fpurethreshold}{%
  name={\ensuremath{\operatorname{fpt}_x((X,\Delta);\mathfrak{a})}},
  description={$F$-pure threshold of a pair with respect to an ideal},
  sort={fptx((X,D);a)}
}
\newglossaryentry{generalizedtightclosure}{%
  name={\ensuremath{(-)^{*}_{(-)}}},
  description={tight closure},
  sort={-*M}
}
\newglossaryentry{bigconex}{%
  name={\ensuremath{\operatorname{Big}^{\{x\}}_{\mathbf{R}}(X)}},
  description={subcone of big cone consisting of $\xi$ such that $x \notin
  \Bplus(\xi)$},
  sort={BigR\{x\}(X)}
}
\newglossaryentry{neronseveri}{%
  name={\ensuremath{N^1_{\mathbf{R}}(X)}},
  description={N\'eron--Severi space},
  sort={N1R(X)}
}
\newglossaryentry{linearequivalence}{%
  name={\ensuremath{\sim_{\mathbf{k}}}},
  description={$\mathbf{k}$-linear equivalence},
  sort={~k},
  nonumberlist
}
\newglossaryentry{numericalequivalence}{%
  name={\ensuremath{\equiv_{\mathbf{k}}}},
  description={$\mathbf{k}$-numerical equivalence},
  sort={=k},
  nonumberlist
}
\newglossaryentry{hilbertsamuel}{%
  name={\ensuremath{e(R)}},
  description={Hilbert--Samuel multiplicity of a local ring},
  sort={e(R)},
  nonumberlist
}
\newglossaryentry{cotangentbundle}{%
  name={\ensuremath{\Omega_X}},
  description={cotangent bundle},
  nonumberlist,
  sort={ zgreek24Omegax}
}
\newglossaryentry{tangentbundle}{%
  name={\ensuremath{T_X}},
  description={tangent bundle},
  nonumberlist,
  sort={TX}
}
\newglossaryentry{gradedlinearsystem}{%
  name={\ensuremath{V_\bullet}},
  description={graded linear system},
  sort={V.}
}
\newglossaryentry{volvbullet}{%
  name={\ensuremath{\operatorname{vol}_X(V_\bullet)}},
  description={volume of a graded linear system},
  sort={volX(V.)}
}
\newglossaryentry{multiplicity}{%
  name={\ensuremath{\operatorname{mult}_WD}},
  description={multiplicity of divisor along a subscheme},
  sort={multWD}
}
\newglossaryentry{degreemap}{%
  name={\ensuremath{\deg f}},
  description={degree of generically finite map},
  sort=degf
}
\newglossaryentry{nonkltlocus}{%
  name={\ensuremath{\operatorname{Nklt}(X,\Delta)}},
  description={non-klt locus of a pair},
  sort={nklt(x,d)}
}
\newglossaryentry{annihilator}{%
  name={\ensuremath{\operatorname{Ann}_R M}},
  description={annihilator of a module},
  nonumberlist,
  sort=AnnRM
}
\newglossaryentry{injectivehull}{%
  name={\ensuremath{E_R(M)}},
  description={injective hull of a module},
  nonumberlist,
  sort=ER(M)
}
\newglossaryentry{exceptionalpullback}{%
  name={\ensuremath{(-)^!}},
  description={exceptional pullback of Grothendieck duality},
  sort={-!}
}
\newglossaryentry{rfstar}{%
  name={\ensuremath{\mathbf{R}f_*}},
  description={derived pushforward},
  nonumberlist,
  sort={Rf*}
}
\newglossaryentry{ordE}{%
  name={\ensuremath{\operatorname{ord}_E}},
  description={divisorial valuation defined by a prime divisor},
  sort={ordE}
}
\newglossaryentry{derivedcat}{%
  name={\ensuremath{\mathbf{D}_{\mathrm{qc}}(X)}},
  description={derived category of quasi-coherent sheaves},
  sort={Dqc(X)}
}
\newglossaryentry{boundedderivedcat}{%
  name={\ensuremath{\mathbf{D}^+_{\mathrm{qc}}(X)}},
  description={bounded-below derived category of quasi-coherent sheaves},
  sort={Drc(X)}
}
\newglossaryentry{sepeftcat}{%
  name={\ensuremath{\mathbf{S}_{\mathbf{e}}}},
  description={category of noetherian schemes whose morphisms are separated and
  essentially of finite type},
  sort={Se}
}
\newglossaryentry{cohomologysheaf}{%
  name={\ensuremath{\mathbf{h}^i(\mathscr{F})}},
  description={cohomology sheaf of a complex of sheaves},
  sort={hm(F)},
  nonumberlist
}
\newglossaryentry{frobdegen}{%
  name={\ensuremath{I_e^\Delta(\mathfrak{m})}},
  description={$e$th Frobenius--degeneracy ideal},
  sort={IeD(m)}
}
\newglossaryentry{maximalx}{%
  name={\ensuremath{\mathfrak{m}_x}},
  description={ideal defining a closed point},
  sort={mx}
}
\newglossaryentry{idealw}{%
  name={\ensuremath{\mathcal{I}_W}},
  description={ideal defining a closed subscheme},
  sort={IW}
}

\setkomafont{descriptionlabel}{\normalfont}

\usepackage[capitalise,noabbrev]{cleveref}
\crefformat{fact}{Fact~\ensuremath{#2#1#3}}
\crefmultiformat{fact}{Facts~\ensuremath{#2#1#3}}{ and~\ensuremath{#2#1#3}}{, \ensuremath{#2#1#3}}{, and~\ensuremath{#2#1#3}}
\crefformat{remarks}{Remark~\ensuremath{#2#1#3}}
\crefmultiformat{remark}{Remarks~\ensuremath{#2#1#3}}{ and~\ensuremath{#2#1#3}}{, \ensuremath{#2#1#3}}{, and~\ensuremath{#2#1#3}}
\crefformat{equation}{\ensuremath{(#2#1#3)}}
\crefmultiformat{equation}{\ensuremath{(#2#1#3)}}{ and~\ensuremath{(#2#1#3)}}{, \ensuremath{(#2#1#3)}}{, and~\ensuremath{(#2#1#3)}}
\crefformat{enumi}{\ensuremath{(#2#1#3)}}
\crefmultiformat{enumi}{\ensuremath{(#2#1#3)}}{ and~\ensuremath{(#2#1#3)}}{, \ensuremath{(#2#1#3)}}{, and~\ensuremath{(#2#1#3)}}
\crefformat{section}{\S\ensuremath{#2#1#3}}
\crefmultiformat{section}{\S\S\ensuremath{#2#1#3}}{, \ensuremath{#2#1#3}}{, \ensuremath{#2#1#3}}{, \ensuremath{#2#1#3}}
\crefformat{question}{Question~\ensuremath{#2#1#3}}
\crefmultiformat{question}{Questions~\ensuremath{#2#1#3}}{ and~\ensuremath{#2#1#3}}{, \ensuremath{#2#1#3}}{, and~\ensuremath{#2#1#3}}
\crefalias{citedthm}{theorem}
\crefalias{innercustomthm}{theorem}

\DeclareMathOperator{\Ann}{Ann}
\DeclareMathOperator{\Ass}{Ass}
\DeclareMathOperator{\Bigcone}{Big}
\DeclareMathOperator{\Bs}{Bs}

\DeclareMathOperator{\Bplus}{\mathbf{B}_+}
\DeclareMathOperator{\Char}{char}
\DeclareMathOperator{\Div}{Cart}
\DeclareMathOperator{\End}{End}
\DeclareMathOperator{\ExcDiv}{ExcDiv}
\DeclareMathOperator{\Ext}{Ext}
\DeclareMathOperator{\Frac}{Frac}
\DeclareMathOperator{\HHom}{\mathscr{H}\!\!\mathit{om}}
\DeclareMathOperator{\Hom}{Hom}

\DeclareMathOperator{\Nklt}{Nklt}
\DeclareMathOperator{\SB}{\mathbf{B}}

\DeclareMathOperator{\SSpec}{\mathbf{Spec}}
\DeclareMathOperator{\Spec}{Spec}
\DeclareMathOperator{\Supp}{Supp}
\DeclareMathOperator{\Sym}{Sym}
\DeclareMathOperator{\TTor}{\mathscr{T}\!\!\mathit{or}}
\DeclareMathOperator{\WDiv}{WDiv}

\DeclareMathOperator{\codim}{codim}
\DeclareMathOperator{\coker}{coker}

\DeclareMathOperator{\eval}{eval}
\DeclareMathOperator{\exc}{Exc}
\DeclareMathOperator{\fpt}{fpt}
\DeclareMathOperator{\height}{ht}
\DeclareMathOperator{\im}{im}
\DeclareMathOperator{\lct}{lct}
\DeclareMathOperator{\length}{length}
\DeclareMathOperator{\mult}{mult}
\DeclareMathOperator{\ord}{ord}
\DeclareMathOperator{\totaldiscrep}{totaldiscrep}
\DeclareMathOperator{\trdeg}{trdeg}
\DeclareMathOperator{\vol}{vol}
\renewcommand{\AA}{\mathbf{A}}
\newcommand{\CC}{\mathbf{C}}

\newcommand{\FF}{\mathbf{F}}
\newcommand{\HH}{\mathbf{H}}
\newcommand{\NN}{\mathbf{N}}
\newcommand{\LL}{\mathbf{L}}
\newcommand{\PP}{\mathbf{P}}
\newcommand{\QQ}{\mathbf{Q}}
\newcommand{\RR}{\mathbf{R}}
\newcommand{\ZZ}{\mathbf{Z}}

\newcommand{\cI}{\mathcal{I}}
\newcommand{\cJ}{\mathcal{J}}
\newcommand{\cO}{\mathcal{O}}
\newcommand{\cP}{\mathcal{P}}
\newcommand{\cQ}{\mathcal{Q}}

\newcommand{\fa}{\mathfrak{a}}
\newcommand{\fb}{\mathfrak{b}}

\newcommand{\fm}{\mathfrak{m}}

\newcommand{\fp}{\mathfrak{p}}
\newcommand{\fq}{\mathfrak{q}}

\newcommand{\sF}{\mathscr{F}}
\newcommand{\sG}{\mathscr{G}}
\newcommand{\sK}{\mathscr{K}}
\newcommand{\sI}{\mathscr{I}}
\newcommand{\sP}{\mathscr{P}}

\newcommand{\Tr}{\mathrm{Tr}}
\newcommand{\cyc}{\mathrm{cyc}}
\newcommand{\id}{\mathrm{id}}

\newcommand{\perf}{\mathrm{perf}}
\newcommand{\pr}{\mathrm{pr}}

\newcommand{\red}{\mathrm{red}}

\newcommand{\isoto}{\overset{\sim}{\to}}
\newcommand{\longisoto}{\overset{\sim}{\longrightarrow}}
\newcommand{\longhookrightarrow}{\lhook\joinrel\longrightarrow}
\newcommand{\longtwoheadrightarrow}{\mathrel{\text{\tikz \draw [-cm double to] (0,0) (0.05em,0.5ex) -- (1.525em,0.5ex);}\hspace{0.05em}}}
\renewcommand{\twoheadrightarrow}{\mathrel{\text{\tikz \draw [-cm double to] (0,0) (0.05em,0.5ex) -- (0.925em,0.5ex);}\hspace{0.05em}}}

\newtheorem{theorem}{Theorem}[section]
\newtheorem{conjecture}[theorem]{Conjecture}
\newtheorem{corollary}[theorem]{Corollary}

\newtheorem{lemma}[theorem]{Lemma}
\newtheorem{principle}[theorem]{Principle}
\newtheorem{proposition}[theorem]{Proposition}
\newtheorem{question}[theorem]{Question}
\newtheorem*{question*}{Question}

\newtheorem{innercustomthm}{Theorem}
\newenvironment{customthm}[1]
  {\renewcommand\theinnercustomthm{#1}\innercustomthm}
  {\endinnercustomthm}

\newtheorem{innercustomquest}{Question}
\newenvironment{customquest}[1]
  {\renewcommand\theinnercustomquest{#1}\innercustomquest}
  {\endinnercustomquest}

\theoremstyle{definition}
\newtheorem{definition}[theorem]{Definition}

\newtheorem{example}[theorem]{Example}
\newtheorem{notation}[theorem]{Notation}
\newtheorem{setup}[theorem]{Setup}

\theoremstyle{remark}
\newtheorem{remark}[theorem]{Remark}

\newtheoremstyle{step}{.25\baselineskip\@plus.1\baselineskip\@minus.1\baselineskip}{.25\baselineskip\@plus.1\baselineskip\@minus.1\baselineskip}{\itshape}{}{\bfseries}{\bfseries .}{5pt plus 1pt minus 1pt}{\thmname{#1}\thmnumber{ #2}\thmnote{ \normalfont(#3)}}
\theoremstyle{step}
\newtheorem{step}{Step}[theorem]

\newtheoremstyle{claim}{.25\baselineskip\@plus.1\baselineskip\@minus.1\baselineskip}{.25\baselineskip\@plus.1\baselineskip\@minus.1\baselineskip}{\itshape}{}{\itshape}{\itshape .}{5pt plus 1pt minus 1pt}{\thmname{#1}\thmnumber{ \normalfont#2}\thmnote{ \normalfont(#3)}}
\theoremstyle{claim}
\newtheorem{claim}[theorem]{Claim}

\newtheoremstyle{cited}{.5\baselineskip\@plus.2\baselineskip\@minus.2\baselineskip}{.5\baselineskip\@plus.2\baselineskip\@minus.2\baselineskip}{\itshape}{}{\bfseries}{\bfseries .}{5pt plus 1pt minus 1pt}{\thmname{#1}\thmnumber{ #2}\thmnote{ \normalfont#3}}
\theoremstyle{cited}
\newtheorem{citedthm}[theorem]{Theorem}
\newtheorem{citedconj}[theorem]{Conjecture}

\newtheorem{citedlem}[theorem]{Lemma}
\newtheorem{citedprop}[theorem]{Proposition}
\newtheorem{citedquestion}[theorem]{Question}
\newtheorem*{customthm*}{Theorem}
\newtheorem{innercustomcitedconj}{Conjecture}
\newenvironment{customcitedconj}[1]
  {\renewcommand\theinnercustomcitedconj{#1}\innercustomcitedconj}
  {\endinnercustomcitedconj}

\newtheoremstyle{citeddef}{.5\baselineskip\@plus.2\baselineskip\@minus.2\baselineskip}{.5\baselineskip\@plus.2\baselineskip\@minus.2\baselineskip}{}{}{\bfseries}{\bfseries .}{5pt plus 1pt minus 1pt}{\thmname{#1}\thmnumber{ #2}\thmnote{ \normalfont#3}}
\theoremstyle{citeddef}
\newtheorem{citeddef}[theorem]{Definition}
\newtheorem{citedex}[theorem]{Example}
\newtheorem{citedconstr}[theorem]{Construction}
\newtheorem{citedsetup}[theorem]{Setup}

\newtheoremstyle{citedrem}{.5\baselineskip\@plus.2\baselineskip\@minus.2\baselineskip}{.5\baselineskip\@plus.2\baselineskip\@minus.2\baselineskip}{}{}{\itshape}{\itshape .}{5pt plus 1pt minus 1pt}{\thmname{#1}\thmnumber{ \normalfont#2}\thmnote{ \normalfont#3}}
\theoremstyle{citedrem}

\renewcommand{\thetheorem}{%
  \ifnum\value{section}=0
    \thechapter.%
  \else
    \thesection.%
  \fi
  \arabic{theorem}%
}

\title{Seshadri Constants and Fujita's Conjecture via Positive
Characteristic Methods}
\author{Takumi Murayama}
\date{2019}
\dedication{To my family}

\begin{document}
\hypersetup{
  pdfsubject={Mathematics},
  pdfkeywords={%
    Seshadri constants,
    Fujita's conjecture,
    positive characteristic methods,
    characterizations of projective space,
    Angehrn--Siu theorem,
    asymptotic cohomological functions%
  }
}

\tikzexternaldisable

\frontmatter

\pagenumbering{Alph}
\begin{titlepage}
 \begin{singlespace} %
  \hbox{\vspace{1.2in}} %
  \begin{center} %
   \makeatletter
   \begin{onehalfspacing}
    {\usekomafont{title} \@title} %
   \end{onehalfspacing} \\[4ex] %
   by \\[2ex] %
   \@author \\ %
   \vfill %
   A dissertation submitted in partial fulfillment \\ %
   of the requirements for the degree of \\ %
   Doctor of Philosophy \\ %
   (Mathematics) \\ %
   in the University of Michigan \\ %
   \@date %
   \makeatother
  \end{center} %
  \vfill %
  \begin{flushleft}
   \hspace{0.7in}Doctoral Committee: \\[2ex] %
   \hspace{1in} %
   \parbox{4.2in}{Professor Mircea Musta\c{t}\u{a}, Chair\\
            Professor Ratindranath Akhoury\\
            Professor Melvin Hochster\\
   				  Professor Mattias Jonsson\\
            Professor Karen E. Smith}
  \end{flushleft} %
 \end{singlespace} %
\end{titlepage}

\clearpage
\begingroup
\renewcommand{\thepage}{F}
\thispagestyle{empty}
\begin{center}
  \hspace{0pt}
  \vfill
  \tikzexternalenable
    \begin{tikzpicture}
      \draw[draw=none, use as bounding box] (-5.25,-1.25) rectangle (6.5,5);
      \draw[umpoppy] (0,4) -- (0,2.5);
      \draw[umpoppy] (0,2.5) -- (1.5,1.13636363636);
      \draw (0,2.5) -- (0,0);
      \draw (0,2.5) to[bend right=10] (-1.25,1);
      \draw (0,2.5) to[bend right=10] (-0.5,1);
      \draw (0,2.5) to[bend left=5] (0.2,1);
      \draw (0,2.5) to[bend left=10] (0.7,1);
      \draw (1.5,1.13636363636) -- (2.75,0);
      \draw (1.5,0) -- (2,0.68181818181);
      \filldraw[umpoppy] (0,4) circle (1.25pt) node[anchor=east,color=umpoppy] {\footnotesize $x_0$};
      \filldraw (0,2.5) circle (1.25pt);
      \draw[open triangle 45-,shorten <=4pt] (0,2.5) -- (-1.25,2.25)
      node[anchor=east] {\footnotesize $x \vee y$};
      \filldraw (0,2.9) circle (1.25pt);

      \filldraw (0,0) circle (1.25pt) node[anchor=south,yshift=-15pt]
      {\footnotesize $0$};
      \filldraw (1.5,0) circle (1.25pt) node[anchor=south,yshift=-15pt]
      {\footnotesize $b$};
      \filldraw (2.75,0) circle (1.25pt) node[anchor=south,yshift=-15pt,xshift=0.5pt]
      {\footnotesize $a$}; 

      \draw[open triangle 45-,shorten <=4pt] (2,0.68181818181) -- (2.625,1)
      node[anchor=west] {\footnotesize Type 2};
      \draw[open triangle 45-,shorten <=10pt] (2.75,0) -- (4,0)
      node[anchor=west] {\footnotesize Type 1 points};
      \draw[open triangle 45-,shorten <=4pt] (0,2.5) -- (0.75,2.6)
      node[anchor=west] {\footnotesize Type 2};
      \draw[open triangle 45-,shorten <=4pt] (0,4) -- (0.75,4.25)
      node[anchor=west] {\footnotesize given norm on $k\{r^{-1}T\}$ $= p(E(r))$};
      \draw[open triangle 45-,shorten <=4pt] (0,2.9) -- (-0.65,2.9)
      node[anchor=east] {\footnotesize Type 3};

      \draw[-open triangle 45] (-3,-0.25) -- (-3,4.5)
      node[anchor=east,yshift=-2pt] {\scriptsize $\rho =$ radius};
      \draw (-3.1,4) -- (-2.9,4) node[anchor=east,xshift=-5pt] {\scriptsize $r$};
      \draw (-3.1,0) -- (-2.9,0) node[anchor=east,xshift=-5pt] {\scriptsize $0$};
      \draw (-3.1,2.9) -- (-2.9,2.9) node[anchor=east,xshift=-5pt] {\scriptsize $\rho
      \notin \lvert k^* \rvert$};
      \draw (-3.1,2.5) -- (-2.9,2.5) node[anchor=east,xshift=-5pt] {\scriptsize $\lvert a
      \rvert$};

      \draw[decoration={brace,mirror,amplitude=5pt},decorate] (-2.25,-0.5) --
      (3.5,-0.5) node [pos=0.5,anchor=north,yshift=-4pt] {\footnotesize
      $\overline{\mathbf{D}}(r)$};

      \filldraw (0,3.45) circle (1.25pt);
      \draw (0,3.45) to[bend left=25] (1,3.1) to[bend left=20] (2,2.9)
      to[bend left=20] (2.8,2.7);
      \draw[thick,loosely dotted] (2.8,2.7) -- (3.5,2.6);
      \filldraw (1,3.1) circle (0.75pt);
      \filldraw (2,2.9) circle (0.75pt);
      \filldraw (2.8,2.7) circle (0.75pt);
      \filldraw (3.5,2.6) circle (1.25pt);
      \draw[open triangle 45-,shorten <=4pt] (3.5,2.6) -- (4.25,2.6)
      node[anchor=west] {\footnotesize Type 4};

      \draw (2,0.68181818181) -- ++(-20:0.45cm);
      \draw (2,0.68181818181) -- ++(-75:0.3cm);
      \draw (1.5,1.13636363636) to[bend left=10] ++(-25:0.45cm);
      \draw (1.5,1.13636363636) -- ++(-110:0.3cm);
      \draw (1.5,1.13636363636) -- ++(-80:0.35cm);
      \draw (1.25,1.3636363636) -- ++(-85:0.4cm);
      \filldraw (1.25,1.3636363636) circle (1.25pt);
      \filldraw[umpoppy] (1.5,1.13636363636) circle (1.25pt) node[anchor=west,
      shift={(-2pt,5pt)},color=umpoppy] {\footnotesize $y$};
      \filldraw (2,0.68181818181) circle (1.25pt);
      \filldraw (0,0.75) circle (1.25pt) node[anchor=east] {\footnotesize $x$};
      \node[umpoppy,anchor=west] at (0.75,2) {\footnotesize$[y,x_0]$};
    \end{tikzpicture}
    \vskip1em
        \begin{tikzpicture}[scale=0.75,every node/.append style={font=\small}]
      \draw[thick] (0,0) arc (100:77:15cm) coordinate (a);
      \draw[umpoppy] (0,0.5) arc (100:87:15cm) edge[out=-2,in=-95] (4,1.5);
      \draw[umpoppy] {(a)+(0,0.5)} arc (77:82:15cm) coordinate (b);
      \node[anchor=east,umpoppy] at (-0.1,0.55) {$a$};
      \path[umlightblue] (0,0.25) edge[out=10,in=120,looseness=0.5] (4.1,0.16)
        (4.1,0.16) edge[out=60,in=170,looseness=0.7] ($(a)+(0,0.25)$);
      \node[anchor=east,umlightblue] at (0.1,0.2) {$a^{-1}$};
      \draw[umpoppy] (b) edge[out=170,in=-87] (4.25,1.48);
      \draw[thick] (5.5,-0.4) arc (170:110:2cm) node[anchor=west,rotate=15] {$\cdots$};
      \node[anchor=west] at (6.2,0.4) {$\operatorname{Spec}(G[a^{-1}])$};
      \fill (4.1,0.16) circle (2pt) node[anchor=north] {$\mathfrak{p}$};
      \draw[line width=2.5pt] (0,0) arc (100:97:15cm) node[pos=0.5,anchor=north,yshift=-3pt] (q) {$\mathfrak{q}$};

      \draw[thick,line cap=round] (0.25,2.5) arc (105:72:10cm) coordinate (a1);
      \draw[thick,line cap=round] (0.25,5) arc (105:72:10cm) coordinate (a2);
      \draw[thick,line cap=round] (0.25,2.5) -- (0.25,5) (a1) -- (a2);
      \node[anchor=west] at (6.2,3.75) {$\operatorname{Spv}(G[a^{-1}])$};
      \draw[-cm to] (3,2.25) -- (3,1.25) node[pos=0.5,anchor=east]
        {\footnotesize $\operatorname{supp}$};
      \fill (0.25,4) circle (2pt) node[anchor=east] (vals) {$s$};
      \draw[open triangle 45-,shorten <=4pt] (vals) -- ($(vals)+(1,0)$)
        node[anchor=west,text width=4cm]
        {\scriptsize Corr.\ to a valuation ring\\
          $R \subset \operatorname{Frac}(G[a^{-1}]/\mathfrak{q})$\\[-0.25em]
          dominating $(G[a^{-1}]/\mathfrak{q})_{\mathfrak{p}/\mathfrak{q}}$};

      \draw[thick] (-7.75,2) arc (100:77:15cm) coordinate (d);
      \draw[thick,line cap=round] (-7.75,3.5) arc (105:72:10cm) coordinate (b1);
      \draw[thick,line cap=round] (-7.75,6) arc (105:72:10cm) coordinate (b2);
      \draw[thick,line cap=round] (-7.75,3.5) -- (-7.75,6) (b1) -- (b2);
      \node[anchor=east] at (-7.75,4.25) {$\operatorname{Spv}(A_a)$};
      \fill (-7.75,5) circle (2pt) node[anchor=east] {$t$};
      \path (-7.75,5) edge[cm bar-cm to,bend right=2,shorten <=6pt,commutative
        diagrams/crossing over] node[pos=0.55,above=1pt,sloped,fill=white]
        {\footnotesize abstract extension} (vals);
      \path (d) edge[-cm to,shorten <=10pt,shorten >=5pt,yshift=-5pt]
        node[pos=0.6,sloped,above=1pt,yshift=2pt] {\footnotesize dominant} (0,1);
      \path ($(b2)-(0,0.25)$) edge[-cm to,yshift=-10pt,shorten <=10pt,shorten >=10pt] (0.25,5);

      \node[anchor=east] at (-8,2) {$\operatorname{Spec}(A_a)$};
      \draw[-cm to] (-5,3.5) -- (-5,2.5) node[pos=0.5,anchor=east]
        {\footnotesize $\operatorname{supp}$};
      \draw[line width=2.5pt] (-8,1.975) arc (100:97:15cm) node[pos=0.5,anchor=north] (tildeq) {$\widetilde{\mathfrak{q}}$};
      \path (tildeq) edge[cm bar-cm to,bend right=12] (q);

      \draw[thick] (-3,7) rectangle (1,9.5);
      \node[anchor=west] at (1,8.5) {$\operatorname{Spv}(A)$};
      \node[rotate=45] at (-3.25,6.7) {\large$\subseteq$};
      \fill (-3,8.75) circle (2pt) node[anchor=east] (valu) {$u$};
      \path (-7.75,5) edge[cm bar-cm to,bend left=5,shorten <=6pt,commutative
        diagrams/crossing over] node[pos=0.5,sloped,above=1pt] {\footnotesize 
        restriction} (valu);
      \fill (0,8.75) circle (2pt) node[anchor=west] {$v$};
      \path (-2.85,8.75) edge[commutative diagrams/rightsquigarrow] node[pos=0.55,above] {$\cdot\rvert_{c\Gamma_u}$} (-0.15,8.75);
    \end{tikzpicture}
    \vskip1em
  \begin{tikzpicture}[scale=1]
    \pgfdeclareshape{strokegenuspic}{
      \anchor{center}{\pgfpointorigin}
      \backgroundpath{\pgfpathmoveto{\pgfqpoint{-1cm}{0cm}}
        \pgfpathcurveto %
          {\pgfpoint{-0.5cm}{-.5cm}}
          {\pgfpoint{0.5cm}{-.5cm}}
          {\pgfpoint{1cm}{0cm}}
        \pgfpathmoveto{\pgfqpoint{-0.75cm}{-0.15cm}}
        \pgfpathcurveto %
            {\pgfpoint{-0.25cm}{.25cm}}
            {\pgfpoint{.25cm}{.25cm}}
            {\pgfpoint{0.75cm}{-0.15cm}}
        \pgfusepath{stroke}
      }
    }
    \pgfdeclareshape{hackgenuspic}{
      \anchor{center}{\pgfpointorigin}
      \backgroundpath{\pgfpathmoveto{\pgfqpoint{-0.78cm}{-.17cm}}
        \pgfpathcurveto %
          {\pgfpoint{-0.35cm}{-.44cm}}
          {\pgfpoint{0.35cm}{-.44cm}}
          {\pgfpoint{.78cm}{-0.17cm}} 
        \pgfpathmoveto{\pgfqpoint{-0.78cm}{-0.17cm}}
        \pgfpathcurveto %
          {\pgfpoint{-0.25cm}{.25cm}}
          {\pgfpoint{.25cm}{.25cm}}
          {\pgfpoint{0.78cm}{-0.17cm}}
        \pgfsetfillcolor{white}
        \pgfusepath{fill}

        \pgfpathmoveto{\pgfqpoint{-1cm}{0cm}}
        \pgfpathcurveto %
          {\pgfpoint{-0.5cm}{-.5cm}}
          {\pgfpoint{0.5cm}{-.5cm}}
          {\pgfpoint{1cm}{0cm}}

        \pgfpathmoveto{\pgfqpoint{-0.75cm}{-0.15cm}}
        \pgfpathcurveto %
          {\pgfpoint{-0.25cm}{.25cm}}
          {\pgfpoint{.25cm}{.25cm}}
          {\pgfpoint{0.75cm}{-0.15cm}}
          \pgfusepath{stroke}
      }
    }
    \path[draw,use Hobby shortcut,closed=true]
      (0,0) .. (1,.75) .. (3,1) .. (4,.3) .. (2,-0.75) .. (.5,-1);
    \begin{scope}
      \clip (1.75,0) -- (2.5,0) -- (2.7,0.4) -- (2,0.4) -- (1.75,0);
      \fill[pattern=north east lines] (2.25,0.2) -- (3.1,0.2) -- (3.05,-0.225) -- (2.05,-0.2) -- (2.25,0.2);
    \end{scope}
    \draw (1.75,0) -- (2.5,0) -- (2.7,0.4) -- (2,0.4) -- cycle;
    \draw (2.25,0.2) -- (3.1,0.2) -- (3.05,-0.225) -- (2.05,-0.2) -- cycle;
    \node[strokegenuspic, draw, scale=0.6, rotate=10] at (1,-0.1) {};
    \node at (-0.3,0.2) {$X$};
    \node[anchor=east,xshift=-3pt] at (2,0.4) {$U_\alpha$};
    \node[anchor=west] at (3.1,0.2) {$U_\beta$};

    \draw (0.5,-2.5) rectangle (1.5,-1.5);
    \draw (3,-2.5) rectangle (4,-1.5);

    \filldraw[pattern=north east lines] (1,-2.5) rectangle (1.5,-2);
    \filldraw[pattern=north east lines] (3,-1.5) rectangle (3.5,-2);

    \path (1.5,-2.25) edge[<->, shorten >=1.5pt, shorten <=1.5pt] node[sloped,
    anchor=center, above, yshift=-2pt] {\scriptsize analytic} (3,-1.75);

    \draw[shorten >=1pt,shorten <=1pt] (1.25,-1.5) -- (1.95,0);
    \draw[shorten >=1pt,shorten <=1pt] (3.25,-1.5) -- (2.85,-0.2);

    \node[anchor=east] at (0.5,-2.1) {$\CC^n \supseteq V_\alpha$};
    \node[anchor=west] at (4,-1.9) {$V_\beta \subseteq \CC^n$};
  \end{tikzpicture}$\quad$%
    \begin{tikzpicture}[y=0.80pt, x=0.80pt, yscale=-0.5, xscale=0.5, inner
      sep=0pt, outer sep=0pt,rotate=-5,scale=0.85]
      \path[draw=black] (1758.3020,890.1180) -- (1649.3230,1104.5890) --
      (2080.5000,1180.0000) node[pos=0.95,above,rotate=-5,yshift=5pt]
      {\small $\mathbf{P}^2_k$} -- (2180.9180,959.6050) -- cycle;
      \path[draw=black,miter limit=10.00,thick] (1719.0000,1035.0000) .. controls
      (1726.1480,1042.4110) and (1743.1240,1052.9010) .. (1752.1580,1057.4780) ..
      controls (1774.8600,1068.9770) and (1781.7920,1070.7240) ..
      (1805.9780,1078.6160) .. controls (1822.5330,1084.0170) and
      (1843.9950,1090.6910) .. (1865.1830,1092.9500) .. controls
      (1886.5600,1095.2280) and (1907.6670,1100.9400) .. (1924.9330,1094.0760) ..
      controls (1940.0400,1088.0700) and (1943.5260,1072.2070) ..
      (1936.4980,1058.9260) .. controls (1931.2570,1049.0200) and
      (1924.3110,1040.0120) .. (1917.7910,1030.7750) .. controls
      (1911.6350,1022.0540) and (1905.8580,1013.1290) .. (1902.2500,1003.0500) ..
      controls (1899.0780,994.1870) and (1897.7800,985.3400) .. (1900.4860,980.0570)
      .. controls (1910.6130,960.2850) and (1956.8040,961.6480) ..
      (1981.7230,967.4400) .. controls (1994.2880,970.3610) and (2023.7800,982.0190)
      .. (2051.7680,997.9190) .. controls (2067.0220,1006.5850) and
      (2078.6970,1019.0710) .. (2090.6830,1028.9620) .. controls
      (2094.8180,1032.3740) and (2109.0440,1043.3490) .. (2111.3740,1045.6270);
      \path[draw=umpoppy,miter limit=10.00,thick] (1719.0000,1065.5000) .. controls
      (1733.9510,1061.8990) and (1751.0450,1061.7650) .. (1762.7560,1050.4180) ..
      controls (1765.7210,1047.5460) and (1774.1330,1032.7620) ..
      (1764.0960,1033.0740) .. controls (1753.5280,1033.4020) and
      (1751.5140,1051.8930) .. (1752.2700,1059.1100) .. controls
      (1753.7830,1073.5650) and (1764.0410,1087.5530) .. (1778.8360,1090.3280) ..
      controls (1793.1390,1093.0110) and (1805.5610,1083.4770) ..
      (1810.2930,1070.3920) .. controls (1812.3420,1064.7250) and
      (1813.1840,1058.4330) .. (1812.3770,1052.4410) .. controls
      (1812.0730,1050.1810) and (1811.2590,1044.3520) .. (1808.3610,1043.6230) ..
      controls (1801.9840,1042.0180) and (1802.9590,1061.1210) ..
      (1803.1740,1063.9300) .. controls (1804.1120,1076.0980) and
      (1808.1610,1088.9740) .. (1817.7120,1097.1380) .. controls
      (1827.8030,1105.7640) and (1842.6190,1107.4600) .. (1854.4890,1101.5740) ..
      controls (1866.0800,1095.8260) and (1874.0440,1082.6430) ..
      (1871.7260,1069.5750) .. controls (1870.5520,1062.9540) and
      (1862.4070,1053.4300) .. (1858.6930,1064.0820) .. controls
      (1854.8930,1074.9790) and (1861.4120,1088.8620) .. (1868.2640,1097.0900) ..
      controls (1875.9200,1106.2810) and (1887.2720,1112.0630) ..
      (1899.3150,1112.0680) .. controls (1910.1240,1112.0720) and
      (1921.4060,1106.9240) .. (1924.5290,1095.8080) .. controls
      (1927.4070,1085.5610) and (1922.6080,1071.2490) .. (1912.8980,1066.0380) ..
      controls (1909.9790,1064.4730) and (1906.1800,1064.0400) ..
      (1905.0050,1067.8320) .. controls (1903.5170,1072.6360) and
      (1907.5440,1078.7080) .. (1910.2900,1082.2280) .. controls
      (1918.9790,1093.3660) and (1936.1520,1102.2940) .. (1949.8280,1094.7940) ..
      controls (1962.7790,1087.6910) and (1964.8370,1069.8860) ..
      (1960.5980,1057.0780) .. controls (1956.4720,1044.6130) and
      (1946.7280,1038.1340) .. (1934.4960,1034.9740) .. controls
      (1921.3580,1031.5800) and (1908.0580,1029.3560) .. (1896.4830,1021.7610) ..
      controls (1887.6330,1015.9540) and (1878.0370,1005.7670) ..
      (1879.9790,994.2190) .. controls (1881.6240,984.4370) and (1891.3390,979.7630)
      .. (1900.4850,980.0560) .. controls (1910.4320,980.3750) and
      (1921.4880,985.5380) .. (1927.6680,993.4580) .. controls (1929.1430,995.3490)
      and (1931.5810,998.8930) .. (1929.7210,1001.2800) .. controls
      (1927.7490,1003.8130) and (1922.9680,1002.3360) .. (1920.5710,1001.4850) ..
      controls (1910.6170,997.9520) and (1899.8040,988.1050) .. (1899.4830,976.8920)
      .. controls (1899.0800,962.8090) and (1915.2880,955.2690) ..
      (1926.5840,951.8160) .. controls (1939.8430,947.7620) and (1957.1030,945.3300)
      .. (1969.9050,952.1580) .. controls (1980.4190,957.7650) and
      (1984.4520,970.0160) .. (1981.9120,981.2850) .. controls (1980.7540,986.4200)
      and (1973.0800,1001.2750) .. (1966.1410,993.5480) .. controls
      (1957.9220,984.3960) and (1973.4500,971.3610) .. (1980.4140,968.0370) ..
      controls (1996.7920,960.2210) and (2017.7080,962.5610) .. (2033.2960,971.1880)
      .. controls (2044.3420,977.3010) and (2055.4320,989.1110) ..
      (2050.7390,1002.6980) .. controls (2047.3380,1012.5480) and
      (2036.9740,1020.2140) .. (2026.9360,1022.1830) .. controls
      (2020.9700,1023.3540) and (2011.8630,1021.3290) .. (2015.1600,1013.3330) ..
      controls (2018.7570,1004.6110) and (2030.9260,1001.0860) ..
      (2039.2200,999.5020) .. controls (2049.6730,997.5050) and (2061.3310,996.9630)
      .. (2071.7910,999.2250) .. controls (2081.3280,1001.2870) and
      (2089.0510,1007.2890) .. (2091.0380,1017.2060) .. controls
      (2093.0090,1027.0440) and (2088.9560,1039.0710) .. (2079.5780,1043.7470) ..
      controls (2075.0930,1045.9830) and (2062.6820,1045.4220) ..
      (2067.2680,1037.4920) .. controls (2072.0380,1029.2420) and
      (2087.4490,1029.0200) .. (2095.5550,1028.7960) .. controls
      (2103.8670,1028.5650) and (2112.1840,1029.0640) .. (2120.4980,1028.9990);
      \fill[fill=umlightblue,miter limit=10.00] (1917.7910,1030.7730) circle
      (0.15cm) node[anchor=south west,xshift=4pt,yshift=2pt,umlightblue] {\footnotesize $p$};
    \end{tikzpicture}
  \tikzexternaldisable
  \vfill
  \hspace{0pt}
\end{center}

\endgroup

\clearpage
\thispagestyle{empty}

\hspace{0pt}

\vfill

\begin{center}
\begin{onehalfspacing}
  \makeatletter
  \@author \\
  \href{mailto:takumim@umich.edu}{\nolinkurl{takumim@umich.edu}} \\
  ORCID iD: \href{https://orcid.org/0000-0002-8404-8540}{\nolinkurl{0000-0002-8404-8540}}\\
  \vspace{0.3in}
  \copyright\;\@author\ \@date
\end{onehalfspacing}
\end{center}

\vfill
\hspace{0pt}

\clearpage

\pagenumbering{roman}
\setcounter{page}{2}

\phantomsection
\addcontentsline{toc}{chapter}{Dedication}
\null\vfill
\makeatletter
{\centering\@dedication \par}%
\vskip \z@ \@plus3fill
\makeatother

\chapter{Acknowledgments}
\begin{onehalfspacing}
  \par First of all, I am tremendously grateful to my advisor Mircea
  Musta\c{t}\u{a} for his constant support throughout the last five years.
  I learned much of the philosophy and many of the methods behind the results
  in this thesis from him, and I am very thankful for his patience and guidance
  as I worked out the many moving pieces that came together to become this
  thesis.
  \par I would also like to thank the rest of my doctoral committee.
  I have benefited greatly from taking courses from Ratindranath Akhoury,
  Melvin Hochster, Mattias
  Jonsson, and Karen E. Smith, and I am especially glad that they are willing to
  discuss physics or mathematics with me when I have questions about their work
  or the classes they are teaching.
  My intellectual debt to Mel and Karen is particularly apparent in this
  thesis, since much of the underlying commutative-algebraic theory was
  developed by Mel and his collaborators, and
  this theory was first applied to algebraic geometry by Karen among others.
  I would also like to thank Bhargav Bhatt, from whom I also learned a lot of
  mathematics through courses and in the seminars that he often led.
  \par Next, I would like to thank my collaborators Rankeya Datta, Yajnaseni
  Dutta, Mihai Fulger, Jeffrey C. Lagarias, Lance E. Miller, David Harry
  Richman, and Jakub Witaszek for their willingness to work on mathematics with
  me.
  While only some of our joint work appears explicitly in this thesis, the ideas
  I learned through conversations with them have had a substantial impact on the
  work presented here.
  \par I would also like to extend my gratitude toward my fellow colleagues at
  Michigan for many useful conversations, including (but not limited to)
  Harold Blum,
  Eric Canton,
  Francesca Gandini,
  Jack Jeffries,
  Zhan Jiang,
  Hyung Kyu Jun,
  Devlin Mallory,
  Eamon Quinlan-Gallego,
  Ashwath Rabindranath,
  Emanuel Reinecke,
  Matthew Stevenson,
  Robert M. Walker,
  Rachel Webb,
  and Ming Zhang.
  I am particularly grateful to Farrah Yhee, without whose support I
  would have had a much harder and definitely more stressful time as a graduate
  student.
  I would also like to thank
  Javier Carvajal-Rojas,
  Alessandro De Stefani,
  Lawrence Ein,
  Krishna Hanumanthu,
  Mitsuyasu Hashimoto,
  J\'anos Koll\'ar,
  Alex K\"uronya,
  Yuchen Liu,
  Linquan Ma,
  Zsolt Patakfalvi,
  Thomas Polstra,
  Mihnea Popa,
  Kenta Sato,
  Karl Schwede,
  Junchao Shentu,
  Daniel Smolkin,
  Shunsuke Takagi,
  Hiromu Tanaka,
  Kevin Tucker,
  and
  Ziquan Zhuang
  for helpful discussions about material in this thesis through the years.
  \par Finally, I would like to thank my parents, sisters, grandparents, and
  other relatives
  for their constant support throughout my life.
  They have always been supportive of me working in mathematics, and of my
  career and life choices.
  I hope that my grandparents in particular are proud of me, even though all of
  them probably have no idea what I work on, and some of
  them did not have the chance to witness me receiving my Ph.D.
  \par This material is based upon work supported by the National Science
  Foundation under Grant Nos.\ DMS-1265256 and DMS-1501461.
  \par
\end{onehalfspacing}

\clearpage
\renewcommand*{\contentsname}{Table of Contents}
\phantomsection\pdfbookmark{\contentsname}{toc}
\tableofcontents

\clearpage
\phantomsection
\addcontentsline{toc}{chapter}{\listfigurename}
\listoffigures

\clearpage
\phantomsection
\listoftables
\addcontentsline{toc}{chapter}{List of Tables}

\clearpage
\listofappendices

\setglossarystyle{long-booktabs}
\setglossarypreamble{Symbols are grouped into three groups,
depending on whether they start with punctuation, Greek letters, or Latin
letters.}
\renewcommand*{\glsgroupskip}{}
\renewcommand{\entryname}{Symbol}
\newlength{\glsattmplen}
\settowidth{\glsattmplen}{$\totaldiscrep(X,\Delta,\fa^t)$}
\setlength{\glsdescwidth}{\dimexpr\linewidth-4\tabcolsep-\glsattmplen}
\renewcommand{\glsxtrprelocation}{}
\makeatletter
\renewcommand*{\@glsxtrpreloctag}{, }
\makeatother
\glsadd{finitefields}
\glsadd{naturalnumbers}
\glsadd{integers}
\glsadd{rationalnumbers}
\glsadd{realnumbers}
\glsadd{complexnumbers}
\glsadd{twistingsheaf}
\glsadd{sheafcohomology}
\glsadd{singcohomology}
\glsadd{dimofcohomology}
\glsadd{degree}
\glsadd{divisorbundle}
\glsadd{projectivebundle}
\glsadd{spectrumring}
\glsadd{spectrumsheaf}
\glsadd{structuresheaf}
\glsadd{linearequivalence}
\glsadd{numericalequivalence}
\glsadd{hilbertsamuel}
\glsadd{cotangentbundle}
\glsadd{tangentbundle}
\glsadd{annihilator}
\glsadd{injectivehull}
\glsadd{rfstar}
\glsadd{cohomologysheaf}

\printglossary[title={List of Symbols}]

\chapter{Abstract}
\onehalfspacing

In 1988, Fujita conjectured that there is an effective and uniform way to turn an
ample line bundle on a smooth projective variety into a globally generated or
very ample line bundle.
We study Fujita's conjecture using Seshadri constants, which were first
introduced by Demailly in 1992 with the hope that they could be used to prove
cases of Fujita's conjecture.
While examples of Miranda seemed to indicate that Seshadri constants could not
be used to prove Fujita's conjecture, we present a new approach to Fujita's
conjecture using Seshadri constants and positive characteristic methods.
Our technique recovers some known results toward Fujita's conjecture over the
complex numbers, without the use of vanishing theorems, and proves new results
for complex varieties with singularities.
Instead of vanishing theorems, we use positive characteristic techniques related
to the Frobenius--Seshadri constants introduced by Musta\c{t}\u{a}--Schwede and
the author.
As an application of our results, we give a characterization of projective
space using Seshadri constants in positive characteristic, which was proved in
characteristic zero by Bauer and Szemberg.

\mainmatter

\chapter{Introduction}\label{s:intro}
Algebraic geometry is the study of
\textsl{algebraic varieties,} which are
geometric spaces defined by polynomial equations.
Some varieties are particularly simple, and the simplest algebraic varieties are
perhaps the $n$-dimensional projective 
spaces\index{projective space, $\mathbf{P}^n$|textbf} \gls*{projectivespace}.
Recall that if $k$ is a field (e.g.\ the complex numbers $\CC$), then the
\textsl{projective space} of dimension $n$ over $k$ is
\[\label{eq:defofpn}
  \PP^n_k \coloneqq \frac{k^{n+1} \smallsetminus \{0\}}{k^*}.
\]
A \textsl{projective variety} over $k$\index{projective variety|textbf} is an
algebraic variety that is isomorphic to a
subset of $\PP^n_k$ defined as the zero set of homogeneous polynomials.
\par Projective spaces are very well understood.
The most relevant property of projective space for us is its intersection theory.
Since at least the Renaissance, artists have used the intersection theory
of $\PP^2_k$ to paint
perspective\index{projective space, $\mathbf{P}^n$!in art}:
in Raphael's \emph{School of Athens} (see
\cref{fig:schoolofathens}), every pair of lines not parallel to the plane of
vision appear to intersect between the two central figures, Plato and Aristotle.
\begin{figure}[t]
  \centering
  \includegraphics[height=3.25in]{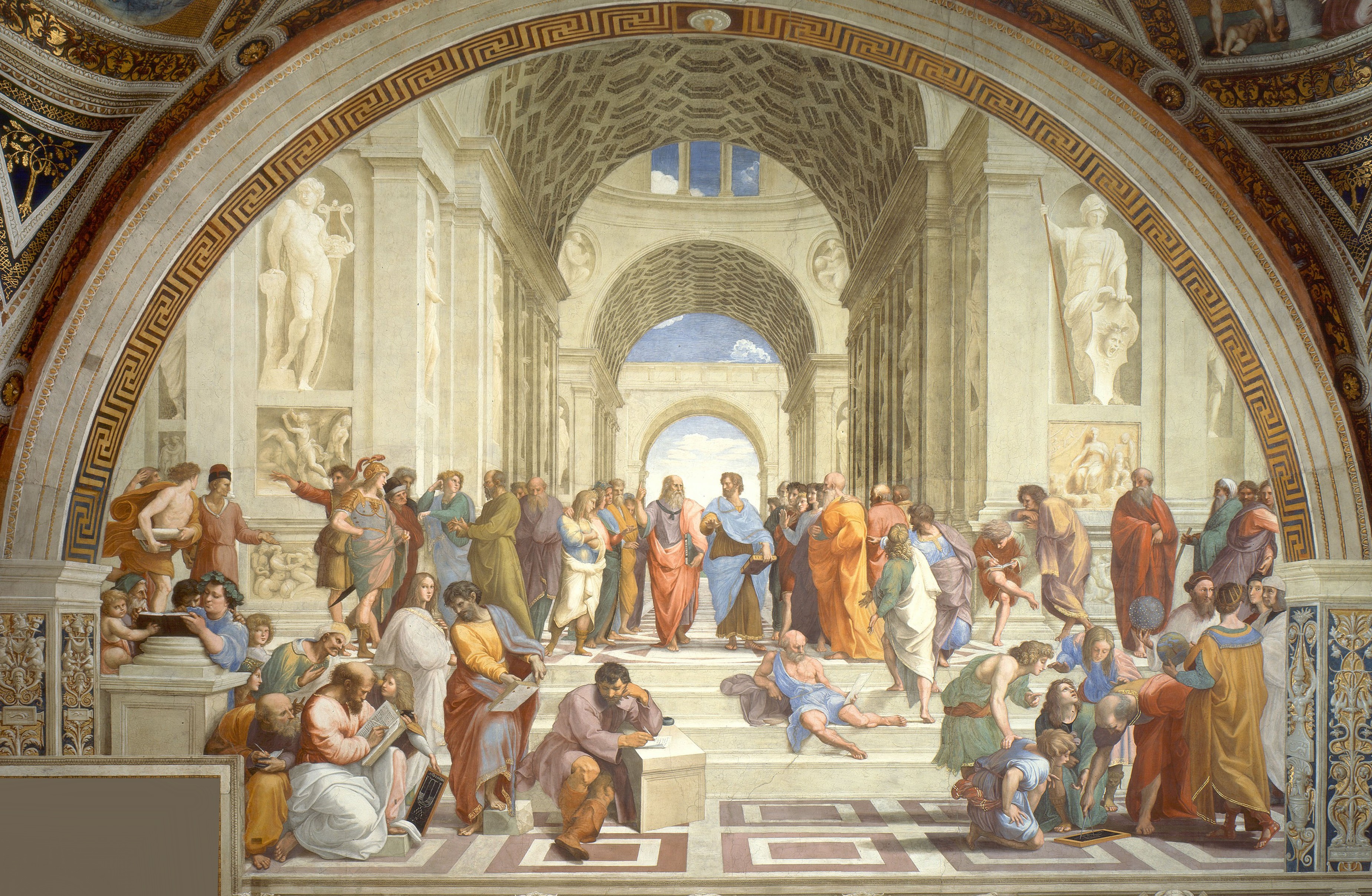}
  \caption{Raphael's \emph{School of Athens} (1509--1511)}
  {\footnotesize Public domain,
  \url{https://commons.wikimedia.org/w/index.php?curid=2194482}}
  \label{fig:schoolofathens}
  \index{projective space, $\mathbf{P}^n$!in art|ff{}}
\end{figure}
Mathematically, a concise way to describe this feature is that the singular
cohomology
ring of $\PP^n_\CC$\index{projective space, $\mathbf{P}^n$!cohomology ring of}
can be described as
\begin{equation}\label{eq:cohringpn}
  H^{*}_{\textup{sing}}\bigl(\PP^n_\CC,\ZZ\bigr) \simeq \frac{\ZZ[h]}{(h^{n+1})},
\end{equation}
where $h \in H^2(\PP^n_\CC,\ZZ)$ is the cohomology class
associated to a hyperplane.
\par In addition to its intersection theory, we understand many more things
about projective spaces, in particular the values of various cohomological
invariants associated to algebraic varieties that come from sheaf cohomology.
It is therefore useful to know when a variety is projective space, prompting the
following:
\begin{question}\label{question:characterization}
  How can we identify when a given projective variety is projective space?
\end{question}
Of course, not every projective variety is a projective space.
For example, the hyperboloid\index{hyperboloid}
\[
  \PP^1_k \times_k \PP^1_k \simeq \bigl\{x^2 + y^2 - z^2 = w^2\bigr\} \subseteq
  \PP^3_k
\]
is an example of a ruled surface\index{ruled surface}, and cannot be isomorphic
to $\PP^2_k$ since two lines in it may not intersect.
See \cref{fig:kobeporttower} for a real-world example of this phenomenon:
adjacent steel trusses that run vertically along the Kobe port tower are
straight, and do not intersect\index{ruled surface!in real life}.
\begin{figure}[t]
  \centering
  \includegraphics[height=3.25in]{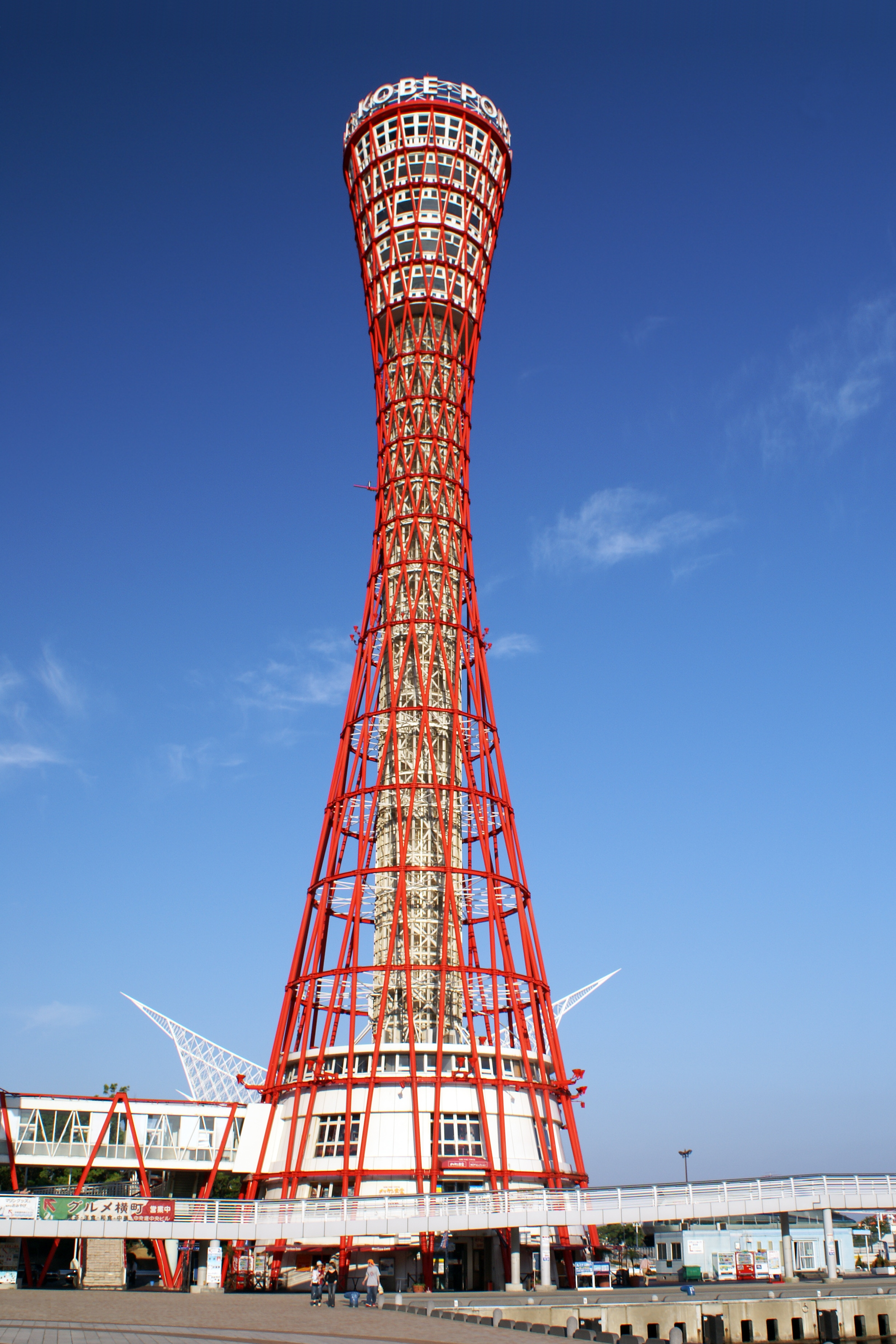}
  \caption{Kobe Port Tower in the Kobe harbor (2006)}
  {\footnotesize By
  \href{https://commons.wikimedia.org/wiki/User:663highland}{\texttt{663highland}},
  \href{https://creativecommons.org/licenses/by/2.5/}{CC BY 2.5},
  \url{https://commons.wikimedia.org/w/index.php?curid=1389137}}
  \index{hyperboloid!in real life|ff{}}
  \index{ruled surface!in real life|ff{}}
  \label{fig:kobeporttower}
\end{figure}
We therefore also ask:
\begin{question}\label{question:embedding}
  Given a projective variety $X$, how can we find an embedding $X
  \hookrightarrow \PP^N_k$, or even just a morphism $X \to \PP^N_k$?
\end{question}
We now state our first result, which gives one answer to
\cref{question:characterization}.
In the statement below, we recall that a smooth projective variety is
\textsl{Fano} if the anti-canonical bundle $\omega_X^{-1} \coloneqq
\bigwedge^{\dim X} T_X$ is ample, where a line
bundle $L$ on a variety $X$ over a field $k$ is \textsl{ample}%
\index{ample line bundle or divisor} if one of the following
equivalent conditions hold (see \cref{def:ampledef,thm:ampleggchar}):
\begin{enumerate}[label=$(\arabic*)$]
  \item There exists an integer $\ell > 0$ such that $L^{\otimes \ell}$ is
    \textsl{very ample,}\index{very ample line bundle or divisor}
    i.e., such that there exists an embedding $X
    \hookrightarrow \PP^N_k$ for some $N$ for which $L^{\otimes \ell} \simeq
    \cO_{\PP^N_k}(1)\rvert_X$.
  \item For every coherent sheaf $\sF$ on $X$, there exists an integer $\ell_0
    \ge 0$ such that the sheaf $\sF \otimes L^{\otimes \ell}$ is globally
    generated for all $\ell \ge \ell_0$.
\end{enumerate}
Additionally, $e(\cO_{C,x})$ denotes the \textsl{Hilbert--Samuel multiplicity}
of $C$ at $x$.
\begin{customthm}{\ref{thm:charpn}}
  Let\index{projective space, $\mathbf{P}^n$!characterization via Seshadri constants}
  $X$ be a Fano variety of dimension $n$ over an
  algebraically closed field $k$ of positive characteristic.
  If there exists a closed point $x \in X$ with
  \[
    \deg\bigl(\omega_X^{-1}\rvert_C\bigr) \ge e(\cO_{C,x}) \cdot (n+1)
  \]
  for every integral curve $C \subseteq X$ passing
  through $x$, then $X$ is isomorphic to the $n$-dimensional projective space
  $\PP^n_k$.
\end{customthm}
An interesting feature of this theorem is that it only requires a positivity
condition on $\omega_X^{-1}$ at \emph{one} point $x \in X$.
Bauer\index{Bauer, Thomas} and Szemberg\index{Szemberg, Tomasz} showed the
analogous statement in characteristic zero.
There have been some recent generalizations of both Bauer and Szemberg's result
and of \cref{thm:charpn} due to Liu\index{Liu, Yuchen} and
Zhuang\index{Zhuang, Ziquan}; see \cref{rem:lz16}.
There is also an interesting connection between \cref{thm:charpn} and the
Mori--Mukai conjecture\index{Mori--Mukai conjecture|(} (see
\cref{conj:morimukai}), which states that if $X$ is a Fano variety of dimension
$n$ such that the anti-canonical bundle $\omega_X^{-1}$ satisfies
$\deg(\omega_X^{-1}\rvert_C) \ge n+1$ for all
rational curves $C \subseteq X$, then $X$ is isomorphic to $\PP^n_k$.
\cref{thm:charpn} strengthens the positivity assumption on $\omega_X^{-1}$ to
incorporate the multiplicity of the curves passing through $x$, but has the
advantage of not having to impose any generality conditions on the point $x$.
See \cref{s:charpncomparison} for further discussion.%
\index{Mori--Mukai conjecture|)}
\medskip
\par Our next result is the main ingredient in proving \cref{thm:charpn},
and gives a partial answer to \cref{question:embedding}.
We motivate this result by first stating Fujita's conjecture, a proof of which
would answer \cref{question:embedding}.
Below, $\omega_X \coloneqq \bigwedge^{\dim X} \Omega_X$ is the \textsl{canonical
bundle}\index{canonical bundle or sheaf} on $X$.
\begin{citedconj}[{\cites[Conj\adddot]{Fuj87}[n\textsuperscript{o} 1]{Fuj88}}]
  \label{conj:fujita}
  Let $X$ be a smooth projective variety of dimension $n$ over an algebraically
  closed field $k$, and let $L$ be an ample line bundle on $X$.
  We then have the following:
  \begin{enumerate}[label=$(\roman*)$,ref=\roman*]
    \item \textup{(Fujita's freeness conjecture)} $\omega_X
      \otimes L^{\otimes \ell}$ is globally generated for all $\ell \ge n+1$.%
      \index{Fujita, Takao!freeness conjecture|textbf}%
      \label{conj:fujitafree}
    \item \textup{(Fujita's very ampleness conjecture)}
      $\omega_X \otimes L^{\otimes \ell}$ is very ample for all $\ell \ge n+2$.%
      \index{Fujita, Takao!very ampleness conjecture|textbf}%
      \label{conj:fujitava}
  \end{enumerate}
\end{citedconj}
\noindent The essence of Fujita's conjecture is that an ample line bundle $L$ can
effectively and uniformly be turned into a globally generated or very ample line
bundle.
Over the complex numbers, Fujita's freeness conjecture holds in dimensions $\le
5$ \cites{Rei88}{EL93fuj}{Kaw97}{YZ15}, and Fujita's very ampleness conjecture
holds in dimensions $\le 2$ \cite{Rei88}.
On the other hand, in arbitrary characteristic, much less is known.
While the same proof as over the complex numbers works for curves, only partial
results are known for surfaces \cites{SB91}{Ter99}{DCF15}, and
in higher dimensions, we only know that Fujita's
conjecture \ref{conj:fujita} holds when $L$ is additionally assumed to be
globally generated \cite{Smi97}.
See \cref{sect:fujitaconj} and especially \cref{tab:knownfujita} for a summary
of existing results.
\par We now describe our approach to Fujita's conjecture \ref{conj:fujita}, and
state our second main result.
In 1992, Demailly\index{Demailly, Jean-Pierre} introduced Seshadri
constants to measure the local positivity of line bundles with the hope that
they could be used to prove cases of Fujita's conjecture \cite[\S6]{Dem92}.
These constants are defined as follows.
Let $L$ be an ample line bundle on a projective variety $X$ over an algebraically
closed field, and consider a closed point $x \in X$.
The \textsl{Seshadri constant}\index{Seshadri constant, $\varepsilon(D;x)$} of
$L$ at $x$ is
\begin{equation}\label{eq:seshdefample}
  \varepsilon(L;x) \coloneqq \sup\bigl\{t \in \RR_{\ge 0} \bigm\vert
  \mu^*L(-tE)\ \text{is ample}\bigr\},
\end{equation}
where $\mu\colon \widetilde{X}\to X$ is the blowup of $X$ at $x$ with
exceptional divisor $E$.
The connection between Seshadri constants and Fujita's conjecture
\ref{conj:fujita} is given by the following result, which says that if the
Seshadri constant $\varepsilon(L;x)$ is sufficiently large, then $\omega_X
\otimes L$ has many global sections.
This is the main ingredient in the proof of \cref{thm:charpn}.
\begin{customthm}{\ref{thm:poscharseshsm}}
  Let\index{Seshadri constant, $\varepsilon(D;x)$!criterion for separation of jets|(}
  $X$ be a smooth projective variety of dimension $n$ over an algebraically
  closed field $k$ of characteristic $p > 0$, and let $L$ be an ample line
  bundle on $X$.
  Let $x \in X$ be a closed point, and consider an integer $\ell \ge 0$.
  If $\varepsilon(L;x) > n+\ell$, then $\omega_X \otimes L$ separates
  $\ell$-jets at $x$, i.e., the restriction morphism
  \[
    H^0(X,\omega_X \otimes L) \longrightarrow H^0(X,\omega_X \otimes L \otimes
    \cO_X/\fm_x^{\ell+1})
  \]
  is surjective, where $\fm_x \subseteq \cO_X$ is the ideal defining $x$.%
  \index{Seshadri constant, $\varepsilon(D;x)$!criterion for separation of jets|)}
\end{customthm}
In particular, then, to show Fujita's freeness conjecture
\ref{conj:fujita}\cref{conj:fujitafree}, it would suffice to show that
$\varepsilon(L;x) > \frac{n}{n+1}$ for every point $x \in X$, where $n = \dim
X$.
\cref{thm:poscharseshsm} was proved over the complex numbers by
Demailly\index{Demailly, Jean-Pierre}; see \cref{prop:demsesh}.
In positive characteristic, the
special case when $\ell = 0$ is due to Musta\c{t}\u{a}%
\index{Mustata, Mircea@Musta\c{t}\u{a}, Mircea} and Schwede\index{Schwede, Karl}
\cite[Thm.\ 3.1]{MS14}.
Our contribution is that the same result holds for all $\ell \ge 0$ in positive
characteristic.
\begin{remark}
  \cref{thm:poscharseshsm} holds more generally for line
  bundles that are not necessarily ample, and for certain
  singular varieties over arbitrary fields; see \cref{thm:demseshsingular}.
  This version of \cref{thm:poscharseshsm} for singular varieties is new even
  over the complex numbers, and we do not know of a proof of this more general
  result that does not reduce to the positive characteristic case.
  Moreover, by combining \cref{thm:demseshsingular} with lower bounds on
  Seshadri constants due to Ein, K\"uchle, and Lazarsfeld (\cref{thm:ekl}), we
  obtain generic results toward Fujita's conjecture \ref{conj:fujita} for
  singular varieties; see \cref{cor:ekl,rem:ekl}.
\end{remark}
\par The main difficulty in proving \cref{thm:poscharseshsm} is that
Kodaira-type vanishing theorems can fail in positive characteristic.
Recall that if $X$ is a smooth projective variety over the complex numbers, and
$L$ is an ample line bundle on $X$, then the Kodaira vanishing
theorem\index{vanishing theorem!Kodaira|textbf} states that
\begin{equation*}\label{eq:kodairavanishing}
  H^i(X,\omega_X \otimes L) = 0
\end{equation*}
for every $i > 0$.
This vanishing theorem was a critical ingredient in Demailly's
proof of \cref{thm:poscharseshsm} over the complex numbers.
In positive characteristic, however, the Kodaira vanishing theorem is often
false, as was first discovered by Raynaud\index{Raynaud, Michel} \cite{Ray78}
(see \cref{ex:raynaud}).
We note that
the strategy behind known cases of Fujita's conjecture \ref{conj:fujita}
is to construct global sections of $\omega_X \otimes L^{\otimes \ell}$
inductively by using versions of the Kodaira vanishing theorem to lift sections
from smaller dimensional subvarieties.
It has therefore been thought that the failure
of vanishing theorems may be the greatest obstacle to making progress on
Fujita's conjecture \ref{conj:fujita} in positive characteristic.
\par In order to replace vanishing theorems, we build on the theory of
so-called ``Frobenius techniques.''
A key insight in positive characteristic algebraic geometry is that while
vanishing theorems are false, there is one major advantage to working in
positive characteristic: every variety $X$ has an interesting endomorphism,
called the \textsl{Frobenius morphism}\index{Frobenius morphism}.
This endomorphism $F\colon X \to X$
is defined as the identity map on points, and the $p$-power
map
\[
  \begin{tikzcd}[row sep=0,column sep=1.475em]
    \cO_X(U) \rar & F_*\cO_X(U)\\
    f \rar[mapsto] & f^p
  \end{tikzcd}
\]
on functions over every open set $U \subseteq X$, where $p$ is the
characteristic of the ground field $k$.
Even if one is only interested in algebraic geometry over the complex numbers,
some results necessitate reducing to the case when the ground field is of
positive characteristic and then using the Frobenius morphism.
For example, this ``reduction modulo $p$'' technique is used in one proof of
the Ax--Grothendieck theorem\index{Ax--Grothendieck theorem},
which says that an injective polynomial endomorphism $\CC^n \to \CC^n$ is
bijective \cites[Thm.\ C]{Ax68}[Prop.\ 10.4.11]{EGAIV3}, and in
Mori's\index{Mori, Shigefumi} bend and break\index{bend and break} technique,
which is used to find rational curves on varieties \cite[\S2]{Mor79}.
The latter in particular is a fundamental technique in modern
birational geometry, but there is no known direct proof of Mori's theorems over
the complex numbers.
\par In its current form, Frobenius techniques were developed simultaneously in
commutative algebra (see, e.g., \cites{HR76}{HH90}) and in representation theory
(see, e.g., \cites{MR85}{RR85}).
Particularly important is the theory of tight closure\index{tight closure}
developed by Hochster\index{Hochster, Melvin} and Huneke\index{Huneke, Craig},
which was used by Smith\index{Smith, Karen E.} to show special cases of Fujita's
conjecture \cites{Smi97}{Smi00fuj}.
\medskip
\par The Frobenius techniques used in proving \cref{thm:poscharseshsm} can be
used to give progress toward Fujita's conjecture \ref{conj:fujita}.
As mentioned above, \cref{thm:poscharseshsm} implies that to show Fujita's
freeness conjecture
\ref{conj:fujita}\cref{conj:fujitafree}, it would suffice to show that
$\varepsilon(L;x) > \frac{n}{n+1}$ for every point $x \in X$, where $n = \dim
X$.
Unfortunately, Miranda\index{Miranda, Rick} showed that the Seshadri constant
$\varepsilon(L;x)$ can get arbitrarily small at special points $x \in X$; see
\cref{ex:mirandaex}.
Nevertheless, we show that the dimension $n$ in the statement of
\cref{thm:poscharseshsm} can be replaced by a smaller number, called the
\textsl{log canonical threshold,} over which one has more control.
See \cref{def:lct} for a precise definition of the log canonical threshold.
This invariant is associated to the data of the variety $X$ together with a
formal $\QQ$-linear combination $\Delta$ of codimension one subvarieties of $X$,
and measures how bad the singularities of $X$ and $\Delta$ are.
We also mention that $\varepsilon(\lVert H \rVert;x)$ below denotes the
\textsl{moving Seshadri constant} of $H$ at $x$, which is a version of the
Seshadri constant defined above in \cref{eq:seshdefample} for line bundles that
are not necessarily ample; see \cref{def:movingsesh}.
\begin{customthm}{\ref{thm:gettingsectionscharzero}}
  Let $(X,\Delta)$ be an effective log pair such that $X$ is a projective normal
  variety over a field $k$ of characteristic zero, and such that
  $K_X+\Delta$ is $\QQ$-Cartier.
  Consider a $k$-rational point $x \in X$ such that $(X,\Delta)$ is klt, and
  suppose that $D$ is a Cartier divisor on $X$ such that $H = D - (K_X+\Delta)$
  satisfies
  \[
    \varepsilon\bigl(\lVert H \rVert;x\bigr) > 
    \lct_x\bigl((X,\Delta);\fm_x\bigr).
  \]
  Then, $\cO_X(D)$ has a global section not vanishing at $x$.
\end{customthm}
While we have stated \cref{thm:gettingsectionscharzero} over a field of
characteristic zero, our proof uses reduction modulo $p$ and Frobenius
techniques to reduce to a similar result
in positive characteristic (\cref{thm:gettingsections}).
\par Using \cref{thm:gettingsectionscharzero}, we then show the following
version of a theorem of Angehrn\index{Angehrn, Urban} and
Siu\index{Siu, Yum Tong} \cite[Thm.\ 0.1]{AS95}.
Our statement is modeled after that in \cite[Thm.\ 5.8]{Kol97}.
Below, $\vol_{X\mid Z}(H)$ denotes the \textsl{restricted volume,} which
measures how many global sections $\cO_Z(mH\rvert_Z)$ has on $Z$ that are
restrictions of global sections of $\cO_X(mH)$ on $X$ as $m \to \infty$; see
\cref{def:restrictedvolume}.
\begin{customthm}{\ref{thm:myangehrnsiu}}
  Let $(X,\Delta)$ be an effective log pair, where $X$ is a normal
  projective variety over an algebraically closed field $k$ of characteristic
  zero, $\Delta$ is a $\QQ$-Weil divisor, and $K_X+\Delta$ is
  $\QQ$-Cartier.
  Let $x \in X$ be a closed
  point such that $(X,\Delta)$ is klt at $x$, and
  let $D$ be a Cartier divisor on $X$ such that setting $H \coloneqq D -
  (K_X+\Delta)$, there exist positive numbers $c(m)$ with the following
  properties:
  \begin{enumerate}[label=$(\roman*)$,ref=\roman*]
    \item For every positive dimensional variety $Z \subseteq X$ containing $x$,
      we have
      \[
        \vol_{X\mid Z}(H) > c(\dim Z)^{\dim Z}.
      \]
    \item The numbers $c(m)$ satisfy the inequality
      \[
        \sum_{m=1}^{\dim X} \frac{m}{c(m)} \le 1.
      \]
  \end{enumerate}
  Then, $\cO_X(D)$ has a global section not vanishing at $x$.
\end{customthm}
A version of this result for smooth complex projective varieties appears in
\cite[Thm.\ 2.20]{ELMNP09}.
As a consequence, we recover the following result, which gives positive evidence
toward Fujita's freeness conjecture \ref{conj:fujita}\cref{conj:fujitafree}.
\begin{corollary}[cf.\ {\cite[Cor.\ 0.2]{AS95}}]\label{cor:as02}
  Let $X$ be a smooth projective variety of dimension $n$ over an algebraically
  closed field of characteristic zero, and let $L$ be an ample line bundle on
  $X$.
  Then, the line bundle $\omega_X \otimes L^{\otimes \ell}$ is globally
  generated for all $\ell \ge \frac{1}{2} n(n+1) + 1$.
\end{corollary}
This corollary is obtained from \cref{thm:myangehrnsiu} by setting $c(m) =
\binom{n+1}{2}$ for every $m$.
Since we prove \cref{cor:as02} without the use of Kodaira-type
vanishing theorems, \cref{thm:myangehrnsiu,cor:as02} support the validity of the
following:
\begin{principle}\label{principle:raynaudnotobstacle}
  The failure of Kodaira-type vanishing theorems is \emph{not} the main
  obstacle to proving Fujita's conjecture \ref{conj:fujita} over fields of
  positive characteristic.
\end{principle}
\noindent Instead, the difficulty is in constructing certain boundary divisors
that are very singular at a point, but have mild singularities elsewhere; cf.\
\cref{thm:kol64}.
\medskip
\par Finally, we mention one intermediate result used in the proofs of
Theorems \ref{thm:gettingsectionscharzero} and \ref{thm:myangehrnsiu},
which is of independent interest.
This statement characterizes ampleness in
terms of asymptotic growth of higher cohomology groups.
It is well known that if $X$ is a projective variety of dimension $n > 0$, then
$h^i(X,\cO_X(mL)) \coloneqq \dim_k H^i(X,\cO_X(mL)) = O(m^n)$ for every Cartier
divisor $L$; see \cite[Ex.\ 1.2.20]{Laz04a}.
It is therefore natural to ask when cohomology groups have submaximal growth.
The following result says that ample Cartier divisors $L$ are characterized by
having submaximal growth of higher cohomology groups for small perturbations of
$L$.
\begin{customthm}{\ref{thm:dfkl41}}
  Let $X$ be a projective variety of dimension $n > 0$ over a field $k$.
  Let $L$ be an $\RR$-Cartier divisor on $X$.
  Then, $L$ is ample if and only if
  there exists a very ample Cartier divisor $A$ on $X$ and a
  real number $\varepsilon > 0$ such that
  \[
    \widehat{h}^i(X,L-tA) \coloneqq \limsup_{m \to \infty}
    \frac{h^i\bigl(X,\cO_X\bigl(\lceil m(L-tA) \rceil\bigr)\bigr)}{m^n/n!} =
    0
  \]
  for all $i > 0$ and for all $t \in [0,\varepsilon)$.
\end{customthm}
Here, the $\widehat{h}^i(X,-)$ are the \textsl{asymptotic higher
cohomological functions} introduced by K\"uronya \cite{Kur06}; see
\cref{sect:burgosgil}.
\cref{thm:dfkl41} was first proved by de Fernex, K\"uronya, and
Lazarsfeld over the complex numbers \cite[Thm.\ 4.1]{dFKL07}.
We note that one can have $\widehat{h}^i(X,L) = 0$ for all $i > 0$ without $L$
being ample, or even pseudoeffective; see \cref{ex:hhatabvar}.
\section{Outline}
This thesis is divided into two parts, followed by two appendices.
The first part consists of \cref{chap:motivation,s:charpn}, and is more
introductory in nature.
In \cref{chap:motivation}, we give more motivation and many examples
illustrating the questions we are studying in this thesis.
After highlighting some difficulties in positive characteristic, we prove
\cref{thm:poscharseshsm}.
We then devote \cref{s:charpn} to proving our characterization of projective
space (\cref{thm:charpn}).
\par The second part of this thesis consists of the remaining chapters.
In \cref{chap:prelims,chap:poschar}, we review some preliminary material that
will be used in the rest of the thesis.
Since almost all of this material is not new, we recommend the reader to skip
ahead to the results they are interested in, and to refer back to these
preliminary chapters as necessary.
We then focus on proving \cref{thm:demseshsingular}, which is a generalization
of \cref{thm:poscharseshsm} for singular varieties, and on proving
Theorems \ref{thm:gettingsectionscharzero} and \ref{thm:myangehrnsiu}.
To do so, we prove \cref{thm:dfkl41} in \cref{ch:dfkl}, which is used when we
study moving Seshadri constants in \cref{chapter:movingseshadri}.
This latter chapter is also where we prove \cref{thm:demseshsingular}.
Finally, we prove Theorems \ref{thm:gettingsectionscharzero} and
\ref{thm:myangehrnsiu} in \cref{ch:angehrnsiu}.
\par The two appendices are devoted to some technical aspects of the theory of
$F$-singularities for rings and schemes whose Frobenius endomorphisms are not
necessarily finite.
\cref{app:nonffinfsings} reviews the definitions of and relationships between
different classes of $F$-singularities, and \cref{app:gamma} develops a
scheme-theoretic version of the gamma construction of Hochster--Huneke, which we
use throughout the thesis to reduce to the case when the ground field $k$
satisfies $[k:k^p]<\infty$, where $\Char k = p > 0$.

\section{Notation and conventions}
We mostly follow the notation and conventions of \cite{Har77} for generalities
in algebraic geometry, of \cites{Laz04a}{Laz04b} for positivity of divisors, line
bundles, and vector bundles, and of \cite{Har66} for Grothendieck
duality theory.
See also the \hyperlink{section*.10}{List of Symbols}.
A notable exception is that we do not assume anything a priori about the ground
field that we work over, and in particular, the ground field may not be
algebraically closed or even perfect.
\par All rings are commutative with identity.
A \textsl{variety}\index{variety|textbf} is a reduced and irreducible scheme
that is separated and of finite type over a field $k$.
A \textsl{complete scheme} is a scheme that is proper over a field $k$.
Intersection products \gls*{intersectionproduct} are defined using Euler
characteristics, following Kleiman\index{Kleiman, Steven L.}; see \cite[App.\
B]{Kle05}.

\chapter{Motivation and examples}\label{chap:motivation}
In this chapter, we motivate the questions posed in the introduction with
some more background and examples.
The new material is a slight modification of Koll\'ar's example
\ref{ex:kollarsurface} to work in arbitrary characteristic, and the proof of
\cref{thm:poscharseshsm}; see \cref{sect:demseshspecialcase}.
A different proof of \cref{thm:poscharseshsm} originally appeared in
\cite[\S3]{Mur18}.
\section{Fujita's conjecture}\label{sect:fujitaconj}
\index{Fujita, Takao!freeness conjecture|(}
\index{Fujita, Takao!very ampleness conjecture|(}
\par To motivate Fujita's\index{Fujita, Takao} conjectural answer to
\cref{question:embedding}, we give some background.
First, we recall the following definition.
\begin{definition}[see {\cite[Def.\ on p.\ 120 and Thm.\ II.7.6]{Har77}}]
  \label{def:ampledef}
  Let $X$ be a scheme over a field $k$, and let $L$ be a line
  bundle on $X$.
  We say that $L$ is
  \textsl{very ample}\index{very ample line bundle or divisor|textbf}
  if there exists an embedding $X \hookrightarrow \PP^N_k$ for some $N$ for
  which $L \simeq \cO_{\PP^N_k}(1)\rvert_X$.
  We say that $L$ is
  \textsl{ample}\index{ample line bundle or divisor|textbf}
  if $L^{\otimes \ell}$ is very ample for some integer $\ell > 0$.
\end{definition}
Ample line bundles can be characterized in the following manner.
\begin{theorem}[Cartan--Serre--Grothendieck; see {\cite[Def.\ on p.\ 153 and
  Thm.\ II.7.6]{Har77}}]\label{thm:ampleggchar}
  Let\index{Cartan--Serre--Grothendieck theorem|textbf}
  $X$ be a scheme of finite type over a field $k$, and let $L$ be a line bundle
  on $X$.
  Then, $L$ is ample if and only if for every coherent sheaf $\sF$ on $X$, there
  exists an integer $\ell_0 \ge 0$ such that the
  sheaf $\sF \otimes L^{\otimes \ell}$ is globally generated for all $\ell \ge
  \ell_0$.
\end{theorem}
\par Because of the defining property in \cref{def:ampledef} and the
characterization in \cref{thm:ampleggchar}, we can ask the following
mathematically precise version of \cref{question:embedding}.
\begin{question}\label{question:whatpower}
  Let $L$ be an ample line bundle on a projective variety $X$.
  What power of $L$ is very ample or globally generated?
\end{question}
The best thing we could hope for is that the power needed in
\cref{question:whatpower} depends on some invariants of $X$.
For curves, we can give a very explicit answer to
\cref{question:whatpower}.
We use the language of divisors instead of line bundles below to
simplify notation.
\begin{example}[Curves I; see {\cite[Cor.\ IV.3.2]{Har77}}]
  \label{ex:matsusakacurves}
  Let $X$ be a smooth curve over an algebraically closed field $k$, i.e., a
  projective variety of dimension $1$ over $k$.
  Let $D$ be a divisor on $X$.
  We claim that the complete linear system $\lvert D \rvert$ is basepoint-free
  if $\deg D \ge 2g$, and is very ample if $\deg D \ge 2g+1$, where $g$ is the
  genus of $X$.
  Recall that by  \cite[Prop.\ IV.3.1]{Har77}, the complete linear system
  $\lvert D \rvert$ is basepoint-free if and only if
  \begin{align*}
    h^0\bigl(X,\cO_X(D-P)\bigr) &= h^0\bigl(X,\cO_X(D)\bigr) - 1
    \intertext{for every closed point $P \in X$, and is very ample if and only
    if}
    h^0\bigl(X,\cO_X(D-P-Q)\bigr) &= h^0\bigl(X,\cO_X(D)\bigr) - 2
  \end{align*}
  for every pair of closed points $P,Q \in X$.
  We will verify these properties below.
  \par Suppose $\deg D \ge 2g$ (resp.\ $\deg D \ge 2g+1$).
  By Serre duality, we have $h^1(X,\cO_X(D)) = 0$, and
  $h^1(X,\cO_X(D-P-Q)) = 0$ for every closed point $P \in X$ (resp.\
  $h^1(X,\cO_X(D-P-Q)) = 0$ for every two closed points $P,Q \in X$).
  We therefore have
  \begin{align*}
    h^0\bigl(X,\cO_X(D-P)\bigr) &= \deg(D - P) + 1 - g\\
    &= \deg D -1 + 1 - g = h^0\bigl(X,\cO_X(D)\bigr) - 1\\
    h^0\bigl(X,\cO_X(D-P-Q)\bigr) &= \deg(D-P-Q) + 1 - g\\
    &= \deg D - 2 + 1 - g = h^0\bigl(X,\cO_X(D)\bigr) - 2
  \end{align*}
  in each case by the Riemann--Roch theorem\index{Riemann--Roch theorem}
  \cite[Thm.\ IV.1.3]{Har77}.
  As a result, we see that if $L$ is an ample divisor on $X$, the complete
  linear system
  $\lvert \ell L \rvert$ is basepoint-free for all $\ell \ge 2g$,
  and is very ample for all $\ell \ge 2g+1$, where $g$ is the genus of $X$.
\end{example}
We can answer \cref{question:whatpower} for abelian varieties as well.
\begin{example}[Abelian varieties]\label{ex:fujitaavs}
  If $L$ is an ample line bundle on an abelian variety $A$,
  then $L^{\otimes \ell}$ is globally generated for $\ell \ge 2$ and is very
  ample for $\ell \ge 3$ by a theorem of
  Lefschetz\index{Lefschetz, Solomon, theorem of}.
  See \cite[App.\ 1 on p.\ 57 and Thm.\ on p.\ 152]{Mum08}.
\end{example}
On the other hand, the following example essentially due to
Koll\'ar\index{Koll\'ar, J\'anos} shows that one
cannot hope for such a simple answer on surfaces: different ample line
bundles on the same surface may need to be raised to different powers to become
very ample.
Note that we have modified Koll\'ar's example to work in arbitrary
characteristic.
\begin{example}[Koll\'ar {\cite[Ex.\ 3.7]{EL93a}}]
  \label{ex:kollarsurface}
  Let\index{Koll\'ar, J\'anos!example of|(}
  $E$ be an elliptic curve over an algebraically closed field $k$.
  Let $X = E \times_k E$, let $F_i$ be the divisors associated to the fibers
  of the projection morphisms $\pr_i\colon X \to E$ for $i \in \{1,2\}$, and
  let $\Delta$ be the divisor associated to the diagonal in $X$.
  Set $R = F_1 + F_2$.
  Since $3R$ is very ample by \cref{ex:matsusakacurves}, we can choose a smooth
  divisor $B \in \lvert 3R \rvert$ by Bertini's theorem \cite[Thm.\
  II.8.18]{Har77}; see \cref{fig:kollarex}.
  \begin{figure}[t]
    \centering
    \tikzexternalenable
    \begin{tikzpicture}[scale=1.25]
      \draw[thick,umlightblue] (-1,-1) -- (3,3);
      \draw[thick,umdarkblue] (0,-1) -- (0,3) node[anchor=south] {\footnotesize $F_1$};
      \draw[thick,umdarkblue] (3,0) -- (-1,0) node[anchor=east] {\footnotesize $F_2$};
      \begin{scope}
        \clip (-1,-1) rectangle (3,3);
        \draw[ultra thick,umpoppy] (0.2,3) .. controls (0.25,0.25) .. (3,0.2);
        \draw[ultra thick,umpoppy] (-0.2,3) .. controls (-0.25,0.25) .. (-3,0.2);
        \draw[ultra thick,umpoppy] (-0.2,-3) .. controls (-0.25,-0.25) .. (-3,-0.2);
        \draw[ultra thick,umpoppy] (0.2,-3) .. controls (0.25,-0.25) .. (3,-0.2);
      \end{scope}
      \draw (-1,-1) rectangle (3,3);

      \node[umlightblue] at (1.75,1.5) {\footnotesize $\Delta$};
      \node[umpoppy] at (1.5,0.5) {\footnotesize $B$};

      \node[anchor=north] at (1.5,3) {$X = E \times E$};
      \draw[<-] (3.15,0.5) -- node[anchor=north] {\scriptsize $3\mathbin{:}1$}
      node[anchor=south] {\scriptsize $f$} (3.95,0.5) node[anchor=west] {$Y$};
    \end{tikzpicture}
    \tikzexternaldisable
    \caption{Koll\'ar's example (\cref{ex:kollarsurface})}
    \index{Koll\'ar, J\'anos!example of|ff{}}
    \label{fig:kollarex}
  \end{figure}
  For each integer $m \ge 2$, consider the divisor
  \[
    A_m \coloneqq m F_1 + (m^2-m+1) F_2 - (m-1) \Delta
  \]
  on $X$.
  We can compute that $(A_m^2) = 2$ and $(A_m\cdot R) = m^2 - 2m+3 > 0$, hence
  $A_m$ is ample: these intersection conditions imply $A_m$ is big by
  \cite[Cor.\ V.1.8]{Har77}, and the fact that $X$ is a homogeneous space
  implies $A_m$ is ample by the Nakai--Moishezon criterion \cite[Thm.\
  1.2.23]{Laz04a} (see \cite[Lem.\ 1.5.4]{Laz04a}).
  \par Now consider the triple cover $f\colon Y \to X$ branched over $B$, as
  constructed in \cite[Prop.\ 4.1.6]{Laz04a}.
  For every $m \ge 2$, the divisors $D_m \coloneqq f^*A_m$ are ample by
  \cite[Prop.\ 1.2.13]{Laz04a}, but we claim that $mD_m$ is not ample.
  It suffices to show that the pullback homomorphism
  \begin{equation}\label{eq:kollarexstar}
    f^*\colon H^0\bigl(X,\cO_X(mA_m)\bigr) \longrightarrow
    H^0\bigl(Y,\cO_Y(mD_m)\bigr)
  \end{equation}
  is an isomorphism, since if this were the case, then the morphism
  \[
    Y \xrightarrow{\lvert mD_m \rvert}
    \PP\bigl(H^0\bigl(Y,\cO_Y(mD_m)\bigr)\bigr)
  \]
  would factor through the $3\mathbin{:}1$ morphism $f$.
  To show that \cref{eq:kollarexstar} is an isomorphism, we first note that
  \begin{equation}\label{eq:doublecoveriso}
    \begin{aligned}
      f_*\bigl(\cO_Y(mD_m)\bigr) &\simeq f_*\cO_Y \otimes \cO_X(mA_m)\\
      &\simeq \cO_X(mA_m) \oplus \cO_X(mA_m-R) \oplus \cO_X(mA_m-2R)
    \end{aligned}
  \end{equation}
  by the projection formula and by the construction of $Y$ (see \cite[Rem.\
  4.1.7]{Laz04a}).
  On global sections, the inclusion $H^0(X,\cO_X(mA_m)) \hookrightarrow
  H^0(X,f_*(\cO_Y(mD_m)))$ induced by the isomorphism \cref{eq:doublecoveriso}
  can be identified with the pullback homomorphism \cref{eq:kollarexstar} by
  the construction of $Y$.
  On the other hand, since $(mA_m-R)^2 < 0$ and $(mA_m-2R)^2 < 0$, we have that
  $H^0(X,\cO_X(mA_m-R)) = H^0(X,\cO_X(mA_m-2R)) = 0$ by \cite[Lem.\
  1.5.4]{Laz04a}.
  Thus, \cref{eq:kollarexstar} is an isomorphism.%
  \index{Koll\'ar, J\'anos!example of|)}
\end{example}
To get bounds only in terms of the dimension of $X$,
Mukai\index{Mukai, Shigeru} suggested that the correct bundles to look at are
\textsl{adjoint line bundles}%
\index{adjoint line bundle or sheaf, $\omega_X \otimes L$|textbf}, i.e., line
bundles of the form $\omega_X \otimes L$, where $\omega_X$ is the
canonical bundle on $X$.
In this direction, Fujita\index{Fujita, Takao|(} conjectured the following:
\begin{customcitedconj}{\ref{conj:fujita}}[{\cites[Conj\adddot]{Fuj87}[n\textsuperscript{o} 1]{Fuj88}}]
  Let $X$ be a smooth projective variety of dimension $n$ over an algebraically
  closed field, and let $L$ be an ample line bundle on $X$.
  We then have the following:
  \begin{enumerate}[label=$(\roman*)$,ref=\roman*]
    \item \textup{(Fujita's freeness conjecture)} $\omega_X
      \otimes L^{\otimes \ell}$ is globally generated for all $\ell \ge n+1$.%
      \index{Fujita, Takao!freeness conjecture}%
    \item \textup{(Fujita's very ampleness conjecture)}
      $\omega_X \otimes L^{\otimes \ell}$ is very ample for all $\ell \ge n+2$.%
      \index{Fujita, Takao!very ampleness conjecture}%
  \end{enumerate}
\end{customcitedconj}
\index{Fujita, Takao|)}%
Note that both properties hold for some $\ell$: \cref{conj:fujitafree} holds
for some $\ell$ by \cref{thm:ampleggchar}, and for
\cref{conj:fujitava}, it suffices to note that if $\omega_X \otimes
L^{\otimes \ell_1}$ is globally generated and $L^{\otimes \ell_2}$ is very
ample, then their tensor product $\omega_X \otimes
L^{\otimes(\ell_1+\ell_2)}$ is very ample \cite[Prop.\ 4.4.8]{EGAII}.
The essence of Fujita's conjecture, then, is that the $\ell$
required can be bounded effectively in terms of only the dimension of $X$.
\par Fujita's conjecture \ref{conj:fujita} is known for some special classes of
varieties.
\begin{example}[Projective spaces and toric varieties]
  If $X = \PP^n_k$ for a field $k$ and $L = \cO_{\PP^n_k}(1)$, then $\omega_X =
  \cO_{\PP^n_k}(-n-1)$
  \index{Fujita, Takao!freeness conjecture!for P n@for $\PP^n$}
  \index{Fujita, Takao!very ampleness conjecture!for P n@for $\PP^n$}
  \index{canonical bundle or sheaf, $\omega_X$!of Pn@of $\PP^n$}
  \cite[Ex.\ II.8.20.2]{Har77}.
  Thus, the bounds in Fujita's conjecture \ref{conj:fujita} are in some sense
  optimal.
  \par Fujita's conjecture also holds for toric varieties.
  \index{Fujita, Takao!freeness conjecture!for toric varieties}%
  \index{Fujita, Takao!very ampleness conjecture!for toric varieties}%
  In the smooth case, this follows from Mori's cone
  theorem\index{Mori, Shigefumi!cone theorem} (see, e.g.,
  \cite[Rem.\ 10.4.6]{Laz04a} and see \cite[Thm.\ 0.3]{Mus02} for a stronger
  statement), and in the singular case, see \cites[Cor.\ 0.2]{Fuj03}[Thm.\
  1]{Pay06}.
\end{example}
\begin{example}[Curves II]\label{ex:fujitacurves}
  \index{Fujita, Takao!freeness conjecture!for curves}
  \index{Fujita, Takao!very ampleness conjecture!for curves}
  Let $X$ be a smooth curve over an algebraically closed field as in
  \cref{ex:matsusakacurves}, and let $L$ be an ample line bundle on $X$.
  By \cref{ex:matsusakacurves}, since $\deg \omega_X = 2g-2$ where $g$
  is the genus of $X$ \cite[Ex.\ IV.1.3.3]{Har77}, the line bundle $\omega_X
  \otimes L^{\otimes\ell}$ is globally generated if $\ell \ge 2$, and is very
  ample if $\ell \ge 3$.
\end{example}
\begin{example}[Abelian varieties]
  Since\index{Fujita, Takao!freeness conjecture!for abelian varieties}%
  \index{Fujita, Takao!very ampleness conjecture!for abelian varieties}
  the canonical bundle $\omega_A$ is isomorphic to the structure sheaf $\cO_A$
  on an abelian variety $A$,
  \cref{ex:fujitaavs} already shows that Fujita's conjecture holds for abelian
  varieties.
\end{example}
\begin{example}[Ample and globally generated line bundles; see {\cite[Ex.\
  1.8.23]{Laz04a}}]\label{ex:fujitaamplegg}
  \index{Fujita, Takao!freeness conjecture!for ample and globally generated line bundles|(}
  \index{Fujita, Takao!very ampleness conjecture!for ample and globally generated line bundles|(}
  Fujita's conjecture \ref{conj:fujita} holds when $L$ is moreover assumed to be
  globally generated.
  In characteristic zero, this can be seen as follows.
  By Castelnuovo--Mumford regularity \cite[Thm.\ 1.8.5]{Laz04a}, a coherent
  sheaf $\sF$ on $X$ is globally generated if $H^i(X,\sF \otimes L^{\otimes-i} )
  = 0$ for all $i > 0$.
  Thus, the sheaf $\sF = \omega_X \otimes L^{\otimes \ell}$ is globally
  generated for $\ell \ge n+1$ since
  $H^i(X,\omega_X \otimes L^{\otimes(\ell-i)} ) = 0$
  by the Kodaira vanishing theorem\index{vanishing theorem, Kodaira} \cite[Thm.\
  4.2.1]{Laz04a}, proving \cref{conj:fujitafree}.
  \cref{conj:fujitava} then follows from \cite[Ex.\ 1.8.22]{Laz04a}.
  \par We also mention generalizations of this example.
  In characteristic zero, the argument above works when $X$ is only assumed to
  have rational singularities by \cite[Ex.\ 4.3.13]{Laz04a}, and in positive
  characteristic, Smith\index{Smith, Karen E.|(} used tight
  closure\index{tight closure, $N^*_M$} methods to
  recover an analogous result when $X$ has $F$-rational singularities
  \cite[Thm.\ 3.2]{Smi97}.
  Keeler\index{Keeler, Dennis S.} gave a proof of Smith's
  \index{Smith, Karen E.|)} result using Castelnuovo--Mumford
  regularity, and also showed
  \cref{conj:fujitava} for smooth varieties
  when $L$ is globally generated \cite[Thm.\ 1.1]{Kee08}.
  Note that Keeler's argument for \cref{conj:fujitafree} also applies to
  varieties with $F$-injective singularities in
  positive characteristic; see \cite[Thm.\ 3.4$(i)$]{Sch14}.
  \index{Fujita, Takao!very ampleness conjecture!for ample and globally generated line bundles|)}
  \index{Fujita, Takao!freeness conjecture!for ample and globally generated line bundles|)}
\end{example}
\begin{table}[t]
  \centering
  \renewcommand\arraystretch{1.25}
  \begin{tabular}{cr@{ }lc}
    \toprule
    Dimension & \multicolumn{2}{c}{Result (over $\CC$)} & Method\\
    \midrule
    1 & Classical & \cite[Cor.\ IV.3.2$(a)$]{Har77} &
    Riemann--Roch\index{Riemann--Roch theorem|ttindex{}}\\
    2 & Reider\index{Reider, Igor|ttindex{}} & \cite[Thm.\ 1$(i)$]{Rei88}
    & Bogomolov instability\index{Bogomolov, Fedor, instability|ttindex{}}\\
    \cmidrule(lr){1-4}
    3 & Ein--Lazarsfeld\index{Ein, Lawrence|ttindex{}}
    \index{Lazarsfeld, Robert|ttindex{}} &
    \cite[Cor.\ 2*]{EL93fuj} &
    \multirow{3}{*}{\makecell{Cohomological method
    of\\Kawamata--Reid--Shokurov
    \index{Kawamata--Reid--Shokurov, cohomological method of|ttindex{}}}}
    \index{Kawamata, Yujiro|ttindex{}}\\
    4 & Kawamata\index{Kawamata, Yujiro|ttindex{}} & \cite[Thm.\ 4.1]{Kaw97}\\
    5 & Ye--Zhu\index{Ye, Fei|ttindex{}}\index{Zhu, Zhixian|ttindex{}} &
    \cite[Main Thm.]{YZ15}\\
    \bottomrule
  \end{tabular}
  \caption{Known cases of Fujita's freeness conjecture over the complex numbers}
  \index{Fujita, Takao!freeness conjecture!in dimensions $\le 5$ over $\CC$|ttindex{}}
  \label{tab:knownfujita}
\end{table}
\par For general smooth complex projective varieties, Fujita's freeness
conjecture \ref{conj:fujita}\cref{conj:fujitafree} holds in dimensions $n \le
5$ (see \cref{tab:knownfujita}) while Fujita's very ampleness conjecture
\ref{conj:fujita}\cref{conj:fujitava} is only known in dimensions $n \le 2$
\cite[Thm.\ 1$(ii)$]{Rei88}%
\index{Fujita, Takao!very ampleness conjecture!for surfaces over $\CC$}.
In positive characteristic, the usual statement of Fujita's conjecture holds for
surfaces that are neither quasi-elliptic nor of general type \cite[Cor.\
8]{SB91}, and weaker bounds are known for quasi-elliptic and general type
surfaces \cites[Thm\adddot]{Ter99}[Thm.\ 1.4]{DCF15}.%
\index{Fujita, Takao!freeness conjecture!for surfaces in characteristic $p$}
\par In arbitrary dimension, one of the best results toward Fujita's conjecture
so far is the following result due to Angehrn\index{Angehrn, Urban} and
Siu\index{Siu, Yum Tong}, which they proved using analytic
methods.
\begin{citedthm}[{\cite[Cor.\ 0.2]{AS95}}]\label{thm:angehrnsiu}
  \index{Angehrn--Siu theorem}
  Let $X$ be a smooth complex projective variety of dimension $n$, and let $L$
  be an ample line bundle on $X$.
  Then, the line bundle $\omega_X \otimes L^{\otimes \ell}$ is globally
  generated for all $\ell \ge \frac{1}{2} n(n+1) + 1$.
\end{citedthm}
Koll\'ar\index{Koll\'ar, J\'anos} later gave an algebraic proof of
\cref{thm:angehrnsiu}, which also applies to klt pairs \cite[Thm.\ 5.8]{Kol97}.
Improved lower bounds for $\ell$ have also been obtained by
Helmke\index{Helmke, Stefan} \cites[Thm.\ 1.3]{Hel97}[Thm.\ 4.4]{Hel99} and
Heier\index{Heier, Gordon} \cite[Thm.\ 1.4]{Hei02}.
Note that \cref{thm:angehrnsiu} is a special case of \cref{thm:myangehrnsiu},
which we will prove later in this thesis, since we can set $c(m) =
\binom{n+1}{2}$ for all $m$; see \cref{cor:as02}.
\index{Fujita, Takao!very ampleness conjecture|)}
\index{Fujita, Takao!freeness conjecture|)}
\section{Seshadri constants}\label{sect:seshadriconstants}
\index{Seshadri constant, $\varepsilon(D;x)$|(}
To study Fujita's conjecture \ref{conj:fujita},
Demailly\index{Demailly, Jean-Pierre} introduced Seshadri
constants, which measure the local positivity of nef divisors.
Recall that an $\RR$-Cartier divisor $D$ is
\textsl{nef}\index{nef line bundle or divisor|textbf}
if $(D \cdot C) \ge 0$ for every curve $C \subseteq X$.
See \cref{def:qrcartierdiv} for the definition of an $\RR$-Cartier divisor.
\begin{definition}[{see \cite[Def.\ 5.1.1]{Laz04a}}]\label{def:seshconst}
  Let $X$ be a complete scheme over a field $k$, and let $D$ be a nef
  $\RR$-Cartier divisor on $X$.
  Let $x \in X$ be a $k$-rational point, and let $\mu\colon \widetilde{X} \to X$
  be the blowup of $X$ at $x$ with exceptional divisor $E$.
  The \textsl{Seshadri constant}%
  \index{Seshadri constant, $\varepsilon(D;x)$|textbf} of $D$ at $x$ is
  \[
    \gls*{seshadriconst} \coloneqq \sup\bigl\{ t \in \RR_{\ge 0} \bigm\vert
    \mu^*D - tE\ \text{is nef}\bigr\}.
  \]
  We use the same notation for line bundles.
\end{definition}
We will see later that this definition matches the definition in
\cref{eq:seshdefample} (\cref{lem:seshample}).
This definition was motivated by
Seshadri's\index{Seshadri, C. S.!criterion for ampleness|(} criterion for
ampleness, which says that when $k$ is algebraically closed, an $\RR$-Cartier
divisor $D$ is ample if and only if $\inf_{x \in X}\varepsilon(D;x) > 0$
\cite[Thm.\ 1.4.13]{Laz04a}.\index{Seshadri, C. S.!criterion for ampleness|)}
While originally defined in the context of Fujita's conjecture, Seshadri
constants have also attracted attention as interesting geometric invariants in
their own right; see \cites[Ch.\ 5]{Laz04a}{BDRHKKSS09}.
\par Before describing the connection between Seshadri constants and Fujita's
conjecture \ref{conj:fujita}, we compute a simple example.
Note that Seshadri constants are very difficult to compute in general.
We will use the fact from \cite[Prop.\ 5.1.5]{Laz04a} that
\begin{equation}\label{eq:seshinfcurves}
  \varepsilon(D;x) = \inf_{C \ni x} \biggl\{ \frac{(D \cdot C)}{e(\cO_{C,x})}
  \biggr\},%
  \index{Seshadri constant, $\varepsilon(D;x)$!via intersections with curves}
\end{equation}
where the infimum runs over all integral curves $C \subseteq X$ containing
$x$, and $e(\cO_{C,x})$ is the Hilbert--Samuel
multiplicity\index{Hilbert--Samuel multiplicity, $e(R)$} of $C$ at $x$.
\begin{example}[Projective spaces; see \cref{fig:seshprojspace}]
  \label{ex:seshprojspace}
  \begin{figure}[t]
    \centering
    \tikzexternalenable
    \begin{tikzpicture}[scale=2]
      \draw (-0.3,-1) -- (1,-0.3) node [very near start, above, sloped] {$H$};
      \draw (1,-0.3) -- (0.2,1) -- (-1,0.3) -- (-0.3,-1);
      \draw[ultra thick,smooth,color=umlightblue,domain=-1:1,samples=50,variable=\t,rotate=20,commutative diagrams/crossing over] plot ({(\t)^2}, {\t^3}) node[anchor=north west] {$C$};
      \fill (0,0) circle (1pt) node[anchor=north,yshift=-1pt] {\footnotesize $x$};
    \end{tikzpicture}
    \tikzexternaldisable
    \caption{Computing the Seshadri constant of the hyperplane class on
    $\PP^n_k$}
    \label{fig:seshprojspace}
  \end{figure}
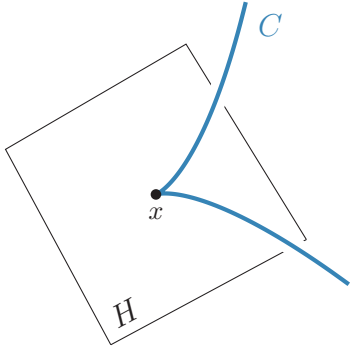
  Consider%
  \index{Seshadri constant, $\varepsilon(D;x)$!on $\PP^n$}
  $\PP^n_k$ for an algebraically closed field $k$, and let $D = H$ be
  the hyperplane class.
  We claim that $\varepsilon(H;x) = 1$ for every closed point $x \in \PP^n_k$.
  By B\'ezout's theorem \cite[Thm.\ I.7.7]{Har77}, we have $(H \cdot C) \ge
  e(\cO_{C,x})$ for every such curve $C \ni x$, hence $\varepsilon(H;x) \ge 1$
  by \cref{eq:seshinfcurves}.
  The inequality $\varepsilon(H;x) \le 1$ also holds by considering the case
  when $C$ is a line containing $x$.
\end{example}
\cref{ex:seshprojspace} can be generalized as follows.
\begin{example}[Ample and globally generated line bundles; see {\cite[Ex.\
  5.1.18]{Laz04a}}]
  We%
  \index{Seshadri constant, $\varepsilon(D;x)$!for ample and globally generated divisors}
  claim that if $D$ is an ample and free Cartier divisor, then
  $\varepsilon(D;x) \ge 1$.
  Let $C \ni x$ be a curve; it suffices to show that $(D \cdot C) \ge
  e(\cO_{C,x})$.
  Since the complete linear system $\lvert D \rvert$ is basepoint-free, there
  exists a divisor $H \in \lvert D \rvert$ such that $H$ does not
  contain $C$.
  We then see that
  \[
    (D \cdot C) = \deg(D\rvert_C) \ge \ell(\cO_{D\rvert_C,x}) \ge e(\cO_{C,x}),
  \]
  where the first inequality follows from definition (see \cite[Def.\
  15.29]{GW10}) and the second inequality is a consequence of \cite[Thm.\
  14.10]{Mat89}.
\end{example}
\par Demailly's\index{Demailly, Jean-Pierre|(} original motivation for defining
Seshadri constants seems to have been its potential application to
Fujita's conjecture \ref{conj:fujita}.
Before we state the result realizing this connection, we make the following
definition.
\begin{definition}\label{def:sepljets}
  Let $X$ be a scheme, and let $\sF$ be a coherent sheaf on $X$.
  Fix a closed point $x \in X$, and denote by $\gls*{maximalx} \subseteq \cO_X$
  the ideal sheaf defining $x$.
  For every integer $\ell \ge -1$, we say that $\sF$ \textsl{separates
  $\ell$-jets}\index{separation of jets|textbf} at $x$ if the restriction
  morphism
  \begin{equation}\label{eq:demaillyrestriction}
    H^0(X,\sF) \longrightarrow H^0(X,\sF/\fm_x^{\ell+1}\sF)
  \end{equation}
  is surjective.
  We denote by \gls*{sfx} the largest integer $\ell \ge -1$ such that $\sF$
  separates $\ell$-jets at $x$.
  If $\sF = \cO_X(D)$ for a Cartier divisor $D$, then we denote $\gls*{sfxd}
  \coloneqq s(\cO_X(D);x)$.
\end{definition}
\begin{remark}
  The convention that $s(\sF;x) = -1$ if $\sF$ does not separate
  $\ell$-jets for every $\ell \ge 0$ is from \cite[Def.\ 6.1]{FM}.
  This differs from the convention $s(\sF;x) = -\infty$, which is used in
  \cite[p.\ 96]{Dem92} and \cite[Def.\ 2.1]{Mur18}, and the convention
  $s(\sF;x) = 0$, which is used in \cite[p.\ 646]{ELMNP09}.
  Our convention is chosen to make a variant of the Seshadri constant defined
  using jet separation (\cref{def:jetsesh}) detect augmented base loci
  (\cref{lem:jetseshbplus}), while distinguishing whether or not $\sF$ has any
  non-vanishing global sections.
\end{remark}
We now prove the following result due to Demailly,
which connects Seshadri constants to separation of jets.
\begin{citedprop}[{\cite[Prop.\ 6.8$(a)$]{Dem92}}]\label{prop:demsesh}
  Let\index{Seshadri constant, $\varepsilon(D;x)$!criterion for separation of jets}
  $X$ be a smooth projective variety of dimension $n$ over an algebraically
  closed field of characteristic zero, and let $L$ be a big and nef
  divisor on $X$.
  Let $x \in X$ be a closed point, and consider an integer $\ell \ge 0$.
  If $\varepsilon(L;x) > n+\ell$, then $\omega_X \otimes \cO_X(L)$ separates
  $\ell$-jets at $x$.
\end{citedprop}
\begin{proof}
  Consider the short exact sequence
  \[
    0 \longrightarrow \fm_x^{\ell+1} \cdot \omega_X \otimes \cO_X(L)
    \longrightarrow \omega_X \otimes \cO_X(L) \longrightarrow \omega_X \otimes
    \cO_X(L) \otimes \cO_X/\fm_x^{\ell+1} \longrightarrow 0.
  \]
  By the associated long exact sequence on sheaf cohomology, to show the
  surjectivity of the restriction morphism \cref{eq:demaillyrestriction},
  it suffices to show that
  \begin{align*}
    H^1\bigl(X,\fm_x^{\ell+1} \cdot \omega_X \otimes \cO_X(L)\bigr) &\simeq
    H^1\bigl(\widetilde{X},\mu^*\bigl(\omega_X \otimes
    \cO_X(L)\bigr)\bigl(-(\ell+1)E\bigr)\bigr)\\
    &\simeq H^1\bigl(\widetilde{X},\omega_{\widetilde{X}} \otimes
    \cO_{\widetilde{X}}\bigl(\mu^*L-(n+\ell)E\bigr)\bigr) = 0,
  \end{align*}
  where $\mu\colon \widetilde{X} \to X$ is the blowup of $X$ at $x$.
  Here, the first isomorphism follows from the Leray spectral
  sequence and the
  quasi-isomorphism $\fm_x^{\ell+1} \simeq
  \RR\mu_*\cO_{\widetilde{X}}(-(\ell+1)E)$ \cite[Lem.\ 4.3.16]{Laz04a},
  and the second isomorphism follows from how the canonical bundle transforms
  under a blowup with a smooth center \cite[Exer.\ II.8.5$(b)$]{Har77}.
  The vanishing of the last group follows from the Kawamata--Viehweg vanishing
  theorem\index{vanishing theorem!Kawamata--Viehweg} \cite[Thm.\ 4.3.1]{Laz04a}
  since $\mu^*L-(n+\ell)E$ is nef by the assumption $\varepsilon(L;x) > n +
  \ell$, and is big because by \cite[Prop.\ 5.1.9]{Laz04a}, we have
  \[
    \bigl(\mu^*L-(n+\ell)E\bigr)^n = (L^n) - (n+\ell)^n \ge
    \bigl(\varepsilon(L;x)\bigr)^n - (n+\ell)^n > 0.\qedhere
  \]
\end{proof}
Demailly showed that a similar technique can be
used to deduce separation of points from the existence of lower
bounds on Seshadri constants, and in particular, that if $\inf_{x \in
X}\varepsilon(L;x) > 2n$ where $n = \dim X$, then $\omega_X \otimes \cO_X(L)$ is
very ample \cite[Prop.\ 6.8$(b)$]{Dem92}.
Because of these results, Demailly asked:
\begin{citedquestion}[{\cite[Quest.\ 6.9]{Dem92}}]\label{quest:demailly}
  Given a smooth projective variety $X$ over an algebraically closed field and
  an ample divisor $L$ on $X$, does there exist a lower bound for
  \[
    \varepsilon(L) \coloneqq \inf_{x \in X} \varepsilon(L;x)\,?
  \]
  If such a lower bound were to exist, could we compute this lower bound
  explicitly in terms of geometric invariants of $X$?%
  \index{Demailly, Jean-Pierre|)}
\end{citedquestion}
\begin{remark}
  We note that the divisors constructed in Koll\'ar's example
  \ref{ex:kollarsurface}\index{Kollar, Janos@Koll\'ar, J\'anos!example of}
  do not give a counterexample to \cref{quest:demailly}.
  In the notation of \cref{ex:kollarsurface}, the divisor $2A_m$ is free on $X$
  by \cref{ex:fujitaavs}.
  The pullback $2D_m = f^*(2A_m)$ is therefore ample and free, hence
  $\varepsilon(2D_m;x) \ge 1$ for every point $x \in Y$.
  By the homogeneity of Seshadri constants \cite[Ex.\ 5.1.4]{Laz04a}, we have
  $\varepsilon(D_m;x) \ge 1/2$.
\end{remark}
A very optimistic answer to \cref{quest:demailly} would be that $\varepsilon(L)
> \frac{n}{n+1}$ where $n = \dim X$, since if this were the case,
\cref{prop:demsesh} would then imply Fujita's freeness conjecture
\ref{conj:fujita}\cref{conj:fujitafree}.
The following example of Miranda\index{Miranda, Rick}, however, shows that
$\varepsilon(L)$ can become arbitrarily small, even on smooth surfaces.
\begin{example}[Miranda {\cite[Ex.\ 3.1]{EL93a}}]\label{ex:mirandaex}
  Let\index{Miranda, Rick!example of|(} $\delta > 0$ be arbitrary.
  We will construct a smooth projective surface $X$ over an
  algebraically closed field $k$ such that $\varepsilon(L;x) < \delta$ for an
  ample divisor $L$ on $X$ and a closed point $x \in X$.
  \par Choose an integer $m \ge 1$ such that $\frac{1}{m} < \delta$, and let
  $\Gamma \subseteq \PP^2_k$ be an integral curve of degree $d \ge 3$ and
  multiplicity $m$ at a closed point $\xi \in \PP^2_k$.
  Let $\Gamma' \subseteq \PP^2_k$ be a general curve of degree $d$, which by
  generality we may assume is integral and intersects $\Gamma$ in $d^2$
  reduced points.
  We moreover claim that for general $\Gamma'$, every curve in the
  pencil $\lvert W \rvert$ spanned by $\Gamma$ and
  $\Gamma'$ is irreducible.
  Note that such a pencil is a one-dimensional linear system, while
  the codimension of the space of reducible curves in $\lvert dH \rvert$
  is
  \begin{align*}
    \binom{d+2}{2} \MoveEqLeft- \max_{1 \le i \le d-1} \biggl\{
    \binom{i+2}{2} + \binom{d-i+2}{2} \biggr\} + 1\\
    &\ge \frac{(d+1)(d+2)}{2} - \biggl(
    \frac{d}{2} + 1 \biggr) \biggl( \frac{d}{2} + 2 \biggr) + 1 =
    \frac{d^2}{4} \ge 2,
  \end{align*}
  by the assumption $d \ge 3$.
  Thus, for general $\Gamma'$, the pencil $\lvert W \rvert$ does not contain any
  reducible curves.
  \begin{figure}[t]
    \centering
    \tikzexternalenable
    \begin{tikzpicture}[scale=1.85,
      continuation/.style={insert path={coordinate[pos=-#1] (aux1)
        coordinate[pos=0] (aux2) coordinate[pos=1+#1] (aux3) coordinate[pos=1]
        (aux4) (aux1) -- (aux2) (aux3) -- (aux4)}
      }]
      \begin{scope}[yscale=0.85,rotate=15]
        \foreach \t in {0.1,0.2,...,0.9}{
          \draw[umlightgray] (-1,1) .. controls ({-(1-\t)*0.5+\t},{(1-\t)*1.75}) and ({(1-\t)*0.5-\t},{(1-\t)*1.75}) .. (1,1) [continuation={((1-\t)*0.07+\t*0.03)}];
        }
        \draw[ultra thick,umlightblue] (-1,1) .. controls (-0.5,1.75) and (0.5,1.75) .. (1,1) node[pos=0.8,anchor=south,yshift=2pt,yscale=0.85,rotate=15] {\footnotesize $\Gamma'$} [continuation=0.075];
        \draw[ultra thick,umpoppy] (-1,1) .. controls (1,0) and (-1,0) .. (1,1) node[midway] (multpt) {} node[pos=0.9,anchor=west,xshift=5pt,yscale=0.85,rotate=15] {\footnotesize $\Gamma$} [continuation=0.035];
        \fill (multpt) circle (1pt) node[anchor=north,yscale=0.85,rotate=15] {\scriptsize $\xi$};
        \fill (-1,1) circle (1pt);
        \fill (1,1) circle (1pt);
        \draw (-1.5,-0.25) rectangle (1.5,2);
        \node[anchor=south west,yscale=0.85,rotate=15] at (-1.5,-0.25) {$\mathbf{P}^2_k$};
      \end{scope}
      \draw[thick,umdarkblue] (3,0.75) node[anchor=east] {\footnotesize $E$} -- (6,0.75);
      \draw[ultra thick,umpoppy] (3.75,0.25) .. controls (3.25,2.5) and (3.25,0.25) .. (3.75,2.5) node[midway] (multptonx) {};
      \draw[ultra thick,umlightblue] (5.25,0.25) .. controls (5.5,1) and (5.5,1.75) .. (5.25,2.5);
      \foreach \t in {0.1,0.2,...,0.9}{
        \draw[umlightgray] ({3.75+1.5*\t},0.25) .. controls ({(1-\t)*3.25+\t*5.5},{(1-\t)*2.5+\t*1}) and ({(1-\t)*3.25+\t*5.5},{(1-\t)*0.25+\t*1.75}) .. ({3.75+1.5*\t},2.5);
      }
      \fill (multptonx) circle (1pt) node[anchor=east] {\scriptsize $x$};
      \node[umpoppy] at (3.45,2) {\footnotesize $C$};
      \node[umlightblue] at (5.6,2) {\footnotesize $C'$};
      \draw (3,0.25) rectangle (6,2.5);
      \node at (6.25,1.5) {$X$};
      \draw[thick] (3,-0.75) -- (6,-0.75);
      \node at (6.25,-0.75) {$\mathbf{P}^1_k$};
      \path[->] (4.5,0.1) edge node[right] {\footnotesize $\pi$} (4.5,-0.6);
      \path[->] (2.75,1.6) edge[bend right=15] node[above,sloped] {\footnotesize blowup} node[below,sloped] {\footnotesize $\Gamma \cap \Gamma'$} (1.4,1.25);
    \end{tikzpicture}
    \tikzexternaldisable
    \caption{Miranda's example (\cref{ex:mirandaex})}
    {\footnotesize Illustration inspired by \cite[Fig.\ 5.1]{Laz04a}}
    \index{Miranda, Rick!example of|ff{}}
    \label{fig:mirandaex}
  \end{figure}
  \par We now consider the blowup $X \to \PP^2_k$ along $\Gamma \cap \Gamma'$
  (see \cref{fig:mirandaex}).
  Since we have blown up the base locus of $\lvert W \rvert$, there is an
  induced morphism $\pi\colon X \to \PP^1_k$ whose fibers correspond to curves
  in the pencil $\lvert W \rvert$.
  Let $C$ and $C'$ be the strict transforms of $\Gamma$ and $\Gamma'$ in $X$,
  respectively, let $x \in C$ be the strict transform of $\xi \in \Gamma$, and
  let $E$ be an exceptional divisor of the blowup $X \to \PP^2_k$.
  We claim that the divisor $L = aC+E$ on $X$ is ample for $a \ge 2$.
  First, note that since $(C \cdot E) = 1$, we have $(L^2) = 2a-1$ and $(L \cdot
  E) = a - 1$.
  If $Z$ is a curve on $X$ different from $E$, we then have
  \begin{equation}\label{eq:mirandanakai}
    (L \cdot Z) = (C \cdot Z) + (E \cdot Z) \ge 0
  \end{equation}
  since $C$ is basepoint-free and $(E \cdot Z) \ge 0$.
  By the Nakai--Moishezon criterion
  \cite[Thm.\ 1.2.23]{Laz04a}, to show that
  $L$ is ample, it suffices to show that equality cannot hold in
  \cref{eq:mirandanakai}.
  If equality holds, then $(C \cdot Z) = 0$, in which case $\pi(Z)$ is a point.
  On the other hand, since every curve in the pencil $\lvert W \rvert$ is
  irreducible, this implies $Z$ is a fiber of $\pi$, in which case $(E \cdot Z)
  > 0$, a contradiction.
  Thus, $L$ is ample.
  Finally, we note that
  \[
    \varepsilon(L;x) \le \frac{(L \cdot C)}{m} = \frac{1}{m} < \delta.
  \]
\end{example}
\begin{remark}
  As noted by Viehweg\index{Viehweg, Eckart}, Miranda's example
  can be used to
  construct varieties of any dimension with arbitrarily small Seshadri
  constants \cite[Ex.\ 3.2]{EL93a}.
  Letting $X$ be as constructed in Miranda's example \ref{ex:mirandaex}, for
  every $n \ge 2$, the $n$-dimensional smooth projective variety $X \times_k
  \PP^{n-2}_k$ satisfies
  \[
    \varepsilon\bigl(p_1^*L \otimes p_2^*\cO(1);(x,z)\bigr) \le
    \varepsilon(L;x)
  \]
  for every $z \in \PP^{n-2}_k$ by considering the curve $C \times_k \{z\}$,
  where $p_1,p_2$ are the first and second projection morphisms, respectively.
  \par Bauer\index{Bauer, Thomas} has also shown that Miranda's
  example is not as exceptional as it might appear: suitable
  blowups of any surface with Picard number one have
  arbitrarily small Seshadri constants \cite[Prop.\ 3.3]{Bau99}.%
  \index{Miranda, Rick!example of|)}
\end{remark}
Despite Miranda's example, Ein\index{Ein, Lawrence}, K\"uchle\index{Kuchle,
Oliver@K\"uchle, Oliver}, and Lazarsfeld\index{Lazarsfeld, Robert} were
able to prove that at very general points on complex projective
varieties, lower bounds for $\varepsilon(L;x)$ do exist.
\begin{citedthm}[{\cite[Thm.\ 1]{EKL95}}]\label{thm:ekl}
  Let $X$ be a complex projective variety of dimension $n$, and let $L$ be a big
  and nef divisor on $X$.
  Then, for all closed points $x \in X$ outside of a countable union of proper
  closed subvarieties in $X$, we have
  \[
    \varepsilon(L;x) \ge \frac{1}{n}.
  \]
  Moreover, for every $\delta > 0$, the locus
  \[
    \biggl\{ x \in X \biggm\vert \varepsilon(L;x) > \frac{1}{n+\delta} \biggr\}
  \]
  contains a Zariski-open dense set in $X(\CC)$.
\end{citedthm}
When $L$ is ample and $X$ is smooth of dimension $n \le 3$, the lower bound in
\cref{thm:ekl} can be improved to $\varepsilon(L;x) \ge 1/(n-1)$
\cites[Thm\adddot]{EL93a}[Thm.\ 1.2]{CN14}.
The case $n = 2$ supports the following strengthening of \cref{thm:ekl}.
\begin{citedconj}[{\cite[p.\ 194]{EKL95}}]
  Let $X$ be a projective variety over an algebraically closed field, and let
  $L$ be a big and nef divisor on $X$.
  Then, for all closed points $x \in X$ outside
  of a countable union of proper closed subvarieties of $X$, we have
  $\varepsilon(L;x) \ge 1$.
\end{citedconj}
By combining \cref{prop:demsesh,thm:ekl}, we obtain the following:
\begin{corollary}\label{cor:ekl}
  Let $X$ be a smooth complex projective variety of dimension $n$, and let $L$
  be a big and nef divisor on $X$.
  Then, the bundle $\omega_X \otimes L^{\otimes m}$ separates $\ell$-jets at
  all general points $x \in X$ for all $m \ge n(n+\ell)+1$.
  In particular, we have
  \[
    h^0(X,\omega_X \otimes L^{\otimes m}) \ge \binom{n+\ell}{n}
  \]
  for all $m \ge n(n+\ell)+1$.
\end{corollary}
\begin{remark}\label{rem:ekl}
  By replacing \cref{prop:demsesh} with \cref{thm:demseshsingular}, we see that
  \cref{cor:ekl} holds for $X$ with singularities of at worst dense
  $F$-injective type.
  See \cref{def:densefsings} for the definition of this
  class of singularities.
  In particular, \cref{cor:ekl} holds for $X$ with at worst rational
  singularities by \cref{fig:singsvsfsings}.%
  \index{Seshadri constant, $\varepsilon(D;x)$|)}
\end{remark}
\section{A relative Fujita-type conjecture}
We also mention the following relative version of Fujita's conjecture.
Inspired by Koll\'ar and Viehweg's work on weak positivity, which partially
answers an analogue of \cref{question:embedding} for \emph{families} of
varieties, Popa and Schnell proposed the following:
\begin{citedconj}[{\cite[Conj.\ 1.3]{PS14}}]\label{conj:popaschnell}
  Let $f\colon Y \to X$ be a morphism of smooth complex projective varieties,
  where $X$ is of dimension $n$, and let $L$ be an ample line bundle on $X$.
  Then, for every $k \ge 1$, the sheaf $f_*\omega_Y^{\otimes k} \otimes
  L^{\otimes m}$ is globally generated for all $m \ge k(n + 1)$.
\end{citedconj}
Note that if $f$ is the identity morphism $X \to X$, then
\cref{conj:popaschnell} is identical to Fujita's freeness conjecture
\ref{conj:fujita}\cref{conj:fujitafree}.
Popa and Schnell proved \cref{conj:popaschnell} when $\dim X = 1$ \cite[Prop.\
2.11]{PS14}, or when $L$ is additionally assumed to be globally generated
\cite[Thm.\ 1.4]{PS14}.
This latter result was shown using Castelnuovo--Mumford regularity in a similar
fashion to \cref{ex:fujitaamplegg}, with Ambro and Fujino's Koll\'ar-type
vanishing theorem replacing the Kodaira vanishing theorem in the proof.
\medskip
\par In joint work with Yajnaseni Dutta, we proved the following effective
global generation result in the spirit of \cref{conj:popaschnell}, which we
later extended to higher-order jets in joint work with Mihai Fulger.
Note that the case when $(Y,\Delta)$ is klt and $k = 1$ is due to de Cataldo
\cite[Thm.\ 2.2]{dC98}.
\begin{citedthm}[{\cites[Thm.\ A]{DM}[Cor.\ 8.2]{FM}}]\label{thm:dmfm}
  Let $f\colon Y \to X$ be a surjective morphism of complex projective
  varieties, where $X$ is of dimension $n$.
  Let $(Y,\Delta)$ be a log canonical pair and let $L$ be a big and nef line
  bundle on $X$.
  Consider a Cartier divisor $P$ on $Y$ such that $P \sim_\RR k(K_Y+\Delta)$ for
  some integer $k \ge 1$.
  Then, the sheaf
  \[
    f_*\cO_Y(P) \otimes L^{\otimes m}
  \]
  separates $\ell$-jets at all general points $x \in X$ for all $m \ge
  k(n(n+\ell)+1)$.
\end{citedthm}
The proof of \cref{thm:dmfm} is a relativization of the argument in
\cref{prop:demsesh}, and uses the lower bound on Seshadri constants in
\cref{thm:ekl}.
A generic global generation result in this direction was first obtained by Dutta
for klt $\QQ$-pairs $(Y,\Delta)$ \cite[Thm.\ A]{Dut}.
Using analytic techniques, Deng and Iwai later obtained improvements of Dutta's
original result for klt pairs with better lower bounds, under the additional
assumption that $X$ is smooth and $L$ is ample \cites[Thm.\ C]{Den}[Thm.\
1.5]{Iwa}.
In \cite[Thm.\ B]{DM}, we proved algebraic versions of Deng's and Iwai's
results as a consequence of a new weak positivity result for pairs
\cite[Thms.\ E and F]{DM}.
Note, however, that only our methods in \cites{DM}{FM} apply to log canonical
pairs.
\begin{remark}
  In positive characteristic, there is an example of a curve fibration over
  $\PP^1_k$ which gives a counterexample both to Popa and Schnell's
  relative Fujita-type conjecture \ref{conj:popaschnell}, and
  to the analogue of \cref{thm:dmfm} in positive characteristic.
  The example is based on a construction due to Moret-Bailly \cite{MB81};
  see \cite[Prop.\ 4.11]{SZ17}.
\end{remark}

\section{Difficulties in positive characteristic}\label{sect:poschardifficult}
While most of the questions, conjectures, and examples seen so far have been
stated over fields of arbitrary characteristic, the majority of the results
stated, in particular on Fujita's conjecture
(\cref{tab:knownfujita,thm:angehrnsiu}) and lower bounds on Seshadri constants
(\cref{thm:ekl}), are only known over fields of characteristic zero.
The most problematic situation is when the ground field $k$ is an imperfect
field of characteristic $p > 0$, in which case there are at least three major
difficulties.
First, since $k$ is of characteristic $p > 0$,
\begin{enumerate}[label=(\Roman*)]
  \item Resolutions of singularities are not known to exist (see \cite{Hau10}),
    and%
    \index{resolution of singularities!not known in characteristic $p$}
  \item Kodaira-type vanishing theorems are false \cite{Ray78} (see
    \cref{sect:raynaud}).%
    \index{vanishing theorem!Kodaira!false in characteristic $p$}
\end{enumerate}
A common workaround for the lack of resolutions is to use de
Jong's\index{de Jong, Aise Johan} theory of
alterations \cite{dJ96}.
The lack of vanishing theorems is harder to circumvent, however, since
over the complex numbers, vanishing theorems are a fundamental ingredient used
to construct global sections of line bundles.
A useful workaround is to exploit the Frobenius
morphism\index{Frobenius morphism, $F$} $F\colon X \to X$ and
its Grothendieck\index{Grothendieck, Alexander!trace}
trace\index{Frobenius morphism, $F$!Grothendieck trace of}
$F_*\omega_X^\bullet \to \omega_X^\bullet$; see
\cites{PST17}{Pat18}.
For imperfect fields, however, this approach runs into another problem:
\begin{enumerate}[label=(\Roman*),start=3]
  \item Applications of Frobenius techniques in algebraic geometry usually
    require the ground field $k$ to be
    \textsl{$F$-finite,}\index{F-finite@$F$-finite!field|textbf} i.e., satisfy
    $[k:k^p] < \infty$.
\end{enumerate}
The last issue arises since Grothendieck
duality\index{Grothendieck, Alexander!duality} cannot be applied to the
Frobenius if it is not finite.
Working around this last issue is the focus of \cref{app:gamma}.
\subsection{Proof of Theorem \ref{thm:poscharseshsm}}
\label{sect:demseshspecialcase}
To illustrate how Frobenius techniques can be used in practice, we prove
the following positive characteristic version of \cref{prop:demsesh}.
\begin{customthm}{B}\label{thm:poscharseshsm}
  Let\index{Seshadri constant, $\varepsilon(D;x)$!criterion for separation of jets|(}
  $X$ be a smooth projective variety of dimension $n$ over an algebraically
  closed field $k$ of characteristic $p > 0$, and let $L$ be an ample line
  bundle on $X$.
  Let $x \in X$ be a closed point, and consider an integer $\ell \ge 0$.
  If $\varepsilon(L;x) > n+\ell$, then $\omega_X \otimes L$ separates
  $\ell$-jets at $x$.%
  \index{Seshadri constant, $\varepsilon(D;x)$!criterion for separation of jets|)}
\end{customthm}
The case $\ell = 0$ is due to
Musta\c{t}\u{a}\index{Mustata, Mircea@Musta\c{t}\u{a}, Mircea} and
Schwede\index{Schwede, Karl} \cite[Thm.\ 3.1]{MS14}.
The case for arbitrary $\ell \ge 0$ first appeared in \cite[Thm.\ A]{Mur18}.
These proofs used a positive-characteristic version of Seshadri constants called
\textsl{Frobenius--Seshadri constants} $\varepsilon_F^\ell(L;x)$; see
\cref{rem:frobseshconst}.
Note that we will later prove a generalization of \cref{thm:poscharseshsm}; see
\cref{thm:demseshsingular}.
\par We give a new proof of \cref{thm:poscharseshsm}, which is an adaptation of
the proof in \cite[Exer.\ 6.3]{PST17}, which proves the case when $\ell = 0$.
As in the proof of \cite[Thm.\ 3.1]{MS14}, the main ingredient in the proof is
the \textsl{Grothendieck\index{Grothendieck, Alexander!trace}
trace}\index{Frobenius morphism, $F$!Grothendieck trace of}
\[
  \Tr_X\colon F_*\omega_X \longrightarrow \omega_X\glsadd{traceoffrobenius}
\]
associated to the (absolute) Frobenius morphism $F\colon X \to X$.
Recall that the Frobenius morphism is defined as the identity map on points, and
the $p$-power map
\[
  \begin{tikzcd}[row sep=0,column sep=1.475em]
    \cO_X(U) \rar & F_*\cO_X(U)\\
    f \rar[mapsto] & f^p
  \end{tikzcd}
\]
on functions over every open set $U \subseteq X$.
This map $\Tr_X$ is a morphism of $\cO_X$-modules, which can be obtained by
applying Grothendieck duality for finite flat
morphisms to the (absolute)
Frobenius\index{Frobenius morphism, $F$} morphism $F\colon X \to X$;
see \cref{sect:grothendieckduality}.
Note that $F$ is finite since $k$ is $F$-finite (see \cref{ex:eftoverffin}),
and is flat by Kunz's theorem \cite[Thm.\ 2.1]{Kun69} since $X$ is smooth.
By \cite[Lem.\ 1.3.6]{BK05},  we can also
describe the trace map locally by
\begin{equation}\label{eq:tracemaplocal}
  \prod_{i=1}^n x_i^{a_i}\,dx \longmapsto
  \prod_{i=1}^n x_i^{\frac{a_i-p+1}{p}}\,dx,
\end{equation}
where $x_1,x_2,\ldots,x_n \in \cO_X(U)$ is a choice of local coordinates on an
affine open subset $U \subseteq X$, and $dx \coloneqq dx_1 \wedge dx_2 \wedge
\cdots \wedge dx_n$.
By convention, the expression on the right-hand side of \cref{eq:tracemaplocal}
is zero unless all exponents are integers.
See \cite[\S1.3]{BK05} for the definition and basic properties of the morphism
$\Tr_X$ from this point of view, where it is also called the \textsl{Cartier
operator}\index{Cartier operator}.
\par The trace map $\Tr_X$ satisfies the following key properties needed for our
proof:
\begin{enumerate}[label=$(\alph*)$,ref=\ensuremath{\alph*}]
  \item\label{list:trace1}
    Since $X$ is smooth, the trace map $\Tr_X$ and its $e$th iterates $\Tr^e_X
    \colon F_*^e\omega_X \to \omega_X$ are surjective for every $e \ge 0$
    \cite[Thm.\ 1.3.4]{BK05}.
  \item If $\fa \subseteq \cO_X$ is a coherent ideal sheaf, then for every $e
    \ge 0$, the map $\Tr^e_X$ satisfies
    \begin{equation}\label{eq:tracemapideals}
      \Tr^e_X\bigl(F_*^e(\fa^{[p^e]}\cdot \omega_X)\bigr) = \fa
      \cdot \Tr^e_X(F^e_*\omega_X) = \fa \cdot \omega_X.
    \end{equation}
    Here, $\fa^{[p^e]}$ is the $e$th \textsl{Frobenius
    power}\index{Frobenius power, $\mathfrak{a}^{[p^e]}$} of $\fa$, which is the
    ideal sheaf
    locally generated by $p^e$th powers of local generators of $\fa$.
    Note that \cref{eq:tracemapideals} follows from \cref{list:trace1} by
    considering the $\cO_X$-module structure on $F_*^e\omega_X$.
\end{enumerate}
\par We need one more general result about Seshadri constants of ample divisors.
Note that this result shows that the definition of the Seshadri constant in
\cref{eq:seshdefample} matches that in \cref{def:seshconst}.
\begin{lemma}\label{lem:seshample}
  Let $X$ be a projective scheme over a field $k$, and let $D$ be an ample
  $\RR$-Cartier divisor on $X$.
  Consider a $k$-rational point $x \in X$, and let $\mu\colon \widetilde{X} \to
  X$ be the blowup of $X$ at $x$ with exceptional divisor $E$.
  For every $\delta \in (0,\varepsilon(D;x))$, the
  $\RR$-Cartier divisor $\mu^*D-\delta E$ is ample.
\end{lemma}
\begin{proof}
  Let $V \subseteq \widetilde{X}$ be a subvariety.
  If $V \not\subseteq E$, then $V$ is the strict transform of a subvariety $V_0
  \subseteq X$, and
  \begin{align*}
    \Bigl(\bigl(\mu^*D-\delta E\bigr)^{\dim V} \cdot V\Bigr) &= \bigl(D^{\dim V}
    \cdot V_0\bigr)
    - \delta\,e(\cO_{V_0,x}) > 0
    \intertext{by the assumption $\varepsilon(D;x) > \delta$ and \cite[Prop.\
    5.1.9]{Laz04a}.
    Otherwise, if $V \subseteq E$, then}
    \Bigl(\bigl(\mu^*D-\delta E\bigr)^{\dim V} \cdot V\Bigr) &=
    \Bigl(\bigl(-\delta E\rvert_E\bigr)^{\dim V} \cdot V\Bigr) > 0
  \end{align*}
  since $\cO_E(-E\rvert_E) \simeq \cO_{\PP^{n-1}}(1)$ is very ample.
  Thus, the divisor $\mu^*D-\delta E$ is ample by the Nakai--Moishezon criterion
  \cite[Thm.\ 1.2.23]{Laz04a}.
\end{proof}
\par We can now prove \cref{thm:poscharseshsm}.
\begin{proof}[Proof of \cref{thm:poscharseshsm}]
  First, we claim that it suffices to show that the restriction morphism
  \[
    \varphi_e\colon H^0(X,\omega_X \otimes L^{\otimes p^e})
    \longrightarrow H^0\bigl(X,\omega_X \otimes L^{\otimes p^e} \otimes
    \cO_X/\fm_x^{\ell p^e+n(p^e-1)+1}\bigr)
  \]
  is surjective for some $e \ge 0$.
  By \cref{eq:tracemapideals}, the map $\Tr^e_X$ induces a morphism
  \begin{align}
    F_*^e\bigl((\fm_x^{\ell+1})^{[p^e]}\cdot\omega_X\bigr)
    &\longrightarrow \fm_x^{\ell+1} \cdot \omega_X.\nonumber
  \intertext{Twisting this morphism by $L$ and applying the projection formula
  yields a morphism}
    F_*^e\bigl((\fm_x^{\ell+1})^{[p^e]}\cdot\omega_X \otimes L^{\otimes p^e}\bigr)
    &\longrightarrow \fm_x^{\ell+1} \cdot \omega_X \otimes L.
    \label{eq:tracemapprojformula}
  \end{align}
  Here, we use the fact that $F^{*}L \simeq L^{\otimes p}$ since pulling back by
  the
  Frobenius morphism raises the transition functions defining $L$ to the $p$th
  power.
  Since the Frobenius morphism $F$ is affine, the pushforward functor $F^e_*$ is
  exact, hence we obtain the exactness of the left column in the following
  commutative diagram:
  \begin{equation}\label{eq:poscharseshsmdiag1}
    \begin{tikzcd}
      0\dar & 0\dar\\
      F_*^e\bigl((\fm_x^{\ell+1})^{[p^e]}\cdot \omega_X \otimes
      L^{\otimes p^e}\bigr) \rar\dar & \fm_x^{\ell+1} \cdot \omega_X \otimes
      L\dar\\
      F_*^e\bigl(\omega_X \otimes L^{\otimes p^e}\bigr) \rar[twoheadrightarrow]
      \dar & \omega_X \otimes L\dar\\
      F_*^e\bigl(\omega_X \otimes L^{\otimes p^e} \otimes
      \cO_X/(\fm_x^{\ell+1})^{[p^e]}\bigr) \rar[twoheadrightarrow] \dar
      & \omega_X \otimes L \otimes \cO_X/\fm_x^{\ell+1} \dar\\
      0 & 0
    \end{tikzcd}
  \end{equation}
  The top horizontal arrow is the map in \cref{eq:tracemapprojformula}; the
  middle horizontal arrow is obtained from $\Tr^e_X$ in a similar fashion by
  twisting by $L$ and by applying the projection formula, hence is surjective
  by \cref{list:trace1}.
  The surjectivity of the middle horizontal arrow also implies the bottom
  horizontal arrow is surjective by the snake lemma.
  Now by the pigeonhole principle (see \cite[Lem.\ 2.4$(a)$]{HH02} or
  \cref{lem:monomials}), we have the inclusion
  $\fm_x^{\ell p^e+n(p^e-1)+1} \subseteq (\fm_x^{\ell+1})^{[p^e]}$
  for every $e \ge 0$, which yields the following commutative diagram:
  \begin{equation}\label{eq:poscharseshsmdiag2}
    \mathclap{\begin{tikzcd}[ampersand replacement=\&]
      0\dar \& 0\dar\\
      \fm_x^{\ell p^e+n(p^e-1)+1}\cdot \omega_X \otimes
      L^{\otimes p^e} \rar[hook]\dar \&
      (\fm_x^{\ell+1})^{[p^e]}\cdot \omega_X \otimes
      L^{\otimes p^e}\dar\\
      \omega_X \otimes L^{\otimes p^e} \rar[equal] \dar \&
      \omega_X \otimes L^{\otimes p^e}\dar\\
      \omega_X \otimes L^{\otimes p^e} \otimes
      \cO_X/\fm_x^{\ell p^e+n(p^e-1)+1} \rar[twoheadrightarrow] \dar
      \& \omega_X \otimes L^{\otimes p^e} \otimes
      \cO_X/(\fm_x^{\ell+1})^{[p^e]} \dar\\
      0 \& 0
    \end{tikzcd}}
  \end{equation}
  By applying $F^e_*(-)$ to \cref{eq:poscharseshsmdiag2}, combining it with
  \cref{eq:poscharseshsmdiag1}, and taking
  global sections in the bottom half of both diagrams, we obtain the following
  commutative square:
  \[
    \begin{tikzcd}
      H^0(X,\omega_X \otimes L^{\otimes p^e}) \dar[swap]{\varphi_e}\rar
      & H^0(X,\omega_X \otimes L) \dar{\rho}\\
      H^0\bigl(X,\omega_X \otimes L^{\otimes p^e} \otimes
        \cO_X/\fm_x^{\ell p^e+n(p^e-1)+1}\bigr)
      \rar[twoheadrightarrow]{\psi}
      & H^0\bigl(X,\omega_X \otimes L \otimes \cO_X/\fm_x^{\ell+1}\bigr)
    \end{tikzcd}
  \]
  Note that $\psi$ is surjective because the kernel of the corresponding
  morphism of sheaves is a skyscraper sheaf supported at $x$.
  Now assuming that $\varphi_e$ is surjective, we see that the composition from
  the top left corner to the bottom right corner is surjective, hence the
  restriction morphism $\rho$ is necessarily surjective as well.
  \par We now show that $\varphi_e$ is surjective for some $e$.
  By the long exact sequence on sheaf cohomology, it suffices to show that
  \begin{align*}
    \MoveEqLeft[4]
    H^1\bigl(X,\fm_x^{\ell p^e+n(p^e-1)+1} \cdot \omega_X \otimes L^{\otimes
    p^e} \bigr)\\
    &\simeq H^1\Bigl(\widetilde{X},\mu^*\bigl(\omega_X \otimes L^{\otimes p^e}
    \bigr)\bigl(-\bigl(\ell p^e+n(p^e-1)+1\bigr)E\bigr)\Bigr)\\
    &\simeq H^1\Bigl(\widetilde{X},\omega_{\widetilde{X}} \otimes
    \bigl(\mu^*L\bigl(-(n+\ell)E\bigr)\bigr)^{\otimes p^e}\Bigr)
    = 0,
  \end{align*}
  where $\mu\colon \widetilde{X} \to X$ is the blowup of $X$ at $x$.
  The first isomorphism follows from the Leray spectral
  sequence and the quasi-isomorphism
  \[
    \fm_x^{\ell p^e+n(p^e-1)+1} \simeq
    \mathbf{R}\mu_*\cO_{\widetilde{X}}\bigl(-\bigl(\ell
    p^e+n(p^e-1)+1\bigr)E\bigr)
  \]
  from \cite[Lem.\ 4.3.16]{Laz04a}, and
  the second isomorphism follows from how the canonical bundle transforms under
  a blowup with a smooth center \cite[Exer.\ II.8.5$(b)$]{Har77}.
  The vanishing of the last group follows from Serre vanishing for $e$
  sufficiently large%
  \index{vanishing theorem!Serre} \cite[Prop.\ III.5.3]{Har77} since
  $\mu^*L-(n+\ell)E$ is ample by \cref{lem:seshample}.
\end{proof}
\begin{remark}
  The proof of \cref{thm:poscharseshsm} works under the weaker assumption that
  $X$ is regular and $k$ is $F$-finite.
  Moreover, by using the gamma construction (\cref{thm:gammaconstintro}) to
  reduce to the case when $k$ is $F$-finite, the proof of
  \cref{thm:poscharseshsm} yields a statement over arbitrary fields of
  characteristic $p > 0$.
  Since this more general version of \cref{thm:poscharseshsm} follows from
  \cref{thm:demseshsingular}, we have chosen to prove this weaker result
  for simplicity.
\end{remark}
\begin{remark}\label{rem:szpirolmvanishing}
  If $\dim X = 2$, then it suffices for $L$ to be big and nef instead of ample
  in \cref{thm:poscharseshsm}.
  To prove this, it suffices to replace Serre
  vanishing\index{vanishing theorem!Serre} with a vanishing theorem of
  Szpiro\index{Szpiro, Lucien} \cite[Prop.\ 2.1]{Szp79} and
  Lewin-M\'en\'egaux\index{Lewin-Menegaux, Renee@Lewin-M\'en\'egaux, Ren\'ee}%
  \index{vanishing theorem!Szpiro--Lewin-M\'en\'egaux|(}
  \cite[Prop.\ 2]{LM81}, which asserts that for a big and
  nef divisor $L$ on a smooth projective surface $X$, we have
  \[
    H^1\bigl(X,\cO_X(-mL)\bigr) = 0
  \]
  for $m$ sufficiently large.
  Fujita\index{Fujita, Takao} has shown that a similar vanishing theorem also
  holds for higher-dimensional projective varieties that are only assumed to be
  normal \cite[Thm.\ 7.5]{Fuj83}, although the positivity condition on $L$ is
  stronger.\index{vanishing theorem!Szpiro--Lewin-M\'en\'egaux!in dimensions $\ge 3$}
  Fujita's theorem cannot be used to prove \cref{thm:poscharseshsm} in higher
  dimensions for big and nef divisors $L$, however, since the required vanishing
  $H^{n-1}(X,\cO_X(-mL)) = 0$ does not hold in general, even as $m \to \infty$;
  see \cref{ex:raynaudfujita}.%
  \index{vanishing theorem!Szpiro--Lewin-M\'en\'egaux|)}
\end{remark}
\subsection{Raynaud's counterexample to Kodaira vanishing}\label{sect:raynaud}
To\index{Raynaud, Michel!counterexample to Kodaira vanishing|(}%
\index{vanishing theorem!Kodaira!false in characteristic $p$|(}
illustrate what goes wrong in positive characteristic, we give a version of
Raynaud's\index{Raynaud, Michel} original example showing that Kodaira vanishing
is false in positive characteristic, with some changes in
presentation following Mukai\index{Mukai, Shigeru} \cites{Muk79}{Muk13}.
See also \cites{Tak10}{Zhe17}.
Note that Mukai also constructs versions of Raynaud's example in higher
dimensions.
\begin{example}[Raynaud {\cites{Ray78}{Muk79}{Muk13}}]\label{ex:raynaud}
  Let $k$ be an algebraically closed field of characteristic $p > 0$.
  The construction proceeds in four steps.
  \begin{step}\label{step:raynaud1}
    Construction of a smooth projective curve $C$ over $k$ and a Cartier
    divisor $D$ on $C$ such that the morphism
    \begin{equation}\label{eq:raynaudeszfails}
      F^*\colon H^1\bigl(C,\cO_C(-D)\bigr) \longrightarrow
      H^1\bigl(C,\cO_C(-pD)\bigr)
    \end{equation}
    induced by the Frobenius morphism is not injective.
  \end{step}
  Let $h > 0$ be an integer, let $P$ be a polynomial of degree $h$ in one
  variable over $k$, and consider the plane curve
  \[
    C = \overline{\bigl\{ P(x^p) - x = y^{ph-1}\bigr\}} \subseteq \PP^2_k
  \]
  of degree $ph$, where $\PP^2_k$ has variables $x,y,z$, and $P(x^p) - x =
  y^{ph-1}$ is the equation defining $C$ on the open set $\{z \ne 0\}$.
  Note that $C$ has exactly one point $\infty$ along $\{z = 0\}$.
  By the Jacobian criterion\index{Jacobian criterion}
  \cite[Exer.\ I.5.8]{Har77}, since the homogeneous Jacobian
  $(-z^{ph-1},\ y^{ph-2}z,\ xz^{ph-2}-y^{ph-1})$
  associated to $C$ has full rank along $C$, we see that $C$ is smooth.
  \par We claim that the morphism \cref{eq:raynaudeszfails} is not injective for
  the divisor $D = h(ph-3)\cdot\infty$.
  By \cite[Lem.\ 12]{Tan72}, since the kernel of the morphism in
  \cref{eq:raynaudeszfails} can be described by
  \[
    \ker(F^*) \simeq \bigl\{ df \in \Omega_{K(C)} \bigm\vert f \in K(C)\
    \text{such that}\ (df) \ge pD\bigr\},
  \]
  it suffices to construct a rational function $f \in K(C)$ satisfying $(df) \ge
  pD$.
  Here, $(df)$ is the divisor of zeroes and poles of the differential form $df$.
  Consider the rational function $y \in K(C)$.
  By the relation $-dx = -y^{ph-2}dy$ on $C \smallsetminus \{\infty\}$, we see
  that $\Omega_C$ is generated by $dy$ over $C \smallsetminus \{\infty\}$, hence
  $dy$ has no poles or zeroes away from $\infty$.
  Since by \cite[Ex.\ V.1.5.1]{Har77}, we have
  \begin{equation}\label{eq:genusoftango}
    \deg \Omega_C = 2g(C) - 2 = ph(ph-3),
  \end{equation}
  we obtain $(dy) = ph(ph-3)\cdot\infty = pD$, as desired.
  We note that $C$ is an example of a \textsl{Tango
  curve}\index{Tango, Hiroshi, curve}.
  \begin{step}
    Construction of a projective bundle $\pi\colon\PP(E) \to C$ with two
    distinguished divisors $F$ and $G$ arising from sections of $\PP(E)$ and of
    $\PP(E^{(p)})$.
  \end{step}
  By identifying the sheaf cohomology groups in \cref{eq:raynaudeszfails} with
  $\Ext^1$ groups \cite[Prop.\ III.6.3]{Har77}, we obtain a short exact sequence
  \begin{equation}\label{eq:mukai2star}
    0 \longrightarrow \cO_C \longrightarrow E \longrightarrow \cO_C(D)
    \longrightarrow 0
  \end{equation}
  such that after pulling back via the Frobenius morphism $F\colon C \to C$ on
  $C$, the resulting short exact sequence
  \begin{equation}\label{eq:mukai3star}
    0 \longrightarrow \cO_C \longrightarrow E^{(p)} \longrightarrow \cO_C(pD)
    \longrightarrow 0
  \end{equation}
  splits.
  The projective bundles of one-dimensional quotients $\PP(E)$ and
  $\PP(E^{(p)})$ associated to $E$ and $E^{(p)}$ fit into the
  pullback diagram
  \[
    \begin{tikzcd}
      \PP(E) \arrow[dashed]{dr}{\varphi} \arrow[bend left=15]{drr}{F}
      \arrow[bend right=15]{ddr}[swap]{f}\\
      & \PP(E^{(p)}) \rar\dar & \PP(E)\dar{f}\\
      & C \rar{F} & C
    \end{tikzcd}
  \]
  where $\varphi\colon \PP(E) \to \PP(E^{(p)})$ is the relative
  Frobenius morphism for $\PP(E)$ over $C$.
  \begin{figure}[t]
    \centering
    \tikzexternalenable
    \begin{tikzpicture}
      \draw (-4,0) rectangle (-1,2);
      \draw (1,0) rectangle (4,2);
      \draw[thick] (-1.75,-1.75) -- (1.75,-1.75) node[anchor=west] {$C$};
      \node[anchor=south] at (-4,2) {$\PP(E)$};
      \node[anchor=south] at (1,2) {$\PP(E^{(p)})$};

      \draw[umlightblue,ultra thick] (-4,1.5) node[anchor=east] {\footnotesize $F$} -- (-1,1.5);
      \draw[umlightblue,ultra thick] (1,1.5) -- (4,1.5) node[anchor=west] {\footnotesize $F'$};

      \draw[umpoppy,ultra thick,decorate,decoration={zigzag,amplitude=1pt,segment length=1.5mm,pre=lineto,pre length=1pt,post=lineto,post length=1pt}] (-4,0.5) node[anchor=east] {\footnotesize $G$} -- (-1,0.5);
      \draw[umpoppy,ultra thick] (1,0.5) -- (4,0.5) node[anchor=west] {\footnotesize $G'$};

      \path[->] (-0.75,1) edge node[above] {\scriptsize $\varphi$} (0.75,1);
      \path[->] (-2,-0.5) edge node[anchor=north east] {\scriptsize $f$} (-1.2,-1.25);
      \path[->] (2,-0.5) edge (1.2,-1.25);
    \end{tikzpicture}
    \tikzexternaldisable
    \caption{Raynaud's example (Example \ref{ex:raynaud})}
    {\footnotesize Illustration from \cite[Fig.\ on p.\ 18]{Muk79}}
    \label{fig:raynaudprojbundle}
  \end{figure}
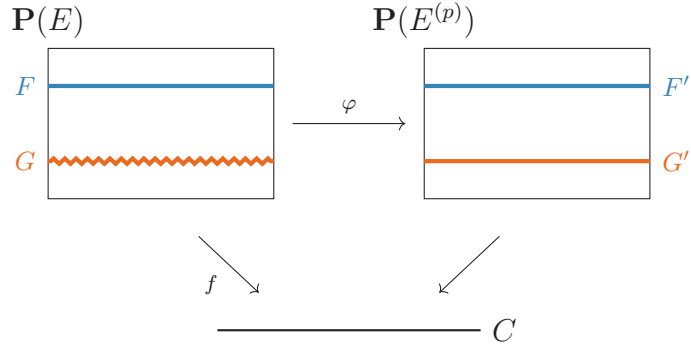
  \par We now note that $f\colon \PP(E) \to C$ has a section $C \to
  \PP(E)$ with image $F \simeq C$, which corresponds to the surjection in
  \cref{eq:mukai2star}.
  The image $F' = \varphi(F)$ of $F$ gives the section of $\PP(E^{(p)}) \to
  C$ corresponding to the surjection in \cref{eq:mukai3star}, and the fact that
  \cref{eq:mukai3star} splits implies $\PP(E^{(p)})$
  also has another section $C \to \PP(E^{(p)})$ with image $G' \simeq C$
  such that $F' \cap G' = \emptyset$ \cite[Exer.\ V.2.2]{Har77}.
  We denote by $G \coloneqq \varphi^{-1}(G')$ the scheme-theoretic inverse of
  $G'$, which is a smooth variety by \cite[Prop.\ 1.7]{Muk13}.\footnote{The
  fact that \cref{eq:mukai2star} does not split but \cref{eq:mukai3star}
  does split is used here.
  We mention that \cite[Prop.\ 1.7]{Muk13} is proved for a higher-dimensional
  generalization of our example.
  See \cite[Thm.\ 3]{Tak10} for a simpler statement that suffices for our
  purposes.}
  Note that $F \cap G = \emptyset$ since $F' \cap G' = \emptyset$; see
  \cref{fig:raynaudprojbundle}.
  By \cite[Prop.\ V.2.6]{Har77}, we have the linear equivalences $0 \sim
  \xi - F$ and $f^*(pD) \sim p\xi - G$ on $\PP(E)$, where $\xi$ is the divisor
  class associated to $\cO_{\PP(E)}(1)$, hence
  \begin{equation}\label{eq:mukai7}
    G - pF \sim -f^*(pD).
  \end{equation}
  \begin{step}
    Construction of a cyclic cover $\pi\colon X \to \PP(E)$ where $X$ is
    a smooth surface and a suitable ample divisor $\widetilde{D}$ on $X$.
  \end{step}
  Let $r \ge 2$ be an integer such that $r\mid p + 1$ and $r \mid h(ph-3)$ for
  some choice of integer $h > 0$ in \cref{step:raynaud1}.
  For example, if $p \ne 2$ then we can set $r = 2$ for arbitrary $h
  > 0$; if $p = 2$, then we can set $r = 3$ for $h > 0$ such that $3
  \mid h$.
  By adding $(p+1)F$ to \eqref{eq:mukai7}, we have $G + F \sim (p+1)F - f^*(pD)
  \sim rM$, where
  \[
    M \coloneqq \frac{p+1}{r}F - f^*\biggl( \frac{ph(ph-3)}{r}\cdot\infty
    \biggr),
  \]
  since $D = h(ph-3) \cdot \infty$.
  We can therefore construct a degree $r$ cyclic cover
  \[
    X \coloneqq \SSpec_{\PP(E)}\Bigl(\cO_{\PP(E)} \oplus \cO_{\PP(E)}(-M) \oplus
    \cdots \oplus \cO_{\PP(E)}\bigl(-(r-1)M\bigr)\Bigr)
    \overset{\pi}{\longrightarrow} \PP(E)
  \]
  branched along $G+F$, which is smooth by the fact that both
  $\PP(E)$ and $G+F$ are smooth \cite[Prop.\ 4.1.6]{Laz04a}.
  We then set
  \[
    \widetilde{D} = \widetilde{F} + (f \circ \pi)^*\biggl(
    \frac{h(ph-3)}{r}\cdot\infty \biggr),
  \]
  where $\widetilde{F} = \pi^{-1}(F)_\red$ is the inverse image of $F$ with
  reduced scheme structure.
  To show that $\widetilde{D}$ is ample, we first note that $F +
  f^*(h(ph-3)\cdot\infty)$ is ample by the Nakai--Moishezon
  criterion \cite[Thm.\ 1.2.23]{Laz04a}
  since it intersects both the
  section $F$ and the fibers of the ruled surface $f\colon \PP(E) \to C$
  positively.
  The pullback $r\widetilde{D} \sim \pi^*(rF) + \pi^*f^*(h(ph-3)\cdot\infty)$
  is therefore also ample by \cite[Prop.\ 4.1.6]{Laz04a}, hence $\widetilde{D}$
  is ample.
  We note that $X$ is an example of a \textsl{Raynaud 
  surface}\index{Raynaud, Michel!surface}.
  \begin{step}
    Proof that $H^1(X,\cO_X(-\widetilde{D})) \ne 0$.
  \end{step}
  First, we note that
  \begin{align*}
    H^1\bigl(X,\cO_X(-\widetilde{D})\bigr) &\simeq
    H^1\bigl(\PP(E),\pi_*\cO_X(-\widetilde{D})\bigr)\\
    &\simeq H^1\biggl(\PP(E),\cO_{\PP(E)}(-F) \oplus \bigoplus_{i=1}^{r-1}
    \cO_{\PP(E)}(-iM)\biggr).
  \end{align*}
  The first isomorphism holds by the fact that $\pi$ is finite.
  The second isomorphism holds by properties of cyclic covers; see \cite[Prop.\
  6.3.4]{Zhe16} for a proof using local coordinates, or see \cite[Prop.\
  3.3]{Zhe17} for a shorter proof.
  Now consider the Leray spectral sequence
  \begin{equation}\label{eq:raynaudlearay}
    \begin{aligned}
      E^{p,q}_2 \MoveEqLeft[6]= H^p\biggl(C,R^qf_*\biggl(\cO_{\PP(E)}(-F)
      \oplus \bigoplus_{i=1}^{r-1} \cO_{\PP(E)}(-iM)\biggr)\biggr)\\
      &\Rightarrow H^{p+q}\biggl(\PP(E),\cO_{\PP(E)}(-F) \oplus
      \bigoplus_{i=1}^{r-1} \cO_{\PP(E)}(-iM)\biggr)
    \end{aligned}
  \end{equation}
  which already degenerates on the $E_2$ page.
  We have that $E^{1,0}_2 = 0$, since the pushforward
  \begin{align*}
    \MoveEqLeft[6]f_*\biggl(\cO_{\PP(E)}(-F) \oplus \bigoplus_{i=1}^{r-1}
    \cO_{\PP(E)}(-iM)\biggr)\\
    &\simeq f_*\cO_{\PP(E)}(-1)
    \oplus \bigoplus_{i=1}^{r-1}
    f_*\cO_{\PP(E)}\biggl(-\frac{i(p+1)}{r}\biggr) \otimes
    \cO_C\biggl(\frac{iph(ph-3)}{r}\cdot\infty\biggr)
  \end{align*}
  is zero by the fact that $f_*\cO_{\PP(E)}(-n) = 0$ for $n > 0$.
  Thus, the Leray spectral sequence
  \cref{eq:raynaudlearay} implies that
  \begin{equation}\label{eq:raynaudkeyiso}
    H^1\bigl(X,\cO_X(-\widetilde{D})\bigr) \simeq
    H^0\biggl(C,R^1f_*\biggl(\cO_{\PP(E)}(-F) \oplus
    \bigoplus_{i=1}^{r-1} \cO_{\PP(E)}(-iM)\biggr)\biggr).
  \end{equation}
  Since $R^1f_*(\cO_{\PP(E)}(-F)) \simeq R^1f_*(\cO_{\PP(E)}(-1)) = 0$,
  we will consider the summands containing
  $\cO_{\PP(E)}(-iM)$.
  By \cite[Exer.\ II.8.4$(c)$]{Har77}, we have
  \begin{align*}
    R^1f_*\biggl(\cO_{\PP(E)}\biggl(-\dfrac{i(p+1)}{r}\biggr)\biggr)
    &\simeq \biggl(f_*\cO_{\PP(E)}\biggl(\dfrac{i(p+1)-2r}{r}\biggr)\biggr)^\vee
    \simeq \bigl(\Sym^{\frac{i(p+1)-2r}{r}}E\bigr)^\vee,
    \intertext{hence the projection formula implies}
    R^1f_*\bigl(\cO_{\PP(E)}(-iM)\bigr) &\simeq
    \bigl(\Sym^{\frac{i(p+1)-2r}{r}}E\bigr)^\vee \otimes
    \cO_C\biggl(\frac{iph(ph-3)}{r}\cdot\infty\biggr).
  \end{align*}
  The short exact sequence \cref{eq:mukai2star} implies that there is a
  surjection $\Sym^{\frac{i(p+1)-2r}{r}}E \twoheadrightarrow
  \cO_C\bigl(\frac{i(p+1)-2r}{r}D\bigr)$, hence there is an injection
  \begin{align*}
    \cO_C\biggl(\frac{(2r-i)h(ph-3)}{r} \cdot \infty \biggr)
    &\longhookrightarrow R^1f_*\bigl(\cO_{\PP(E)}(-iM)\bigr).
    \intertext{We therefore have}
    H^0\biggl(C,\cO_C\biggl(\frac{(2r-i)h(ph-3)}{2} \cdot \infty
    \biggr)\biggr) &\longhookrightarrow
    H^0\bigl(C,R^1f_*\bigl(\cO_{\PP(E)}(-iM)\bigr)\bigr),
  \end{align*}
  where the left-hand side is nonzero as long as $2r-i \ge 0$.
  By the assumption $r \ge 2$, the left-hand side is nonzero for $i
  =1$, hence \cref{eq:raynaudkeyiso} implies $H^1(X,\cO_X(-\widetilde{D})) \ne
  0$.
\end{example}
\index{vanishing theorem!Kodaira!false in characteristic $p$|)}%
As noted by Fujita\index{Fujita, Takao}, Raynaud's example \ref{ex:raynaud}
also gives counterexamples to the vanishing theorem in
\cref{rem:szpirolmvanishing} for smooth projective varieties of dimension $3$.
\begin{citedex}[{\cite[(7.10)]{Fuj83}}]\label{ex:raynaudfujita}
  Let\index{vanishing theorem!Szpiro--Lewin-M\'en\'egaux!in dimensions $\ge 3$|(}
  $X$ and $\widetilde{D}$ be as constructed in \cref{ex:raynaud}, and
  consider the $\PP^1$-bundle
  \[
    Y \coloneqq \PP\bigl(\cO_X(\widetilde{D}) \oplus \cO_X\bigr)
    \overset{\pi}{\longrightarrow} X
  \]
  over $X$.
  Note that $\cO_Y(1)$ is big and nef by \cite[Lems.\ 2.3.2$(iii)$ and
  2.3.2$(iv)$]{Laz04a}.
  We claim that $H^2(Y,\cO_Y(-m)) \ne 0$ for all $m \ge 2$.
  We have
  \begin{align*}
    H^2\bigl(Y,\cO_Y(-m)\bigr)^\vee &\simeq
    H^1\bigl(Y,\omega_Y \otimes \cO_Y(m)\bigr)\\
    &\simeq H^1\bigl(Y,\cO_Y(m-2) \otimes \pi^*\bigl(\omega_X \otimes
    \cO_X(\widetilde{D})\bigr)\bigr)
  \end{align*}
  by Serre duality \cite[Cor.\ 7.7]{Har77} and by \cite[Exer.\
  II.8.4$(b)$]{Har77}, respectively.
  By the projection formula, we therefore have
  \[
    H^2\bigl(Y,\cO_Y(-m)\bigr)^\vee
    \simeq H^1\bigl(X,\Sym^{m-2}\bigl(\cO_X(\widetilde{D}) \oplus \cO_X\bigr)
    \otimes \omega_X \otimes \cO_X(\widetilde{D})\bigr).
  \]
  The right-hand side contains $H^1(X,\omega_X \otimes \cO_X(\widetilde{D}))
  \simeq H^1(X,\cO_X(-\widetilde{D}))^\vee$ as a direct summand
  for all $m \ge 2$.
  Since $H^1(X,\cO_X(-\widetilde{D})) \ne 0$ by the construction in
  \cref{ex:raynaud}, we see that $H^2(Y,\cO_Y(-m)) \ne 0$ for all $m \ge
  2$.\index{vanishing theorem!Szpiro--Lewin-M\'en\'egaux!in dimensions $\ge 3$|)}
\end{citedex}%
\index{Raynaud, Michel!counterexample to Kodaira vanishing|)}

\chapter{Characterizations of projective~space}\label{s:charpn}
In this chapter, we describe how Seshadri constants can be used to study the
following:
\begin{customquest}{\ref{question:characterization}}
  How can we identify when a given projective variety is projective space?
\end{customquest}
\par Recall that a smooth projective variety $X$ of dimension $n$ is
\textsl{Fano}\index{Fano variety|textbf} if its anti-canonical bundle
$\omega_X^{-1} \coloneqq \bigwedge^nT_X$%
\index{anti-canonical bundle, $\omega_X^{-1}$|textbf}
is ample.
In this chapter, we prove the following characterization of projective space
amongst Fano varieties using Seshadri constants.
Note that the lower bound $\deg(\omega_X^{-1}\rvert_C) \ge
e(\cO_{C,x})\cdot(n+1)$ below is equivalent to $\varepsilon(\omega_X^{-1};x) \ge
n+1$; see the statement of \cref{thm:altcharpn}.
\begin{customthm}{A}\label{thm:charpn}
  Let\index{projective space, $\mathbf{P}^n$!characterization via Seshadri constants|textbf}
  $X$ be a Fano variety of dimension $n$ over an
  algebraically closed field $k$ of positive characteristic.
  If there exists a closed point $x \in X$ with
  \[
    \deg\bigl(\omega_X^{-1}\rvert_C\bigr) \ge e(\cO_{C,x}) \cdot(n+1)
  \]
  for every integral curve $C \subseteq X$ passing
  through $x$, then $X$ is isomorphic to the $n$-dimensional projective space
  $\PP^n_k$.
\end{customthm}
This result is originally due to Bauer\index{Bauer, Thomas} and
Szemberg\index{Szemberg, Tomasz} in characteristic zero \cite[Thm.\ 1.7]{BS09}.
The material in this chapter is from \cite[\S4]{Mur18}.
\section{Background}\label{s:charpncomparison}
We start by motivating the statement of \cref{thm:charpn}.
Our story begins with the following observation about projective space.
\begin{principle}
  The\index{tangent bundle, $T_X$!of $\PP^n$|(} $n$-dimensional projective space $\PP^n_k$ over a field $k$ has a
  ``positive'' tangent bundle $T_X$.
  For example, we have the following:
  \begin{enumerate}[label=$(\arabic*)$,ref=\arabic*]
    \item $T_X$ is an ample vector bundle \emph{\cite[Prop.\
      6.3.1$(i)$]{Laz04b}}.\label{moricond}
    \item There exists an ample line bundle $H$ on $\PP^n_k$ such that
      $\bigwedge^n T_X \simeq H^{\otimes(n+1)}$.\label{kocond}
    \item $\bigwedge^n T_X = \omega_X^{-1}$ is ample (i.e., $X$ is
      Fano\index{Fano variety}) and
      $\deg(\omega_X^{-1}\rvert_C) \ge n+1$ for all integral curves
      $C \subseteq X$.\label{lengthcond}%
      \index{tangent bundle, $T_X$!of $\PP^n$|)}
  \end{enumerate}
\end{principle}
\noindent Note that \cref{kocond,lengthcond} hold since $\omega_{\PP^n_k}^{-1} =
\cO_{\PP^n_k}(n+1)$\index{anti-canonical bundle, $\omega_X^{-1}$!of Pn@of $\PP^n$}
\cite[Ex.\ II.8.20.2]{Har77}.
We recall that a vector bundle $E$ on $X$ is \textsl{ample} if $\cO_{\PP(E)}(1)$
is ample \cite[Def.\ 6.1.1]{Laz04b}.\index{ample vector bundle|textbf}
\par These properties seem very special, and lead us to ask the following more
specific version of \cref{question:characterization}.
\begin{question}
  Let $X$ be a smooth projective variety of dimension $n$ over an algebraically
  closed field $k$.
  If $X$ satisfies one of \cref{moricond}--\cref{lengthcond}, is $X$ isomorphic
  to $\PP^n_k$?
\end{question}
Many results in this direction are known.
The first result, due to Mori\index{Mori, Shigefumi|(}, is in some sense
the birthplace of modern birational geometry and the minimal model program.
\begin{citedthm}[{\cite[Thm.\ 8]{Mor79}}]\label{thm:morichar}
  Let\index{projective space, $\mathbf{P}^n$!Mori's characterization of}
  $X$ be a smooth projective variety of dimension $n$ over an algebraically
  closed field $k$.
  If \cref{moricond} holds, then $X$ is isomorphic to $\PP^n_k$.
\end{citedthm}
The sufficiency of \cref{moricond} was first conjectured by
Frankel\index{Frankel, Theodore} \cite[Conj.]{Fra61} in the analytic context,
and by Hartshorne\index{Hartshorne, Robin} \cite[Prob.\ III.2.3]{Har70} in the
algebraic context.
The idea of Mori's proof is to produce many copies of $\PP^1_k$ inside $X$
passing through a point $x_0 \in X$ using \textsl{bend and
break}\index{bend and break} techniques.
\begin{figure}[t]
  \centering
  \tikzexternalenable
  \begin{tikzpicture}[scale=1]
    \begin{scope}[rotate around={-12:(0,3.75)}]
      \path[thick,umpoppy] (1,4) edge[bend right] (0,3.75) (0,3.75) edge[bend left] (-1,3.5);
    \end{scope}
    \path[thick,umpoppy] (1,5.5) edge[bend left=10] (0,5.25) (0,5.25) edge[bend right=20] (-0.75,5);
    \path[thick,umpoppy] (-1,6) edge[bend right] (0,5.75) (0,5.75) edge[bend left] (0.75,5.6);
    \path[thick,umpoppy] (-1.1,3.25) edge[bend left] (0,3) (0,3) edge[bend right] (1.1,2.875);

    \draw[ultra thick,umlightblue] (0,{tan(1.4 r)/3+4.5}) node[anchor=south] {\footnotesize $E$} -- (0,{tan(-1.4 r)/3+4.5});
    \draw[domain=-1.395:1.405,variable=\t,smooth,samples=200] plot ({\t},{tan(\t r)/3+4.5}) -- plot[domain=-1.395:1.405] ({\t},{-tan(\t r)/3+4.5}) -- cycle;

    \draw[ultra thick] (5,6.25) -- (5,2.75);
    \node[anchor=west] at (5.125,3) {$\mathbf{P}^{n-1}_k$};

    \draw[->] (0,2.25) -- node[anchor=east]{\scriptsize $\mu$} (0,1.25);
    \draw[->] (2,4.5) -- node[anchor=south]{\scriptsize $\mathbf{P}^1_k$-bundle} (4,4.5);

    \draw[rotate=5] (0,0) ellipse (2cm and 1cm);

    \node at (-2.75,3) {$\widetilde{X} = \operatorname{Bl}_{x_0}X$};
    \node at (-2.5,-0.25) {$X$};

    \path[thick,umpoppy] (1,0.5) edge[bend right] (0,0) (0,0) edge[bend left] (-1,-0.5)
      (-0.75,0.25) edge[bend left=20] (0,0) (0,0) edge[bend right=30] (0.75,-0.5)
      (-0.25,0.8) edge[bend right=15] (0,0) (0,0) edge[bend left=10] (0.25,-0.7);

    \fill (0,0) circle (2pt) node[anchor=west,yshift=-0.5pt] {\footnotesize $x_0$};
  \end{tikzpicture}
  \tikzexternaldisable
  \caption{Mori's characterization of $\PP^n_k$}
  \label{fig:morichar}
\end{figure}
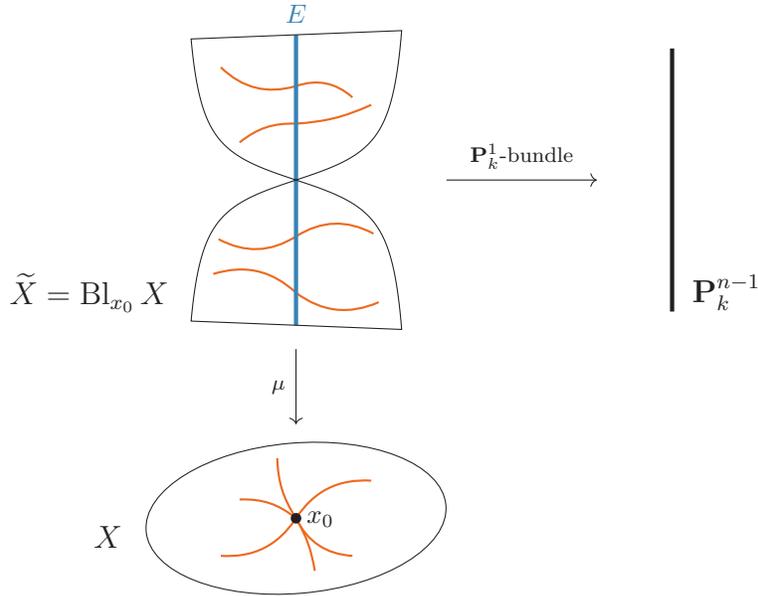
Letting $\mu\colon \widetilde{X} \to X$ be the blowup of $X$ at $x_0$, Mori
shows that $\widetilde{X}$ has the structure of a $\PP^1_k$-bundle over
$\PP^{n-1}_k$; see \cref{fig:morichar} for an illustration.
This $\PP^1_k$-bundle structure for $\widetilde{X}$ forces $X \simeq \PP^n_k$
\cite[Lem.\ V.3.7.8]{Kol96}.
An interesting feature of Mori's bend and break techniques is that in
characteristic zero, Mori's techniques require reducing modulo
$p$\index{reduction modulo $p$} and utilizing the Frobenius
morphism\index{Frobenius morphism, $F$}.
It is unknown whether one can prove \cref{thm:morichar} directly, without
reducing modulo $p$.\index{Mori, Shigefumi|)}
\medskip
\par The next result was actually known before Mori's theorem \ref{thm:morichar}.
The analogous result in positive characteristic, however, took much longer.
\begin{citedthm}[{\cites[Cor.\ to Thm.\ 1.1]{KO73}[Cor.\ 2]{KK00}}]
  \label{thm:kochar}
  Let\index{projective space, $\mathbf{P}^n$!Kobayashi and Ochiai's characterization of}
  $X$ be a smooth projective variety of dimension $n$ over an algebraically
  closed field $k$.
  If \cref{kocond} holds, then $X$ is isomorphic to $\PP^n_k$.
\end{citedthm}
This result is due to Kobayashi\index{Kobayashi, Shoshichi} and
Ochiai\index{Ochiai, Takushiro} in characteristic zero \cite{KO73}, and
to Kachi\index{Kachi, Yasuyuki} and
Koll\'ar\index{Kollar, Janos@Koll\'ar, J\'anos} in positive characteristic
\cite{KK00}.
The methods of \cite{KO73} are topological and complex analytic in nature,
while \cite{KK00} uses
Mori's bend and break\index{bend and break} techniques.
\cref{thm:kochar} illustrates the general philosophy that methods from
topology and complex analysis can often be replaced by Frobenius techniques in
positive characteristic.
\medskip
\par Finally, we consider the following:\index{Mori--Mukai conjecture|(}
\begin{conjecture}[{Mori\index{Mori, Shigefumi}--Mukai\index{Mukai, Shigeru}
  \cite[Conj.\ V.1.7]{Kol96}}]
  \label{conj:morimukai}\index{Mori--Mukai conjecture|textbf}
  Let $X$ be a smooth projective variety of dimension $n$ over an algebraically
  closed field $k$.
  If \cref{lengthcond} holds, then $X$ is isomorphic to $\PP^n_k$.
\end{conjecture}
By using results of Kebekus\index{Kebekus, Stefan} \cite{Keb02} on families of
singular rational curves, Cho\index{Cho, Koji},
Miyaoka\index{Miyaoka, Yoichi}, and
Shepherd-Barron\index{Shepherd-Barron, N. I.} proved this conjecture in
characteristic zero \cite{CMSB02}.
More precisely, they showed the following statement, which is stronger than the
Mori--Mukai conjecture \ref{conj:morimukai} since Fano varieties are
uniruled \cite[Cor.\ IV.1.15]{Kol96}\index{Fano variety!is uniruled}.
\begin{citedthm}[{\cite[Cor.\ $0.4(11)$]{CMSB02}}]\label{thm:cmsb02}
  Let\index{Mori--Mukai conjecture!is true in characteristic zero}
  $X$ be a smooth projective variety of dimension $n$ over an
  algebraically closed field $k$ of characteristic zero.
  If $X$ is uniruled, and the inequality
  \[
    \deg\bigl(\omega_X^{-1}\rvert_C\bigr) \ge n+1
  \]
  holds for every rational curve $C \subseteq X$ passing through a general
  closed point $x_0 \in X$, then $X$ is isomorphic to $\PP^n_k$.
\end{citedthm}
Because of the assumption on the characteristic, we ask the following:
\begin{question}
  Is the Mori--Mukai conjecture \ref{conj:morimukai} true in positive
  characteristic?
\end{question}
\par In arbitrary characteristic, as far as we know the only result in this
direction is the following result due to Kachi\index{Kachi, Yasuyuki} and
Koll\'ar\index{Kollar, Janos@Koll\'ar, J\'anos}, which we state using the
language of divisors.
\begin{citedthm}[{\cite[Cor.\ 3]{KK00}}]\label{thm:kkcharpn}
  Let $X$ be a smooth projective variety of dimension $n$ over an
  algebraically closed field $k$.
  Suppose $K_X$ is not nef.
  If
  \begin{enumerate}[label=$(\alph*)$,ref=\ensuremath{\alph*}]
    \item\label{list:kkcond1} $(-K_X \cdot C) \ge n+1$ for every
      rational curve $C \subseteq X$; and
    \item\label{list:kkcond2} $(-K_X)^n \ge (n+1)^n$,
  \end{enumerate}
  then $X$ is isomorphic to $\PP^n_k$.
\end{citedthm}
The major issue in trying to mimic the proof in \cite{CMSB02} is the use of
deformation theory and, more specifically, the use of generic
smoothness\index{generic smoothness} in studying deformations of curves.
Note that generic smoothness is false in positive characteristic, since
the absolute Frobenius morphism\index{Frobenius morphism, $F$!is nowhere smooth}
$F\colon X \to X$ for a smooth variety $X$ is nowhere smooth \cite[Ex.\
III.10.5.1]{Har77}.
One way to interpret \cref{thm:charpn} is that we avoid issues with
generic smoothness and deformation theory by building singularities into the
statement of the Mori--Mukai conjecture \ref{conj:morimukai}.
The advantage of this modification is that \cref{thm:charpn} can be interpreted
in terms of Seshadri constants\index{Seshadri constants} in the following
manner.
\begin{customthm}{\ref*{thm:charpn}*}\label{thm:altcharpn}
  Let $X$ be a Fano variety of dimension $n$ over an
  algebraically closed field $k$ of positive characteristic. If there exists a
  closed point $x \in X$ with
  $\varepsilon(\omega_X^{-1};x) \ge (n+1)$,
  then $X$ is isomorphic to the $n$-dimensional projective space $\PP^n_k$.
\end{customthm}
\noindent Note that the conditions in \cref{thm:charpn} and in
\cref{thm:altcharpn} are equivalent by \cref{eq:seshinfcurves}.
\par Using this reinterpretation, we can show that \cref{thm:charpn} is a
consequence of \cref{thm:cmsb02} in characteristic zero, and a version of
\cref{thm:charpn} assuming a slightly stronger lower bound on
$\varepsilon(\omega_X^{-1};x)$ at \emph{all} points $x \in X$ is a consequence
of \cref{thm:kkcharpn}.
The statement in characteristic zero gives a different proof of \cite[Thm.\
1.7]{BS09}.
\begin{proposition}\label{prop:charpnallpts}
  Let $X$ be a Fano variety of dimension $n$ over an algebraically
  closed field $k$.
  Suppose one of the following is satisfied:
  \begin{enumerate}[label=$(\roman*)$,ref=\ensuremath{\roman*}]
    \item\label{prop:charpnallptschar0}
      We have $\Char k = 0$ and $\varepsilon(\omega_X^{-1};x) > n$ holds for a
      single closed point $x \in X$; or
    \item\label{prop:charpnallptscharp}
      We have $\Char k = p > 0$ and $\varepsilon(\omega_X^{-1};x) \ge n+1$ holds
      for \emph{all} closed points $x \in X$.
  \end{enumerate}
  Then, $X$ is isomorphic to $\PP^n_k$.
\end{proposition}
\begin{proof}
  For \cref{prop:charpnallptschar0}, we use \cref{thm:cmsb02}.
  It suffices to verify the condition $\deg(\omega_X^{-1}\rvert_C) > n$.
  First, the locus $\{x \in X \mid
  \varepsilon(\omega_X^{-1};x) > n\}$
  contains a Zariski open set \cite[Lem.\ 1.4]{EKL95}, hence we have
  $\varepsilon(\omega_X^{-1};x_0) > n$ at a general point $x_0 \in X$.
  By the alternative characterization of Seshadri constants in terms of curves
  in \eqref{eq:seshinfcurves}, we have the chain of inequalities
  \[
    n < \varepsilon(\omega_X^{-1};x_0) \le
    \frac{\deg(\omega_X^{-1}\rvert_C)}{e(\cO_{C,x_0})}
    \le \deg(\omega_X^{-1}\rvert_C)
  \]
  for any rational curve $C$ containing $x_0$.
  \par For \cref{prop:charpnallptscharp}, we use Theorem \ref{thm:kkcharpn}.
  The verification of condition \cref{list:kkcond1}
  proceeds as in \cref{prop:charpnallptschar0} by applying
  \eqref{eq:seshinfcurves} to a closed point
  $x \in C$ contained in a given rational curve $C \subseteq X$.
  For condition $(\ref{list:kkcond2})$, we use the inequality
  $\varepsilon(\omega_X^{-1};x) \le \sqrt[n]{(-K_X)^n}$,
  which is \cite[eq.\ 5.2]{Laz04a}.
  The inequality $\varepsilon(\omega_X^{-1};x) \ge n+1$ then implies condition
  \cref{list:kkcond2}.
\end{proof}
\par Given the similarity between the Mori--Mukai conjecture
\ref{conj:morimukai} and \cref{thm:charpn}, we ask the following:
\begin{question}\label{quest:wouldprovecmsb}
  Let $X$ be a Fano variety of dimension $n$ over an
  algebraically closed field $k$.
  If the inequality
  \[
    \deg\bigl(\omega_X^{-1}\rvert_C\bigr) \ge n+1
  \]
  holds for every curve $C \subseteq X$, then does there exist a closed
  point $x \in X$ with
  \[
    \deg(\omega_X^{-1}\rvert_C)\ge e(\cO_{C,x}) \cdot (n+1)
  \]
  for every curve $C \subseteq X$ containing $x$?
\end{question}
The answer to this question is ``yes'' in
characteristic zero by using \cref{thm:cmsb02}, since \cref{thm:cmsb02}
implies $X \simeq \PP^n_k$, and therefore the required positivity property on
$\omega_X^{-1}$ holds by \cref{ex:seshprojspace}.
If one could answer \cref{quest:wouldprovecmsb} affirmatively independently of
\cref{thm:cmsb02}, then \cite[Thm.\ 1.7]{BS09} would give an alternative proof of
the Mori--Mukai conjecture \ref{conj:morimukai} in characteristic zero, and
\cref{thm:charpn} would resolve their conjecture in positive
characteristic.\index{Mori--Mukai conjecture|)}
\medskip
\par Finally, we mention another conjectural characterization of projective
space that ties in with our discussion in \cref{s:intro}.
\begin{remark}
  In\index{Fujita, Takao!conjecture on characterizing $\PP^n$|(}
  \cref{s:intro}, we noted that the cohomology ring of $\PP^n_\CC$ is
  $\ZZ[h]/(h^{n+1})$; see \cref{eq:cohringpn}.
  Fujita\index{Fujita, Takao} conjectured that a smooth complex
  projective variety of dimension $n$ with this singular cohomology ring
  is isomorphic to $\PP^n_\CC$ \cite[Conj.\ C\textsubscript{n}]{Fuj80}.
  Fujita himself proved this conjecture in dimensions $n \le 5$ \cite[Thm.\
  1]{Fuj80} (under the additional assumption that $X$ is Fano), and Libgober and
  Wood proved this conjecture in dimensions $n \le 6$ \cite[Thm.\ 1]{LW90}.
  See also \cite[Thm.\ 2]{Deb15}.
  It is unclear what the right formulation of this conjecture would be in
  positive characteristic.%
  \index{Fujita, Takao!conjecture on characterizing $\PP^n$|)}
\end{remark}
\section{Proof of Theorem \ref{thm:charpn}}
We now turn to the proof of Theorem \ref{thm:charpn}. The main technical
tool is the notion of bundles of principal parts,
which are also known as jet bundles in the literature.
See \cite[\S4]{LT95} or \cite[\S16]{EGAIV4} for a detailed discussion.
\begin{definition}
  Let $X$ be a variety over an algebraically closed field $k$.
  Denote by $p_1$ and $p_2$ the projections
  \[
    \begin{tikzcd}[column sep=0]
      & X \times_k X\arrow{dl}[swap]{p_1}\arrow{dr}{p_2}\\
      X & & X
    \end{tikzcd}
  \]
  Let $\sI \subset \cO_{X \times X}$ be the ideal defining the diagonal, and
  let $L$ be a line bundle on $X$. For each integer $\ell \ge 0$, the
  \textsl{$\ell$th bundle of principal parts}%
  \index{bundle of principal parts, $\mathscr{P}^\ell(L)$|textbf}
  associated to $L$ is the sheaf
  \[
    \gls*{bundlepparts} \coloneqq p_{1*}(p_2^*L \otimes \cO_{X \times X}/\sI^{\ell+1}).
  \]
  Note that $\sP^0(L) \simeq L$, since the diagonal in $X \times X$ is isomorphic
  to $X$.
\end{definition}
We will use the following facts about these sheaves from \cite[\S4]{LT95},
under the assumption that $X$ is a smooth variety over an algebraically closed
field.
\begin{enumerate}[label=$(\alph*)$,ref=\ensuremath{\alph*}]
  \item There exists a short exact sequence
    \cite[n\textsuperscript{o}\ 4.2]{LT95}
    \begin{equation}\label{eq:sesbundlepp}
      0 \longrightarrow \Sym^\ell(\Omega_X) \otimes L \longrightarrow \sP^\ell(L)
      \longrightarrow \sP^{\ell-1}(L) \longrightarrow 0,
    \end{equation}
    where \gls*{cotangentbundle} denotes the cotangent bundle%
    \index{cotangent bundle, $\Omega_X$} on $X$.
    By using induction and this short exact sequence, we see that the sheaf
    $\sP^\ell(L)$ is a vector bundle for all integers $\ell \ge 0$.
  \item\label{property:bundlepp2}
    There exists an identification $\sP^\ell(L) \simeq p_{2*}(p_2^*L \otimes
    \cO_{X\times X}/\sI^{\ell+1})$, and by applying adjunction to the
    map $p_2^*L \to p_2^*L \otimes \cO_{X\times X}/\sI^{\ell+1}$, there is a
    morphism
    $d^\ell\colon L \to \sP^\ell(L)$
    of sheaves \cite[n\textsuperscript{o}\ 4.1]{LT95}, such that the
    diagram
    \[
      \begin{tikzcd}
        H^0(X,L) \rar{H^0(d^\ell)}\dar &
        H^0\bigl(X,\sP^\ell(L)\bigr) \dar\\
        H^0(X,L \otimes \cO_X/\fm_x^{\ell+1}) &
        H^0\bigl(X,\sP^\ell(L) \otimes \cO_X/\fm_x\bigr) \lar[swap]{\sim}
      \end{tikzcd}
    \]
    commutes for all closed points $x \in X$ \cite[Lem.\ 4.5(1)]{LT95}, where
    the vertical arrows are the restriction maps.
    Thus, if $L$ separates $\ell$-jets at $x$, then $\sP^\ell(L)$ is globally
    generated at $x$.
\end{enumerate}
\par We will also use the following description of the determinant of the
$\ell$th bundle of principal parts. This
description is stated in \cite[p.\ 1660]{DRS01}.
\begin{lemma}\label{lem:detofpp}
  Let\index{bundle of principal parts, $\mathscr{P}^\ell(L)$!determinant of}
  $X$ be a smooth variety of dimension $n$ over an algebraically closed
  field, and let $L$ be a line bundle on $X$.
  Then, for each $\ell \ge 0$, we have an isomorphism
  \[
    \det(\sP^{\ell}(L)) \simeq \bigl( \omega_X^\ell \otimes L^{\otimes(n+1)}
    \bigr)^{\frac{1}{n+1}\binom{n+\ell}{n}}.
  \]
\end{lemma}
\begin{proof}
  We proceed by induction on $\ell \ge 0$. If $\ell = 0$, then $\sP^0(L) \simeq
  L$, so we are done.
  \par Now suppose $\ell > 0$. Since $X$ is smooth, the cotangent bundle
  $\Omega_X$\index{cotangent bundle, $\Omega_X$} has rank $n$, and we have
  isomorphisms
  \[
    \det\bigl(\Sym^\ell(\Omega_X) \otimes L\bigr)
    \simeq \det\bigl(\Sym^\ell(\Omega_X)\bigr) \otimes
    L^{\otimes\binom{n+\ell-1}{n-1}}_{\vphantom{X}}
    \simeq \omega_X^{\otimes\binom{n+\ell-1}{n}} \otimes
    L^{\otimes\binom{n+\ell-1}{n-1}}_{\vphantom{X}}.
  \]
  By induction and taking top exterior powers in the short exact sequence
  \eqref{eq:sesbundlepp}, we obtain
  \begin{align*}
    \det(\sP^{\ell}(L)) &\simeq \omega_X^{\otimes\binom{n+\ell-1}{n}} \otimes
    L^{\otimes\binom{n+\ell-1}{n-1}}_{\vphantom{X}} \otimes \det(\sP^{\ell-1}(L))\\
    &\simeq\omega_X^{\otimes\binom{n+\ell-1}{n}} \otimes
    L^{\otimes\binom{n+\ell-1}{n-1}}_{\vphantom{X}} \otimes \bigl(
      \omega_X^{\otimes(\ell-1)}
    \otimes L^{\otimes(n+1)} \bigr)^{\otimes \frac{1}{n+1}\binom{n+\ell-1}{n}}\\
    &\simeq\bigl( \omega_X^{\otimes\ell} \otimes L^{\otimes(n+1)}
    \bigr)^{\otimes\frac{1}{n+1}\binom{n+\ell}{n}}.
  \end{align*}
  Note that the last isomorphism holds because of the identities
  \begin{gather*}
    \binom{n+\ell-1}{n} + \frac{\ell-1}{n+1}\binom{n+\ell-1}{n} =
    \frac{n+\ell}{n+1} \binom{n+\ell-1}{n} =
    \frac{\ell}{n+1}\binom{n+\ell}{n},\\[0.5em]
    \binom{n+\ell-1}{n-1} + \binom{n+\ell-1}{n} = \binom{n+\ell}{n}
  \end{gather*}
  involving binomial coefficients.
\end{proof}
\par We now prove \cref{thm:charpn}.
We actually show the equivalent formulation in \cref{thm:altcharpn}.
To prove \cref{thm:altcharpn}, we mostly follow the proof of
\cite[Thm.\ 1.7]{BS09}, although we must be more
careful with tensor operations in positive characteristic.
\begin{proof}[Proof of \cref{thm:altcharpn}]
  \par We first show that $\sP^{n+1}(\omega_X^{-1})$ is a trivial bundle.
  First, $\omega_X^{-1} \simeq \omega_X \otimes (\omega_X^{-1})^{\otimes 2}$
  separates
  $(n+1)$-jets by \cref{thm:demseshsingular} (or the special cases in
  \cref{prop:demsesh,thm:poscharseshsm}) since 
  $\varepsilon(\omega_X^{-1};x) \ge n+1$.
  By property \cref{property:bundlepp2} of bundles of principal parts, we
  therefore have that $\sP^{n+1}(\omega_X^{-1})$ is globally
  generated at $x$.
  On the other hand, by \cref{lem:detofpp} applied to $L =
  \omega_X^{-1}$, we have an isomorphism $\det(\sP^{n+1}(\omega_X^{-1})) \simeq
  \cO_X$.
  Now to show that $\sP^{n+1}(\omega_X^{-1})$ is a trivial bundle,
  consider the following diagram:
  \[
    \begin{tikzcd}
      \det\bigl(\sP^{n+1}(\omega_X^{-1})\bigr) \rar{\sim}\dar[twoheadrightarrow]
      & \cO_X\dar[twoheadrightarrow]\\
      \det\bigl(\sP^{n+1}(\omega_X^{-1}) \otimes \cO_X/\fm_x\bigr) \rar{\sim}
      & \cO_X/\fm_x
    \end{tikzcd}
  \]
  Suppose the isomorphism in the top row is given by a non-vanishing global
  section
  \[
    s \in H^0\bigl(X,\det\bigl(\sP^{n+1}(\omega_X^{-1})\bigr)\bigr).
  \]
  Let $s_{1,x} \wedge s_{2,x} \wedge \cdots \wedge s_{r,x}$ be the image of $s$
  in $\det\bigl(\sP^{n+1}(\omega_X^{-1}) \otimes \cO_X/\fm_x\bigr)$, which gives
  the isomorphism in the bottom row.
  Then, since $\sP^{n+1}(\omega_X^{-1})$ is
  globally generated at $x$, each $s_{i,x}$ can be lifted to a global section
  $\widetilde{s}_i \in H^0\bigl(X,\sP^{n+1}(\omega_X^{-1})\bigr)$.
  Because the
  exterior product $\widetilde{s}_1 \wedge \widetilde{s}_2 \wedge \cdots \wedge
  \widetilde{s}_r$ does not vanish at $x$, this exterior product does not vanish
  anywhere, since $H^0(X,\cO_X) = k$ \cite[Thm.\ I.3.4$(a)$]{Har77}.
  Thus, the global sections
  $\widetilde{s}_i$ give a frame for $\sP^{n+1}(\omega_X^{-1})$, and therefore
  $\sP^{n+1}(\omega_X^{-1})$ is a trivial bundle.
  \par To show $X \simeq \PP^n_k$, we use a generalization of Mori's
  characterization of projective space \cite[Thm.\ V.3.2]{Kol96}.
  It suffices to show that for
  every non-constant morphism $f\colon \PP^1_k \to X$,
  the pull back $f^*T_X$ is a sum of line bundles of positive degree.
  Since every vector bundle on $\PP^1_k$ splits \cite[Exer.\
  V.2.6]{Har77}, we may write
  \[
    f^*(T_X) \simeq\bigoplus_{i=1}^n \cO(a_i) \qquad \text{and} \qquad
    f^*(\omega_{X}^{-1}) \simeq\cO(b),
  \]
  where $b$ is positive since $\omega_X^{-1}$ is ample.
  We want to show that each $a_i$ is positive.
  We have
  \[
    f^*(\Omega_X) \simeq f^*(T_X)^\vee \simeq\bigoplus_{i=1}^n \cO(-a_i).
  \]
  Dualizing the short exact sequence \eqref{eq:sesbundlepp} for $\ell = n+1$, we
  have the short exact sequence
  \[
    0 \longrightarrow \sP^n(\omega_X^{-1})^\vee \longrightarrow
    \sP^{n+1}(\omega_X^{-1})^\vee \longrightarrow (\Sym^{n+1}\Omega_X)^\vee
    \otimes \omega_X \longrightarrow 0.
  \]
  The quotient on the right is globally generated because it is a quotient of
  the trivial bundle $\sP^{n+1}(\omega_X^{-1})^\vee$.
  We have
  isomorphisms
  \begin{align*}
    f^*\bigl((\Sym^{n+1}\Omega_X)^\vee \otimes \omega_X\bigr) &\simeq
    \bigl(\Sym^{n+1} f^*(\Omega_X) \bigr)^\vee \otimes f^*(\omega_X)\\
    &\simeq
    \biggl(\Sym^{n+1} \bigoplus_{i=1}^n \cO(-a_i) \biggr)^\vee \otimes \cO(-b),
  \end{align*}
  and this bundle is globally generated since it is the pullback of a globally
  generated bundle.
  By expanding out the
  symmetric power on the right-hand side, we have a surjection
  \[
    f^*\bigl((\Sym^{n+1}\Omega_X)^\vee \otimes \omega_X\bigr)
    \longtwoheadrightarrow
    \bigoplus_{i=1}^n \cO\bigl((n+1)a_i-b\bigr),
  \]
  hence the direct sum on the right-hand side is also globally generated.
  Finally, this implies $(n+1)a_i - b \ge 0$,
  and therefore since $b > 0$, we have that $a_i > 0$ as required.
\end{proof}
\begin{remark}\label{rem:lz16}
  Liu\index{Liu, Yuchen} and Zhuang's\index{Zhuang, Ziquan|(}
  characteristic zero statement in \cite[Thm.\ 2]{LZ18} is
  stronger than \cref{thm:charpn}: it only assumes that $X$ is
  $\QQ$-Fano, and in particular that $X$ is not necessarily smooth.
  While a version of \cref{thm:poscharseshsm} holds for a large class of
  singular varieties (see 
  \cref{thm:demseshsingular})
  the rest of our approach does not generalize
  to the non-smooth setting, since
  Mori's characterization of projective space uses bend and
  break\index{bend and break} techniques.
  Zhuang has since proved \cite[Thm.\ 2]{LZ18} in positive characteristic
  by studying the global $F$-singularities of the blowup of
  $X$ at $x$ \cite[Thm.\ 3]{Zhu}.
  Zhuang has also shown a version of \cref{thm:charpn} using lower bounds on the
  moving Seshadri constant $\varepsilon(\lVert -K_X \rVert;x)$ without the
  assumption that $X$ is Fano, but
  only in characteristic zero \cite[Thm.\ 1.7]{Zhu18}.\index{Zhuang, Ziquan|)}
\end{remark}

\chapter{Preliminaries in arbitrary characteristic}\label{chap:prelims}
In this chapter, we collect some background material that will be used
throughout the rest of this thesis.
The only new result is \cref{prop:kur13prop27}, which describes how sufficiently
large twists of a
coherent sheaf by a big $\QQ$-Cartier divisor $D$ are globally generated away
from the augmented base locus of $D$.
\section{Morphisms essentially of finite type}
Recall that a ring homomorphism $A \to B$ is \textsl{essentially of finite type}
if $B$ is isomorphic (as $A$-algebras) to a localization of an $A$-algebra of
finite type \cite[(1.3.8)]{EGAIV1}.
The corresponding scheme-theoretic notion is the following:
\begin{citeddef}[{\cite[Def.\ 2.1$(a)$]{Nay09}}]\label{def:eft}
  Let $f\colon X \to Y$ be a morphism of schemes.
  \begin{enumerate}[label=$(\alph*)$,ref=\alph*]
    \item We say that $f$ is \textsl{locally essentially of finite
      type}\index{essentially of finite type morphism!locally|textbf} if
      there is an affine open covering
      \[
        Y = \bigcup_i \Spec A_i
      \]
      such that for every $i$, there is an affine open covering
      \[
        f^{-1}(\Spec A_i) = \bigcup_j \Spec B_{ij}
      \]
      for which the corresponding ring homomorphisms $A_i \to B_{ij}$ are
      essentially of finite type.\label{def:loceft}
    \item We say that $f$ is \textsl{essentially of finite type}%
      \index{essentially of finite type morphism|textbf} if it is
      locally essentially of finite type and quasi-compact.
  \end{enumerate}
\end{citeddef}
We will also use the following alternative characterization of these morphisms.
\begin{lemma}\label{lem:eftchara}
  A morphism $f\colon X \to Y$ of schemes is locally essentially of finite type
  (resp.\ essentially of finite type) if and only if for every affine open
  subset $\Spec A \subseteq Y$, there is an affine open covering (resp.\ finite
  affine open covering)
  \[
    f^{-1}(\Spec A) = \bigcup_i \Spec B_i
  \]
  for which the corresponding ring homomorphisms $A \to B_i$ are essentially of
  finite type.
\end{lemma}
\begin{proof}
  It suffices to show the statement for morphisms locally essentially of finite
  type since a similar statement holds for quasi-compactness
  \cite[p.\ 290]{EGAInew}.
  Moreover, the direction $\Leftarrow$ is clear, hence it remains to prove the
  direction $\Rightarrow$.
  \par Fix coverings for $f$ as in \cref{def:eft}\cref{def:loceft},
  and let $\Spec A \subseteq Y$ be an arbitrary open affine subset.
  By \cite[Lem.\ 3.3]{GW10}, there exist $g_k \in A$ such that $\Spec A =
  \bigcup_k \Spec A_{g_k}$ and such that $\Spec A_{g_k} =
  \Spec{(A_i)_{h_k}}$ as open subsets in $Y$ for some $h_k \in A_i$.
  The preimage of $\Spec{(A_i)_{h_k}}$ is covered by the $\Spec{(B_{ij})_{h_k}}$,
  and the compositions
  \[
    A \longrightarrow A_{g_k} \longisoto (A_i)_{h_k} \longrightarrow
    (B_{ij})_{h_k}
  \]
  are essentially of finite type since the class of ring homomorphisms
  essentially of finite type is stable under composition and base change
  \cite[Prop.\ 1.3.9]{EGAIV1}.
  We therefore use the affine open covering
  \[
    f^{-1}(\Spec A) = \bigcup_{i,j,k} \Spec{(B_{ij})_{h_k}}.\qedhere
  \]
\end{proof}
Using this characterization, we can show the following:
\begin{citedlem}[{\cite[(2.2)]{Nay09}}]\label{lem:nayak22}
  The\index{essentially of finite type morphism!under composition and base change}
  class of morphisms (locally) essentially of finite type is closed under
  composition and base change.
\end{citedlem}
\begin{proof}
  It suffices to show the statement for morphisms locally essentially of finite
  type since the corresponding statement holds for quasi-compactness
  \cite[Prop.\ 6.1.5]{EGAInew}.
  \par For composition, let $f\colon X \to Y$ and $g\colon Y \to Z$ be locally
  essentially of finite type.
  Let $\Spec A \subseteq Z$ be an arbitrary affine open set.
  By \cref{lem:eftchara}, there exists an affine open covering
  \[
    g^{-1}(\Spec A) = \bigcup_i \Spec B_i
  \]
  where the corresponding ring homomorphisms $A \to B_i$ are essentially of
  finite type.
  Applying $f^{-1}$ and using \cref{lem:eftchara} again, there exists an
  affine open covering
  \[
    (g \circ f)^{-1}(\Spec A) = \bigcup_i f^{-1}(\Spec B_i) = \bigcup_{i,j}
    \Spec C_{ij},
  \]
  where the corresponding ring homomorphisms $B_i \to C_{ij}$ are essentially of
  finite type.
  The compositions $A \to B_i \to C_{ij}$ are essentially of finite type by
  \cite[Prop.\ 1.3.9$(i)$]{EGAIV1}, hence $g \circ f$ is locally essentially of
  finite type.
  \par For base change, let $f\colon X \to Y$ be locally essentially of finite
  type and fix coverings for $f$ as in \cref{def:eft}\cref{def:loceft}.
  Let $g\colon Y' \to Y$ be an arbitrary morphism of schemes, and denote the
  base change of $f$ by $f'\colon X' \to Y'$.
  Choose an affine open covering
  \[
    g^{-1}(\Spec A_i) = \bigcup_k \Spec C_{ik},
  \]
  in which case $Y' = \bigcup_{i,k} \Spec C_{ik}$.
  Then, the affine open covering
  \[
    f^{\prime-1}(\Spec C_{ik}) = \bigcup_{j} \Spec(B_{ij} \otimes_{A_i}
    C_{ik})
  \]
  is such that the corresponding ring homomorphisms $C_{ik} \to B_{ij}
  \otimes_{A_i} C_{ik}$ are essentially of finite type by base change
  \cite[Prop.\ 1.3.9$(ii)$]{EGAIV1}.
\end{proof}
\begin{remark}
  This notion of morphisms essentially of finite type is somewhat subtle.
  For example, even if $\Spec B \to \Spec A$ is essentially of finite type, it
  is not known whether the corresponding ring homomorphism $A \to B$ is
  essentially of finite type \cite[(2.3)]{Nay09}.
\end{remark}
\section{Cartier and Weil divisors}\label{sect:qrcartierweildiv}

We will work often with $\QQ$- or $\RR$-coefficients for both Cartier and Weil
divisors.
\begin{definition}[see {\cites[Def.\ 21.1.2]{EGAIV4}[\S1.3]{Laz04a}}]
  \label{def:qrcartierdiv}
  Let $X$ be a locally noetherian scheme.
  A \textsl{Cartier divisor}\index{Cartier divisor|textbf}
  on $X$ is an element of the abelian group
  \[
    \Div(X) \coloneqq H^0\bigl(X,\sK_X^*/\cO_X^*\bigr),\glsadd{cartierdivisors}
  \]
  where \gls*{merofuncsheaf} is the sheaf of total quotient rings of $\cO_X$
  \cite[p.\ 204]{Kle79}, and $\sK_X^*$ (resp.\ $\cO_X^*$) is the subsheaf of
  $\sK_X$ (resp.\ $\cO_X$) consisting of invertible sections.
  Concretely, a Cartier divisor is represented by the data $\{(U_i,f_i)\}_i$,
  where $f_i \in \sK_X^*(U_i)$ are local sections, and $X = \bigcup_i U_i$.
  A Cartier divisor $D$ is \textsl{effective}%
  \index{Cartier divisor!effective|textbf} if the functions $f_i \in
  \sK^*(U_i)$ are regular on $U_i$, i.e., if $f_i \in \cO_X(U_i)$.
  \par A \textsl{$\QQ$-Cartier divisor}%
  \index{Cartier divisor!Q-@$\mathbf{Q}$-|textbf} (resp.\ \textsl{$\RR$-Cartier
  divisor}\index{Cartier divisor!R-@$\mathbf{R}$-|textbf}) is an element of the
  group $\Div_\QQ(X) \coloneqq
  \Div(X) \otimes_\ZZ \QQ$ (resp.\ $\Div_\RR (X)\coloneqq \Div(X) \otimes_\ZZ
  \RR$).
  A $\QQ$-Cartier divisor (resp.\ $\RR$-Cartier divisor) is \textsl{effective}
  if it is a $\QQ_{\ge0}$-linear combination (resp.\ $\RR_{\ge0}$-linear
  combination) of effective Cartier divisors.
  A $\QQ$-Cartier divisor (resp.\ \textsl{$\RR$-Cartier divisor}) \textsl{is a
  Cartier divisor} if it is in the image of the map $\Div(X) \to
  \Div_\QQ(X)$ (resp.\ $\Div(X) \to \Div_\RR(X)$).
\end{definition}
\begin{definition}[see {\cites[\S21.6]{EGAIV4}[\S1.3]{Laz04a}}]
  Let $X$ be a locally noetherian scheme.
  A \textsl{Weil divisor}\index{Weil divisor|textbf}
  on $X$ is a formal $\ZZ$-linear combination of
  codimension $1$ cycles on $X$.
  These form an abelian group, which we denote by $\WDiv(X)$.%
  \glsadd{weildivisors}
  A Weil divisor $D$ on $X$ is \textsl{effective}%
  \index{Weil divisor!effective|textbf} if $D$ is a formal
  $\ZZ_{\ge0}$-linear combination of codimension $1$ cycles on $X$.
  \par A \textsl{$\QQ$-Weil
  divisor}\index{Weil divisor!Q-@$\mathbf{Q}$-|textbf} (resp.\
  \textsl{$\RR$-Weil divisor}\index{Weil divisor!R-@$\mathbf{R}$-|textbf}) is an
  element of the group $\WDiv_\QQ(X) \coloneqq \WDiv(X) \otimes_\ZZ \QQ$
  (resp.\ $\WDiv_\RR(X) \coloneqq \WDiv(X) \otimes_\ZZ \RR$).
  A $\QQ$-Weil divisor (resp.\ $\RR$-Weil divisor) is \textsl{effective}
  if it is a $\QQ_{\ge0}$-linear combination (resp.\ $\RR_{\ge0}$-linear
  combination) of effective Weil divisors.
  A $\QQ$-Weil divisor (resp.\ \textsl{$\RR$-Weil divisor}) \textsl{is a
  Weil divisor} if it is in the image of the map $\WDiv(X) \to
  \WDiv_\QQ(X)$ (resp.\ $\WDiv(X) \to \WDiv_\RR(X)$).
\end{definition}
If $X$ is a locally noetherian scheme, then there is a \textsl{cycle map}%
\index{cycle map, $\mathrm{cyc}$|textbf}
\begin{equation*}\label{eq:cyclemap}
  \cyc\colon \Div(X) \longrightarrow \WDiv(X)
\end{equation*}
sending a Cartier divisor to its associated Weil divisor; see
\cite[\S21.6]{EGAIV4}.
If $X$ is locally factorial, then the cycle map $\cyc$ is bijective%
\index{cycle map, $\mathrm{cyc}$!bijective on locally factorial schemes}
\cite[Thm.\ 21.6.9$(ii)$]{EGAIV4}, hence we can
identify Cartier and Weil divisors, as well as their corresponding versions
with $\QQ$- or $\RR$-coefficients.
\par Even if $X$ is not locally factorial, as long as $X$ is normal, we can pass
from Cartier divisors to Weil divisors:
\begin{definition}
  Let $X$ be a locally noetherian scheme.
  If $X$ is a normal, then the cycle map $\cyc$ is injective
  \cite[Thm.\ 21.6.9$(i)$]{EGAIV4}.%
  \index{cycle map, $\mathrm{cyc}$!injective on normal schemes}
  We then say that a Weil divisor (resp.\
  $\QQ$-Weil divisor, $\RR$-Weil divisor) \textsl{is Cartier} (resp.\
  \textsl{$\QQ$-Cartier,} \textsl{$\RR$-Cartier}) if it is in
  the image of $\cyc$ (resp.\ $\cyc \otimes_\ZZ \QQ$, $\cyc \otimes_\ZZ \RR$).
\end{definition}
Finally, we will use the following conventions for rounding up and down.
\begin{definition}[see {\cites[Def.\ 3.4.1]{BGGJKM}[Def.\ 9.1.2]{Laz04b}}]
  \label{def:roundup}
  Let $X$ be a locally noetherian scheme, and let $D \in \Div_\RR(X)$.
  A \textsl{decomposition}\index{decomposition|textbf}
  $\mathcal{D}$ of $D$ is an expression
  \[
    D = \sum_{i=1}^r a_iD_i
  \]
  for some $a_i \in \RR$ and Cartier divisors $D_i$.
  We note that such a decomposition is not unique, since $\Div(X)$ may have
  torsion.
  The \textsl{round-up}\index{round-up, $\lceil D \rceil$|textbf|(} and
  \textsl{round-down}\index{round-down, $\lfloor D \rfloor$|textbf|(} of $D$
  with respect to
  $\mathcal{D}$ are the Cartier divisors%
  \glsadd{roundup}\glsadd{rounddown}
  \[
    \lceil D \rceil_{\mathcal{D}} \coloneqq \sum_{i=1}^r \lceil a_i \rceil
    D_i \qquad \text{and} \qquad
    \lfloor D \rfloor_{\mathcal{D}} \coloneqq \sum_{i=1}^r \lfloor a_i \rfloor
    D_i,
  \]
  respectively.
  Note that the round-up and round-down depend on the decomposition
  $\mathcal{D}$.
  \par Now let $D \in \WDiv_\RR(X)$.
  By definition, we have
  \[
    D = \sum_{i=1}^r a_iD_i.
  \]
  Then, the \textsl{round-up} and \textsl{round-down} of $D$ are
  the Weil divisors
  \[
    \lceil D \rceil \coloneqq \sum_{i=1}^r \lceil a_i \rceil
    D_i \qquad \text{and} \qquad
    \lfloor D \rfloor \coloneqq \sum_{i=1}^r \lfloor a_i \rfloor
    D_i,
  \]
  respectively.%
  \index{round-up, $\lceil D \rceil$|textbf|)}
  \index{round-down, $\lfloor D \rfloor$|textbf|)}
\end{definition}
\section{Reflexive sheaves}
We will need some basic results on reflexive sheaves, which we collect here.
Our main reference is \cite[\S1]{Har94}.
\begin{definition}
  Let $\sF$ be a coherent sheaf on a scheme $X$.
  The \textsl{dual} of $\sF$ is $\sF^\vee\glsadd{fdual} \coloneqq
  \HHom_{\cO_X}(\sF,\cO_X)$.\index{dual, $\mathscr{F}^\vee$|textbf}
  We say that $\sF$ is \textsl{reflexive}\index{reflexive sheaf|textbf} if the
  natural map $\sF \to
  \sF^{\vee\vee}$ is an isomorphism.
  We say that $\sF$ is \textsl{normal}\index{normal sheaf|textbf} if for every
  open subset $U
  \subseteq X$ and every subset $Y \subseteq U$ of codimension $\ge 2$, the
  restriction map
  $\sF(U) \to \sF(U \smallsetminus Y)$
  is bijective.
\end{definition}
We note that all locally free sheaves are reflexive
\cite[Exer.\ II.5.1$(a)$]{Har77}.
\par By the following result, all reflexive sheaves on reasonably nice schemes
are normal.
Below, we recall that a noetherian scheme $X$ \textsl{satisfies $G_i$} for an
integer $i \ge 0$ if for every point $x \in X$ such that $\dim \cO_{X,x} \le i$,
the local ring $\cO_{X,x}$ is Gorenstein.
\begin{citedprop}[{\cite[Prop.\ 1.11]{Har94}}]\label{prop:reflexivenormal}
  Let\index{reflexive sheaf!is normal}
  $X$ be a noetherian scheme satisfying $G_1$ and $S_2$.
  Then, every reflexive sheaf $\sF$ is normal.
\end{citedprop}
We will also need the following:
\begin{citedlem}
  \label{lem:homreflexive}
  Let $X$ be a locally noetherian scheme satisfying $G_0$ and $S_1$.
  Let $\sF$ and $\sG$ be coherent sheaves on $X$.
  If $\sG$ is reflexive, then $\HHom_{\cO_X}(\sF,\sG)$ is also reflexive.
\end{citedlem}
\begin{proof}
  Since $\sG$ is reflexive, we have
  \[
    \HHom_{\cO_X}(\sF,\sG) \simeq
    \HHom_{\cO_X}\bigl(\sF,\HHom_{\cO_X}(\sG^\vee,\cO_X)\bigr)
    \simeq \HHom_{\cO_X}(\sF \otimes_{\cO_X} \sG^\vee,\cO_X)
  \]
  where the second isomorphism is by tensor--hom adjunction.
  Since the dual of any coherent sheaf is reflexive \cite[Cor.\ 1.8]{Har94}, we
  are done.
\end{proof}
We will often use this fact to extend morphisms from the complement of
codimension at least two.%
\index{reflexive sheaf!determined by codimension one behavior|(}
\begin{corollary}\label{cor:extendsections}
  Let $X$ be a locally noetherian scheme satisfying $G_1$ and $S_2$,
  and let $\sF$ and $\sG$ be coherent sheaves on $X$ such
  that $\sF$ is reflexive.
  If $U \subseteq X$ is an open subset such that $\codim(X \smallsetminus U) \ge
  2$, then every morphism $\varphi\colon \sG\rvert_U \to \sF\rvert_U$ extends
  uniquely to a morphism $\widetilde{\varphi}\colon\sG \to \sF$.
\end{corollary}
\begin{proof}
  The morphism $\varphi$ corresponds to a section of the sheaf
  $\HHom_{\cO_X}(\sG,\sF)$ over $U$.
  The sheaf $\HHom_{\cO_X}(\sG,\sF)$ is reflexive by
  \cref{lem:homreflexive}, hence the section $\varphi$ extends uniquely to a
  section $\widetilde{\varphi}$ of $\HHom_{\cO_X}(\sG,\sF)$ over $X$ by
  \cref{prop:reflexivenormal}.
\end{proof}
The following result says that on noetherian schemes satisfying $G_1$ and $S_2$,
reflexive sheaves are determined by their codimension one behavior.
\begin{citedthm}[{\cite[Thm.\ 1.12]{Har94}}]\label{thm:har94112}
  Let $X$ be a noetherian scheme satisfying $G_1$ and $S_2$, and let $Y
  \subseteq X$ be a closed subset of codimension at least $2$.
  Then, the restriction functor induces an equivalence of categories
  \[
    \begin{tikzcd}[row sep=0,column sep=1.475em]
      \biggl\{
        \begin{tabular}{@{}c@{}}
          reflexive coherent\\
          $\cO_X$-modules
        \end{tabular}
      \biggr\} \rar &
      \biggl\{
        \begin{tabular}{@{}c@{}}
          reflexive coherent\\
          $\cO_{X \smallsetminus Y}$-modules
        \end{tabular}
      \biggr\}\\
      \sF \rar[mapsto] & \sF\rvert_{X \smallsetminus Y}
    \end{tikzcd}%
    \index{reflexive sheaf!determined by codimension one behavior|)}
  \]
\end{citedthm}
\section{Dualizing complexes and Grothendieck duality}
\label{sect:grothendieckduality}
The main references for this section are \cites{Har66}{Con00}, although we need
the extension of the theory to separated morphisms that are essentially of
finite type, following \cite{Nay09}.
In the statement below, recall that for a noetherian scheme $X$,
\gls*{derivedcat} denotes the derived category
of $\cO_X$-modules with quasi-coherent cohomology, and 
\gls*{boundedderivedcat} is the full subcategory of
$\mathbf{D}_{\mathrm{qc}}(X)$ whose objects are complexes $\sF$ such that
$\mathbf{h}^i\sF = 0$ for all $i \ll 0$.
\begin{theorem}[{\cite[Thm.\ 5.3]{Nay09}; cf.\ \cite[Cor.\ V.3.4]{Har66}}]
  \label{thm:nayakshriek}
  Let\index{Grothendieck, Alexander!duality|(}
  \gls*{sepeftcat} denote the subcategory of the category of
  schemes whose objects are noetherian schemes, and whose morphisms are
  separated and essentially of finite type morphisms of schemes.
  Then, there exists a contravariant $\mathbf{D}^+_{\mathrm{qc}}$-valued
  pseudofunctor $(-)^!$\glsadd{exceptionalpullback} on
  $\mathbf{S}_{\mathbf{e}}$ such that
  \begin{enumerate}[label=$(\roman*)$]
    \item For proper morphisms, $(-)^!$ is pseudofunctorially isomorphic to the
      right adjoint of the right-derived direct image pseudofunctor $\RR f_*$;
    \item For essentially \'etale morphisms, $(-)^!$ equals the inverse image
      pseudofunctor $(-)^*$;
    \item For every cartesian diagram
      \[
        \begin{tikzcd}
          U \rar{j}\dar[swap]{g} & X\dar{f}\\
          V \rar{i} & Y
        \end{tikzcd}
      \]
      of noetherian schemes, where $f$ is proper and $i$ is flat, there is a
      flat base change isomorphism $j^*f^! \overset{\sim}{\to} g^!i^*$.%
      \index{Grothendieck, Alexander!duality|)}
  \end{enumerate}
\end{theorem}
We note that a morphism is \textsl{essentially \'etale}%
\index{essentially etale@essentially \'etale|textbf} (resp.\
\textsl{essentially smooth}\index{essentially smooth|textbf}) if it is
separated, formally \'etale (resp.\ formally smooth), and essentially of finite
type \cite[(5.1) and (5.4)]{Nay09}.
Note that for certain classes of morphisms, the pseudofunctor $(-)^!$ has
concrete descriptions; see \cite[Prop.\ III.6.5 and Thm.\ III.6.7]{Har66} for
finite morphisms, and see \cites[III.2]{Har66}[(5.4)]{Nay09} for essentially
smooth morphisms.
\medskip
\par \cref{thm:nayakshriek} allows us to define the following:
\begin{definition}\label{def:canonicalsheaf}
  Let $h\colon X \to \Spec k$ be an equidimensional scheme that is separated and
  essentially of finite type over a field $k$.
  The \textsl{normalized dualizing complex}%
  \index{dualizing complex, $\omega_X^\bullet$!normalized|textbf} for $X$ is
  $\gls*{dualizingcomplex}
  \coloneqq h^!k$, where $h^!$ is the functor in \cref{thm:nayakshriek}.
  The \textsl{canonical sheaf}%
  \index{canonical bundle or sheaf, $\omega_X$|textbf} on $X$ is the coherent
  sheaf
  \[
    \gls*{canonicalsheaf} \coloneqq \mathbf{h}^{-\dim X} \omega_X^\bullet.
  \]
  Note that the canonical sheaf is reflexive if $X$ satisfies $G_1$ and $S_2$,
  since it is $S_2$ by \cite[Lem.\ 1.3]{Har07}, hence reflexive \cite[Thm.\
  1.9]{Har94}.
  If $X$ is normal, we can therefore define a \textsl{canonical
  divisor}\index{canonical divisor, $K_X$|textbf} \gls*{canonicaldivisor}
  as a choice of Weil divisor whose associated sheaf $\cO_X(K_X)$ is isomorphic
  to $\omega_X$.
  Note that $K_X$ is only well-defined up to linear equivalence.
\end{definition}
\par When $X$ is essentially smooth, the canonical sheaf $\omega_X$ is
isomorphic to the invertible sheaf of top differential forms
$\bigwedge^{\dim X}\Omega_X$ \cites[III.2]{Har66}[(5.4)]{Nay09}.
When $X$ is Gorenstein, the canonical sheaf $\omega_X$ is invertible
\cites[Exer.\ V.9.7]{Har66}[(5.10)]{Nay09}.

\section{Base ideals and base loci}
We define the classical notions of base ideals and base loci for Cartier
divisors.
\begin{definition}[see {\cite[\S1.1.B]{Laz04a}}]\label{def:baseideal}
  Let $X$ be a scheme over a field $k$, and let $L$ be a line bundle on $X$.
  If $V \subseteq H^0(X,L)$ is a finite-dimensional $k$-vector space,
  then the associated projective space $\gls*{linearsystem} \coloneqq
  \PP(V^\vee)$ of one-dimensional subspaces of $V$ is called a \textsl{linear
  system}\index{linear system, $\lvert V \rvert$|textbf}.
  If $V = H^0(X,L)$, then $\lvert V \rvert$ is the \textsl{complete linear
  system}%
  \index{linear system, $\lvert V \rvert$!complete, $\lvert D \rvert$|textbf}
  associated to $L$.
  The \textsl{base ideal}\index{base!ideal, $\fb(\lvert V \rvert)$|textbf}
  of $\lvert V \rvert$ is
  \begin{equation}\label{eq:baseidealeval}
    \fb\bigl(\lvert V \rvert\bigr) \coloneqq \im\bigl(V \otimes_k L^{-1} 
    \xrightarrow{\eval} \cO_X\bigr).\glsadd{baseideal}
  \end{equation}
  The \textsl{base scheme}\index{base!scheme, $\Bs(\lvert V \rvert)$|textbf}
  of $\lvert V \rvert$ is the closed subscheme \gls*{basescheme} of $X$
  defined by $\fb(\lvert V \rvert)$, and the
  \textsl{base locus}\index{base!locus, $\Bs(\lvert V \rvert)_\red$|textbf}
  of $\lvert V \rvert$ is the underlying closed subset \gls*{baselocus} of
  $\Bs(\lvert V \rvert)$.
  \par If the line bundle $L$ is of the form $\cO_X(D)$ for a Cartier divisor
  $D$, then the complete linear system associated to $\cO_X(D)$ is denoted by
  \gls*{completelinearsystem}.
\end{definition}
Note that if $X$ is either projective over a field or reduced, then every line
bundle $L$ on $X$ is of the form $\cO_X(D)$ for a Cartier divisor $D$ \cite[Ex.\
1.1.5]{Laz04a}.
\medskip
\par We will need the following description for how base ideals transform under
birational morphisms.
\begin{lemma}\label{lem:baselocusnormal}
  Let\index{base!ideal, $\fb(\lvert V \rvert)$!under birational morphisms}
  $f\colon X' \to X$ be a birational morphism between complete varieties,
  where $X$ is normal.
  Then, for every Cartier divisor $D$ on $X$, we have
  \[
    f^{-1}\fb\bigl(\lvert D \rvert\bigr)\cdot\cO_{X'} = \fb\bigl(\lvert f^*D
    \rvert\bigr).
  \]
\end{lemma}
\begin{proof}
  Since $X$ is normal, we have $f_*\cO_{X'} \simeq \cO_X$ \cite[Proof of Cor.\
  III.11.4]{Har77}.
  By the projection formula, we then have $H^0(X,\cO_X(D)) =
  H^0(X',\cO_{X'}(f^*D))$, and the lemma then follows by pulling back the
  evaluation map \eqref{eq:baseidealeval}.
\end{proof}
We define the notion of a graded family of ideals, of which base ideals
will be an important example.
\begin{definition}[{see \cite[Def.\ 2.4.14]{Laz04b}}]
  Let $X$ be a locally noetherian scheme.
  A \textsl{graded family of ideals}%
  \index{graded family of ideals, $\mathfrak{a}_\bullet$|textbf}
  $\gls*{gradedfamily} = \{\fa_m\}_{m\in\NN}$ on
  $X$ is a collection of coherent ideal sheaves $\fa_m \subseteq \cO_X$ such
  that $\fa_0 = \cO_X$, and such that for all $m,n\ge0$, we have
  $\fa_m \cdot \fa_n \subseteq \fa_{m+n}$.
\end{definition}
We now describe how base ideals can form a graded family of ideals.
\begin{example}[{see \cite[Ex.\ 1.1.9]{Laz04a}}]\label{ex:triplelinsys}
  Let $X$ be a complete scheme over a field $k$, and let $D$ be a
  $\QQ$-Cartier divisor on $X$.
  We define a graded family of ideals \gls*{gradedfamilydivisor} by setting
  \[
    \fa_m(D) = \begin{cases}
      \fb\bigl(\lvert mD \rvert\bigr) & \text{if $mD$ is integral}\\
      \hfil 0 & \text{otherwise}
    \end{cases}
  \]
  where $\fb(\lvert mD \rvert)$ is the base ideal of the complete linear series
  $\lvert mD \rvert$ (\cref{def:baseideal}).
  Note that $\fa_\bullet(D)$ is a graded family since the multiplication map
  $H^0(X,mD) \otimes_k H^0(X,nD) \to H^0(X,(m+n)D)$ induces an inclusion
  $\fb(\lvert mD \rvert) \cdot \fb(\lvert nD \rvert) \subseteq \fb(\lvert (m+n)D
  \rvert)$.
\end{example}
\section{Asymptotic invariants of line bundles}
In this section, we review some aspects of the theory of asymptotic invariants
of Cartier divisors and their base loci.
We have taken care to work over arbitrary fields; see \cite{ELMNP05} for an
overview on the theory of asymptotic invariants for smooth complex varieties.
\subsection{Stable base loci}
We start by defining a stable ``asymptotic'' version of the base locus due to
Fujita\index{Fujita, Takao}.
\begin{citeddef}[{\cite[Def.\ 1.17]{Fuj83}}]
  Let $X$ be a complete scheme over a field, and let $D$ be a Cartier
  divisor on $X$.
  The \textsl{stable base locus}\index{stable base locus, $\SB(D)$|textbf}
  of $D$ is the closed subset
  \begin{equation}\label{eq:defstablebaselocus}
    \gls*{stablebaselocus} \coloneqq \bigcap_m \Bs\bigl( \lvert mD \rvert \bigr)_\red
  \end{equation}
  of $X$, where the intersection runs over every integer $m > 0$.
  The noetherian property implies $\SB(D) = \SB(nD)$ for every
  integer $n > 0$ \cite[Ex.\ 2.1.23]{Laz04a}, hence the formula
  \cref{eq:defstablebaselocus} can be used for $\QQ$-Cartier divisors
  $D$ by taking the intersection over every integer $m > 0$ such that
  $mD$ is a Cartier divisor.
\end{citeddef}
The stable base locus is not a numerical invariant of $D$, as we can see in the
following example.
\begin{example}[cf.\ {\cite[Ex.\ 10.3.3]{Laz04b}}]
  \index{stable base locus, $\SB(D)$!not a numerical invariant}
  Let $C$ be an elliptic curve, and
  let $P_1$ and $P_2$ be degree zero divisors on $C$ that are torsion and
  non-torsion, respectively.
  Then, $P_1$ is semiample, hence $\SB(P_1) = \emptyset$.
  On the other hand, we have $\SB(P_2) = C$ since no multiple of $P_2$ has
  global sections.
  \par Using this observation, we also construct an example where the stable
  base locus is not a numerical invariant, even for big and nef divisors.
  The construction below is due to
  Cutkosky\index{Cutkosky, Dale, construction of}; see
  \cite[\S2.3.B]{Laz04a}.
  Let $A$ be divisor on $C$ of degree $1$, and
  for every degree zero divisor $P$ on $C$, consider the projective space bundle
  \[
    X_P \coloneqq \PP\bigl( \cO_C(P) \oplus \cO_C(A+P)\bigr) \longrightarrow C.
  \]
  Denote by $\xi_P$ the divisor on $X_P$ corresponding to the tautological line
  bundle $\cO_{X_P}(1)$, which is big and nef by \cite[Lems.\
  2.3.2$(iii)$ and 2.3.2$(iv)$]{Laz04a}.
  Since the bundles defining $X_P$ only differ by a twist by the divisor $P$,
  the varieties $X_P$ are all naturally isomorphic to $X_0$, and
  under this identification, the divisors $\xi_P$ are numerically equivalent
  divisors on $X_0$.
  Now let $P_1$ and $P_2$ be as in the previous paragraph.
  Then, since $P_1$ is semiample, $\xi_{P_1}$ is also semiample by \cite[Lem.\
  2.3.2$(v)$]{Laz04a}, hence $\SB(\xi_{P_1}) = \emptyset$.
  On the other hand, $\SB(\xi_{P_2})$ contains the section $C \simeq
  \PP(\cO_C(P_2)) \subseteq X_{P_2}$ corresponding to the first projection
  $\cO_C(P_2) \oplus \cO_C(A+P_2) \to \cO_C(P_2)$, since $\SB(P_2) = C$.
\end{example}
\par We will see in the next subsection how one can define a numerically
invariant approximation of $\SB(D)$.
\subsection{Augmented base loci}\label{sect:bplus}
We define a numerically invariant upper approximation of the stable base locus,
which was first introduced by Nakamaye\index{Nakamaye, Michael} \cite{Nak00}.
\begin{definition}[see {\cite[Def.\ 1.2]{ELMNP06}}]
  Let $X$ be a projective scheme over a field, and let $D$ be an
  $\RR$-Cartier divisor on $X$.
  The \textsl{augmented base locus}
  \index{augmented base locus, $\Bplus(D)$|textbf} of $D$ is the closed subset
  \[
    \gls*{augmentedbaselocus} \coloneqq \bigcap_A \SB(D-A)
  \]
  of $X$, where the intersection runs over all ample $\RR$-Cartier
  divisors $A$ such that $D - A$ is a $\QQ$-Cartier divisor.
  If $X$ is a variety, then by \cite[Rem.\ 1.3]{ELMNP06}, we have
  \[
    \Bplus(D) = \bigcap_{D \equiv_\RR A + E} \Supp E
  \]
  where the intersections runs over all $\RR$-numerical equivalences $D
  \equiv_\RR A + E$ where $A$ is an ample $\QQ$-Cartier divisor and $E$ is an
  effective $\RR$-Cartier divisor.
\end{definition}
We note that $D$ is ample if and only if $\Bplus(D) = \emptyset$, and
if $X$ is a variety, then $D$ is big if and only if $\Bplus(D) \ne
X$; see \cite[Ex.\ 1.7]{ELMNP06}.%
\index{augmented base locus, $\Bplus(D)$!detects ampleness}%
\index{augmented base locus, $\Bplus(D)$!detects bigness}
\medskip
\par We will need the following birational transformation rule for
augmented base loci.
\begin{proposition}[cf.\ {\cite[Prop.\ 2.3]{BBP13}}]\label{prop:bbp23}
  \index{augmented base locus, $\Bplus(D)$!under birational morphisms}
  Let $f\colon X' \to X$ be a birational morphism between normal projective
  varieties.
  If $D$ is an $\RR$-Cartier divisor on $X$, then we have
  \begin{equation}\label{eq:bbp23prop}
    \Bplus(f^*D) = f^{-1}\bigl(\Bplus(D)\bigr) \cup \exc(f).
  \end{equation}
\end{proposition}
The proof of \cite[Prop.\ 2.3]{BBP13} applies in this setting after setting $F =
0$ in their notation, since this makes the application of the negativity
lemma\index{negativity lemma} unnecessary.
\begin{remark}
  If one works over an algebraically closed field, then the augmented base locus
  on the left-hand side of \cref{eq:bbp23prop} can be replaced by $\Bplus(f^*D +
  F)$, where $F$ is any $f$-exceptional $\RR$-Cartier divisor on $X'$.
  The proof of this follows \cite[Prop.\ 2.3]{BBP13}, after proving the
  negativity lemma\index{negativity lemma} in arbitrary characteristic (see
  \cite[(2.3)]{Bir16}).
\end{remark}
We also need the following description for the augmented base locus for nef
Cartier divisors, which is originally due to Nakamaye\index{Nakamaye, Michael}
for smooth projective varieties over algebraically closed fields of
characteristic zero.
\begin{theorem}[{\cite[Thm.\ 1.4]{Bir17}; cf.\ \cite[Thm.\ 0.3]{Nak00}}]
  \label{thm:nakamaye}
  Let\index{augmented base locus, $\Bplus(D)$!for nef divisors}
  $X$ be a projective scheme over a field, and suppose
  $D$ is a nef $\RR$-Cartier divisor.
  Then, we have
  \[
    \Bplus(D) = \bigcup_{(L^{\dim V} \cdot V) = 0} V,
  \]
  where $V$ runs over all positive-dimensional subvarieties $V \subseteq
  X$ such that $(L^{\dim V} \cdot V) = 0$.
\end{theorem}
We will also need the following result, which describes how $\Bplus(D)$ is the
locus where $D$ is ample.
Regularity in the proof below is in the sense of Castelnuovo and
Mumford;
see \cite[Def.\ 1.8.4]{Laz04a} for the definition.
\begin{proposition}[cf.\ {\cites[Prop.\ 2.7]{Kur13}[Lem.\ 7.12]{FM}}]
  \label{prop:kur13prop27}
  Let\index{augmented base locus, $\Bplus(D)$!detects ampleness}
  $X$ be a projective scheme over a field, and let $D$ be a $\QQ$-Cartier
  divisor on $X$ with a decomposition $\mathcal{D}$.
  Then, $\Bplus(D)$ is the smallest closed subset of $X$ such that for every
  coherent sheaf $\sF$ on $X$ and for every $\RR$-Cartier divisor $E$ with 
  decomposition $\mathcal{E}$, there exists an integer $n_0$ such
  that the sheaves
  \[
    \sF \otimes \cO_X\bigl(\lfloor E+nD \rfloor + P\bigr)
    \qquad \text{and} \qquad \sF \otimes \cO_X\bigl(\lceil E+nD
    \rceil + P\bigr)
  \]
  are globally
  generated on $X \smallsetminus \Bplus(D)$ for every integer $n \ge n_0$
  and every nef Cartier divisor $P$, where the rounding is done with respect to
  $\mathcal{D}$ and $\mathcal{E}$.
  \par If $X$ is normal, then the same conclusion holds for
  $\QQ$-Cartier $\QQ$-Weil divisors $D$ and $\RR$-Weil divisors $E$, where the
  rounding is done in the sense of $\RR$-Weil divisors.
\end{proposition}
\begin{proof}
  We first show that $\Bplus(D)$ satisfies the condition in the proposition.
  If $\Bplus(D) = X$, then the condition trivially holds.
  We therefore assume that $\Bplus(D) \ne X$.
  \par Let $A$ be an ample and free Cartier divisor on $X$.
  By \cite[Prop.\ 1.5]{ELMNP06} and \cite[Prop.\ 2.1.21]{Laz04a}, there exist
  positive integers $q$ and $r$ such that $qrD$ is a Cartier divisor and
  \begin{equation}\label{eq:bplusstabilize}
    \Bplus(D) = \SB(rD-A) = \Bs\bigl(\bigl\lvert q(rD-A)\bigr\rvert\bigr)_\red.
  \end{equation}
  After possibly replacing $A$ and $r$ by $qr$ and $qA$, respectively, we can
  assume that $r$ is an integer such that $rD$ is Cartier, and $\Bplus(D) =
  \Bs(rD-A)_\red$.
  \par Now we claim that there exists an integer $m_0$ such that
  $\sF \otimes \cO_X(mA+\lfloor E+jD \rfloor+P)$ (resp.\ $\sF \otimes
  \cO_X(mA+\lceil E+jD \rceil+P)$) is globally generated for every
  $m \ge m_0$, every $1 \le j < r$, and every nef Cartier divisor $P$, where
  $\lfloor E+jD \rfloor$ (resp.\ $\lceil E+jD \rceil$)
  should either be interpreted in the sense of $\RR$-Cartier
  divisors with respect to the decomposition $\mathcal{D}$ and $\mathcal{E}$,
  or interpreted in
  the sense of $\RR$-Cartier $\RR$-Weil divisors in the situation when $X$ is
  normal.
  By Fujita's vanishing theorem\index{vanishing theorem!Fujita}
  \cite[Thm.\ 5.1]{Fuj83}, there exists an integer $m_1$ such that for all
  integers $m \ge m_1$ and all $i > 0$, we have
  \begin{align*}
    H^i\bigl(X,\sF \otimes \cO_X\bigl(mA+\lfloor E+jD \rfloor+P\bigr)\bigr) &= 0\\
    H^i\bigl(X,\sF \otimes \cO_X\bigl(mA+\lceil E+jD \rceil+P\bigr)\bigr) &= 0
  \end{align*}
  for all $0 \le j < r$, and every nef Cartier divisor $P$.
  Thus, if $m \ge m_1 + \dim X$, then the coherent sheaf $\sF \otimes
  \cO_X(mA+\lfloor E+jD \rfloor+P)$ (resp.\ $\sF \otimes
  \cO_X(mA+\lceil E+jD \rceil+P)$) is
  $0$-regular with respect to $A$,
  hence is globally generated by \cite[Thm.\ 1.8.5$(i)$]{Laz04a}.
  It therefore suffices to set $m_0 = m_1 + \dim X$.
  \par To prove that $\Bplus(D)$ satisfies the condition in the proposition,
  we note that by the above, the sheaves
  \begin{align*}
    \sF \otimes \cO_X\bigl(mA+\lfloor E+jD \rfloor+P\bigr) &\otimes
    \cO_X\bigl(q(rD-A)\bigr)\\
    \sF \otimes \cO_X\bigl(mA+\lceil E+jD \rceil+P\bigr) &\otimes
    \cO_X\bigl(q(rD-A)\bigr)
  \end{align*}
  are globally generated away from
  $\Bplus(D)$ 
  for all $m \ge m_0$, all $q \ge 1$, all $0 \le j < r$, and every
  nef Cartier divisor $P$.
  Setting $q=m$, we see that the sheaves 
  \begin{align*}
    \sF \otimes \cO_X\bigl(mrD+\lfloor E+jD \rfloor+P\bigr) &\simeq
    \sF \otimes \cO_X\bigl(\lfloor E+(mr+j)D \rfloor+P\bigr)\\
    \sF \otimes \cO_X\bigl(mrD+\lceil E+jD \rceil+P\bigr) &\simeq
    \sF \otimes \cO_X\bigl(\lceil E+(mr+j)D \rceil+P\bigr)
  \end{align*}
  are globally generated away from $\Bplus(D)$ 
  for all $m \ge m_0$, all $0 \le j < r$, and every nef
  Cartier divisor $P$.
  It therefore suffices to set $n_0 = m_0r$.
  \par Finally, we show $\Bplus(D)$ is the smallest closed subset satisfying
  the condition in the proposition.
  Let $x \in \Bplus(D)$; it suffices to show that for $\sF = \cO_X(-A)$ where
  $A$ is ample, the sheaf $\sF \otimes \cO_X(mD)
  = \cO_X(nD-A)$ is not globally generated at $x$
  for all $n \ge 0$ such that $nD$ is a Cartier divisor.
  This follows from \cite[Prop.\ 1.5]{ELMNP06} since $x \in \Bplus(D)$.
\end{proof}
\subsection{Asymptotic cohomological functions}
\label{sect:burgosgil}
We now review K\"uronya's\index{Kuronya, Alex@K\"uronya, Alex} asymptotic
cohomological functions with suitable modifications to work over arbitrary
fields, following \cite[\S2]{Kur06} and \cite[\S3]{BGGJKM}.
Asymptotic cohomological functions are defined as follows.
\begin{citeddef}[{\cite[Def.\ 3.4.6]{BGGJKM}}]\label{def:asymptoticcoh}
  Let $X$ be a projective scheme of dimension $n$ over a field.
  For every integer $i \ge 0$, the \textsl{$i$th asymptotic cohomological
  function}%
  \index{asymptotic cohomological function, $\widehat{h}^i(X,D)$|textbf}
  $\widehat{h}^i(X,-)$ on $X$ is the function defined by setting
  \[
    \gls*{asymptoticcohfunction} \coloneqq \limsup_{m \to \infty}
    \frac{h^i\bigl(X,\cO_X\bigl(\lceil mD
    \rceil_{\mathcal{D}}\bigr)\bigr)}{m^n/n!}
  \]
  for an $\RR$-Cartier divisor $D$,
  where $\mathcal{D}$ is a decomposition of $D$ (see \cref{def:roundup}).
  The numbers $\widehat{h}^i(X,D)$ only
  depend on the $\RR$-linear equivalence class of $D$ and are independent of
  the decomposition $\mathcal{D}$ by \cite[Rem.\ 3.4.5]{BGGJKM}, hence
  $\widehat{h}^i(X,-)$ gives rise to a well-defined function
  $\Div_\RR(X) \to \RR$ and $\Div_\RR(X)/\mathord{\sim_\RR} \to \RR$.
\end{citeddef}
\par A key property of asymptotic cohomological functions is the following:
\begin{citedprop}[{\cite[Prop.\ 3.4.8]{BGGJKM}}]\label{prop:nahomogcont}
  Let\index{asymptotic cohomological function, $\widehat{h}^i(X,D)$!continuity}
  \index{asymptotic cohomological function, $\widehat{h}^i(X,D)$!homogeneity}
  $X$ be a projective scheme of dimension $n$ over a field.
  For every $i \ge 0$, the function $\widehat{h}^i(X,-)$ on
  $\Div_\RR(X)$ is homogeneous of degree $n$, and is
  continuous on every finite-dimensional $\RR$-subspace of
  $\Div_\RR(X)$ with respect to every norm.
\end{citedprop}
\cref{prop:nahomogcont} shows that \cref{def:asymptoticcoh}
is equivalent to K\"uronya's original definition in \cite[Def.\ 2.1]{Kur06},
and that when $i = 0$, the asymptotic cohomological function
$\widehat{h}^i(X,D)$ matches the \textsl{volume function}%
\index{volume, $\vol_X(D)$} \gls*{volume} from
\cite[\S2.2]{Laz04a}.
\cref{prop:nahomogcont} also allows us to prove that asymptotic cohomological
functions behave well with respect to generically finite morphisms.
\begin{proposition}[cf.\ {\cite[Prop.\ 2.9(1)]{Kur06}}]\label{prop:kur291}
  Let\index{asymptotic cohomological function, $\widehat{h}^i(X,D)$!under generically finite morphisms}
  $f\colon Y \to X$ be a surjective morphism of projective
  varieties, and consider an $\RR$-Cartier divisor $D$ on $X$.
  Suppose $f$ is generically finite of degree $d$.
  Then, for every $i$, we have
  \[
    \widehat{h}^i(Y,f^*D) = d \cdot \widehat{h}^i(X,D).
  \]
\end{proposition}
\begin{proof}
  The proof of \cite[Prop.\ 2.9(1)]{Kur06} works in our setting with the
  additional hypothesis that $D$ is a Cartier divisor.
  It therefore suffices to reduce to this case.
  If the statement holds for Cartier divisors $D$, then it also
  holds for $D \in \Div_\QQ(X)$ by homogeneity of 
  $\widehat{h}^i$ (\cref{prop:nahomogcont}).
  Moreover, the subspace of $\Div_\RR(X)$ spanned by the Cartier divisors
  appearing in a decomposition of $D$ is finite-dimensional, hence by
  approximating each coefficient in $D$ by rational numbers,
  \cref{prop:nahomogcont} implies the statement for $D \in \Div_\RR(X)$ by
  continuity.
\end{proof}
\begin{remark}\label{rem:kur291reductions}
  We will repeatedly use the same steps as in the proof of
  \cref{prop:kur291} to prove statements about $\widehat{h}^i(X,D)$ for
  arbitrary $\RR$-Cartier divisors by reducing to the case when $D$ is a Cartier
  divisor.
  If $D$ is an $\RR$-Cartier divisor, we can write $D$ as the limit of
  $\QQ$-Cartier divisors by approximating each coefficient in a decomposition of
  $D$ by rational numbers, and continuity of asymptotic cohomological functions
  (\cref{prop:nahomogcont}) then allows us to reduce to the case when $D$ is a
  $\QQ$-Cartier divisor.
  By homogeneity of asymptotic cohomology functions (\cref{prop:nahomogcont}),
  one can then reduce to the case when $D$ is a Cartier divisor.
\end{remark}
We also need the following:
\begin{proposition}[Asymptotic Serre duality; cf.\ {\cite[Cor.\ 2.11]{Kur06}}]
  \label{prop:asympoticserre}
  Let\index{asymptotic cohomological function, $\widehat{h}^i(X,D)$!duality for}
  $X$ be a projective variety of dimension $n$, and let $D$ be an
  $\RR$-Cartier divisor on $X$.
  Then, for every $0 \le i \le n$, we have
  \[
    \widehat{h}^i(X,D) = \widehat{h}^{n-i}(X,-D).
  \]
\end{proposition}
\begin{proof}
  As in \cref{rem:kur291reductions}, it suffices to consider the
  case when $D$ is a Cartier divisor.
  Let $f\colon Y \to X$ be a regular alteration of degree $d$ \cite[Thm.\
  4.1]{dJ96}.
  We then have
  \[
    \widehat{h}^i(Y,f^*D) = \limsup_{m \to \infty}
    \frac{h^{n-i}\bigl(Y,\cO_Y\bigl(K_Y-f^*(mD)\bigr)\bigr)}{m^n/n!}
    = \widehat{h}^{n-i}(Y,-f^*D)
  \]
  by Serre duality and \cite[Lem.\ 3.2.1]{BGGJKM}, respectively.
  By \cref{prop:kur291}, the left-hand side is equal to
  $d\cdot\widehat{h}^i(X,D)$ and the right-hand side is equal to
  $d\cdot\widehat{h}^{n-i}(X,-D)$, hence the statement follows after dividing by
  $d$.
\end{proof}
\subsection{Restricted volumes}
We will also need the following variant of the volume function $\vol_X(D) =
\widehat{h}^0(X,D)$.
\begin{citeddef}[{\cite[Def.\ 2.1]{ELMNP09}}]\label{def:restrictedvolume}
  Let $X$ be a projective variety of dimension $n$ over a field $k$, and let $V
  \subseteq X$ be a subvariety of dimension $d \ge 1$.
  Consider a $\QQ$-Cartier divisor $D$ on $X$.
  The \textsl{restricted volume}%
  \index{volume, $\vol_X(D)$!restricted, $\vol_{X \vert V}(D)$|textbf}
  of $D$ along $V$ is
  \[
    \gls*{restrictedvolume} \coloneqq \limsup_{m \to \infty} \frac{h^0(X\vert
    V,\cO_X(mD))}{m^d/d!},
  \]
  where
  \[
    H^0\bigl(X\vert V,\cO_X(mD)\bigr) \coloneqq
    \im\Bigl(H^0\bigl(X,\cO_X(mD)\bigr) \to
    H^0\bigl(V,\cO_V(mD\rvert_V)\bigr)\Bigr),\glsadd{restrictedsections}
  \]
  and $h^0(X\vert V,\cO_X(mD)) \coloneqq \dim_k H^0(X\vert V,\cO_X(mD))$.%
  \glsadd{dimrestrictedsections}
\end{citeddef}
\section{Log pairs and log triples}
To simplify notation, we will use the following conventions
for log pairs and log triples.
Recall that if $R$ is a ring, then $R^\circ$\glsadd{rcirc} is the
complement of the union of the minimal primes of $R$.
\begin{definition}\label{def:triples}
  A \textsl{log triple}
  \index{log triple, $(X,\Delta,\mathfrak{a}_\bullet^\lambda)$|textbf}
  $\gls*{triple}$ consists of
  \begin{enumerate}[label=$(\roman*)$,ref=\roman*]
    \item an excellent reduced noetherian scheme $X$;
    \item an $\RR$-Weil divisor $\Delta$ on $X$; and
    \item a symbol $\fa_\bullet^\lambda$ where $\fa_\bullet$ is a graded family
      of ideals on $X$ such that for every open affine subset $U = \Spec R
      \subseteq X$, we have $\fa_m(U) \cap R^\circ \ne \emptyset$ for some $m >
      0$, and $\lambda$ is a real number;
  \end{enumerate}
  where we assume that $X$ is normal and integral if $\Delta \ne 0$.
  We say that $(X,\Delta,\fa_\bullet^\lambda)$ is \textsl{effective} if $\Delta
  \ge 0$ and $\lambda \ge 0$.
  We drop $\lambda$ from our notation if $\lambda = 1$.
  If $\fa_\bullet = \{\fa^m\}_{m=0}^\infty$ for some fixed ideal sheaf $\fa$,
  then we denote the log triple by $(X,\Delta,\fa^t)$ where $t = \lambda$.
  If $X = \Spec R$ for a ring $R$, then we denote
  the log triple by $(R,\Delta,\fa_\bullet^\lambda)$,
  and denote by $R(\lfloor \Delta \rfloor)$ (resp.\
  $R(\lceil \Delta \rceil)$) the ring of global sections of $\cO_{\Spec
  R}(\lfloor \Delta \rfloor)$ (resp.\ $\cO_{\Spec R}(\lceil \Delta \rceil)$).
  A \textsl{log pair}%
  \index{log pair, $(X,\Delta)$ or $(X,\mathfrak{a}_\bullet^\lambda)$|textbf}
  $(X,\Delta)$ (resp.\ $(X,\fa_\bullet^\lambda)$) is a
  log triple such that $\fa_m = \cO_X$ for all $m$ (resp.\ $\Delta = 0$).
  \par We will often call log triples (resp.\ log pairs) \textsl{triples}
  (resp.\ \textsl{pairs}) when there is no risk of confusion.
\end{definition}
\section{Singularities of pairs and triples}
We will need the notion of singularities of log pairs and log triples.
We mostly follow the conventions of \cite[\S3]{Kol97}, with some adaptations to
work with log triples as well.
\begin{definition}[Discrepancies; cf.\ {\cite[Defs.\ 3.3 and 3.4]{Kol97}}]
  \label{def:discrepancy}
  Let $(X,\Delta,\fa^t)$ be a log triple, where $X$ is a normal variety and
  $K_X+\Delta$ is $\RR$-Cartier.
  Write $\Delta = \sum d_iD_i$.
  Suppose $f\colon Y \to X$ is a birational morphism from a normal
  variety $Y$, and choose canonical divisors $K_Y$ and $K_X$ such that $f_*K_Y =
  K_X$.
  In this case, we may write
  \begin{equation}\label{eq:kollar331}
    K_Y = f^*(K_X+\Delta) + \sum_E a(E,X,\Delta)E,
  \end{equation}
  where the $E$ are distinct prime Weil divisors over $X$.
  The right-hand side is not unique since we allow non-exceptional divisors to
  appear on the right-hand side.
  To make the sum on the right-hand side unique, we adopt the convention that a
  non-exceptional divisor $E$ appears on the right-hand side of
  \cref{eq:kollar331} if and only if $E = f_*^{-1}D_i$ for some $i$, in which
  case we set $a(E,X,\Delta) = -d_i$.
  \par For each $E$, the real number $a(E,X,\Delta)$ is called the
  \textsl{discrepancy of $E$} with respect to $(X,\Delta)$.
  Note that if $f'\colon Y' \to X$ is another birational morphism and $E'
  \subseteq Y'$ is the birational transform of $E$, then $a(E,X,\Delta) =
  a(E',X,\Delta)$, hence the discrepancy of $E$ only depends on $E$ and not on
  $Y$.
  \par The \textsl{discrepancy of $E$} with respect to $(X,\Delta,\fa^t)$%
  \index{discrepancy of $E$, $a(E,X,\Delta,\mathfrak{a}^t)$|textbf} is
  \[
    \gls*{discrepancyofE} \coloneqq a(E,X,\Delta) -
    t\cdot\ord_E(\fa)
  \]
  where $\ord_E$\glsadd{ordE} is the divisorial valuation on the function field
  of $X$ defined by $E$.
  \par The \textsl{total discrepancy} of $(X,\Delta,\fa^t)$%
  \index{total discrepancy, $\totaldiscrep(X,\Delta,\mathfrak{a}^t)$|textbf}
  is
  \[
    \totaldiscrep(X,\Delta,\fa^t) \coloneqq \inf_{f\colon Y \to X} \bigl\{
      a(E,X,\Delta,\fa^t) \bigm\vert
      \text{$E$ is a Weil divisor on $Y$}
    \bigr\}\glsadd{totaldiscrepancy}
  \]
  where the infimum runs over all birational morphisms $f\colon Y \to X$
  as above.
\end{definition}
\begin{definition}[Singularities of pairs and triples; cf.\ {\cite[Def.\ 3.5]{Kol97}}]
  Let $(X,\Delta,\fa^t)$ be a log triple, where $X$ is a normal
  variety and $K_X+\Delta$ is $\RR$-Cartier.
  We say that $(X,\Delta,\fa^t)$ is
  \textsl{sub-klt}\index{sub-klt|textbf} if $\totaldiscrep(E,X,\Delta,\fa^t)
  > -1$, and is
  \textsl{sub-log canonical}\index{sub-log canonical|textbf} if
  $\totaldiscrep(E,X,\Delta,\fa^t) \ge -1$.
  A sub-klt (resp.\ sub-log canonical) log triple $(X,\Delta,\fa^t)$ is
  \textsl{klt}\index{klt|textbf} (resp.\ \textsl{log
  canonical}\index{log canonical|textbf}) if $(X,\Delta,\fa^t)$ is
  effective.
  \par When we say that $(X,\Delta,\fa^t)$ is sub-klt (resp.\ sub-log canonical,
  klt, log canonical) at a point $x \in X$, we mean that there exists an open
  neighborhood $U \subseteq X$ of $x$ such that
  $(U,\Delta\rvert_U,\fa\rvert_U^t)$ is sub-klt (resp.\ sub-log canonical, klt,
  log canonical).
  \par We note that \textsl{klt} is short for \textsl{Kawamata log terminal}.
\end{definition}
\par Next, we recall the following:
\begin{definition}
  A \textsl{log resolution}\index{log resolution|textbf}
  of a log triple $(X,\Delta,\fa^t)$ is a projective,
  birational morphism $f\colon Y \to X$, with $Y$ regular, such that
  \begin{enumerate}[label=$(\roman*)$,ref=\roman*]
    \item We have $f^{-1}\fa\cdot\cO_Y = \cO_Y(-F)$ for an effective Cartier
      divisor $F$;
    \item If $\Delta = \sum_i d_iD_i$, and $\widetilde{D}_i$ is the strict
      transform of $D_i$, then the divisor $\ExcDiv(f) + F + \sum_i
      \widetilde{D}_i$ has simple normal crossing support, where $\ExcDiv(f)$ is
      the sum of exceptional divisors of $f$.\label{def:logressnc}
  \end{enumerate}
\end{definition}
Note that log resolutions exist for varieties over a field of characteristic
zero \cites{Hir64a}{Hir64b}, and even for reduced noetherian quasi-excellent
$\QQ$-schemes \cite{Tem18}.
\begin{remark}\label{rem:logresexists}
  Since the existence of log resolutions is not stated explicitly in
  \cite{Tem18}, we describe how this follows from results therein.
  First, apply the
  principalization result in \cite[Thm.\ 1.1.11]{Tem18} to the closed subscheme
  $\Supp\Delta \cup Z(\fa)$ to obtain a resolution $g \colon X' \to X$ such that
  $g^{-1}(\Supp\Delta \cup Z(\fa))$ is a divisor with simple normal crossing
  support.
  Then, one can apply \cite[Thm.\ 1.1.9]{Tem18} to the
  subscheme $g^{-1}(\Supp\Delta \cup Z(\fa)) \cup \ExcDiv(g)$ to ensure
  that the simple normal crossing condition in \cref{def:logressnc} holds.
\end{remark}
The result below says that to check what singularities a given log triple
has, it suffices to check on a log resolution.
\begin{lemma}\label{thm:singpairsvialogres}
  Let $(X,\Delta,\fa^t)$ be a log triple, and consider a log resolution
  $f\colon Y \to X$ for $(X,\Delta,\fa^t)$.
  Choose canonical divisors $K_Y$ and $K_X$ such that $f_*K_Y = K_X$, and write
  \[
    K_Y - f^*(K_X+\Delta) - tF = \sum_E a(E,X,\Delta)E
  \]
  using our conventions in \cref{def:discrepancy} for the right-hand side,
  where $F$ is the effective Cartier divisor defined by $f^{-1}\fa\cdot\cO_Y$.
  Then, we have that
  $(X,\Delta,\fa^t)$ is sub-klt (resp.\ sub-log canonical) if and only if
  $\min_E\{a(E,X,\Delta)\} > -1$ (resp.\ $\ge -1$), where $E$ runs over all
  prime divisors on $Y$.
\end{lemma}
\begin{proof}
  The statement for log pairs is \cite[Cor.\ 3.13]{Kol97}.
  The statement for log triples then follows, since $(X,\Delta,\fa^t)$ is
  sub-klt (resp.\ sub-log canonical) if and only $(Y,\Delta_Y+tF)$ is sub-klt
  (resp.\ sub-log canonical), where $\Delta_Y$ is defined by $K_Y+\Delta_Y =
  f^*(K_X+\Delta)$.
\end{proof}
\subsection{Log canonical thresholds}
We also define the following:
\begin{definition}[Log canonical threshold; {cf.\ \cite[Def.\ 8.1]{Kol97}}]
  \label{def:lct}
  Let $(X,\Delta,\fa)$ be a triple and let $x \in X$ be a closed point.
  The \textsl{log canonical threshold}%
  \index{log canonical threshold, $\lct_x((X,\Delta);\fa)$|textbf}
  of $(X,\Delta)$ at $x$ with respect
  to $\fa$ is
  \[
    \lct_x\bigl( (X,\Delta);\fa \bigr) \coloneqq \sup\bigl\{c \in \RR_{\ge0}
    \bigm\vert
    \text{$(X,\Delta,\fa^c)$ is sub-log canonical at $x$}
    \bigr\},\glsadd{lct}
  \]
  where if $(X,\Delta)$ is not sub-log canonical, then we set
  $\lct_x((X,\Delta);\fa) = -\infty$.
  If $\fa = \cO_X(-D)$ for a Cartier divisor $D$, then we denote
  \[
    \lct_x\bigl( (X,\Delta);D\bigr) \coloneqq
    \lct_x\bigl((X,\Delta);\cO_X(-D)\bigr).
  \]
  We also drop $\Delta$ from our notation if $\Delta = 0$.
\end{definition}
Log canonical thresholds can be computed on a log resolution:
\begin{proposition}[cf.\ {\cite[Prop.\ 8.5]{Kol97}}]\label{prop:lctvialogres}
  Let\index{log canonical threshold, $\lct_x((X,\Delta);\fa)$!via log resolutions}
  $(X,\Delta,\fa)$ be a log triple such that $(X,\Delta)$ is sub-log
  canonical, and let $x \in X$ be a closed point.
  Consider a log resolution $f\colon Y \to X$ for $(X,\Delta,\fa)$.
  Using our conventions in \cref{def:discrepancy}, write
  \[
    K_Y - f^*(K_X+\Delta) = \sum_j a_jE_j \qquad \text{and} \qquad F = \sum_j
    b_jE_j,
  \]
  where $F$ is the effective Cartier divisor defined by $f^{-1}\fa\cdot\cO_Y$.
  Then,
  \[
    \lct_x\bigl( (X,\Delta);\fa \bigr) = \min_{\{j \mid f(E_j)=\{x\}\}} \biggl\{
    \frac{a_j+1}{b_j} \biggr\}.
  \]
\end{proposition}
\begin{proof}
  This follows from \cref{thm:singpairsvialogres}, since $(X,\Delta,\fa^c)$ is
  sub-log canonical if and only if $a_j-cb_j\ge-1$ for all $j$.
\end{proof}
We compute one example of a log canonical threshold.
\begin{example}[Cuspidal cubic]\label{ex:cuspidalcubic}
  Let\index{cuspidal cubic|(}
  $k$ be an algebraically closed field, and consider\index{cuspidal cubic}
  the cuspidal cubic $C = \{x^2+y^3=0\} \subseteq \AA^2_k$.
  \begin{figure}[t]
    \centering
    \tikzexternalenable
    \begin{tikzpicture}[scale=1.2]
      \draw[domain=-1:1,variable=\t] plot ({(\t)^2},{(\t)^3});
      \node[anchor=west] at (1,-1) {$C$};
      
      \draw[->] (2.25,0) -- (1.5,0);
      
      \draw[domain=-0.75:0.75,variable=\t] plot ({2*(\t)^2+2.75},{(\t)});
      \node[anchor=west] at (3.875,-0.75) {$\tilde{C}$};
      \draw (2.75,-1) -- (2.75,1);
      \node[anchor=north] at (2.75,-1) {\footnotesize $E_1$};

      \draw (6.5,-1) -- (6.5,1);
      \node[anchor=north] at (6.5,-1) {\footnotesize $E_1$};
      \draw (5.5,0.7) -- (7.5,-0.7) node[anchor=west] {\footnotesize $\tilde{C}$};
      \draw (5.5,-0.7) -- (7.5,0.7) node[anchor=west] {\footnotesize $E_2$};
      \draw[->] (5,0) -- (4.25,0);

      \draw (9.25,0) -- (11.75,0) node[anchor=west] {\footnotesize $E_3$};
      \draw (9.5,1) -- (9.5,-1) node[anchor=north] {\footnotesize $E_1$};
      \draw (10.5,1) -- (10.5,-1) node[anchor=north,yshift=1.25pt] {\footnotesize $\tilde{C}$};
      \draw (11.5,1) -- (11.5,-1) node[anchor=north] {\footnotesize $E_2$};
      \draw[->] (8.75,0) -- (8,0);

      \node (a) at (0.5,-1.75) {$\AA^2_k$};
      \node (b) at (3.25,-1.75) {$X$};
      \node (c) at (6.5,-1.75) {$Y$};
      \node (d) at (10.5,-1.75) {$W$};

      \path[->] (d) edge (c) (c) edge (b) (b) edge (a);

      \path[->] (d) edge[bend left=10,transform canvas={yshift=-0.15cm}]
      node[fill=white] {$\pi$} (a);
    \end{tikzpicture}
    \tikzexternaldisable
    \caption{Log resolution of a cuspidal cubic}
    \label{fig:cuspidalcubic}
    \index{cuspidal cubic|ff{}}
  \end{figure}
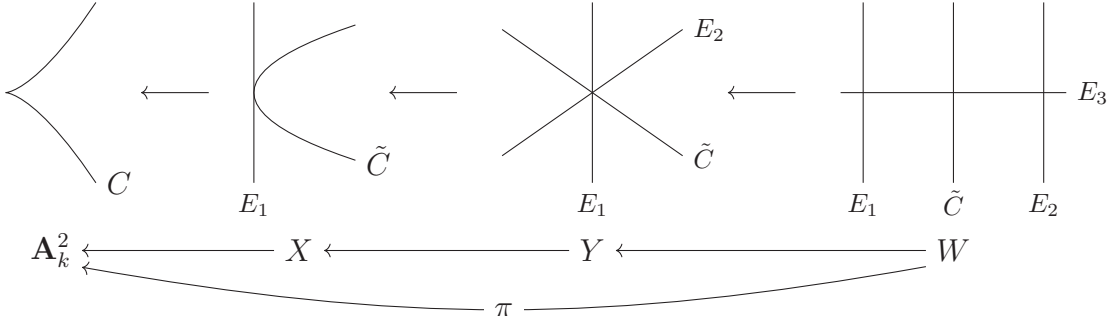
  We would like to compute the log canonical threshold $\lct_0(\AA^2_k;C)$,
  where $0 \in \AA^2_k$ is the origin.
  First, there is a log resolution $\pi\colon W \to \AA^2_k$ as in
  \cref{fig:cuspidalcubic}, which is constructed as a sequence of blowups at the
  intersection of the divisors shown, where
  \[
    \pi^*C = 2E_1 + 3E_2 + 6E_3 + \widetilde{C} \qquad \text{and} \qquad
    K_W - \pi^*K_{\AA^2_k} = E_1 + 2E_2 + 4E_3.
  \]
  By \cref{prop:lctvialogres}, we then have
  \[
    \lct_0(\AA^2_k;C) = \min\biggl\{
      \frac{1+1}{2},\frac{2+1}{3},\frac{4+1}{6}
    \biggr\} = \frac{5}{6}.\index{cuspidal cubic|)}
  \]
\end{example}
\section{Multiplier ideals}
We briefly review the theory of multiplier ideals.
Multiplier ideals were first defined by Nadel\index{Nadel, Alan Michael}
in the analytic setting \cite[Def.\ 2.5]{Nad90}.
We recommend \cite[Pt.\ 3]{Laz04b} for an overview on this topic.
We work in the more general setting of excellent $\QQ$-schemes, following
\cite[\S2]{dFM09} and \cite[App.\ A]{JM12}.
\begin{citeddef}[{\cite[Def.\ 9.3.60]{Laz04b}}]
  Let $(X,\Delta,\fa^t)$ be an effective log triple such that $X$ is a
  $\QQ$-scheme and such that $K_X+\Delta$ is $\RR$-Cartier.
  Fix a log resolution $f\colon Y \to X$ of $(X,\Delta,\fa^t)$ so that
  $\fa \cdot \cO_Y = \cO_{\widetilde{X}}(-D)$ for an effective divisor $D$.
  Note that such a log resolution exists by \cref{rem:logresexists}.
  The \textsl{multiplier ideal}%
  \index{multiplier ideal, $\mathcal{J}(X,\Delta,\mathfrak{a}^t)$|textbf} is
  \begin{align*}
    \gls*{multiplierideal} &\coloneqq
    f_*\cO_Y \bigl(K_Y - \bigl\lfloor f^*(K_X+\Delta) +
    tD \bigr\rfloor \bigr).
    \intertext{This definition does not depend on the choice of log resolution;
    see \cites[Thm.\ 9.2.18]{Laz04b}[Prop.\ 2.2]{dFM09}.\endgraf
    If $(X,\Delta,\fb^s)$ is another effective log triple on $X$, then we can
    analogously define the multiplier ideal}
    \cJ(X,\Delta,\fa^t\cdot\fb^s) &\coloneqq
    f_*\cO_Y \bigl(K_Y - \bigl\lfloor f^*(K_X+\Delta) +
    tD + sF \bigr\rfloor \bigr).
  \end{align*}
  where $f$ is a simultaneous log resolution for $(X,\Delta,\fa)$ and
  $(X,\Delta,\fb)$ such that $\fb \cdot \cO_Y = \cO_{\widetilde{X}}(-F)$ for an
  effective divisor $F$.
\end{citeddef}
Note that if $(X,\Delta,\fa^t)$ is an effective log triple such that $X$ is a
normal variety and $K_X+\Delta$ is $\RR$-Cartier, then $(X,\Delta,\fa^t)$ is klt
if and only if $\cJ(X,\Delta,\fa^t) = \cO_X$ \cite[p.\ 165]{Laz04a}.
See \cite[Ex.\ 9.2.30, Prop.\ 9.2.32, and p.\ 185]{Laz04b} for some other basic
properties of multiplier ideals, which carry over to the setting of excellent
$\QQ$-schemes.%
\index{multiplier ideal, $\mathcal{J}(X,\Delta,\mathfrak{a}^t)$!basic properties of}
Some more subtle properties in our context are checked in \cites[Prop.\
2.3]{dFM09}[App.\ A]{JM12}.
In particular, the subadditivity theorem of
Demailly\index{Demailly, Jean-Pierre}--%
Ein\index{Ein, Lawrence}--%
Lazarsfeld\index{Lazarsfeld, Robert} \cite[Thm.\ on p.\ 137]{DEL00} holds for
regular excellent $\QQ$-schemes; see \cite[Thm.\ A.2]{JM12}.%
\index{subadditivity theorem!for multiplier ideals}
\par We will also need an asymptotic version of multiplier ideals for graded
families of ideals.
Note that an asymptotic version of multiplier ideals first appeared in the work
of Siu \cite[pp.\ 668--669]{Siu98} in the analytic setting.
\begin{definition}[see {\cite[Def.\ 11.1.15]{Laz04a}}]
  \label{def:asymptoticmultideal}
  Let $(X,\Delta,\fa_\bullet^\lambda)$ be an effective log triple such that $X$
  is a $\QQ$-scheme and such that $K_X+\Delta$ is $\RR$-Cartier.
  If $m$ and $r$ are positive integers, then
  \begin{equation*}
    \cJ(X,\Delta,\fa_{m}^{\lambda/m})
    = \cJ\bigl(X,\Delta,(\fa_m^r)^{\lambda/(mr)}\bigr)
    \subseteq \cJ\bigl(X,\Delta,\fa_{mr}^{\lambda/(mr)}\bigr),
  \end{equation*}
  by \cite[Rem.\ 9.2.4 and Prop.\ 9.2.32$(iii)$]{Laz04b}, and by the graded
  property $\fa_m^r \subseteq \fa_{mr}$.
  Thus, the set of ideals
  \begin{equation}\label{eq:multidealposet}
    \bigl\{\cJ(X,\Delta,\fa_{m}^{\lambda/m})\bigr\}_{m=1}^\infty
  \end{equation}
  is partially ordered, and has a unique maximal element by the noetherian
  property that coincides with $\cJ(X,\Delta,\fa_m^{\lambda/m})$
  for $m$ sufficiently large and divisible.
  The \textsl{asymptotic multiplier ideal}%
  \index{asymptotic multiplier ideal, $\mathcal{J}(X,\Delta,\mathfrak{a}_\bullet^\lambda)$|textbf}
  \[
    \gls*{asymptoticmultiplierideal} \subseteq \cO_X
  \]
  is the maximal element of the partially ordered set \eqref{eq:multidealposet}.
  \par If $(X,\Delta,\fb_\bullet^\lambda)$ is another effective log triple on
  $X$, then we can analogously define the multiplier ideal
  $\cJ(X,\Delta,\fa_\bullet^\lambda\cdot\fb_\bullet^\mu)$.
\end{definition}
The following examples of multiplier ideals will be the most useful in our
applications.
\begin{example}[see {\cites[Defs.\ 9.2.10 and 11.1.2]{Laz04b}}]
  \label{ex:multidealdiv}
  Suppose $(X,\Delta)$ is an effective log pair where $X$ is a complete scheme
  over a field of characteristic zero.
  If $D$ is a Cartier divisor such that $H^0(X,\cO_X(mD)) \ne 0$ for some
  positive integer $m$, then for every real number $t \ge 0$, we set%
  \index{multiplier ideal, $\mathcal{J}(X,\Delta,\mathfrak{a}^t)$!of a divisor, $\mathcal{J}(X,\Delta,t \cdot \lvert D \rvert)$|textbf} 
  \begin{align*}
    \cJ\bigl(X,\Delta,t \cdot \lvert D \rvert\bigr) &\coloneqq
    \cJ\bigl(X,\Delta,\fb\bigl(\lvert D \rvert\bigr)^t\bigr).%
    \glsadd{multidealofls}
    \intertext{If $D$ is a $\QQ$-Cartier divisor such that $H^0(X,\cO_X(mD)) \ne
    0$ for some sufficiently divisible $m > 0$, then for every real number
    $\lambda \ge 0$, we set}
    \cJ\bigl(X,\Delta,\lambda \cdot \lVert D \rVert\bigr) &\coloneqq
    \cJ\bigl(X,\Delta,\fa_\bullet(D)^\lambda\bigr),\glsadd{asymmultidealofls}
  \end{align*}
  \index{asymptotic multiplier ideal, $\mathcal{J}(X,\Delta,\mathfrak{a}_\bullet^\lambda)$!of a $\mathbf{Q}$-divisor, $\mathcal{J}(X,\Delta,\lambda \cdot \lVert D \rVert)$|textbf}%
  where $\fa_\bullet(D)$ is the graded family of ideals defined in
  \cref{ex:triplelinsys}.
\end{example}
Finally, we have the following uniform global generation result for the
(asymptotic) multiplier ideals defined in Example \ref{ex:multidealdiv}.
\begin{theorem}[cf.\ {\cite[Prop.\ 9.4.26 and Cor.\ 11.2.13]{Laz04b}}]
  \label{thm:multidealuniformgg}
  Let\index{multiplier ideal, $\mathcal{J}(X,\Delta,\mathfrak{a}^t)$!uniform global generation of|(}%
  \index{asymptotic multiplier ideal, $\mathcal{J}(X,\Delta,\mathfrak{a}_\bullet^\lambda)$!uniform global generation of|(}
  $(X,\Delta)$ be an effective log pair where $X$ is a normal projective
  variety over a field $k$ of characteristic zero and $\Delta$ is an
  effective $\QQ$-Weil divisor such that $K_X+\Delta$ is $\QQ$-Cartier.
  Let $D$, $L$, and $H$ be Cartier divisors on $X$ such that $H$ is ample and
  free.
  If $\lambda$ is a non-negative real number such that
  $L - (K_X + \Delta + \lambda \cdot D)$
  is ample, then the sheaves
  \[
    \cJ\bigl(X,\Delta,\lambda\cdot\lvert D \rvert\bigr) \otimes \cO_X(L+dH)
    \quad \text{and} \quad
    \cJ\bigl(X,\Delta,\lambda\cdot\lVert D \rVert\bigr) \otimes \cO_X(L+dH)
  \]
  are globally generated for every integer $d > \dim X$.%
  \index{asymptotic multiplier ideal, $\mathcal{J}(X,\Delta,\mathfrak{a}_\bullet^\lambda)$!uniform global generation of|)}%
  \index{multiplier ideal, $\mathcal{J}(X,\Delta,\mathfrak{a}^t)$!uniform global generation of|)}
\end{theorem}
\begin{proof}
  By choosing $n$ sufficiently divisible such that
  \[
    \cJ\bigl(X,\Delta,\lambda\cdot\lVert D \rVert\bigr)
    = \cJ\bigl(X,\Delta,(\lambda/n)\cdot\lvert nD \rvert\bigr),
  \]
  it suffices to consider the case for the usual multiplier ideals.
  Global generation follows from Nadel vanishing\index{vanishing theorem!Nadel}
  when $k$ is algebraically
  closed \cite[Prop.\ 9.4.26 and Rem.\ 9.4.27]{Laz04b}, hence it suffices to
  reduce to this case.
  \par Let $\pi\colon \overline{X} \coloneqq X \times_k \overline{k} \to X$
  denote the base extension to the algebraic closure $\overline{k}$ of $k$.
  Since the extension $k \subseteq \overline{k}$ is faithfully flat, the
  sheaf $\cJ(X,\Delta,\lambda\cdot\lvert D \rvert) \otimes \cO_X(L+dH)$ is
  globally generated if its pullback to $\overline{X}$ is globally generated.
  The pullback to $\overline{X}$ is isomorphic to
  \begin{equation}\label{eq:multidealafterbasechange}
    \cJ\bigl(\overline{X},\pi^*\Delta,\lambda\cdot\lvert \pi^*D \rvert\bigr)
    \otimes \cO_{\overline{X}}(\pi^*L+d\pi^*H)
  \end{equation}
  since the formation of multiplier ideals commutes with faithfully flat base
  change \cite[Prop.\ 1.9]{JM12}.
  Moreover, the $\RR$-Cartier divisor
  \[
    \pi^*\bigl(L - (K_X + \Delta + \lambda \cdot D)\bigr)
    = \pi^*L - (K_{\overline{X}} + \pi^*\Delta + \lambda \cdot \pi^*D)
  \]
  is ample and $\pi^*H$ is ample and free by faithfully flat base change
  \cite[Cor.\ 2.7.2]{EGAIV2}.
  To apply the special case when $k$ is algebraically closed, we note that while
  $\overline{X}$ may not be irreducible, it is still the disjoint union of
  normal varieties \cite[Rem.\ on pp.\ 64--65]{Mat89}.
  Thus, by applying \cite[Prop.\ 9.4.26 and Rem.\ 9.4.27]{Laz04b} to each
  connected component of $\overline{X}$ individually, we see that the sheaf in
  \cref{eq:multidealafterbasechange} is globally generated.
\end{proof}

\chapter{Preliminaries in positive characteristic}\label{chap:poschar}
In this chapter, we review some preliminaries on commutative algebra and
algebraic geometry in positive characteristic.
See \cite{ST12} and \cite{TW18} for overviews on the topic.
See also \cite{PST17} and \cite{Pat18} for more geometric applications.
\par The only new material is a new, short proof of the subadditivity theorem
for test ideals (\cref{thm:testidealsubadditivity}), and some material on
$F$-pure triples in \cref{sect:fsingpairs}.
\section{Conventions on the Frobenius morphism}
\label{sect:frobconventions}
We start by establishing our conventions for the Frobenius morphism.
\begin{definition}
  Let $X$ be a scheme of characteristic $p > 0$.
  The \textsl{(absolute) Frobenius
  morphism}\index{Frobenius morphism, $F$|textbf} is the morphism
  $F\colon X \to X$ of schemes given by the identity on points
  and the $p$-power map
  \[
    \begin{tikzcd}[row sep=0,column sep=1.475em]
      \cO_X(U) \rar & F_*\cO_X(U)\\
      f \rar[mapsto] & f^p
    \end{tikzcd}
  \]
  on structure sheaves for every open subset $U \subseteq X$.
  If $R$ is a ring of characteristic $p > 0$, we denote the corresponding ring
  homomorphism by $F\colon R \to F_*R$.
  For every integer $e \ge 0$, the $e$th iterate of the Frobenius morphism
  is denoted by $\gls*{frobenius}\colon X \to X$ and $F^e\colon R\to F^e_*R$.
  If $\fa \subseteq \cO_X$ is a coherent ideal sheaf, we define the
  \textsl{$e$th Frobenius power} $\fa^{[p^e]}$\glsadd{frobeniuspower}
  \index{Frobenius power, $\mathfrak{a}^{[p^e]}$|textbf} to be the inverse
  image of $\fa$ under the
  $e$th iterate of the Frobenius morphism.
  Locally, if $\fa$ is generated by $(h_i)_{i \in I}$, then $\fa^{[p^e]}$ is
  generated by $(h_i^{p^e})_{i \in I}$.
\end{definition}
We note that the notation $\gls*{frobpushforward}$ is used to remind us that
the $R$-algebra structure on $F^e_*R$ is given by the ring homomorphism $F^e$.
\section{The pigeonhole principle}
A\index{pigeonhole principle|(}
surprisingly important fact in this thesis is the following combinatorial
result based on the pigeonhole principle.
\begin{lemma}[cf.\ {\cite[Lem.\ 2.4$(a)$]{HH02}}]\label{lem:monomials}
  Let $R$ be a commutative ring of characteristic $p > 0$.
  Then, for any ideal $\fa$ generated by $n$ elements and for all non-negative
  integers $e$ and $\ell$, we have the sequence of inclusions
  \begin{equation}
    \fa^{\ell p^e+n(p^e-1)+1} \subseteq (\fa^{\ell+1})^{[p^e]} \subseteq
    \fa^{(\ell+1)p^e}.\label{eq:calemincl}
  \end{equation}
  Moreover, if $R$ is a regular local ring of dimension $n$ and
  $\fm$ is the maximal ideal of $R$, then
  \[
    \fm^{\ell p^e+n(p^e-1)} \not\subseteq (\fm^{\ell+1})^{[p^e]}.
  \]
\end{lemma}
\begin{proof}
  The second inclusion in \cref{eq:calemincl} is clear by the definition of
  Frobenius powers.
  We want to show the first inclusion.
  Let $y_1,y_2,\ldots,y_n$ be a set of generators for
  $\fa$.
  The ideal $\fa^{\ell p^e+n(p^e-1)+1}$ is generated
  by all elements of the form
  \begin{alignat}{4}
    &\prod_{i=1}^{n} y_i^{a_i} &&\quad& \text{such that} &\quad 
    \sum_{i=1}^n a_i &{}={}& \ell p^e+n(p^e-1)+1,
    \label{eq:monoordpower}
    \intertext{and the ideal $(\fa^{\ell+1})^{[p^e]}$ is generated by all
    elements of the form}
    &\prod_{i=1}^{n} y_i^{p^e b_i} &&\quad& \text{such that} &\quad 
    \sum_{i=1}^n b_i &{}={}& \ell+1.\label{eq:monomixedpower}
  \end{alignat}
  We want to show that the elements \cref{eq:monoordpower} are divisible by
  some elements of the form \cref{eq:monomixedpower}.
  By the division algorithm, we may write $a_i = a_{i,0} + p^e a_{i}'$ for some
  non-negative integers $a_{i,0}$ and $a_i'$
  such that $0 \le a_{i,0} \le p^{e}-1$.
  We then have
  \[
    \prod_{i=1}^{n} y_i^{a_i} = \prod_{i=1}^n y_i^{a_{i,0}} \cdot
    \prod_{i=1}^n y_i^{p^e a_{i}'},
  \]
  and since $a_{i,0} \le p^e - 1$, we have that $\sum_{i=1}^n a_{i,0} \le n(p^e -
  1)$.
  Thus, we have the inequality
  \[
    \ell p^e + n(p^e - 1) + 1 = \sum_{i=1}^n a_i \le n(p^e - 1) +
    \sum_{i=1}^n p^e a_i',
  \]
  which implies $\ell + p^{-e} \le \sum_{i=1}^n a_i'$. Since the right-hand side
  of this inequality is an integer, we have that $\ell + 1 \le
  \sum_{i=1}^n a_i'$, i.e., the element $\prod_{i=1}^n y_i^{p^ea_i'}$ is
  divisible by one of the form \cref{eq:monomixedpower}. Thus, each element of
  the form in \cref{eq:monoordpower} is divisible by one of the form in
  \cref{eq:monomixedpower}.
  \par Now suppose $R$ is a regular local ring of dimension $n$, and $\fm$
  is the maximal ideal of $R$. Let $y_1,y_2,\ldots,y_n$ be a regular system of
  parameters. Then, we have
  \[
    y_{i_0}^{\ell p^e} \cdot \prod_{i=1}^n y_i^{p^e-1} \in
    \fm^{\ell p^e+n(p^e-1)}
  \]
  for all $i_0 \in \{1,2,\ldots,n\}$. This monomial
  does not lie in $(\fm^{\ell+1})^{[p^e]}$ since its image is not in the
  extension of $(\fm^{\ell+1})^{[p^e]}$ in the completion of $R$ at $\fm$,
  which is isomorphic to a formal power series ring with variables
  $y_1,y_2,\ldots,y_n$ by the Cohen structure theorem.
\end{proof}
We moreover show that asymptotically, the number $n$ of elements generating
$\fa$ can be replaced by the analytic spread\index{analytic spread} of $\fa$.
See \cite[Def.\ 5.1.5]{HS06} for the definition of analytic spread.
\begin{lemma}\label{lem:mur23improved}
  Let $(R,\fm,k)$ be a noetherian local ring of characteristic $p > 0$.
  Then, for every ideal $\fa$ of analytic spread $h$,
  there exists an integer $t \ge 0$ such that for all non-negative integers
  $e$ and $\ell$, we have the sequence of inclusions
  \begin{equation}\label{eq:inclusions}
    \fa^{\ell p^e+h(p^e-1)+1+t} \subseteq (\fa^{\ell+1})^{[p^e]} \subseteq
    \fa^{(\ell+1)p^e}.
  \end{equation}
  In particular, if $\fa = \fm$, then \cref{eq:inclusions} holds for $h = \dim
  R$.
\end{lemma}
\begin{proof}
  The right inclusion in \cref{eq:inclusions} is clear as in
  \cref{lem:monomials}.
  It therefore suffices to prove the left inclusion in \cref{eq:inclusions}.
  We first reduce to the case when $k$ is infinite.
  Consider the ring $S = R[x]_{\fm R[x]}$ as in \cite[\S8.4]{HS06}.
  Then, $S$ is a noetherian local ring of characteristic $p > 0$ such that $R
  \subseteq S$ is faithfully flat and $S/\fm S \simeq k(x)$ is
  infinite.
  Since we can check the inclusions in \cref{eq:inclusions} after a faithfully
  flat extension \cite[Thm.\ 7.5$(ii)$]{Mat89}, and since analytic spread does
  not change after passing to $S$ \cite[Lem.\ 8.4.2(4)]{HS06}, we can replace
  $R$ with $S$ to assume that $k$ is infinite.
  \par We now prove the left inclusion in \cref{eq:inclusions} under the
  assumption that $k$ is infinite.
  Recall that since $k$ is infinite, there exists
  an ideal $\fq \subseteq \fa$ called a \textsl{minimal reduction
  ideal}\index{minimal reduction ideal} and an
  integer $t > 0$ such that $\fq$ is generated by $h$ elements, and
  $\fa^{s+t} = \fq^s\cdot\fa^t$ for every integer $s \ge 0$; see \cite[Def.\
  1.2.1 and Prop.\ 8.3.7]{HS06}.
  Setting $s = \ell p^e + h(p^e-1)+1$, we have
  \[
    \fa^{\ell p^e + h(p^e-1)+1+t} = \fq^{\ell p^e + h(p^e-1)+1} \cdot \fa^t
    \subseteq (\fq^{\ell+1})^{[p^e]} \cdot \fa^t \subseteq
    (\fa^{\ell+1})^{[p^e]}
  \]
  for all non-negative integers $e$ and $\ell$, where the first
  inclusion holds by \cref{lem:monomials}.
  The special case for $\fa = \fm$ follows from \cite[Cor.\ 8.3.9]{HS06}.%
\index{pigeonhole principle|)}
\end{proof}
\section{\textit{F}-finite schemes}
As mentioned in \cref{sect:poschardifficult}, in positive characteristic, one
often needs to restrict or reduce to the case when the Frobenius morphism is
finite.
We isolate this class of schemes.
\begin{definition}
  Let $X$ be a scheme of characteristic $p > 0$.
  We say that $X$ is \textsl{$F$-finite}\index{F-finite@$F$-finite|textbf} if
  the (absolute) Frobenius morphism $F\colon X \to X$ is finite.
  We say that a ring $R$ of characteristic $p > 0$ is \textsl{$F$-finite} if
  $\Spec R$ is $F$-finite, or equivalently if $F\colon R \to F_*R$ is
  module-finite.
\end{definition}
\par Note that a field $k$ is $F$-finite if and only if $[k:k^p] <
\infty$.\index{F-finite@$F$-finite!field}
$F$-finite schemes are ubiquitous in geometric contexts because of the
following:
\begin{example}[see {\cites[p.\ 999]{Kun76}[Ex.\ 2.1]{BMS08}}]
  \label{ex:eftoverffin}
  If $X$ is a
  scheme that is locally essentially of finite type over an
  $F$-finite scheme of characteristic $p > 0$, then $X$ is $F$-finite.
  In particular, schemes essentially of finite type over perfect or $F$-finite
  fields are $F$-finite.
\end{example}
\par If a scheme $X$ of characteristic $p > 0$ is $F$-finite, then Grothendieck
duality (\cref{thm:nayakshriek}) can be applied to the Frobenius morphism since
it is finite.
The $F$-finiteness condition implies other desirable conditions as well.
\begin{citedthm}[{\cites[Thm.\ 2.5]{Kun76}[Rem.\ 13.6]{Gab04}}]
  \label{thm:ffiniteaffine}
  Let\index{F-finite@$F$-finite!implies excellent}%
  \index{F-finite@$F$-finite!has a dualizing complex}%
  \index{dualizing complex, $\omega_X^\bullet$!on $F$-finite schemes}
  $R$ be a noetherian $F$-finite ring of characteristic $p > 0$.
  Then, $R$ is excellent and is isomorphic to a quotient of a regular
  ring of finite Krull dimension.
  In particular, $R$ admits a dualizing complex.
\end{citedthm}
See \cite[Def.\ on p.\ 258]{Har66} for the definition of a dualizing
complex.\index{dualizing complex, $\omega_X^\bullet$}
\section{\emph{F}-singularities of pairs and triples}\label{sect:fsingpairs}
We now define $F$-singularities for log triples in the sense of
\cref{def:triples}.
These are common generalizations of the notions for log pairs $(X,\Delta)$ and
$(X,\fa^t)$ due to Hara\index{Hara, Nobuo}--Watanabe\index{Watanabe, Kei-ichi}
\cite{HW02} and Takagi\index{Takagi, Shunsuke} \cite{Tak04inversion},
respectively.
While an equivalent definition of strong $F$-regular triples has appeared
before (see \cref{rem:rounding}), the definition of $F$-pure triples appears to
be new.
\par We note that we assume $F$-finiteness throughout.
See \cref{app:nonffinfsings} for an overview on $F$-singularities for
rings, where we work without $F$-finiteness assumptions. 
\begin{definition}[cf.\ {\cites[Def.\ 2.1]{HW02}[Def.\ 3.1]{Tak04inversion}}]
  \label{def:fsingpair}
  Let $(R,\Delta,\fa^t)$ be an effective log triple such that $R$ is an
  $F$-finite local ring of characteristic $p > 0$.
  \begin{enumerate}[label=$(\alph*)$,ref=\alph*]
    \item\label{def:fpurepair}
      The triple $(R,\Delta,\fa^t)$ is
      \textsl{$F$-pure}\index{F-pure@$F$-pure!triple|textbf} if there exists an
      integer $e' > 0$ such that for all $e \ge e'$, there exists an element $d
      \in \fa^{\lfloor (p^e-1)t \rfloor}$ for which the composition
      \begin{equation}\label{eq:fpurecomp}
        R \overset{F^e}{\longrightarrow} F^e_*R \longhookrightarrow
        F^e_*R\bigl(\lfloor (p^e-1)\Delta \rfloor\bigr) \xrightarrow{F^e_*(-
        \cdot d)} F^e_*R\bigl(\lfloor (p^e-1)\Delta \rfloor\bigr)
      \end{equation}
      splits as an $R$-module homomorphism.
    \item\label{def:sharpfpurepair}
      The triple $(R,\Delta,\fa^t)$ is
      \textsl{sharply $F$-pure}\index{F-pure@$F$-pure!sharply!triple|textbf} if
      there exists an
      integer $e > 0$ and an element $d
      \in \fa^{\lceil (p^e-1)t \rceil}$ for which the composition
      \begin{equation}\label{eq:sharpfpurecomp}
        R \overset{F^e}{\longrightarrow} F^e_*R \longhookrightarrow
        F^e_*R\bigl(\lceil (p^e-1)\Delta \rceil\bigr) \xrightarrow{F^e_*(-
        \cdot d)} F^e_*R\bigl(\lceil (p^e-1)\Delta \rceil\bigr)
      \end{equation}
      splits as an $R$-module homomorphism.
    \item\label{def:sfrpair}
      The triple $(R,\Delta,\fa^t)$ is \textsl{strongly $F$-regular}%
      \index{F-regular@$F$-regular!strongly!triple|textbf} if for all
      $c \in R^\circ$, there exists an integer $e > 0$ and an element $d \in
      \fa^{\lfloor(p^e-1)t\rfloor}$ for which the composition
      \begin{equation}\label{eq:sfrcomp}
        R \overset{F^e}{\longrightarrow} F^e_*R \longhookrightarrow
        F^e_*R\bigl(\lfloor (p^e-1)\Delta \rfloor\bigr) \xrightarrow{F^e_*(-
        \cdot cd)} F^e_*R\bigl(\lfloor (p^e-1)\Delta \rfloor\bigr)
      \end{equation}
      splits as an $R$-module homomorphism.
  \end{enumerate}
  \par Now suppose that $(X,\Delta,\fa^t)$ is an effective log triple such that
  $X$ is an $F$-finite scheme of characteristic $p > 0$, and let $x \in X$ be a
  point.
  The triple $(X,\Delta,\fa^t)$ is \textsl{$F$-pure} (resp.\
  \textsl{sharply $F$-pure}, \textsl{strongly $F$-regular}) at $x$ if the
  localized triple
  \(
    (\cO_{X,x},\Delta\rvert_{\Spec
    \cO_{X,x}},\fa_x^t)
  \)
  is $F$-pure (resp.\ sharply $F$-pure, strongly $F$-regular).
  The triple $(X,\Delta,\fa^t)$ is \textsl{$F$-pure} (resp.\ \textsl{sharply
  $F$-pure}, \textsl{strongly
  $F$-regular}) if it is $F$-pure (resp.\ strongly $F$-regular) at every point
  $x \in X$.
\end{definition}
\begin{remark}
  A triple $(R,0,R^1)$ as in \cref{def:fsingpair} is $F$-pure if
  and only if $R$ is $F$-pure in the sense of Hochster\index{Hochster, Melvin}%
  --Roberts\index{Roberts, Joel L.} (since
  $F$-purity and $F$-splitting coincide $F$-finite rings; see
  \cref{fig:fsingsdiagram}), and is strongly $F$-regular if and only if $R$ is
  strongly $F$-regular in the sense of Hochster\index{Hochster, Melvin}--Huneke%
  \index{Hochster, Craig}
  (\cref{def:fsingssplit}\cref{def:fsingssplitreg}).
\end{remark}
\par We collect some basic properties of $F$-singularities for triples.
\begin{proposition}[cf.\ {\cites[Prop.\ 2.2]{HW02}[Prop.\ 3.3]{Tak04inversion}}]
  \label{prop:fsingbasic}
  Let $(R,\Delta,\fa^t)$ be an effective log triple such that $R$ is an
  $F$-finite local ring of characteristic $p > 0$.
  \begin{enumerate}[label=$(\roman*)$,ref=\roman*]
    \item If $(R,\Delta,\fa^t)$ is $F$-pure (resp.\ sharply $F$-pure, strongly
      $F$-regular), then
      so is $(R,\Delta',\fb^s)$ for every triple such that $\Delta' \le \Delta$,
      $\fb \supseteq \fa$, and $s \in [0,t]$.\label{prop:fsingssmaller}
    \item If $(R,\Delta,\fa^t)$ is $F$-pure, then $\lceil \Delta \rceil$ is
      reduced, i.e., the nonzero coefficients of $\lceil \Delta \rceil$ are
      equal to $1$.\label{prop:fpurered}
    \item $(R,\Delta,\fa^t)$ is strongly $F$-regular if and only if for all $c
      \in R^\circ$, there exists an integer $e' > 0$ such that for all $e \ge
      e'$, there exists $d \in \fa^{\lceil p^e t\rceil}$ for which the
      composition
      \begin{equation}\label{eq:sfrceilingscomp}
        R \overset{F^e}{\longrightarrow} F^e_*R \longhookrightarrow
        F^e_*R\bigl(\lceil p^e\Delta \rceil\bigr) \xrightarrow{F^e_*(- \cdot
        cd)} F^e_*R\bigl(\lceil p^e\Delta \rceil\bigr)
      \end{equation}
      splits as an $R$-module homomorphism.\label{prop:sfrceilings}
    \item We have the implications\label{prop:sfrimpliesfpure}
      \[
        \begin{tikzcd}[column sep=-1em]
          \text{strongly $F$-regular}
          \arrow[Rightarrow]{rr}\arrow[Rightarrow]{dr} & & \text{sharply
          $F$-pure}\arrow[dashed,Rightarrow]{dl}\\
          & \text{$F$-pure}
        \end{tikzcd}
      \]
      where the dashed implication holds when $\Delta$ is Cartier and $\fa$ is
      locally principal.
  \end{enumerate}
\end{proposition}
\begin{proof}
  \cref{prop:fsingssmaller} follows since the splitting conditions for
  $(R,\Delta',\fb^s)$ are weaker than those for $(R,\Delta,\fa^t)$.
  For \cref{prop:fpurered}, we note that $(R,\Delta)$ is $F$-pure by
  \cref{prop:fsingssmaller}, hence \cref{prop:fpurered} follows from
  \cite[Prop.\ 2.2(4)]{HW02}.
  \par For \cref{prop:sfrceilings}, we note that $\Leftarrow$ is clear.
  For $\Rightarrow$, the case when $\Delta = 0$ is shown in \cite[Prop.\
  3.3(3)]{Tak04inversion}, hence it suffices to consider when $\Delta
  \ne 0$, in which case $R$ is a normal domain by our conventions in
  \cref{def:triples}.
  Let $c' \in R^\circ$ be arbitrary; we want to show that for $c = c'$, the
  composition \cref{eq:sfrceilingscomp} splits for some $d \in \fa^{\lceil p^et
  \rceil}$ for all $e \gg 0$.
  Choose nonzero elements $a \in R(-2\lceil \Delta \rceil)$ and $b \in
  \fa^{2\lceil t \rceil}$, in which case
  \begin{gather*}
    a \cdot R\bigl(\lceil p^e\Delta \rceil\bigr) \subseteq R\bigl(\lceil
    p^e\Delta \rceil - 2\lceil \Delta \rceil\bigr) \subseteq R\bigl(\lfloor
    (p^e-1)\Delta \rfloor\bigr)\\
    b \cdot \fa^{\lfloor(p^e-1)t\rfloor} \subseteq \fa^{2\lceil t \rceil +
    \lfloor(p^e-1)t\rfloor} \subseteq \fa^{\lceil p^et \rceil}
  \end{gather*}
  for every integer $e > 0$.
  By the assumption that $(R,\Delta,\fa^t)$ is strongly $F$-regular, there
  exist $e' > 0$ and $d' \in \fa^{\lfloor(p^{e'}-1)t\rfloor}$ such that the
  composition \cref{eq:sfrcomp} splits for $e = e'$ and with $c = abc'$.
  This composition factors as
  \begin{align*}
    R \overset{F^{e'}}{\longrightarrow} F^{e'}_*R &\longhookrightarrow
    F^{e'}_*R\bigl(\lceil p^{e'}\Delta \rceil\bigr)\\
    &\xrightarrow{F^{e'}_*(-\cdot
    bc'd')} F^{e'}_*R\bigl(\lceil p^{e'}\Delta \rceil\bigr)
    \xrightarrow{F^{e'}_*(- \cdot a)} F^{e'}_*R\bigl(\lfloor (p^{e'}-1)\Delta
    \rfloor\bigr),
  \end{align*}
  hence the composition of the first three homomorphisms splits.
  Now since $R$ is $F$-pure by \cref{prop:fsingssmaller}, the homomorphism
  \[
    F^{e'}_*R\bigl(\lceil p^{e'}\Delta \rceil\bigr)
    \xrightarrow{F^{e'}_*(F^{e-e'}(\lceil p^{e'}\Delta \rceil))}
    F^{e}_*R\bigl(p^{e-e'}\lceil p^{e'}\Delta \rceil\bigr),
  \]
  which is obtained by twisting the $(e-e')$th iterate of the Frobenius
  homomorphism by $\lceil p^{e'}\Delta \rceil$ and applying $F^{e'}_*$, also
  splits for every $e \ge e'$.
  The composition
  \begin{align*}
    R \overset{F^{e'}}{\longrightarrow} F^{e'}_*R \longhookrightarrow
    F^{e'}_*R\bigl(\lceil p^{e'}\Delta \rceil\bigr)
    &\xrightarrow{F^{e'}_*(-\cdot
    bc'd')} F^{e'}_*R\bigl(\lceil p^{e'}\Delta \rceil\bigr)\\
    &\xrightarrow{F^{e'}_*(F^{e-e'}(\lceil p^{e'}\Delta \rceil))}
    F^e_*R\bigl(p^{e-e'}\lceil p^{e'}\Delta \rceil\bigr)
  \end{align*}
  therefore splits for $e \ge e'$.
  Finally, this composition factors as
  \[
    R \overset{F^e}{\longrightarrow} F^e_*R \longhookrightarrow
    F^e_*R\bigl(\lceil p^e\Delta \rceil\bigr)
    \xrightarrow{F^e_*(-\cdot
    (bc'd')^{p^{e-e'}})} F^e_*R\bigl(\lceil p^e\Delta \rceil\bigr)
    \longhookrightarrow F^e_*R\bigl(p^{e-e'}\lceil p^{e'}\Delta \rceil\bigr),
  \]
  hence the composition \cref{eq:sfrceilingscomp} splits for $c = c'$ and
  \[
    d = b^{p^{e-e'}}(c')^{p^{e-e'}-1}(d')^{p^{e-e'}} \in \bigl(
    \fa^{2\lceil t \rceil + \lfloor (p^{e'}-1)t \rfloor} \bigr)^{p^{e-e'}}
    \subseteq \bigl( \fa^{\lceil p^{e'} t \rceil} \bigr)^{p^{e-e'}} \subseteq
    \fa^{\lceil p^et \rceil}.
  \]
  \par Finally, for \cref{prop:sfrimpliesfpure}, we note that strong
  $F$-regularity implies sharp $F$-purity by \cref{prop:sfrceilings}, and the
  dashed implication holds by \cite[Prop.\ 3.5]{Sch08}.
\end{proof}
\begin{remark}
  It seems to be unknown whether sharp $F$-purity implies $F$-purity in general
  \cite[Ques.\ 3.8]{Sch08}.
\end{remark}
\begin{example}\label{rem:rounding}
  While \cref{prop:fsingbasic}\cref{prop:sfrceilings} shows that the rounding
  ``$\lfloor (p^e-1)- \rfloor$'' in \cref{def:fsingpair}\cref{def:sfrpair} can
  be replaced by ``$\lceil (p^e-1) - \rceil$'' (this is the convention in
  \cites[Def.\ 2.11]{Sch10}[Def.\ 3.2]{Sch10ref}), this is not the case for
  $F$-purity.
  For example, the pair
  \[
    \bigl(\FF_2\llbracket x,y,z \rrbracket,(x^2+y^5+z^5)^{1/2}\bigr)
  \]
  from \cite[Ex.\ 4.3]{MY09} is $F$-pure but not sharply $F$-pure by \cite[Thm.\
  4.1]{Her12}.
\end{example}
\subsection{The trace of Frobenius}
We now describe variants of the Grothendieck trace map associated to the
Frobenius morphism, and its relationship to $F$-singularities.
This material is essentially contained in \cite{Sch09Fadj}, although we use some
of the notation of \cite[\S2]{Tan15trace} and \cite[\S2.3]{CTX15}.
\begin{citedprop}[{\cite[Def.-Prop.\ 2.5]{CTX15}}]\label{prop:ctx25}
  Let $X$ be a normal scheme essentially of finite type over an $F$-finite field
  of characteristic $p > 0$, let $D$ be an effective Weil divisor on $X$,
  and let $e$ be a positive integer.
  Then, there exists a homomorphism
  \[
    \Tr_{X,D}^e \colon F^e_*\bigl(\cO_X\bigl((1-p^e)K_X-D\bigr)\bigr)
    \longrightarrow \cO_X\glsadd{traceoffrobenius}
  \]
  of $\cO_X$-modules that fits into a commutative diagram
  \begin{equation}\label{eq:ctxcommdiag}
    \begin{tikzcd}[column sep=large]
      F^e_*\bigl(\cO_X\bigl((1-p^e)K_X-D\bigr)\bigr)
      \rar{\Tr_{X,D}^e}\arrow{d}[left]{\theta}[sloped,above]{\sim} & \cO_X
      \arrow{d}[sloped,above]{\sim}\\
      \HHom_{\cO_X}\bigl(F^e_*\bigl(\cO_X(D)\bigr),\cO_X\bigr)
      \rar{(F^e_D)^*} & \HHom_{\cO_X}(\cO_X,\cO_X)
    \end{tikzcd}
  \end{equation}
  of $\cO_X$-modules, where the left vertical arrow is an isomorphism of
  $(F^e_*\cO_X,\cO_X)$-bimodules and the right vertical arrow is an isomorphism
  of $\cO_X$-modules.
\end{citedprop}
\begin{proof}
  Consider the composition map
  \[
    \cO_X \overset{F^e}{\longrightarrow} F^e_*\cO_X \longhookrightarrow
    F^e_*\cO_X(D),
  \]
  which we denote by $F^e_D$.
  Applying the contravariant functor $\HHom_{\cO_X}(-,\cO_X)$, we have the
  top arrow in the commutative diagram
  \[
    \begin{tikzcd}
      \HHom_{\cO_X}\bigl(F^e_*\bigl(\cO_X(D)\bigr),\cO_X\bigr)
      \rar{(F^e_D)^*} &\HHom_{\cO_X}(\cO_X,\cO_X)\\
      \HHom_{\cO_X}\bigl(F^e_*\bigl(\cO_X(p^eK_X+D)\bigr),\cO_X(K_X)\bigr)
      \arrow{u}[sloped,above]{\sim}
      \rar &\HHom_{\cO_X}\bigl(\cO_X(K_X),\cO_X(K_X)\bigr)\arrow{u}[sloped,above]{\sim}
    \end{tikzcd}
  \]
  where the vertical arrows are isomorphisms by restricting to the regular
  locus of $X$, by the fact that $\cO_X(K_X)$ is a reflexive sheaf, and by
  \cref{lem:homreflexive,thm:har94112}.
  By Grothendieck duality for finite morphisms (see \cref{thm:nayakshriek}),
  the sheaf in the bottom left corner satisfies
  \begin{align*}
    \HHom_{\cO_X}\bigl(F^e_*\bigl(\cO_X(p^eK_X+D)\bigr),
    \cO_X(K_X)\bigr)
    &\simeq
    F^e_*\HHom_{\cO_X}\bigl(\cO_X(p^eK_X+D),F^{e!}\cO_X(K_X)\bigr)\\
    &\simeq F^e_*\HHom_{\cO_X}\bigl(\cO_X(p^eK_X+D),\cO_X(K_X)\bigr)\\
    &\simeq F^e_*\cO_X\bigl( (1-p^e)K_X-D\bigr)
  \end{align*}
  where the second isomorphism follows from the fact that $F^{e!}\omega_X \simeq
  \omega_X$ by \cref{def:canonicalsheaf,thm:nayakshriek}, and the last
  isomorphism follows from restricting to the regular locus of $X$ and using the
  reflexivity of the sheaves involved \cite[Prop.\ 2.7]{Har94}.
  We can therefore define $\theta$ to be the composition of isomorphisms
  \begin{align*}
    F^e_*\cO_X\bigl( (1-p^e)K_X-D\bigr)
    &\simeq \HHom_{\cO_X}\bigl(F^e_*\bigl(\cO_X(p^eK_X+D)\bigr),
    \cO_X(K_X)\bigr)\\
    &\simeq \HHom_{\cO_X}\bigl(F^e_*\bigl(\cO_X(D)\bigr),\cO_X\bigr).
  \end{align*}
  Note that $\theta$ is an isomorphism of left-$F^e_*\cO_X$-modules by tracing
  through these isomorphisms, where the left-$F^e_*\cO_X$-module structure comes
  from precomposition by multiplication by an element in $F^e_*\cO_X$.
\end{proof}
We then use \cref{prop:ctx25} to prove the following characterization of
$F$-singularities.
\begin{corollary}\label{cor:fsingstrace}
  Let $(X,\Delta,\fa^t)$ be an effective log triple such that $X$ is $F$-finite
  and of characteristic $p > 0$.
  \begin{enumerate}[label=$(\roman*)$,ref=\roman*]
    \item The triple $(X,\Delta,\fa^t)$ is $F$-pure if and only if there exists
      an integer $e' > 0$ such that for all $e \ge e'$, the morphism
      \begin{align*}
        \MoveEqLeft[5]F^e_*\bigl(\fa^{\lfloor (p^e-1)t \rfloor} \cdot
        \cO_X\bigl((1-p^e)K_X-\lfloor (p^e-1)\Delta\rfloor\bigr)\bigr)\\
        &\longhookrightarrow
        F^e_*\bigl(\cO_X\bigl((1-p^e)K_X-\lfloor
        (p^e-1)\Delta\rfloor\bigr)\bigr)
        \xrightarrow{\Tr_{X,\lfloor
        (p^e-1)\Delta\rfloor}^e} \cO_X
      \end{align*}
      is surjective.\label{cor:fsingstracefpure}
    \item The triple $(X,\Delta,\fa^t)$ is sharply $F$-pure if and only if there
      exists an integer $e > 0$ such that the morphism
      \begin{align*}
        \MoveEqLeft[5]F^e_*\bigl(\fa^{\lceil (p^e-1)t \rceil} \cdot
        \cO_X\bigl((1-p^e)K_X-\lceil (p^e-1)\Delta\rceil\bigr)\bigr)\\
        &\longhookrightarrow
        F^e_*\bigl(\cO_X\bigl((1-p^e)K_X-\lceil
        (p^e-1)\Delta\rceil\bigr)\bigr)
        \xrightarrow{\Tr_{X,\lceil
        (p^e-1)\Delta\rceil}^e} \cO_X
      \end{align*}
      is surjective.
    \item The triple $(X,\Delta,\fa^t)$ is strongly $F$-regular if and only if
      for every Cartier divisor $E$ on $X$, there there exists an
      integer $e > 0$ such that the morphism
      \begin{align*}
        \MoveEqLeft[5]F^e_*\bigl(\fa^{\lfloor (p^e-1)t \rfloor} \cdot
        \cO_X\bigl((1-p^e)K_X-\lfloor (p^e-1)\Delta\rfloor-E\bigr)\bigr)\\
        &\longhookrightarrow
        F^e_*\bigl(\cO_X\bigl((1-p^e)K_X-\lfloor
        (p^e-1)\Delta\rfloor-E\bigr)\bigr)
        \xrightarrow{\Tr_{X,\lfloor
        (p^e-1)\Delta\rfloor+E}^e} \cO_X
      \end{align*}
      is surjective.
  \end{enumerate}
\end{corollary}
\begin{proof}
  We first consider the case of a pair $(X,\Delta)$.
  Let $D$ stand for one of $\lfloor (p^e-1)\Delta\rfloor$, $\lceil
  (p^e-1)\Delta\rceil$, or $\lfloor (p^e-1)\Delta\rfloor+E$.
  By \cref{prop:ctx25}, we see that $\Tr^e_{X,D}$ is surjective
  if and only if
  \begin{align*}
    (F^e_D)^*\colon
    \HHom_{\cO_X}\bigl(F^e_*\bigl(\cO_X(D)\bigr),\cO_X\bigr) &\longrightarrow
    \HHom_{\cO_X}(\cO_X,\cO_X)
    \intertext{is surjective.
    Since $X$ is $F$-finite, these morphisms are surjective if and only if
    for every $x \in X$, the morphism}
    (F^e_D)^*\colon
    \Hom_{\cO_{X,x}}\bigl(F^e_*\bigl(\cO_{X,x}(D)\bigr),\cO_{X,x}\bigr)
    &\longrightarrow \Hom_{\cO_{X,x}}(\cO_{X,x},\cO_{X,x})
  \end{align*}
  is surjective.
  Finally, this condition is equivalent to the splitting of the map
  \[
    \cO_{X,x} \longrightarrow F^e_*\cO_{X,x} \longhookrightarrow
    F^e_*\bigl(\cO_{X,x}(D)\bigr),
  \]
  hence all three statements follow by comparing this condition to
  \cref{def:fsingpair}.
  \par Finally, the case for a triple $(X,\Delta,\fa^t)$ follows from the fact
  that under the isomorphism $\theta$ in \cref{prop:ctx25}, the multiplication
  $F^e_*(-\cdot d)$ or $F^e_*(-\cdot cd)$ in \cref{def:fsingpair} corresponds to
  precomposition of the trace $\Tr^e_{X,D}$ by multiplication by an element in
  $\fa^{\lfloor (p^e-1)t \rfloor}$ or $\fa^{\lceil (p^e-1)t \rceil}$.
\end{proof}
\subsection{\textit{F}-pure thresholds}
We define the $F$-pure threshold, which is the positive
characteristic analogue of the log canonical threshold.
\begin{citeddef}[{\cite[Def.\ 2.1]{TW04}}]
  Let $(X,\Delta,\fa)$ be an effective log triple such that $X$ is $F$-finite
  of characteristic $p > 0$.
  The \textsl{$F$-pure threshold}%
  \index{F-pure threshold@$F$-pure threshold, $\fpt_x((X,\Delta);\fa)$|textbf}
  of the pair $(X,\Delta)$ with respect to $\fa$
  at a point $x \in X$ is
  \[
    \fpt_x\bigl( (X,\Delta);\fa\bigr) \coloneqq \sup\bigl\{ c \in \RR_{\ge 0}
    \bigm\vert \text{the triple}\ (X,\Delta,\fa^c)\ \text{is $F$-pure at
    $x$}\bigr\},\glsadd{fpurethreshold}
  \]
  where if $(X,\Delta)$ is not $F$-pure at $x$,
  then we set $\fpt_x((X,\Delta);\fa) = -\infty$.
\end{citeddef}
$F$-pure thresholds can be very different from log canonical thresholds, even
for the same defining equation.
\begin{citedex}[{\cite[Ex.\ 4.3]{MTW05}}]
  Let\index{cuspidal cubic|(}
  $R = k\llbracket x,y\rrbracket$ with maximal ideal $\fm$, where $k$ is an
  $F$-finite field of characteristic $p > 0$, and let $f = x^2 + y^3$.
  The $F$-pure threshold then depends on the characteristic of $k$:
  \[
    \fpt_\fm\bigl(\Spec R;x^2 + y^3\bigr) =
    \begin{dcases}
      \frac{1}{2} & \text{if}\ p = 2\\
      \frac{2}{3} & \text{if}\ p = 3\\
      \frac{5}{6} & \text{if}\ p \equiv 1 \bmod 3\\
      \frac{5}{6} - \frac{1}{6p} & \text{if}\ p \equiv 2 \bmod 3\
      \text{and}\ p \ne 2
    \end{dcases}
  \]
  We see that as $p \to \infty$, the $F$-pure threshold approaches the log
  canonical threshold as computed in \cref{ex:cuspidalcubic}, as predicted by
  \cref{thm:takagiredmodp}.\index{cuspidal cubic|)}
  For more examples of similar phenomena, see \cites[Exs.\ 2.4 and
  2.5]{TW04}[\S4]{MTW05}{CHSW16}.
\end{citedex}
\section{Test ideals}
We review the theory of test ideals, which are the positive characteristic
analogues of multiplier ideals.
Test ideals for rings were originally defined by Hochster%
\index{Hochster, Melvin} and Huneke\index{Huneke, Craig} \cite[Def.\ 8.22]{HH90}
using tight closure, and
versions for pairs and triples were first
defined by Hara\index{Hara, Nobuo}--Yoshida\index{Yoshida, Ken-ichi}
\cite[Def.-Thm.\ 6.5]{HY03} and Takagi\index{Takagi, Shunsuke} \cites[Def.\
2.6]{Tak04}[Def.\ 2.2]{Tak08} using generalized versions of tight closure; see
\cref{rem:testidealsviatightclosure}.
While test ideals can be defined in this way without $F$-finiteness assumptions,
we will assume $F$-finiteness throughout and define test ideals using the notion
of $F$-compatibility, following Schwede\index{Schwede, Karl} \cite{Sch10}.
We recommend \cites[\S6]{ST12}[\S5]{TW18} for surveys on this topic.
\medskip
\par We start with the following definition.
\begin{citeddef}[{\cite[Def.\ 3.1]{Sch10}}]
  Let $(R,\Delta,\fa^t)$ be an effective log triple such that $R$ is an
  $F$-finite ring of characteristic $p > 0$.
  An ideal $J \subseteq R$ is \textsl{uniformly
  $(\Delta,\fa^t,F)$-compatible}%
  \index{uniformly $(\Delta,\mathfrak{a}^t,F)$-compatible|textbf}
  if for every integer $e > 0$ and
  every $\varphi \in \Hom_R(F^e_*R(\lceil(p^e-1)\Delta\rceil),R)$, we have
  \begin{equation}\label{eq:fcompatdef}
    \varphi\Bigl(F^e_*\bigl(J \cdot \fa^{\lceil t(p^e-1)\rceil}
    \bigr)\Bigr) \subseteq J.
  \end{equation}
  We drop $\Delta$ or $\fa^t$ from our notation when working with pairs or the
  ring itself.
  \par If $(R,\Delta,\fb^s)$ is another effective log triple on $R$, then we can
  analogously
  define uniform $(\Delta,\fa^t \cdot \fb^s,F)$-compatibility by using the ideal
  $\fa^{\lceil t(p^e-1)\rceil}\cdot\fb^{\lceil s(p^e-1)\rceil}$ in
  \cref{eq:fcompatdef}.
\end{citeddef}
We can now define test ideals.
\begin{citeddef}[{\cite[Def.\ 3.1 and Thm.\ 6.3]{Sch10}}]\label{def:testideal}
  Let $(R,\Delta,\fa^t)$ be an effective log triple such that $R$ is an
  $F$-finite ring of characteristic $p > 0$.
  The \textsl{test ideal}%
  \index{test ideal, $\tau(X,\Delta,\mathfrak{a}^t)$|textbf}
  \[
    \tau(R,\Delta,\fa^t) \subseteq R\glsadd{testideal}
  \]
  is the smallest ideal which is uniformly
  $(\Delta,\fa^t,F)$-compatible and whose intersection with
  $R^\circ$ is nonempty.
  We drop $\Delta$ or $\fa^t$ from our notation when working with pairs or the
  ring itself.
  We also often drop the ring $R$ from our notation if it is clear from context.
  \par If $(R,\Delta,\fb^s)$ is another effective log triple on $R$, then we can
  analogously define the test ideal $\tau(R,\Delta,\fa^t\cdot\fb^s)$ as the
  smallest uniformly $(\Delta,\fa^t \cdot \fb^s,F)$-compatible ideal that
  intersects $R^\circ$.
\end{citeddef}
The test ideal as defined in \cref{def:testideal} exists since it
matches the earlier notion (see \cref{rem:testidealsviatightclosure}) defined
using tight closure \cite[Thm.\ 6.3]{Sch10}.
We briefly describe a direct proof of existence, following \cite{Sch11}.
The key ingredient is the following:
\begin{definition}[cf.\ {\cite[Def.\ 3.19]{Sch11}}]\label{def:hdbstest}
  Let $(R,\Delta,\fa^t)$ be an effective log triple such that
  $R$ is an $F$-finite ring of characteristic $p > 0$.
  An element $c \in R^\circ$ is a \textsl{big sharp test element}%
  \index{big sharp test element|textbf} for
  $(R,\Delta,\fa^{t})$ if, for every $d \in R^\circ$, there exists
  $\varphi \in \Hom_R(F^e_*R(\lceil(p^e-1)\Delta\rceil),R)$ for some integer $e
  > 0$ such that
  \begin{equation}\label{eq:hdbstest}
    c \in \varphi\Bigl(F^e_*\bigl(d \cdot \fa^{\lceil
    t(p^e-1)\rceil} \bigr)\Bigr).
  \end{equation}
  If $c$ is a big sharp test element, then $c'c$ is also for all $c' \in
  R^\circ$ by considering the composition $(c' \cdot -) \circ \varphi$ for
  $\varphi$ as in \cref{eq:hdbstest}.
  \par If $(R,\Delta,\fb^s)$ is another effective log triple on $R$, then we can
  analogously define the big sharp test elements for $(R,\Delta,\fa^{t})$ and
  $(R,\Delta,\fb^{s})$ by
  using the ideal $\fa^{\lceil t(p^e-1)\rceil}\cdot\fb^{\lceil s(p^e-1)\rceil}$
  in \cref{eq:hdbstest}.
\end{definition}
Various versions of test elements were shown to exist in the context of tight
closure; see \cites[\S6]{HH90}[Thm.\ 6.4]{HY03}[Thm.\ 2.5(2)]{Tak04}[Thm.\ on
p.\ 90]{Hoc07}.
Big sharp test elements as defined in \cref{def:hdbstest} exist
by \cites[Lem.\ 2.17]{Sch10}[Prop.\ 3.21]{Sch11}.%
\index{big sharp test element!existence of}
Assuming this fact, we can show that test ideals exist.
Note that the description \cref{eq:haratakagi} is originally due to
Hara\index{Hara, Nobuo} and Takagi\index{Takagi, Shunsuke} \cite[Lem.\
2.1]{HT04}.
\begin{citedthm}[{\cite[Thm.\ 3.18]{Sch11}}]
  \label{thm:testidealexist}
  Let\index{test ideal, $\tau(X,\Delta,\mathfrak{a}^t)$!existence of|(}
  $(R,\Delta,\fa^{t})$ be an effective log triple such that
  $R$ is an $F$-finite ring of characteristic $p > 0$.
  Then, for every choice of big sharp test element $c \in R^\circ$ for the
  triple $(R,\Delta,\fa^{t})$, we have
  \begin{equation}\label{eq:haratakagi}
    \tau(R,\Delta,\fa^{t})
    = \adjustlimits{\sum^\infty}_{e = 0}{\sum}_{\varphi_e}
    \varphi_e\Bigl(F^e_*\bigl(c\cdot\fa^{\lceil t(p^e-1)\rceil}
    \bigr)\Bigr),
  \end{equation}
  where $\varphi_e$ ranges over all elements in
  $\Hom_R(F^e_*R(\lceil(p^e-1)\Delta\rceil),R)$.
  In particular, the test ideal $\tau(R,\Delta,\fa^{t})$ exists.
  If $(R,\Delta,\fb^s)$ is another effective log triple, then
  $\tau(R,\Delta,\fa^{t}\cdot \fb^s)$ exists by replacing $c$ with a big sharp
  test element for
  $(R,\Delta,\fa^{t})$ and $(R,\Delta,\fb^{s})$ and by
  using the ideal $\fa^{\lceil t(p^e-1)\rceil}\cdot\fb^{\lceil s(p^e-1)\rceil}$
  in \cref{eq:haratakagi}.%
  \index{test ideal, $\tau(X,\Delta,\mathfrak{a}^t)$!existence of|)}
\end{citedthm}
\begin{proof}
  By definition of a big sharp test element (\cref{def:hdbstest}), we
  have $c \in J$ for every $(\Delta,\fa^t,F)$-compatible ideal $J
  \subseteq R$.
  On the other hand, the ideal on the right-hand side of \cref{eq:haratakagi}
  is the smallest $(\Delta,\fa^{t},F)$-compatible ideal containing
  $c$, hence must coincide with $\tau(R,\Delta,\fa^{t})$.
  The proof for $\tau(R,\Delta,\fa^{t}\cdot\fb^s)$ is similar.
\end{proof}
To define test ideals on schemes, we use the following consequence of the proof
of the existence of big sharp test elements.
\begin{citedprop}[{\cite[Prop.\ 3.23$(ii)$]{Sch11}}]\label{prop:testideallocal}
  Let $(R,\Delta,\fa^t)$ be an effective log triple such that $R$ is
  an $F$-finite ring of characteristic $p > 0$.
  For every multiplicative set $W \subseteq R$, we have
  \[
    W^{-1}\,\tau(R,\Delta,\fa^t)
    = \tau\bigl(W^{-1}R,\Delta\rvert_{\Spec
    W^{-1}R},(W^{-1}\fa)^t\bigr),
  \]
  and similarly for $\tau(R,\Delta,\fa^t\cdot\fb^s)$.%
  \index{test ideal, $\tau(X,\Delta,\mathfrak{a}^t)$!localizes}
\end{citedprop}
We can now define test ideals on schemes.
\begin{definition}
  Let $(X,\Delta,\fa^t)$ be an effective log triple such that $X$ is
  an $F$-finite scheme of characteristic $p > 0$.
  By \cref{prop:testideallocal}, we can define the
  \textsl{test ideal}\index{test ideal, $\tau(X,\Delta,\mathfrak{a}^t)$|textbf}
  $\gls*{testideal} \subseteq \cO_X$
  locally on every open affine subset $U = \Spec R \subseteq X$ by
  \[
    \tau(X,\Delta,\fa^t)(U) =
    \tau\bigl(R,\Delta\rvert_U,\fa(U)^t\bigr).
  \]
  We drop $\Delta$ or $\fa^t$ from our notation when working with pairs or the
  scheme itself.
  We also often drop the scheme $X$ from our notation if it is clear from
  context.
  \par If $(X,\Delta,\fb^s)$ is another effective log triple on $X$, then we can
  analogously define the test ideal $\tau(X,\Delta,\fa^t\cdot\fb^s)$.
\end{definition}
We now state some properties of test ideals that we will use often, which are
reminiscent of those for multiplier ideals in \cite[\S9.2]{Laz04b}.
\begin{proposition}[see {\cite[Prop.\ 5.6]{TW18}}]
  \label{thm:testidealprops}
  Let\index{test ideal, $\tau(X,\Delta,\mathfrak{a}^t)$!basic properties of|(}
  $(X,\Delta,\fa^t)$ be an effective log triple such that $X$ is an
  $F$-finite scheme of characteristic $p > 0$.
  \begin{enumerate}[label=$(\roman*)$,ref=\roman*]
    \item\label{thm:testideallarger}
      If $(X,\fb)$ is a log pair on $X$, then $\tau(\Delta,\fa^t) \cdot
      \fb\subseteq \tau(\Delta,\fa^t \cdot \fb)$.
    \item\label{thm:testidealincl}
      Let $(X,\Delta',\fb^s)$ be another effective log triple on $X$.
      If $\Delta \ge \Delta'$ and $\fa^{\lceil t(p^e-1) \rceil}
      \subseteq \fb^{\lceil s(p^e-1) \rceil}$ for every integer $e >
      0$, then $\tau(\Delta,\fa^t) \subseteq \tau(\Delta',\fb^s)$.
    \item\label{thm:testidealunambig}
      For every non-negative real number $s$, we have $\tau(\Delta,\fa^s \cdot
      \fa^t) = \tau(\Delta,\fa^{s+t})$.
    \item\label{thm:testidealunambiginteger}
      For every non-negative integer $m$, we have $\tau(\Delta,(\fa^m)^t) =
      \tau(\Delta,\fa^{mt})$.
    \item\label{thm:testidealperturb}
      There exists $\varepsilon > 0$ such that for all $s \in
      [t,t+\varepsilon]$, we have $\tau(\Delta,\fa^t) = \tau(\Delta,\fa^s)$.
    \item\label{thm:testidealperturbdiv}
      Suppose that $X$ is normal.
      For every effective Cartier divisor $D$ on $X$, there exists
      $\varepsilon > 0$ such that for all $\delta \in [0,\varepsilon]$, we have
      $\tau(\Delta,\fa^t) = \tau(\Delta+\delta D,\fa^t)$.
    \item\label{thm:testidealstrongfreg}
      The triple $(X,\Delta,\fa^t)$ is strongly $F$-regular if and only if
      $\tau(\Delta,\fa^t) = \cO_X$.\index{test ideal, $\tau(X,\Delta,\mathfrak{a}^t)$!basic properties of|)}
  \end{enumerate}
\end{proposition}
We will define strong $F$-regularity in \cref{def:fsingpair}\cref{def:sfrpair}.
\begin{proof}
  Since test ideals are defined locally, it suffices to consider the case when
  $X = \Spec R$.
  Fix a big sharp test element $c \in R^\circ$.
  We will freely use the description of the test ideal in
  \cref{thm:testidealexist}.
  \par To show \cref{thm:testideallarger}, we note that
  \begin{align*}
    \tau(\Delta,\fa^t) \cdot \fb &= \adjustlimits{\sum^\infty}_{e =
    0}{\sum}_{\varphi_e}
    \varphi_e\Bigl( F^e_*\bigl(c \cdot \fa^{\lceil t(p^e-1)
    \rceil}\bigr)\Bigr) \cdot \fb\\
    &= \adjustlimits{\sum^\infty}_{e =
    0}{\sum}_{\varphi_e}
    \varphi_e\Bigl( F^e_*\bigl(c \cdot \fa^{\lceil t(p^e-1)
    \rceil} \cdot \fb^{[p^e]}\bigr)\Bigr)\\
    &\subseteq \adjustlimits{\sum^\infty}_{e =
    0}{\sum}_{\varphi_e}
    \varphi_e\Bigl( F^e_*\bigl(c \cdot \fa^{\lceil t(p^e-1)
    \rceil} \cdot \fb^{p^e-1}\bigr)\Bigr)
    = \tau(\Delta,\fa^t \cdot \fb).
  \end{align*}
  \par To show \cref{thm:testidealincl}, it suffices to note that if an ideal
  $J \subseteq R$ is $(\Delta',\fb^{s},F)$-compatible, then $J$ is
  $(\Delta,\fa^{t},F)$-compatible, since
  \begin{gather*}
    \varphi\Bigl(F^e_*\bigl(J \cdot \fa^{\lceil t(p^e-1)\rceil}
    \bigr)\Bigr) \subseteq 
    \varphi\Bigl(F^e_*\bigl(J \cdot \fb^{\lceil s(p^e-1)\rceil}
    \bigr)\Bigr) \subseteq J
    \shortintertext{for all}
    \varphi \in
    \Hom_R\bigl(F^e_*R\bigl(\bigl\lceil(p^e-1)\Delta\bigr\rceil\bigr),R\bigr)
    \subseteq
    \Hom_R\bigl(F^e_*R\bigl(\bigl\lceil(p^e-1)\Delta'\bigr\rceil\bigr),R\bigr).
  \end{gather*}
  \par To show \cref{thm:testidealunambig}, we first note that the inclusion
  $\subseteq$ holds by \cref{thm:testidealincl} since
  \[
    \fa^{\lceil s(p^e-1) \rceil} \cdot \fa^{\lceil t(p^e-1) \rceil}
    \subseteq \fa^{\lceil (s+t)(p^e-1) \rceil}
  \]
  for every integer $e > 0$.
  To show the reverse inclusion $\supseteq$, note that
  \begin{align*}
    \emptyset \ne{} \MoveEqLeft[3]\fa^{\lceil s(p^e-1) \rceil + \lceil t(p^e-1)
    \rceil - \lceil (s+t)(p^e-1) \rceil} \cap R^\circ\\
    &\subseteq \Bigl( \bigl(\fa^{\lceil
      s(p^e-1) \rceil} \cdot \fa^{\lceil t(p^e-1) \rceil}\bigr) : \fa^{\lceil
    (s+t)(p^e-1) \rceil} \Bigr) \cap R^\circ,
  \end{align*}
  hence we can choose an element $c'$ in the set on the right-hand side.
  Then, the product $cc'$ is a big sharp test element, hence
  \begin{align*}
    \tau(\Delta,\fa^{s+t})
    &= \adjustlimits{\sum^\infty}_{e=0}{\sum}_{\varphi_e}
    \varphi_e\Bigl( F^e_*\bigl(cc' \cdot
    \fa^{\lceil (s+t)(p^e-1) \rceil} \bigr)\Bigr)\\
    &\subseteq \adjustlimits{\sum^\infty}_{e=0}{\sum}_{\varphi_e}
    \varphi_e\Bigl( F^e_*\bigl(c
    \cdot \fa^{\lceil s(p^e-1) \rceil} \cdot \fa^{\lceil t(p^e-1)
    \rceil} \bigr)\Bigr)
    = \tau(\Delta,\fa^s\cdot\fa^t).
  \end{align*}
  \cref{thm:testidealunambiginteger} then follows from applying
  \cref{thm:testidealunambig} $m$ times.
  \par See \cite[Lem.\ 6.1]{ST14} and \cite[Prop.\ 2.14(2)]{Sat18} for proofs
  of \cref{thm:testidealperturb} and \cref{thm:testidealperturbdiv},
  respectively.
  Note that in the proof of \cite[Prop.\ 2.14(2)]{Sat18}, one should follow the
  proof of \cite[Lem.\ 6.1]{ST14} to reduce to the case when
  $(p^e-1)(K_X+\Delta)$ is Cartier for some integer $e > 0$.
  \par See \cite[Prop.\ 3.23$(iii)$]{Sch11} for a proof of
  \cref{thm:testidealstrongfreg}.
\end{proof}
We give a new proof of the following very important property of test ideals.
\begin{theorem}[Subadditivity {\cite[Thm.\ 6.10(2)]{HY03}}]
  \label{thm:testidealsubadditivity}
  Let\index{subadditivity theorem!for test ideals|(}
  $(X,\fa^t)$ and $(X,\fb^s)$ be two effective log pairs where $X$ is an
  $F$-finite regular scheme of characteristic $p > 0$.
  Then, we have
  \[
    \tau(\fa^t \cdot \fb^s) \subseteq \tau(\fa^t) \cdot \tau(\fb^s).
  \]
\end{theorem}
\begin{proof}
  By \cref{prop:testideallocal}, it suffices to consider the case when $X =
  \Spec R$ for a regular local ring $R$.
  By \cite[Prop.\ 3.11]{Sch10}, for a regular ring $R$, an ideal $J \subseteq R$
  is uniformly
  $(\fa^t\cdot\fb^s,F)$-compatible if and only if for every integer $e \ge
  0$, we have
  \[
    \fa^{\lceil t(p^e-1) \rceil} \cdot \fb^{\lceil s(p^e-1) \rceil} \subseteq
    (J^{[p^e]}\mathbin{:}J).
  \]
  It therefore suffices to show the chain of inclusions
  \begin{align*}
    \fa^{\lceil t(p^e-1) \rceil} \cdot \fb^{\lceil s(p^e-1) \rceil}
    &\subseteq \bigl( \tau(\fa^t)^{[p^e]} \mathbin{:} \tau(\fa^t) \bigr) \cdot \bigl(
    \tau(\fb^s)^{[p^e]} \mathbin{:} \tau(\fb^s) \bigr)\\
    &\subseteq \Bigl( \bigl( \tau(\fa^t) \cdot \tau(\fb^s) \bigr)^{[p^e]} \mathbin{:}
    \bigl( \tau(\fa^t)^{[p^e]} \cdot \tau(\fb^s)^{[p^e]} \bigr)\Bigr)
  \end{align*}
  since $\tau(\fa^t \cdot \fb^s)$ is the smallest
  $(\fa^t\cdot\fb^s,F)$-compatible ideal by
  definition.
  The first inclusion follows from the fact that $\tau(\fa^t)$ and
  $\tau(\fb^s)$ are uniformly $(\fa^t,F)$- and
  $(\fb^s,F)$-compatible, respectively.
  The second inclusion follows from the fact that in general,
  $(I_1 \mathbin{:} J_1) \cdot (I_2 \mathbin{:} 
  J_2) \subseteq (I_1I_2 \mathbin{:} J_1J_2)$.
\end{proof}
\begin{remark}
  Subadditivity (\cref{thm:testidealsubadditivity}) was originally
  proved by Hara and Yoshida \cite[Thm.\ 6.10(2)]{HY03} using tight closure.
  We have included a proof purely in the language of $F$-compatibility to be
  consistent with our choice of definition (\cref{def:testideal}); see
  \cite[Prop.\ 2.11$(iv)$]{BMS08} for another approach.
  Our proof can also be used to show a more general form of subadditivity: if
  $(X,\fa_\bullet)$ and $(X,\fb_\bullet)$ are two pairs as in \cite[Def.\
  2.3]{Sch10}, then
  \[
    \tau(\fa_\bullet \cdot \fb_\bullet) \subseteq
    \tau(\fa_\bullet) \cdot \tau(\fb_\bullet).
  \]
  Here, the test ideal is as described in \cite[Thm.\
  6.3]{Sch11}.
  We have avoided this notation since it clashes with that of
  asymptotic test ideals below.\index{subadditivity theorem!for test ideals|)}
\end{remark}
Because of the formal properties of test ideals in \cref{thm:testidealprops}, we
can define the following asymptotic version of test ideals.
\begin{citeddef}[{\cite[Prop.-Def.\ 2.16]{Sat18}}]
  \label{def:asymptotictestideal}
  Let $(X,\Delta,\fa_\bullet^\lambda)$ be an effective log triple such that $X$
  is $F$-finite and of characteristic $p > 0$.
  If $m$ and $r$ are positive integers, then
  \begin{equation*}
    \tau(X,\Delta,\fa_{m}^{\lambda/m})
    = \tau\bigl(X,\Delta,(\fa_m^r)^{\lambda/(mr)}\bigr)
    \subseteq \tau\bigl(X,\Delta,\fa_{mr}^{\lambda/(mr)}\bigr),
  \end{equation*}
  by Propositions \ref{thm:testidealprops}\cref{thm:testidealunambiginteger}
  and \ref{thm:testidealprops}\cref{thm:testidealincl}, and by the graded
  property $\fa_m^r \subseteq \fa_{mr}$.
  Thus, the set of ideals
  \begin{equation}\label{eq:testidealposet}
    \bigl\{\tau(X,\Delta,\fa_{m}^{\lambda/m})\bigr\}_{m=1}^\infty
  \end{equation}
  is partially ordered, and has a unique maximal element by the noetherian
  property that coincides with $\tau(X,\Delta,\fa_m^{\lambda/m})$
  for $m$ sufficiently large and divisible.
  The \textsl{asymptotic test ideal}\index{asymptotic test ideal, $\tau(X,\Delta,\fa_\bullet^{\lambda})$}
  \[
    \tau(X,\Delta,\fa_\bullet^{\lambda}) \subseteq
    \cO_X\glsadd{asymptotictestideal}
  \]
  is the maximal element of the partially ordered set \eqref{eq:testidealposet}.
  \par If $(X,\Delta,\fb_\bullet^\mu)$ is another effective log triple on
  $X$, then we can analogously define the test ideal
  $\tau(X,\Delta,\fa_\bullet^\lambda\cdot\fb_\bullet^\mu)$.
\end{citeddef}
Asymptotic test ideals satisfy properties analogous to those in
\cref{thm:testidealprops,thm:testidealsubadditivity}.
\begin{remark}
  \cref{def:asymptotictestideal} is due to
  Musta\c{t}\u{a}\index{Mustata, Mircea@Musta\c{t}\u{a}, Mircea} when $X$ is
  regular and $\Delta = 0$ \cite[pp.\ 540--541]{Mus13}.
  An asymptotic version of the test ideal was first defined by
  Hara\index{Hara, Nobuo}
  \cite[Prop.-Def.\ 2.9]{Har05}, although this ideal differs from that in
  \cref{def:asymptotictestideal} in general; see \cite[Rem.\
  1.4]{TY08}.
\end{remark}
\par The following examples of test ideals will be the most useful in our
applications.
\begin{example}[see {\cite[Def.\ 2.36]{Sat18}}]\label{ex:testidealdiv}
  Suppose $(X,\Delta)$ is an effective log pair where $X$ is a complete scheme
  over an $F$-finite field of characteristic $p > 0$.
  If $D$ is a Cartier divisor such that $H^0(X,\cO_X(D)) \ne 0$, then for every
  real number $t \ge 0$, we set%
  \index{test ideal, $\tau(X,\Delta,\mathfrak{a}^t)$!of a divisor, $\tau(X,\Delta,t\cdot\lvert D \rvert)$|textbf}
  \begin{align*}
    \tau\bigl(X,\Delta,t \cdot \lvert D \rvert\bigr) &\coloneqq
    \tau\bigl(X,\Delta,\fb\bigl(\lvert D \rvert\bigr)^t\bigr).%
    \glsadd{testidealofls}
    \intertext{If $D$ is a $\QQ$-Cartier divisor such that $H^0(X,\cO_X(mD)) \ne
    0$ for some $m > 0$ such that $mD$ is Cartier, then for every real number
    $\lambda \ge 0$, we set}
    \tau\bigl(X,\Delta,\lambda \cdot \lVert D \rVert\bigr) &\coloneqq
    \tau\bigl(X,\Delta,\fa_\bullet(D)^\lambda\bigr),\glsadd{asymtestidealofls}
  \end{align*}
  \index{asymptotic test ideal, $\tau(X,\Delta,\fa_\bullet^{\lambda})$!of a $\QQ$-divisor, $\tau(X,\Delta,\lambda \cdot \lVert D \rVert)$|textbf}%
  where $\fa_\bullet(D)$ is the graded family of ideals defined in
  \cref{ex:triplelinsys}.
\end{example}
Finally, we have the following uniform global generation result for
(asymptotic) test ideals.
\begin{theorem}[{\cite[Prop.\ 4.1]{Sat18}; cf.\ 
  \cites[Thm.\ 4.3]{Sch14}[Thm.\ 4.1]{Mus13}}]
  \label{thm:testidealuniformgg}
  Let%
  \index{test ideal, $\tau(X,\Delta,\mathfrak{a}^t)$!uniform global generation of}%
  \index{asymptotic test ideal, $\tau(X,\Delta,\fa_\bullet^{\lambda})$!uniform global generation of}
  $(X,\Delta)$ be an effective log pair where $X$ is a normal projective
  variety over an $F$-finite field of characteristic $p > 0$ and $\Delta$ is an
  effective $\QQ$-Weil divisor such that $K_X+\Delta$ is $\QQ$-Cartier.
  Let $D$, $L$, and $H$ be Cartier divisors on $X$ such that $H$ is ample and
  free.
  If $\lambda$ is a non-negative real number such that
  $L - (K_X + \Delta + \lambda \cdot D)$
  is ample, then the sheaves
  \[
    \tau\bigl(X,\Delta,\lambda\cdot\lvert D \rvert\bigr) \otimes \cO_X(L+dH)
    \quad \text{and} \quad
    \tau\bigl(X,\Delta,\lambda\cdot\lVert D \rVert\bigr) \otimes \cO_X(L+dH)
  \]
  are globally generated for every integer $d > \dim X$.
\end{theorem}
\begin{remark}[Test ideals via tight closure]
  \label{rem:testidealsviatightclosure}
  We\index{test ideal, $\tau(X,\Delta,\mathfrak{a}^t)$!via tight closure|(}
  briefly recall an alternative definition for test ideals via tight closure,
  following \cite[\S2]{Tak08} and \cite[\S2.2]{Sch10}.
  Let $(R,\Delta,\fa^t)$ be an effective log triple such that $R$ is a ring of
  characteristic $p > 0$, and let $\iota\colon N \hookrightarrow M$ be an
  inclusion of $R$-modules.
  For every integer $e > 0$, let
  \[
    N_M^{[p^e],\Delta} \coloneqq \im\Bigl( N \otimes_R F^e_*R \xrightarrow{\iota
      \otimes_R \id} M \otimes_R F^e_*R \longrightarrow
      M \otimes_R F_*^eR\bigl(\lfloor (p^e-1)\Delta \rfloor \bigr)
    \Bigr).\glsadd{frobeniuspower}
  \]
  The \textsl{$(\Delta,\fa^t)$-tight closure}
  \index{tight closure, $N^*_M$!for pairs and triples, $N^{*(\Delta,\fa^t)}_M$}
  of $N$ in $M$ is the $R$-module
  \[
    N^{*(\Delta,\fa^t)}_M \coloneqq \biggl\{
      z \in M \biggm\vert
      \begin{tabular}{@{}c@{}}
        there exists $c\in R^\circ$ such that\\
        $z \otimes c\fa^{\lceil p^et \rceil} \subseteq N_M^{[p^e],\Delta}$
        for all $e \gg0$
      \end{tabular}
    \biggr\}.\glsadd{generalizedtightclosure}
  \]
  Now let $E \coloneqq \bigoplus_\fm E_R(R/\fm)$ be the direct sum of the
  injective hulls of the residue fields $R/\fm$ for every maximal ideal $\fm
  \subseteq R$.
  The \textsl{test ideal} of $(R,\Delta,\fa^t)$ is
  \[
    \tau(R,\Delta,\fa^t) \coloneqq \Ann_R\bigl(
    0_E^{*(\Delta,\fa^t)} \bigr) \subseteq R.
  \]
  By \cite[Thm.\ 6.3]{Sch10}, this test ideal is equal to the test ideal defined
  in \cref{def:testideal} as long as $R$ is $F$-finite.%
  \index{test ideal, $\tau(X,\Delta,\mathfrak{a}^t)$!via tight closure|)}
\end{remark}
\begin{remark}[Big vs.\ finitistic test ideals]
  Our test ideals correspond to the (\textsl{non-finitistic} or \textsl{big})
  test ideal defined by Lyubeznik\index{Lyubeznik, Gennady} and
  Smith\index{Smith, Karen E.} \cite[\S7]{LS01} when $\Delta = 0$,
  $\fa = R$, and $t = 1$, instead of the original (\textsl{finitistic}) test
  ideal defined by Hochster\index{Hochster, Melvin} and
  Huneke\index{Huneke, Craig} \cite[Def.\ 8.22]{HH90}.
  Note that the definitions for test ideals of pairs in \cites[Def.\
  1.1]{HY03}[Def.\ 2.1]{Tak04} specialize to the finitistic test ideal.
  The corresponding non-finitistic notion first appears in \cite[Def.\
  1.4]{HT04}.
\end{remark}
\section{Reduction modulo \texorpdfstring{$\mathfrak{p}$}{p}}
Finally, we review the theory of reduction modulo $\fp$, and the relationship
between singularities in characteristic zero and characteristic $p > 0$.
What follows is a small part of the general discussion in \cite[\S8]{EGAIV3}.
\begin{citedsetup}[{\cite[(8.2.2), (8.5.1), and (8.8.1)]{EGAIV3}}]
  \label{setup:spreadingout}
  We will denote by $\{(S_\lambda,u_{\lambda\mu})\}_{\lambda \in \Lambda}$
  a filtered inverse system of schemes with affine transition morphisms
  $u_{\lambda\mu}\colon S_\mu \to S_\lambda$ for $\lambda \le \mu$, where
  $\Lambda$ has a unique minimal element $0$.
  We then set $S \coloneqq \varprojlim_{\lambda \in \Lambda} S_\lambda$ with
  projection morphisms $u_\lambda\colon S \to S_\lambda$.
  \par Now suppose an element $\alpha \in \Lambda$ and schemes
  $X_\alpha$ and $Y_\alpha$ over $S_\alpha$ are given.
  We then denote by
  \begin{alignat*}{2}
    \omit\hfil$\displaystyle\bigl\{(X_\lambda,v_{\lambda\mu})\bigr\}_{\lambda
    \in \Lambda}$\hfil\ignorespaces &\qquad\text{and}\qquad&
    \omit\hfil$\displaystyle\bigl\{(Y_\lambda,w_{\lambda\mu})\bigr\}_{\lambda
    \in \Lambda}$\hfil\ignorespaces
  \intertext{the inverse systems induced by $\{(S_\lambda,u_{\lambda\mu})\}$,
  where}
    \begin{aligned}
      X_\lambda &\coloneqq X_\alpha \times_{S_\alpha} S_\lambda\\
      v_{\lambda\mu} &\coloneqq \id_{X_\alpha} \times u_{\lambda\mu}
    \end{aligned}
    &\qquad\text{and}\qquad&
    \begin{aligned}
      Y_\lambda &\coloneqq Y_\alpha \times_{S_\alpha} S_\lambda\\
      w_{\lambda\mu} &\coloneqq \id_{Y_\alpha} \times u_{\lambda\mu}
    \end{aligned}
  \end{alignat*}
  for $\alpha \le \lambda \le \mu$.
  The inverse limits of these inverse systems are $X = X_\alpha
  \times_{S_\alpha} S$ and $Y = Y_\alpha \times_{S_\alpha} S$, respectively,
  with projection morphisms $v_\lambda\colon X \to X_\lambda$ and
  $w_\lambda\colon Y \to Y_\lambda$.
  We then have the following canonical map of sets:
  \begin{equation}\label{eq:egaiv38811}
    \varinjlim_{\lambda \in \Lambda}
    \Hom_{S_\lambda}(X_\lambda,Y_\lambda) \longrightarrow \Hom_S(X,Y).
  \end{equation}
  \par Similarly, suppose an element $\alpha \in \Lambda$, a scheme $X_\alpha$,
  and $\cO_{X_\alpha}$-modules $\sF_\alpha$ and
  $\sG_\alpha$ are given.
  We then denote by
  \begin{alignat*}{2}
    \omit\hfil$\displaystyle\{\sF_\lambda\}_{\lambda
    \in \Lambda}$\hfil\ignorespaces &\qquad\text{and}\qquad&
    \omit\hfil$\displaystyle\{\sG_\lambda\}_{\lambda
    \in \Lambda}$\hfil\ignorespaces
  \intertext{the inverse systems induced by $\{(X_\lambda,v_{\lambda\mu})\}$ and
  $\{(Y_\lambda,w_{\lambda\mu})\}$, where}
    \sF_\lambda \coloneqq v_{\alpha\lambda}^*(\sF_\alpha)
    &\qquad\text{and}\qquad&
    \sG_\lambda \coloneqq w_{\alpha\lambda}^*(\sG_\alpha)
  \end{alignat*}
  for $\alpha \le \lambda \le \mu$.
  Note that these families verify the conditions $\sF_\mu =
  v_{\lambda\mu}^*(\sF_\lambda)$ and $\sG_\mu =
  w_{\lambda\mu}^*(\sG_\lambda)$.
  The $\cO_X$-modules $\sF = v_\alpha^*(\sF_\alpha)$
  and $\sG = w_\alpha^*(\sG_\alpha)$ then satisfy $\sF =
  v_\lambda^*(\sF_\lambda)$ and $\sG = w_\lambda^*(\sG_\lambda)$ for every
  $\lambda \ge \alpha$, and we
  have the following canonical map of abelian groups: 
  \begin{equation}\label{eq:egaiv38513}
    \varinjlim_{\lambda \in \Lambda}
    \Hom_{\cO_{X_\lambda}}(\sF_\lambda,\sG_\lambda) \longrightarrow
    \Hom_{\cO_{X}}(\sF,\sG).
  \end{equation}
\end{citedsetup}
\begin{theorem}[Spreading out; see
  {\cite[Thms.\ 8.8.2 and 8.5.2]{EGAIV3}}]
  \label{thm:spreadingout}
  Fix\index{spreading out|(} notation as in \cref{setup:spreadingout}.
  \begin{enumerate}[label=$(\roman*)$,ref=\roman*]
    \item\label{thm:spreadingoutschemes}
      Suppose $S_0$ is quasi-compact and quasi-separated.
      For every scheme $X$ of finite presentation over $S$, there exists
      $\lambda \in \Lambda$, a scheme $X_\lambda$ of finite presentation over
      $S_\lambda$, and an $S$-isomorphism $X \isoto X_\lambda \times_{S_\lambda}
      S$.
    \item\label{thm:spreadingoutmorphisms}
      Suppose $X_\alpha$ is quasi-compact (resp.\ quasi-compact and
      quasi-separated) over $S_\alpha$, and $Y_\alpha$ is locally of finite type
      (resp.\ locally of finite presentation) over $S_\alpha$ for some $\alpha
      \in \Lambda$.
      Then, the map \cref{eq:egaiv38811} is injective (resp.\ bijective).
    \item\label{thm:spreadingoutsheaves}
      Suppose $X_\alpha$ is quasi-compact and quasi-separated over
      $S_\alpha$, and that $S_\alpha$ is quasi-compact and quasi-separated.
      For every quasi-coherent $\cO_X$-module $\sF$ of finite
      presentation, there exists $\lambda \in \Lambda$ and a quasi-coherent
      $\cO_{X_\lambda}$-module $\sF_\lambda$ of finite presentation such that
      $\sF$ is isomorphic to $u_\lambda^*(\sF_\lambda)$.
    \item\label{thm:spreadingoutsheafmorphisms}
      Suppose $X_\alpha$ is quasi-compact (resp.\ quasi-compact and
      quasi-separated) and that $\sF_\lambda$ is quasi-coherent of finite type
      (resp.\ of finite presentation) and $\sG_\lambda$ is quasi-coherent for
      some $\alpha \in \Lambda$.
      Then, the map \cref{eq:egaiv38513} is injective (resp.\ bijective).
  \end{enumerate}
\end{theorem}
\begin{table}[t]
  \centering
  \begin{tabular}[h]{ll}
    \toprule
    \multicolumn{1}{c}{Property of morphism of schemes} &
    \multicolumn{1}{c}{Proof}
    \\\cmidrule(lr){1-1}\cmidrule(lr){2-2}
    closed immersion & \cite[Thm.\ 8.10.5$(iv)$]{EGAIV3}\\
    flat & \cite[Thm.\ 11.2.6$(ii)$]{EGAIV3}\\
    projective & \cite[Thm.\ 8.10.5$(xiii)$]{EGAIV3}\\
    proper & \cite[Thm.\ 8.10.5$(xii)$]{EGAIV3}\\
    separated & \cite[Thm.\ 8.10.5$(v)$]{EGAIV3}\\
    smooth & \cite[Thm.\ 17.7.8$(ii)$]{EGAIV4}\\
    \midrule
    \multicolumn{1}{c}{Property of sheaf} & \multicolumn{1}{c}{Proof}
    \\\cmidrule(lr){1-1}\cmidrule(lr){2-2}
    flat & \cite[Thm.\ 11.2.6$(ii)$]{EGAIV3}\\
    locally free of rank $n$ & \cite[Prop.\ 8.5.5]{EGAIV3}\\
    \bottomrule
  \end{tabular}
  \caption{Some properties preserved under spreading out}
  {\footnotesize We assume that $S_0$ is quasi-compact and quasi-separated\\
  and that $X_\alpha$ and $Y_\alpha$ are of finite presentation over
  $S_\alpha$.}
  \label{table:spreadingout}
  \index{spreading out!properties under|ttindex{}}
\end{table}
We give the resulting objects in \cref{thm:spreadingout} a name.
\begin{definition}\label{def:spreadingout}
  Fix notation as in \cref{setup:spreadingout}.
  We say that $X_\lambda$ (resp.\
  $\sF_\lambda$) is a \textsl{model}\index{model|textbf|(}
  of $X$ (resp.\ $\sF$) over $S_\lambda$ in
  the situation of \cref{thm:spreadingout}\cref{thm:spreadingoutschemes} (resp.\
  \ref{thm:spreadingout}\cref{thm:spreadingoutsheaves}).
  If in the situation of
  \cref{thm:spreadingout}\cref{thm:spreadingoutmorphisms} (resp.\
  \ref{thm:spreadingout}\cref{thm:spreadingoutsheafmorphisms}), the map
  in \cref{eq:egaiv38811} (resp.\ \cref{eq:egaiv38513}) is bijective,
  and $f_\lambda$ (resp.\ $\varphi_\lambda$) is a lift of $f \in \Hom_S(X,Y)$
  (resp.\ $\varphi \in \Hom_{\cO_X}(\sF,\sG)$) under this map, then we also say
  that $f_\lambda$ (resp.\ $\varphi_\lambda$) is a
  \textsl{model}\index{model|textbf|)} of $f$ (resp.\ $\varphi$) over
  $S_\lambda$.
  \par Now let $\mathcal{P}$ be a property of schemes (resp.\ morphisms of
  schemes, modules, morphisms of modules).
  If a model $X_\lambda$ (resp.\ $f_\lambda$, $\sF_\lambda$, $\varphi_\lambda$)
  can always be chosen such that $X$ (resp.\ $f$, $\sF$, $\varphi$) has
  $\mathcal{P}$ if and only if $X_\lambda$ (resp.\ $f_\lambda$, $\sF_\lambda$,
  $\varphi_\lambda$) has $\mathcal{P}$, then we say that $\mathcal{P}$ is
  \textsl{preserved under spreading out}.\index{spreading out|)}
\end{definition}
We record in \cref{table:spreadingout} some properties of schemes, morphisms,
sheaves, and morphisms of sheaves that can be descended to a model that we will
use.
See the properties labeled (IND) in \cite[App.\ C]{GW10} and the properties in
the ``spreading out'' column in \cite[App.\ C.1, Table 1]{Poo17} for more
exhaustive lists.
\medskip
\par We now specialize to the case where $S = \Spec k$ for a field $k$ of
characteristic zero.
\begin{definition}\label{def:reductionmodulop}
  Let $k$ be a field of characteristic zero, and write $k =
  \varinjlim_{\lambda \in \Lambda} A_\lambda$, where the rings $A_\lambda$ are
  finite type extensions of $\ZZ$ in $k$.
  Let $S_\lambda = \Spec A_\lambda$ in \cref{setup:spreadingout}.
  Given models over $S_\lambda$ as in \cref{def:spreadingout}, for every closed
  point $\fp \in \Spec A_\lambda$,
  we say that $X_{\fp} \coloneqq X_\lambda \times_{A_\lambda} \kappa(\fp)$
  (resp.\ $\sF_\fp \coloneqq \sF\rvert_{X_\fp}$, $f_\fp \coloneqq
  f_\lambda\rvert_{X_\fp} \colon X_\fp \to Y_\fp$, $\varphi_\fp \coloneqq
  \varphi\rvert_{\sF_\fp}\colon \sF_\fp \to \sG_\fp$) is the
  \textsl{reduction modulo $\fp$}\index{reduction modulo $\fp$|textbf}
  of $X$ (resp.\ $\sF$, $f$, $\varphi$).
  \par Now let $\mathcal{P}$ be a property of schemes (resp.\ morphisms of
  schemes, modules, morphisms of modules).
  If a model $X_\lambda$ (resp.\ $f_\lambda$, $\sF_\lambda$, $\varphi_\lambda$)
  can always be chosen such that $X$ (resp.\ $f$, $\sF$, $\varphi$) has
  $\mathcal{P}$ if and only if $X_\fp$ (resp.\ $f_\fp$, $\sF_\fp$,
  $\varphi_\fp$) has $\mathcal{P}$ for every $\fp \in \Spec A_\lambda$, then we
  say that $\mathcal{P}$ is \textsl{preserved under reduction modulo $\fp$}.
\end{definition}
\par An important feature of reduction modulo $\fp$ is the following:
\begin{lemma}\label{lem:reductionmodulopischarp}
  With notation as in \cref{setup:spreadingout,def:reductionmodulop}, for every
  $\lambda \in \Lambda$, the residue fields $\kappa(\fp)$ of $A_\lambda$ are
  finite fields for every $\fp \in \Spec A_\lambda$.
  Moreover, the set $\{\Char \kappa(\fp)\}_{\fp \in \Spec A_\lambda} \subseteq
  \NN$ is unbounded for every $\lambda \in \Lambda$.
\end{lemma}
\begin{proof}
  The first statement is \cite[Lem.\ 10.4.11.1]{EGAIV3}.
  For the second, consider the morphism $u_\lambda\colon \Spec A_\lambda \to
  \Spec \ZZ$, which is of finite type.
  By Chevalley's theorem\index{Chevalley, Claude!theorem} \cite[Thm.\
  1.8.4]{EGAIV1}, the image of $u_\lambda$ is constructible.
  Moreover, since $\ZZ \to A_\lambda$ is injective, the morphism $u_\lambda$ is
  dominant, and in particular the image contains $(0) \in \Spec \ZZ$.
  Thus, the image of $u_\lambda$ is open, and therefore contains points $\fp \in
  \Spec A_\lambda$ with residue fields of unbounded characteristic.
\end{proof}
We record in \cref{table:reductionmodulop} some properties of schemes,
morphisms, sheaves, and morphisms of sheaves that are preserved under reduction
modulo $\fp$.
Note that these properties are constructible on $\Spec A_\lambda$, hence for
arbitrary models, as long as the original object over $k$ satisfied the property
listed, these properties will hold when $\Char\kappa(\fp)$ is sufficiently
large.
\begin{table}[t]
  \centering
  \begin{tabular}{ll}
    \toprule
    \multicolumn{1}{c}{Property of scheme} & \multicolumn{1}{c}{Proof}
    \\\cmidrule(lr){1-1}\cmidrule(lr){2-2}
    dimension $n$ (when $X$ irreducible) & \cite[Cor.\ 9.5.6]{EGAIV3}\\
    geometrically irreducible & \cite[Thm.\ 9.7.7$(i)$]{EGAIV3}\\
    geometrically normal & \cite[Prop.\ 9.9.4$(iii)$]{EGAIV3}\\
    geometrically reduced & \cite[Thm.\ 9.7.7$(iii)$]{EGAIV3}\\
    \midrule
    \multicolumn{1}{c}{Property of sheaf} & \multicolumn{1}{c}{Proof}
    \\\cmidrule(lr){1-1}\cmidrule(lr){2-2}
    (very) ample over $k$ (when $X/k$ proper) & \cite[Prop.\ 9.6.3]{EGAIV3}\\
    \midrule
    \multicolumn{1}{c}{Property of morphism of sheaves} &
    \multicolumn{1}{c}{Proof}
    \\\cmidrule(lr){1-1}\cmidrule(lr){2-2}
    bijective & \cite[Cor.\ 9.4.5]{EGAIV3}\\
    injective & \cite[Cor.\ 9.4.5]{EGAIV3}\\
    surjective & \cite[Cor.\ 9.4.5]{EGAIV3}\\
    \bottomrule
  \end{tabular}
  \caption{Some properties preserved under reduction modulo $\fp$}
  \label{table:reductionmodulop}
  \index{reduction modulo $\fp$!properties under|ttindex{}}
\end{table}
\medskip
\par We will also need to spread out more than what we have discussed above.
We discuss these operations below.
\begin{remark}[Spreading out and reduction modulo $\fp$ for other objects]
  \label{rem:spreadingoutothers}
  Fix notation as in \cref{def:reductionmodulop}.
  We will freely use the properties in
  \cref{table:spreadingout,table:reductionmodulop}.
  \begin{enumerate}[label=$(\alph*)$,ref=\alph*]
    \item (Ideal sheaves)\index{spreading out!ideal sheaves}
      Let $\fa \subseteq \cO_X$ be a coherent ideal sheaf.
      We can then spread out $\fa$ and the inclusion into $\cO_X$ to a model
      $\fa_\lambda \to \cO_{X_\lambda}$.
      We can further assume that $\fa_\fp \to \cO_{X_\fp}$ is injective for all
      $\fp \in \Spec A_\lambda$.
    \item (Cartier divisors)\index{spreading out!Cartier divisors}
      Let $D$ be an effective Cartier divisor on $X$.
      We can then spread out the ideal sheaf $\cO_X(-D)$ to a model
      $\cO_{X_\lambda}(-D_\lambda)$ on $X_\lambda$, which remains invertible.
      Thus, $D_\fp$ is an effective Cartier divisor for all $\fp \in
      \Spec A_\lambda$, since $\cO_{X_\fp}(-D_\fp) \to \cO_{X_\fp}$ is injective
      for all $\fp \in \Spec A_\lambda$.
      This can be extended to arbitrary Cartier divisors and to $\QQ$- and
      $\RR$-coefficients by linearity.
    \item (Weil divisors)\index{spreading out!Weil divisors}
      Suppose $X$ is irreducible, and suppose
      $D$ is a prime Weil divisor on $X$.
      Then, one can find $\lambda \in \Lambda$ such that $X$ and
      $D$ have models $X_\lambda$ and $D_{\lambda}$
      over $S_\lambda$ such that every $D_{\fp}$ is a prime Weil divisor
      (by preserving dimension, integrality, and the fact that $D
      \hookrightarrow X$ is a closed immersion) on $X_\fp$ (by preserving
      irreducibility and dimension of $X$).
      This can be extended to arbitrary Weil divisors and to $\QQ$- and
      $\RR$-coefficients by linearity.
      \par If $X$ is normal, and $D$ is a Cartier divisor (resp.\ $\QQ$-Cartier
      divisor, $\RR$-Cartier divisor) on $X$, then by simultaneously choosing
      models for a Cartier divisor (resp.\ $\QQ$-Cartier divisor,
      $\RR$-Cartier divisor) and the Weil divisor (resp.\ $\QQ$-Weil divisor,
      $\RR$-Weil divisor) associated to it, we can preserve the
      property of being a Cartier divisor (resp.\ $\QQ$-Cartier divisor,
      $\RR$-Cartier divisor) under reduction modulo $\fp$.
  \end{enumerate}
\end{remark}
\subsection{Singularities vs.\ \emph{F}-singularities}
We can now define the following notions in characteristic zero obtained via
reduction modulo $\fp$.
See \cref{def:finjective} for the definition of $F$-injective singularities in
positive characteristic.
\begin{definition}\label{def:densefsings}
  Fix notation as in
  \cref{setup:spreadingout,def:reductionmodulop,rem:spreadingoutothers}.
  Let $X$ be a scheme of finite type over a field $k$ of characteristic zero.
  We say that $X$ is of \textsl{$F$-injective type} (resp.\ \textsl{dense
  $F$-injective type})%
  \index{dense F-injective type@dense $F$-injective type|textbf}
  if there exists a model
  $X_\lambda$ over $A_\lambda$ such that $X_\fp$ is $F$-injective for an open
  dense (resp.\ dense) set of closed points $\fp \in \Spec A_\lambda$.
  \par Now let $(X,\Delta,\fa)$ be an effective log triple such that $X$ is
  normal and of finite type over a field $k$ of characteristic zero.
  Fix models $X_\lambda$, $\Delta_\lambda$, and $\fa_\lambda$ over $\Spec
  A_\lambda$.
  We say that $(X,\Delta,\fa)$ is of \textsl{$F$-pure type} (resp.\
  \textsl{dense $F$-pure type}%
  \index{dense F-pure type@dense $F$-pure type|textbf})
  if $(X_\fp,\Delta_\fp,\fa_\fp^t)$ is $F$-pure
  for an open dense (resp.\ dense) set of closed points $\fp \in \Spec A_\lambda$.
  We say that $(X,\Delta,\fa)$ is
  of \textsl{strongly $F$-regular type} (resp.\ 
  \textsl{dense strongly $F$-regular
  type}\index{dense strongly F-regular type@dense strongly $F$-regular type|textbf})
  if $(X_\fp,\Delta_\fp,\fa_\fp^t)$ is strongly $F$-regular
  for an open dense (resp.\ dense) set of closed points $\fp \in \Spec A_\lambda$.
\end{definition}
One can define similar notions for all $F$-singularities of rings and of pairs
and triples.
The notions defined above are those that appear in the sequel.
\medskip
\par We will need the following result connecting singularities of pairs and
$F$-singularities of pairs, which relates multiplier ideals and test
ideals under reduction modulo $\fp$.
For rings, this result is due to
Smith\index{Smith, Karen E.} \cite[Thm.\ 3.1]{Smi00} and Hara\index{Hara, Nobuo}
\cite[Thm.\ 5.9]{Har01}, and for pairs, this result is due to
Takagi\index{Takagi, Shunsuke} \cite[Thm.\ 3.2]{Tak04} and
Hara\index{Hara, Nobuo}--Yoshida\index{Yoshida, Ken-ichi}
\cite[Thm.\ 6.8]{HY03}.
\begin{figure}[t]
  \centering
  \begin{tikzcd}
    \text{klt} \rar[Rightarrow]\dar[Rightarrow]
    & \text{rational} \dar[Rightarrow]\arrow[Leftrightarrow,bend left=25]{rrr}
    & & \text{strongly $F$-regular} \rar[Rightarrow]\dar[Rightarrow]\arrow[Leftrightarrow,bend right=25,end anchor={[xshift=0.75em]north}]{lll}
    & \text{$F$-rational}\dar[Rightarrow]\\
    \text{log canonical} \rar[Rightarrow]\arrow[Leftarrow,bend right=25]{rrr}
    & \text{Du Bois}
    & & \text{$F$-pure} \rar[Rightarrow]
    & \text{$F$-injective}\arrow[Rightarrow,bend left=25]{lll}
  \end{tikzcd}
  \caption{Singularities vs.\ \emph{F}-singularities}
  {\footnotesize The left- and right-hand sides of the diagram are connected via
  reduction modulo $\fp$.\\
  This is a simplified version of \cite[Fig.\ on p.\ 86]{ST14}.
  See \cite[p.\ 86]{ST14} for references for each implication.}
  \label{fig:singsvsfsings}
  \index{klt!and reduction modulo $\fp$|ff{}}
  \index{log canonical!and reduction modulo $\fp$|ff{}}
  \index{rational singularities!and reduction modulo $\fp$|ff{}}
  \index{Du Bois singularities!and reduction modulo $\fp$|ff{}}
\end{figure}
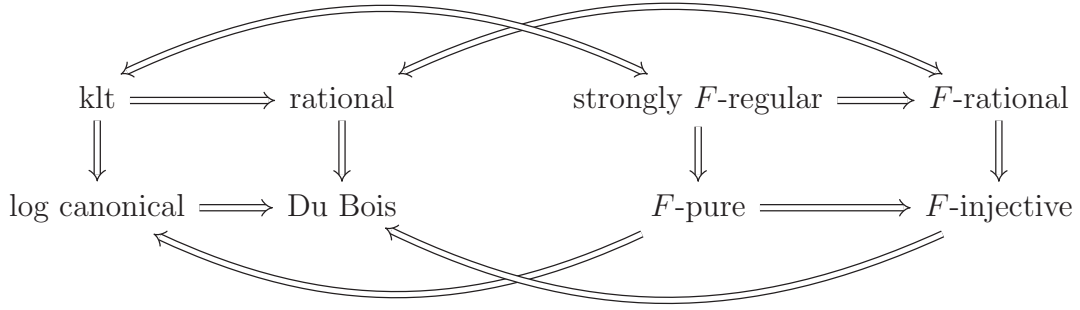
There are many more results describing how singularities and
$F$-singularities are related, which we will not state explicitly;
see \cref{fig:singsvsfsings} for a summary of what is known.
\begin{theorem}[see {\cite[Thm.\ 2.5]{Tak08}}]\label{thm:takagiredmodp}
  Let\index{multiplier ideal, $\mathcal{J}(X,\Delta,\mathfrak{a}^t)$!and reduction modulo $\fp$}
  $(X,\Delta,\fa)$ be an effective log triple such that $X$ is normal and
  finite type over a field $k$ of characteristic zero, and such
  that $K_X+\Delta$ is $\RR$-Cartier.
  With notation as in
  \cref{setup:spreadingout,def:reductionmodulop,rem:spreadingoutothers}, fix
  models $X_\lambda$, $\Delta_\lambda$, and $\fa_\lambda$ over $\Spec
  A_\lambda$.
  Then, for all $t \ge 0$, we have
  \begin{equation}\label{eq:multisuniversaltest}
    \cJ\bigl( (X,\Delta);\fa^t\bigr)_\fp = \tau\bigl(
    (X_\fp,\Delta_\fp);\fa_\fp^t\bigr)
  \end{equation}
  when $\Char \kappa(\fp)$ is sufficiently large.
  In particular, $(X,\Delta,\fa^t)$ is klt if and only if $(X,\Delta,\fa^t)$ is
  of strongly $F$-regular type.
  Moreover, for every sequence
  of closed points $\fp \in \Spec A_\lambda$ such that the characteristic of
  $\kappa(\fp)$ goes to infinity, we have that the limit of the $F$-pure
  thresholds $\fpt_x((X_\fp,\Delta_\fp);\fa_\fp)$ is the log canonical threshold
  $\lct_x( (X,\Delta);\fa)$.
\end{theorem}
Note that implicit in the statement of \cref{thm:takagiredmodp} is that both
objects in \cref{eq:multisuniversaltest} make sense.
For the left-hand side, this requires choosing a model of a log resolution as
well, from which one obtains a model of $\cJ( (X,\Delta);\fa^t)$.
We also note that the characteristic of $\kappa(\fp)$ is unbounded by
\cref{lem:reductionmodulopischarp}.
\begin{proof}
  All but the last part of the statement of \cref{thm:takagiredmodp} is proved
  in \cite[Thm.\ 2.5]{Tak08}.
  To prove this last statement, let $\{\fp_i\}_{i \in \NN}$ be a sequence of
  closed points in $\Spec A_\lambda$ such that $\Char\kappa(\fp_i) \to \infty$
  as $i \to \infty$.
  We claim that for every $s \ge 0$, we have
  \begin{equation}\label{eq:lctseqcond}
    \lct_x\bigl((X,\Delta);\fa\bigr) \ge 
    \fpt_x\bigl((X_{\fp_i},\Delta_{\fp_i});\fa_{\fp_i}\bigr) > s
  \end{equation}
  for $i \gg 0$.
  The first inequality automatically holds since the inclusion $\supseteq$ in
  \cref{eq:multisuniversaltest} holds for every $\fp \in \Spec A_\lambda$; see
  \cite[Thm.\ 6.7]{Sch10}.
  The second inequality holds for $i \gg 0$ since \cref{eq:multisuniversaltest}
  holds for $s = t$ when $\Char\kappa(\fp_i)$ is sufficiently large.
\end{proof}

\chapter{The ampleness criterion of de~Fernex--K\"uronya--Lazarsfeld}
\label{ch:dfkl}
In this chapter, we prove a criterion for ampleness using asymptotic
cohomological functions (\cref{thm:dfkl41}), which is originally due to
de Fernex\index{de Fernex, Tommaso},
K\"uronya\index{Kuronya, Alex@K\"uronya, Alex}, and
Lazarsfeld\index{Lazarsfeld, Robert} over the complex numbers \cite[Thm.\
4.1]{dFKL07}.
A key ingredient is a lemma asserting that the base ideals associated to
multiples of a non-nef divisor grow at least like powers of an ideal defining a
curve (\cref{prop:dfkl31}).
This material is mostly from \cite{Mur}, with some modifications in the proof of
\cref{prop:dfkl31} using ideas from \cite[Lems.\ 4.3 and 4.4]{MPST}.
\par We briefly describe the main
difficulties in adapting the proof of \cite[Thm.\ 4.1]{dFKL07} to positive
characteristic.
First, the proof of \cite[Prop.\ 3.1]{dFKL07} requires resolutions of
singularities, and because of this, we need to adapt the proof to use
alterations instead.
Second, we need to replace asymptotic multiplier ideals with asymptotic test
ideals in the same proof, which requires reducing to the case when the ground
field is $F$-finite by using the gamma construction (\cref{thm:gammaconstintro}).
Finally, \cite{dFKL07} uses the assumption that the ground field is uncountable
to choose countably many very general divisors that facilitate an inductive
argument.
Our version of \cite[Thm.\ 4.1]{dFKL07} therefore needs to reduce to this case.
\section{Motivation and statement}
We start by motivating the statement of our ampleness criterion.
Let $X$ be a projective variety of dimension $n > 0$.
For every Cartier divisor $L$ on $X$, we have
\begin{align*}
  h^i\bigl(X,\cO_X(mL)\bigr) &= O(m^n)
  \intertext{for every $i$; see \cite[Ex.\ 1.2.20]{Laz04a}.
  In \cite[Thm.\ 4.1]{dFKL07}, de Fernex, K\"uronya, and Lazarsfeld asked when
  the higher cohomology groups have submaximal growth, i.e., when}
  h^i\bigl(X,\cO_X(mL)\bigr) &= o(m^n).
\end{align*}
They proved that over the complex numbers, ample Cartier
divisors $L$ are characterized by having submaximal growth of higher cohomology
groups for small perturbations of $L$.
\par We prove the following version of their result, which is valid over
arbitrary fields, and in particular, is valid over possibly imperfect fields of
positive characteristic.
\begin{customthm}{E}\label{thm:dfkl41}
  Let $X$ be a projective variety of dimension $n > 0$ over a field $k$.
  Let $L$ be an $\RR$-Cartier divisor on $X$, and consider the following
  property:
  \begin{enumerate}[label=$(\star)$,ref=\star]
    \item\label{thm:dfkl41cond}
      There exists a very ample Cartier divisor $A$ on $X$ and a
      real number $\varepsilon > 0$ such that
      \[
        \widehat{h}^i(X,L-tA) \coloneqq \limsup_{m \to \infty}
        \frac{h^i\bigl(X,\cO_X\bigl(\lceil m(L-tA) \rceil\bigr)\bigr)}{m^n/n!} =
        0
      \]
      for all $i > 0$ and for all $t \in [0,\varepsilon)$.
  \end{enumerate}
  Then, $L$ is ample if and only if $L$ satisfies $(\ref{thm:dfkl41cond})$ for
  some pair $(A,\varepsilon)$.
\end{customthm}
We note that one can have $\widehat{h}^i(X,L) = 0$ for all $i > 0$ without $L$
being ample, or even pseudoeffective, as seen in the following example.
\begin{citedex}[{\cite[Ex.\ 3.3]{Kur06}}]\label{ex:hhatabvar}
  Let $A$ be a abelian variety of dimension $g$ over an algebraically closed
  field $k$, and let $L$ be a line bundle on $A$.
  We recall that $K(L) \subseteq A$ is defined to be the maximal closed
  subscheme of $A$ such that the
  \textsl{Mumford bundle}\index{Mumford, David, bundle|textbf}
  \[
    \Lambda(L) \coloneqq m^*(L) \otimes p_1^*(L)^{-1} \otimes p_2^*(L)^{-1}
  \]
  is trivial on $A \times_k A$ \cite[p.\ 115]{Mum08}, where $m$ is the
  multiplication map and $p_1,p_2$ are the first and second projections,
  respectively.
  We also recall that $L$ is \textsl{non-degenerate} if $K(L)$ is finite
  \cite[p.\ 145n]{Mum08}.\index{non-degenerate line bundles and divisors|textbf}
  By Mumford's index theorem \cite[Thm.\ on p.\ 140]{Mum08}, we have
  \begin{equation}\label{eq:mumindex}
    h^i(A,L) = \begin{cases}
      (-1)^{i(L)}\cdot(L^g) & \text{if $i = i(L)$}\\
      \hfil 0 & \text{otherwise}
    \end{cases}
  \end{equation}
  for non-degenerate line bundles $L$, where $i(L)$ is the index of
  $L$ \cite[p.\ 145]{Mum08}.
  In particular, this holds for ample line bundles $L$ on $A$ by \cite[App.\ 1
  on p.\ 57]{Mum08}, in which case $i(L) = 0$ by the proof of \cite[Thm.\ on p.\
  140]{Mum08}.
  \begin{figure}[t]
    \centering
    \tikzexternalenable
    \begin{tikzpicture}

      \newcommand{\radiusx}{2}
      \newcommand{\radiusy}{.66}
      \newcommand{\theheight}{2}

      \coordinate (a) at (-{\radiusx*sqrt(1-(\radiusy/\theheight)*(\radiusy/\theheight))},{\radiusy*(\radiusy/\theheight)});

      \coordinate (b) at ({\radiusx*sqrt(1-(\radiusy/\theheight)*(\radiusy/\theheight))},{\radiusy*(\radiusy/\theheight)});

      \coordinate (a1) at (-{\radiusx*sqrt(1-(\radiusy/\theheight)*(\radiusy/\theheight))},{2*\theheight-\radiusy*(\radiusy/\theheight)});

      \coordinate (b1) at ({\radiusx*sqrt(1-(\radiusy/\theheight)*(\radiusy/\theheight))},{2*\theheight-\radiusy*(\radiusy/\theheight)});

      \draw[fill=gray!30,thick] (a)--(b1)--(a1)--(b)--cycle;

      \fill[gray!5] circle (\radiusx{} and \radiusy);

      \begin{scope}
      \clip ([xshift=-2mm]a) rectangle ($(b)+(2mm,-2*\radiusy)$);
      \draw[thick] circle (\radiusx{} and \radiusy);
      \end{scope}

      \begin{scope}
      \clip ([xshift=-2mm]a) rectangle ($(b)+(2mm,2*\radiusy)$);
      \draw[thick] circle (\radiusx{} and \radiusy);
      \end{scope}

      \fill[gray!5] (0,2*\theheight) circle (\radiusx{} and \radiusy);

      \begin{scope}
      \clip ([xshift=-2mm]a1) rectangle ($(b1)+(2mm,-2*\radiusy)$);
      \draw[thick] (0,2*\theheight) circle (\radiusx{} and \radiusy);
      \end{scope}

      \begin{scope}
      \clip ([xshift=-2mm]a1) rectangle ($(b1)+(2mm,2*\radiusy)$);
      \draw[thick] (0,2*\theheight) circle (\radiusx{} and \radiusy);
      \end{scope}

      \node at (0,2*\theheight) {\footnotesize $\widehat{h}^0 \ne 0$};
      \node at (0,0) {\footnotesize $\widehat{h}^2 \ne 0$};

    \end{tikzpicture}
    \tikzexternaldisable
    \caption{Asymptotic cohomological functions on an abelian surface}
    {\footnotesize Illustration from \cite[Fig.\ 4]{ELMNP05}}
    \label{fig:elmnpabeliansurface}
  \end{figure}
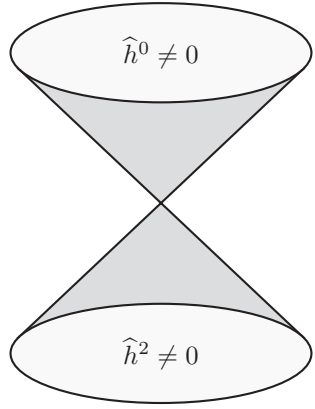
  \par Now let $\xi$ be a nef $\RR$-Cartier divisor on $A$.
  Then, $\xi$ can be written as the limit of ample $\QQ$-Cartier divisors on
  $A$.
  Thus, by using \cref{eq:mumindex} and the 
  homogeneity and continuity of asymptotic cohomological functions
  (see \cref{rem:kur291reductions}), we have
  \[
    \widehat{h}^i(A,\xi) = \begin{cases}
      (\xi^g) & \text{if $i = 0$}\\
      \hfil 0 & \text{otherwise}
    \end{cases}
  \]
  and we note that $(\xi^g) = 0$ if $\xi$ is nef but not ample \cite[Cor.\
  1.5.18]{Laz04a}.
  By asymptotic Serre duality (\cref{prop:asympoticserre}), we therefore see
  that for nef but not ample $\RR$-Cartier divisors $\xi$, we have
  $\widehat{h}^i(A,-\xi) = 0$ for all $i$, even though $-\xi$ is not ample, or
  even pseudoeffective.
  \par We now illustrate this phenomenon in a more concrete situation.
  Recall that if $X$ is a complete scheme over a field, then the
  \textsl{N\'eron--Severi space} is the $\RR$-vector space
  \begin{equation}\label{eq:neronseveri}
    \gls*{neronseveri} \coloneqq \Div_\RR(X)/\mathord{\equiv_\RR},
  \end{equation}
  where $\equiv_\RR$ denotes $\RR$-linear equivalence.
  This vector space is finite-dimensional by \cite[Prop.\ 2.3]{Cut15}.
  Now if $A$ is an abelian surface, then the ample cone in $N^1_\RR(X)$ is
  $\{\xi \in N^1_\RR(X) \mid \widehat{h}^0(\xi) \ne 0\}$.
  By \cite[Lem.\ 1.5.4]{Laz04a}, the classes $-\xi$ considered above for nef but
  not ample $\RR$-Cartier divisors $\xi$ correspond to classes in the boundary
  of the cone $\{\xi \in N^1_\RR(X) \mid \widehat{h}^2(\xi) \ne 0\}$.
  See \cref{fig:elmnpabeliansurface} for an illustration of the case when
  the Picard rank $\rho(A)$ of $A$ is $3$.
  We note that if $A = E \times_k E$ for a sufficiently general elliptic curve
  $E$, then $\rho(A) = 3$.
  This follows from the fact that $\End_k(E) \otimes_\ZZ \QQ \simeq \QQ$ for
  sufficiently general $E$ by a
  theorem of Deuring\index{Deuring, Max} \cite[Thm.\ on p.\ 201]{Mum08}, hence
  $\rho(A) = 3$ by a lemma of Murty\index{Murty, V. Kumar}
  \cite[Prop.\ 2.3]{Laf}.
\end{citedex}
\section{A lemma on base loci}\label{sect:lemonbaseloci}
A key ingredient in our proof of \cref{thm:dfkl41} is the following result on
base loci, which is the analogue of \cite[Prop.\ 3.1]{dFKL07} over arbitrary
fields.
The lemma says that base ideals associated to multiples of non-nef divisors
grow like powers of an ideal defining a curve.
\begin{proposition}\label{prop:dfkl31}
  Let $V$ be a normal projective variety of dimension at least two over
  a field $k$.
  Let $D$ be a Cartier divisor on $V$, and suppose there exists an integral
  curve $Z \subseteq V$ such that $(D \cdot Z) < 0$.
  Denoting by $\fa \subseteq \cO_V$ the ideal sheaf defining $Z$,
  there exist positive integers $q$ and $c$ such that for every integer $m
  \ge c$, we have
  \[
    \fb\bigl(\lvert mqD \rvert \bigr) \subseteq \fa^{m-c}.
  \]
\end{proposition}
Here, $\fb(\lvert D \rvert)$ denotes the base ideal of the Cartier
divisor $D$; see \cref{def:baseideal}.
\par To use Bertini theorems, we need to reduce to the case when the ground
field $k$ is infinite.
Moreover, in positive characteristic, we use asymptotic test ideals instead of
asymptotic multiplier ideals, which requires also reducing to the case where the
ground field is $F$-finite.
\begin{lemma}\label{lem:dfkl31ffinred}
  To prove \cref{prop:dfkl31}, we may assume that the ground field $k$ is
  infinite, and in positive characteristic, we may also assume that $k$ is
  $F$-finite.
\end{lemma}
\begin{proof}
  We first construct a sequence $k \subseteq k' \subseteq K$
  of two field extensions such that $V \times_k K$ is integral and normal, where
  $k'$ is infinite and $K$ is $F$-finite in positive characteristic.
  If $k$ is already infinite, then let $k' = k$.
  Otherwise, consider the purely transcendental extension $k \subseteq k(x)$.
  To show that $V \times_k k'$ is integral and normal, let $\bigcup_j U_j$ be an
  affine open covering of $V$.
  Then, $V \times_k k'$ is covered by affine open subsets that are localizations
  of the normal varieties $U_j \times_k \Spec k[x]$,
  which pairwise intersect, hence $V \times_k k'$ is integral and normal.
  The same argument shows that $Z \times_k k'$ is an integral curve.
  We set $K = k'$ in characteristic zero, and in positive characteristic, the
  gamma construction (\cref{thm:gammaconst}) shows that there is a
  field extension $k' \subseteq K$ such that $K$ is $F$-finite,
  $V \times_k K$ is integral and normal, and $Z \times_k K$ is integral.
  Note that $K$ is infinite since it contains the infinite field $k'$.
  \par We now show that the special case when $k$ is infinite and $F$-finite
  implies the general case.
  Let $\pi\colon V \times_k k' \to V$ be the first projection morphism, which we
  note is faithfully flat by base change.
  Since $(\pi^*D \cdot \pi^*Z) = (D \cdot Z) < 0$ by \cite[Prop.\ B.17]{Kle05},
  the special case of \cref{prop:dfkl31} implies 
  \[
    \fb\bigl(\lvert mq\,\pi^*D \rvert\bigr) \subseteq (\pi^{-1}\fa \cdot
    \cO_{V \times_k k'})^{m-c}.
  \]
  Then, since $\pi$ is faithfully flat and since $\fb(\lvert mq\,\pi^*D \rvert )
  = \pi^{-1} \fb(\lvert mqD \rvert) \cdot \cO_{V \times_k k'}$ by flat base
  change, we have $\fb(\lvert mqD \rvert ) \subseteq \fa^{m-c}$ by \cite[Thm.\
  7.5$(ii)$]{Mat89}.
\end{proof}
\begin{remark}
  When $k$ is $F$-finite of characteristic $p > 0$, then one can set $K$ to be
  $k(x^{1/p^\infty})$ in the proof of
  \cref{lem:dfkl31ffinred}, since integrality and normality are preserved under
  limits of schemes with affine and flat transition morphisms \cite[Cor.\
  5.13.4]{EGAIV2}.
\end{remark}
We now focus on proving \cref{prop:dfkl31} in positive characteristic; see
\cref{rmk:dfkl31char0} for the characteristic zero case.
We have incorporated some ideas from \cite[Lems.\ 4.3 and 4.4]{MPST}.
In the proof below, we will use the fact \cite[Lem.\ B.12]{Kle05} that if
$W$ is a one-dimensional subscheme of a complete scheme $X$ over a field, and if
$D$ is a Cartier divisor on $X$, then
\begin{equation}\label{eq:nonredintersection}
  (D \cdot W) = \sum_\alpha \length_{\cO_{X,\eta_\alpha}} \bigl(
  \cO_{W_\alpha,\eta_\alpha} \bigr) \cdot (D \cdot W_\alpha),
\end{equation}
where the $W_\alpha$ are the one-dimensional components of $W$ with generic
points $\eta_\alpha \in W_\alpha$.
\begin{proof}[Proof of \cref{prop:dfkl31} in positive characteristic]
  By \cref{lem:dfkl31ffinred}, it suffices to consider the case when the ground
  field $k$ is infinite and $F$-finite.
  The statement is trivial if $H^0(V,\cO_V(mD)) = 0$ for every integer $m > 0$,
  since in this case $\fb(\lvert mqD \rvert) = 0$ for all positive integers
  $m,q$.
  We therefore assume $H^0(V,\cO_V(mD)) \ne 0$ for some integer $m > 0$.
  \par We first set some notation.
  Let $\eta\colon V_1 \to V$ be the normalized blowup of $Z \subseteq
  V$, and denote $E \coloneqq \eta^{-1}(Z)$.
  Consider a regular alteration $\varphi\colon V' \to V_1$ for
  $(V_1,E)$ as in \cite[Thm.\ 4.1]{dJ96}, and set $D' \coloneqq (\eta \circ
  \varphi)^*D$.
  Note that in this case, $E' \coloneqq \varphi^*E = (\eta \circ
  \varphi)^{-1}(Z)$ is a Cartier divisor with simple normal crossing support.
  The proof proceeds in four steps.
  \begin{step}\label{step:dfkl31mpst}
    It suffices to show that there exists a positive integer $a$ such that
    for every integer $m > 0$, we have
    \begin{equation}\label{eq:mpstgoal}
      \fb\bigl(\lvert maD' \rvert\bigr) \subseteq \cO_{V'}(-mE'_\red).
    \end{equation}
  \end{step}
  Consider the commutative diagram
  \[
    \begin{tikzcd}
      V' \rar{\varphi_2}\arrow{dr}[swap]{\varphi} & V_2\dar{\varphi_1}\\
      & V_1 \rar{\eta} &V
    \end{tikzcd}
  \]
  where the triangle is the Stein factorization for $\varphi$ \cite[Cor.\
  III.11.5]{Har77}.
  Note that by construction of the Stein factorization, the scheme $V_2$ is a
  normal projective variety.
  Now by setting $b$ to be the largest coefficient appearing in $E'$, we see that
  $\cO_{V'}(-bE_\red) \subseteq \cO_{V'}(-E)$.
  Thus, we have
  \begin{equation}\label{eq:mpstintermediate}
    \varphi_2^{-1}\fb\bigl(\lvert
    mab\,(\eta \circ \varphi_1)^*D \rvert\bigr) \cdot \cO_{V'} =
    \fb\bigl(\lvert mabD' \rvert\bigr) \subseteq
    \cO_{V'}(-mbE'_\red) \subseteq \cO_{V'}(-mE')
  \end{equation}
  by \cref{eq:mpstgoal}, where the first equality holds by
  \cref{lem:baselocusnormal} since $\varphi_2$ is birational.
  Setting $q = ab$ and pushing forward by $\varphi_2$, we have
  \begin{align*}
    \fb\bigl(\lvert mq\,(\eta \circ \varphi_1)^*D \rvert\bigr) &=
    \fb\bigl(\lvert mq\,(\eta \circ \varphi_1)^*D \rvert\bigr) \cdot
    \varphi_{2*}\cO_{V'}\\
    &= \varphi_{2*}\Bigl( \varphi_2^{-1}\fb\bigl(\lvert
    mq\,(\eta \circ \varphi_1)^*D \rvert\bigr) \cdot \cO_{V'}\Bigr)
    \subseteq \cO_{V_2}(-m\varphi_1^*E)
  \end{align*}
  where the first equality and last inclusion hold by the fact that $V_2$ is
  normal, hence $\varphi_{2*}\cO_{V'} = \cO_{V_2}$ \cite[Proof of Cor.\
  III.11.4]{Har77}, and the second equality holds by definition of restriction
  of scalars.
  Next, we push forward by $\varphi_1$ and intersect with the subsheaf
  $\cO_{V_1} \subseteq \varphi_{1*}\cO_{V_2}$ to obtain the chain of inclusions
  \begin{align*}
    \fb\bigl( \lvert mq\,\eta^*D \rvert\bigr) &\subseteq
    \varphi_{1*}\Bigl(\fb\bigl(\lvert mq\,(\eta \circ \varphi_1)^*D
    \rvert\bigr)\Bigr) \cap \cO_{V_1}\\
    &\subseteq \varphi_{1*}\bigl(\cO_{V_2}(-m\,\varphi_1^*E)\bigr) \cap
    \cO_{V_1} = \cO_{V_1}(-mE),
  \end{align*}
  where the last equality holds by the fact that $\varphi_1$ is finite, hence
  integral, and then by properties of integral closure \cite[Props.\ 1.5.2 and
  1.6.1]{HS06}.
  Finally, we push forward by $\eta$ to obtain
  \[
    \fb\bigl( \lvert mqD \rvert\bigr) \subseteq
    \eta_*\fb\bigl( \lvert mq\,\eta^*D \rvert\bigr)
    \subseteq \overline{\fa^m},
  \]
  where $\overline{\fa^m}$ is the integral closure of $\fa^m$ \cite[Rem.\
  9.6.4]{Laz04b}.
  By \cite[Cor.\ 1.2.5]{HS06}, there exists an integer $c$ such that
  $\overline{\fa^{\ell+1}} = \fa \cdot \overline{\fa^\ell}$ for all
  $\ell \ge c$ \cite[Cor.\ 1.2.5]{HS06}.
  We therefore have  $\overline{\fa^m} \subseteq
  \fa^{m-c}$ for all $m \ge c$, concluding \cref{step:dfkl31mpst}.
  \medskip
  \par In the rest of the proof, we consider another Stein factorization
  \cite[Cor.\ III.11.5]{Har77}
  \begin{equation}\label{eq:step2factor}
    \begin{tikzcd}
      V' \rar{\mu}\arrow{dr}[swap]{\psi} & \widetilde{V}\dar{\nu}\\
      & V
    \end{tikzcd}
  \end{equation}
  this time for the morphism $\psi = \eta\circ\varphi$, in which case
  $\widetilde{V}$ is a normal projective variety.
  Let $\widetilde{Z} \coloneqq \nu^{-1}(Z)$ be the scheme-theoretic inverse
  image of $Z$ under $\nu$, and write
  \[
    \widetilde{Z} = \bigcup_\alpha \widetilde{Z}_\alpha
  \]
  where $\widetilde{Z}_\alpha$ are the irreducible components of
  $\widetilde{Z}$.
  Since $\nu$ is finite, every $\widetilde{Z}_\alpha$ is one-dimensional and
  dominates $Z$, hence the projection formula and
  \cref{eq:nonredintersection} imply $\nu^*D \cdot \widetilde{Z}_\alpha < 0$.
  We also note that $E' = \mu^{-1}(\widetilde{Z})$ is a Cartier divisor with
  simple normal crossing support by the factorization \cref{eq:step2factor}.
  \par We also fix the following notation.
  Fix a very ample Cartier divisor $H$ on $V'$, and set $A =
  K_{V'} + (\dim V' + 1)H$.
  For every subvariety $W \subseteq V'$, a \textsl{complete intersection
  curve} is a curve formed by taking the intersection of $\dim W - 1$
  hyperplane sections in $\bigl\lvert H\rvert_W \bigr\rvert$, and a
  \textsl{general complete intersection curve} is one formed by taking these
  hyperplane sections to be general in $\bigl\lvert H\rvert_W \bigr\rvert$.
  For each positive integer $q$, we will consider the asymptotic test ideal
  \[
    \tau\bigl(V',\lVert qD' \rVert\bigr) = \tau\bigl(\lVert qD'
    \rVert\bigr) \subseteq \cO_{V'}.
  \]
  By uniform global generation for test ideals (\cref{thm:testidealuniformgg}),
  the sheaf
  \begin{equation}\label{eq:prop31gg}
    \tau\bigl(\lVert qD' \rVert\bigr) \otimes \cO_{V'}(qD'+A)
  \end{equation}
  is globally generated for every integer $q > 0$.
  \begin{step}\label{step:dfkl31step2}
    There exists an integer $\ell_0 > 0$ such that for every integer $\ell >
    \ell_0$ and for every irreducible component $F$ of $E'_\red$ that dominates
    $(Z_\alpha)_\red$ for some $\alpha$, we have
    \[
      \tau\bigl(\lVert \ell D' \rVert\bigr) \subseteq \cO_{V'}(-F).
    \]
  \end{step}
  Let $C \subseteq F$ be a general complete intersection curve; note that
  $C$ is integral by Bertini's theorem \cite[Thm.\ 3.4.10 and Cor.\
  3.4.14]{FOV99} and dominates $(Z_\alpha)_\red$ for some $\alpha$, hence $(D'
  \cdot C) < 0$ by the projection formula and \eqref{eq:nonredintersection}.
  If for some integer $q > 0$, the curve $C$ is not contained in the zero locus
  of $\tau(\lVert qD' \rVert)$, then the fact that the sheaf \eqref{eq:prop31gg}
  is globally generated implies
  \[
    \bigl( (qD' + A) \cdot C \bigr) \ge 0.
  \]
  Letting $\ell_{0F} = -(A \cdot C)/(D' \cdot C)$, we see that the ideal
  $\tau(\lVert \ell D' \rVert)$ vanishes everywhere along $C$ for every integer
  $\ell > \ell_{0F}$.
  By varying $C$, the ideal $\tau(\lVert \ell D' \rVert)$ must vanish everywhere
  along $F$ for every integer $\ell > \ell_{0F}$, hence we can set $\ell_0 =
  \max_F\{\ell_{0F}\}$.
  \begin{step}\label{step:dfkl31step3}
    Let $E'_i$ be an irreducible component of $E'_\red$ not dominating
    $Z_\alpha$ for every $\alpha$.
    Suppose $E'_j$ is another irreducible component of $E'_\red$ such that
    $E'_i \cap E'_j \ne \emptyset$ and for which there exists an integer $\ell_j$
    such that for every integer $\ell > \ell_j$, we have
    \begin{align*}
      \tau\bigl(\lVert \ell D' \rVert\bigr) &\subseteq \cO_{V'}(-E'_j).
      \intertext{Then, there is an integer $\ell_i \ge \ell_j$ such that for
      every integer $\ell > \ell_i$, we have}
      \tau\bigl(\lVert \ell D' \rVert\bigr) &\subseteq \cO_{V'}(-E'_i).
    \end{align*}
  \end{step}
  Let $C \subseteq E'_i$ be a complete intersection curve.
  By the assumption that $E'$ has simple normal crossing support, there exists
  at least one closed point $P \in C \cap E_j'$.
  For every $\ell > \ell_j$ and every $m > 0$, we have the sequence of
  inclusions
  \begin{equation}\label{eq:dfkl31incl}
    \begin{aligned}
      \MoveEqLeft[4]\Bigl( \tau\bigl(\lVert m\ell D'\rVert\bigr) \otimes
      \cO_{V'}(m\ell D'+A) \Bigr) \cdot \cO_C
      \subseteq \Bigl( \tau\bigl(\lVert \ell D'\rVert\bigr)^m \otimes
      \cO_{V'}(m\ell D'+A) \Bigr) \cdot \cO_C\\
      &\subseteq \Bigl( \cO_{V'}(-mE_j') \otimes
      \cO_{V'}(m\ell D'+A) \Bigr) \cdot \cO_C
      \subseteq \cO_C(A\rvert_C - mP)
    \end{aligned}
  \end{equation}
  where the first two inclusions follow from subadditivity
  (\cref{thm:testidealsubadditivity}) and by assumption, respectively.
  The last inclusion holds since $C$ maps to a closed point in $V$, hence
  $\cO_C(D') = \cO_C$.
  By the global generation of the sheaf in \eqref{eq:prop31gg} for $q = m\ell$,
  the inclusion \eqref{eq:dfkl31incl} implies that for every integer $\ell >
  \ell_j$, if $\tau(\lVert m\ell D' \rVert)$ does not vanish everywhere along
  $C$, then $(A \cdot C) \ge m$.
  Choosing $\ell_i = \ell_j \cdot ((A \cdot C) + 1 )$, we see that
  $\tau(\lVert \ell D' \rVert)$ vanishes everywhere along $C$ for every integer
  $\ell > \ell_i$.
  By varying $C$, we have $\tau(\lVert \ell D' \rVert) \subseteq
  \cO_{V'}(-E_i')$ for every integer $\ell > \ell_i$.
  \begin{step}\label{lem:dfkl32}
    There exists a positive integer $a$ such that
    for every integer $m > 0$, we have
    $\fb(\lvert maD' \rvert)
    \subseteq \cO_{V'}(-mE'_\red)$.
  \end{step}
  Write
  \[
    E'_\red = \bigcup_j \bigcup_{i \in I_j} E_{ij}',
  \]
  where the $E_{ij}'$ are the irreducible components of $E'_\red$, and the
  $\bigcup_{i \in I_j} E_{ij}'$ are the connected components of $E'_\red$.
  Since $V$ is normal, each preimage $\mu^{-1}(Z_\alpha)$ is connected by
  Zariski's main theorem \cite[Cor.\ III.11.4]{Har77}, hence each connected
  component $\bigcup_{i \in I_j} E_{ij}'$ of $E'_\red$ contains an irreducible
  component $E_{i_0j}'$ that dominates $(Z_\alpha)_\red$ for some $\alpha$.
  By \cref{step:dfkl31step2}, there exists an integer $\ell_0$
  such that for every $j$, we have $\tau(\lVert \ell D' \rVert) \subseteq
  \cO_{V'}(-E'_{i_0j})$ for every integer $\ell > \ell_0$.
  For each $j$, by applying \cref{step:dfkl31step3} $(\lvert I_j
  \rvert-1)$ times to the $j$th connected component $\bigcup_{i \in I_j}
  E_{ij}'$ of $E'$, we can find $\ell_j$ such that
  $\tau(\lVert \ell D' \rVert) \subseteq \cO_{V'}(-E_{ij}')$
  for every $i \in I_j$ and for every integer $\ell > \ell_j$.
  Setting $a = \max_j\{\ell_j\}+1$, we have $\tau(\lVert aD'
  \rVert) \subseteq \cO_{V'}(-E'_\red)$.
  Thus, for every integer $m > 0$, we have
  \[
    \fb\bigl(\lvert maD' \rvert\bigr) \subseteq \tau\bigl(\lvert
    maD'\rvert\bigr) \subseteq \tau\bigl(\lVert maD'\rVert\bigr) \subseteq
    \tau\bigl(\lVert aD'\rVert\bigr)^m \subseteq \cO_{V'}(-mE'_\red),
  \]
  where the first inclusion follows by the fact that $V'$ is regular hence
  strongly $F$-regular (Propositions
  \ref{thm:testidealprops}\cref{thm:testideallarger} and
  \ref{thm:testidealprops}\cref{thm:testidealstrongfreg}), the second
  inclusion is by definition of the asymptotic test ideal, and the third
  inclusion is by subadditivity (\cref{thm:testidealsubadditivity}).
  This concludes the proof of \cref{lem:dfkl32}, hence also of
  \cref{prop:dfkl31} by \cref{step:dfkl31mpst}.
\end{proof}
\begin{remark}\label{rmk:dfkl31char0}
  When $\Char k = 0$, it suffices to replace asymptotic test ideals in the proof
  above with asymptotic multipliers ideals $\cJ(\lVert D \rVert)$ as defined in
  \cref{def:asymptoticmultideal} by replacing \cref{thm:testidealprops},
\cref{thm:testidealsubadditivity}, and \cref{thm:testidealuniformgg}
  with \cite[Prop.\ 2.3]{dFM09}, \cite[Thm.\ A.2]{JM12}, and
  \cref{thm:multidealuniformgg}, respectively.
\end{remark}
\section{Proof of Theorem \ref{thm:dfkl41}}\label{sect:dfkl41proof}
We now prove Theorem \ref{thm:dfkl41}.
We first note that the direction $\Rightarrow$ in \cref{thm:dfkl41}
follows from existing results.
\begin{proof}[Proof of $\Rightarrow$ in \cref{thm:dfkl41}]
  Let $A$ be a very ample Cartier divisor.
  Then, for all $t$ such that $L-tA$ is ample, we have $\widehat{h}^i(X,L-tA) =
  0$ by Serre vanishing and by homogeneity and
  continuity (see \cref{rem:kur291reductions}).
\end{proof}
For the direction $\Leftarrow$, it suffices to show \cref{thm:dfkl41} for
Cartier divisors $L$ by continuity and homogeneity (see \cref{rem:kur291reductions}).
We also make the following two reductions.
Recall that an $\RR$-Cartier divisor $L$ on $X$ \textsl{satisfies
\cref{thm:dfkl41cond}} for a pair $(A,\varepsilon)$ consisting of a very
ample Cartier divisor $A$ on $X$ and a real number $\varepsilon > 0$ if
$\widehat{h}^i(X,L-tA) = 0$ for all $i > 0$ and all $t \in [0,\varepsilon)$.
\begin{lemma}\label{lem:dfklfieldred}
  To prove the direction $\Leftarrow$ in \cref{thm:dfkl41}, we may assume
  that the ground field $k$ is uncountable.
\end{lemma}
\begin{proof}
  Consider the purely transcendental extension
  \[
    k' \coloneqq k(x_\alpha)_{\alpha \in A}
  \]
  where $\{x_\alpha\}_{\alpha \in A}$ is an uncountable set of indeterminates;
  note that $k'$ is uncountable by construction.
  We claim that $X \times_k k'$ is integral.
  Let $\bigcup_j U_j$ be an affine
  open covering of $X$.
  Then, $X \times_k k'$ is covered by affine open subsets that are localizations
  of the integral varieties $U_j \times_k \Spec k[x_\alpha]_{\alpha \in A}$,
  which pairwise intersect, hence $X \times_k k'$ is integral.
  \par Now suppose $X$ is a projective variety over $k$, and let $L$ be an
  Cartier divisor satisfying \cref{thm:dfkl41cond} for some pair
  $(A,\varepsilon)$.
  Let
  \[
    \pi\colon X \times_k k' \longrightarrow X
  \]
  be the first projection map, which we note is faithfully flat by base change.
  Then, the pullback $\pi^*A$ of $A$ is very ample,
  and to show that $L$ is ample, it suffices to show that $\pi^*L$ is ample by
  flat base change and Serre's criterion for ampleness.
  By the special case of \cref{thm:dfkl41} over the ground field $k'$, it
  therefore suffices to show that $\pi^*L$ satisfies \cref{thm:dfkl41cond}
  for the pair $(\pi^*A,\varepsilon)$.
  \par We want to show that for every $i > 0$ and for all
  $t \in [0,\varepsilon)$, we have
  \begin{equation}\label{eq:hhatinvfieldext}
    \widehat{h}^i(X,L-tA) = \widehat{h}^i\bigl(X \times_k K,\pi^*(L-tA)\bigr) =
    0.
  \end{equation}
  For every $D \in \Div(X)$ and every $i \ge 0$, the number $h^i(X,\cO_X(D))$ is
  invariant under ground field extensions by flat base change,
  hence $\widehat{h}^i(X,D)$ is also.
  By homogeneity and continuity (see \cref{rem:kur291reductions}),
  the number $\widehat{h}^i(X,D)$ is also invariant under
  ground field extensions for $D \in \Div_\RR(X)$, hence
  \cref{eq:hhatinvfieldext} holds.
\end{proof}
\begin{lemma}\label{lem:nonnefcounterex}
  To prove the direction $\Leftarrow$ in \cref{thm:dfkl41}, it suffices
  to show that every Cartier divisor satisfying \cref{thm:dfkl41cond} is nef.
\end{lemma}
\begin{proof}
  Suppose $L$ is a Cartier divisor satisfying \cref{thm:dfkl41cond} for a
  pair $(A,\varepsilon)$.
  Choose $\delta \in (0,\varepsilon) \cap \QQ$ and let $m$ be a positive integer
  such that $m\delta$ is an integer.
  Then, the Cartier divisor $m(L-\delta A)$ is nef since
  \begin{align*}
    \widehat{h}^i\bigl(X,m(L-\delta A) - tA\bigr)
    &= \widehat{h}^i\bigl(X,mL - (t+m\delta)A\bigr)\\
    &= m \cdot \widehat{h}^i\Bigl(X,L - \Bigl(\frac{t}{m}+\delta\Bigr)A\Bigr)
    = 0
  \end{align*}
  for all $t \in [0,m\varepsilon-\delta)$ by homogeneity
  (\cref{prop:nahomogcont}).
  Thus, the Cartier divisor $L = (L-\delta A) + \delta A$ is ample by
  \cite[Cor.\ 1.4.10]{Laz04a}.
\end{proof}
We will also need the following result to allow for an inductive proof.
Note that the proof in \cite{dFKL07} works in our setting.
\begin{citedlem}[{\cite[Lem.\ 4.3]{dFKL07}}]\label{lem:dfkl43}
  Let $X$ be a projective variety of dimension $n > 0$ over an uncountable
  field, and let $L$ be a Cartier divisor on $X$.
  Suppose $L$ satisfies \cref{thm:dfkl41cond} for a pair $(A,\varepsilon)$,
  and let $E \in \lvert A \rvert$ be a very general divisor.
  Then, the restriction $L\rvert_E$ satisfies \cref{thm:dfkl41cond} for the
  pair $(A\rvert_E,\varepsilon)$.
\end{citedlem}
We can now show the direction $\Leftarrow$ in \cref{thm:dfkl41}; by
\cref{lem:nonnefcounterex}, we need to show that every Cartier divisor
satisfying \cref{thm:dfkl41cond} is nef.
Recall that by \cref{lem:dfklfieldred}, we may assume that the ground field
$k$ is uncountable.
Our proof follows that in \cite[pp.\ 450--454]{dFKL07} after reducing to a
setting where \cref{prop:dfkl31} applies, although we have to
be more careful in positive characteristic.
\begin{proof}[Proof of $\Leftarrow$ in \cref{thm:dfkl41}]
  We proceed by induction on $\dim X$.
  Suppose $\dim X = 1$; we will show the contrapositive.
  If $L$ is not nef, then $\deg L < 0$ and $-L$ is ample.
  Thus, by asymptotic Serre duality (\cref{prop:asympoticserre}), we
  have $\widehat{h}^1(X,L) = \widehat{h}^0(X,-L) \ne 0$,
  hence \cref{thm:dfkl41cond} does not hold for every choice of
  $(A,\varepsilon)$.
  \par We now assume $\dim X \ge 2$.
  Suppose by way of contradiction that there is a non-nef Cartier
  divisor $L$ satisfying \cref{thm:dfkl41cond}, and let $Z \subseteq X$ be an
  integral curve such that $(L \cdot Z) < 0$.
  Our goal is to show that
  \begin{equation}\label{eq:dfklgoal}
    \widehat{h}^1(X,L-\delta A) \ne 0
  \end{equation}
  for $0 < \delta \ll 1$, contradicting \cref{thm:dfkl41cond}.
  Let $F \in \lvert A \rvert$ be a very general divisor.
  By Bertini's theorem \cite[Thm.\ 3.4.10 and Cor.\ 3.4.14]{FOV99}, we may
  assume that $F$ is a subvariety of $X$, in which case by
  inductive hypothesis and \cref{lem:dfkl43}, we have that $L\rvert_F$
  is ample.
  Since ampleness is an open condition in families \cite[Cor.\ 9.6.4]{EGAIV3},
  there exists an integer $b > 0$ such that $bL$ is very ample along the
  generic divisor $F_\eta \in \lvert A \rvert$.
  By possibly replacing $b$ with a multiple, we may also assume that
  $mbL\rvert_{F_\eta}$ has vanishing higher cohomology for every integer
  $m > 0$.
  Since the ground field $k$ is uncountable, we can then choose a sequence of
  very general Cartier divisors $\{E_\beta\}_{\beta=1}^\infty \subseteq \lvert
  A \rvert$ such that the following properties hold:
  \begin{enumerate}[label=$(\alph*)$]
    \item $E_\beta$ is a subvariety of $X$ for all $\beta$
      (by Bertini's theorem \cite[Thm.\ 3.4.10 and Cor.\ 3.4.14]{FOV99});
    \item For all $\beta$, $bL\rvert_{E_\beta}$ is very ample
      and $mbL\rvert_{E_\beta}$ has vanishing higher cohomology for every
      integer $m > 0$ (by the constructibility of very ampleness in families
      \cite[Prop.\ 9.6.3]{EGAIV3} and by semicontinuity); and
    \item For every positive integer $r$, the $k$-dimension of cohomology groups
      of the form
      \begin{equation}\label{eq:indepcohgp}
        H^j\bigl(E_{\beta_1} \cap E_{\beta_2} \cap \cdots \cap
          E_{\beta_r},\cO_{E_{\beta_1} \cap E_{\beta_2} \cap \cdots \cap
        E_{\beta_r}}(mL)\bigr)
      \end{equation}
      for non-negative integers $j$ and $m$ is independent of the $r$-tuple
      $(\beta_1,\beta_2,\ldots,\beta_r)$ (by semicontinuity; see \cite[Prop.\
      5.5]{Kur06}).
  \end{enumerate}
  We will denote by $h^j(\cO_{E_1 \cap E_2 \cap \cdots \cap E_r}(mL))$ the
  dimensions of the cohomology groups \cref{eq:indepcohgp}.
  By homogeneity (\cref{prop:nahomogcont}), we can
  replace $L$ by $bL$ so that $L\rvert_{E_\beta}$ is very ample
  with vanishing higher cohomology for all $\beta$.
  \par To show \cref{eq:dfklgoal}, we now follow the proof in
  \cite[pp.\ 453--454]{dFKL07} with appropriate modifications.
  Given positive integers $m$ and $r$, consider the complex
  \begin{align*}
    K^\bullet_{m,r} \coloneqq{}& \biggl(
    \bigotimes_{\beta=1}^r\bigl(\cO_{X} \longrightarrow
    \cO_{E_\beta}\bigr) \biggr) \otimes \cO_{X}(mL)\\
    ={}& \biggl\{ \cO_{X}(mL) \longrightarrow
    \bigoplus_{\beta=1}^r \cO_{E_\beta}(mL) \longrightarrow \bigoplus_{1
    \le \beta_1 < \beta_2 \le r} \cO_{E_{\beta_1} \cap E_{\beta_2}}(mL)
    \longrightarrow \cdots \biggr\}.
  \end{align*}
  By \cite[Cor.\ 4.2]{Kur06}, this complex is acyclic away from
  $\cO_{X}(mL)$, hence is a resolution for
  $\cO_{X}(mL - rA)$.
  In particular, we have
  \[
    H^j\bigl(X,\cO_{X}(mL - rA)\bigr) =
    \HH^j(X,K^\bullet_{m,r}).
  \]
  The right-hand side is computed by an $E_1$-spectral sequence whose first page
  is shown in \cref{fig:hypercohspectralsequence}.
  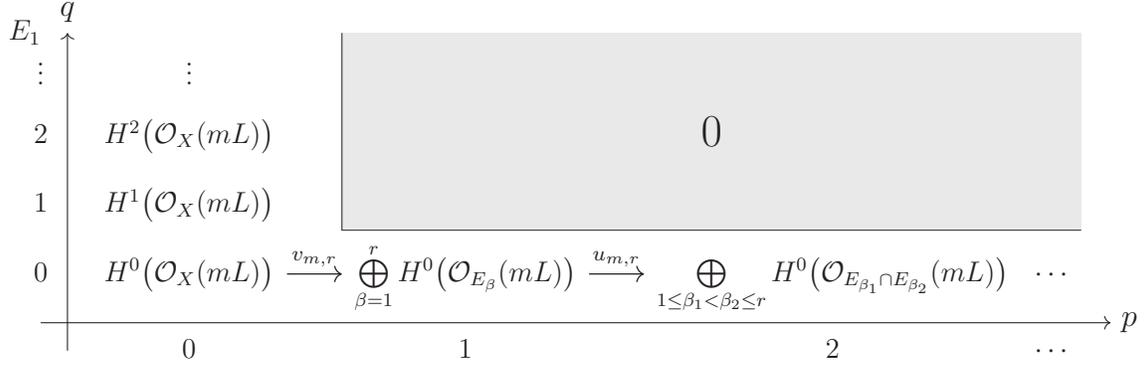
\begin{figure}[t]
    \centering
    \tikzexternalenable
    \begin{tikzpicture}
      \matrix (m) [matrix of math nodes,commutative diagrams/every cell,column
      sep=1.8em,transform shape,nodes={scale=0.9}]{
        \strut\makebox[1em][c]{$\mathllap{E_1}$} &[-1em] &
        \hphantom{\bigoplus\limits_{\beta=1}^r
        H^0\bigl(\cO_{E_\beta}(mL)\bigr)}
        & &[-1.75em]
        &[-1.75em] {}\\[-0.75em]
        \vdots & \vdots & & & \\
        2 & \vphantom{\bigoplus\limits^r}
        H^2\bigl(\cO_{X}(mL)\bigr)\\
        1 & \vphantom{\bigoplus\limits^r}
        H^1\bigl(\cO_{X}(mL)\bigr) &
        & & & \vphantom{\bigoplus\limits^r}\hphantom{\cdots}\\
        0 & H^0\bigl(\cO_{X}(mL)\bigr)
        \vphantom{\bigoplus\limits_{\beta=1}^r} & \bigoplus\limits_{\beta=1}^r
        H^0\bigl(\cO_{E_\beta}(mL)\bigr) & \bigoplus\limits_{1
        \le \beta_1 < \beta_2 \le r} H^0\bigl(\cO_{E_{\beta_1} \cap
        E_{\beta_2}}(mL)\bigr) & \cdots
        &\vphantom{\bigoplus\limits_{\beta=1}^r}\\
        \quad\strut & 0 & 1 & 2 & \cdots&\strut \\};
      \draw[<-] (m-1-1.east) node[anchor=south]{$q$} -- (m-6-1.east);
      \draw[->] (m-6-1.north) -- (m-6-6.north) node[anchor=west]{$p$};
      \draw (m-1-3.west) |- (m-4-6.south west);
      \draw[draw=none,fill=gray,fill opacity=0.2] (m-1-3.west) rectangle
        node[midway,opacity=1] {\Large $0$} (m-4-6.south west);
      \path[commutative diagrams/.cd, every arrow, every label]
        (m-5-2) edge node {$v_{m,r}$} (m-5-3)
        (m-5-3) edge node {$u_{m,r}$} (m-5-4);
    \end{tikzpicture}
    \tikzexternaldisable
    \caption{Hypercohomology spectral sequence computing $H^j(X,\cO_X(mL-rA))$}
    \label{fig:hypercohspectralsequence}
  \end{figure}
  This spectral sequence yields a natural inclusion
  \begin{equation}\label{eq:dfkl41incl}
    \frac{\ker(u_{m,r})}{\im(v_{m,r})} \subseteq
    H^1\bigl(X,\cO_{X}(mL - rA)\bigr).
  \end{equation}
  \par We want to bound the left-hand side of \cref{eq:dfkl41incl} from
  below.
  First, there exists a constant $C_1 > 0$ such that $h^0(\cO_{E_1\cap
  E_2}(mL)) \le C_1 \cdot m^{n-2}$ for all $m \gg 0$ \cite[Ex.\
  1.2.20]{Laz04a}.
  Thus, we have
  \[
    \codim\Bigl(\ker(u_{m,r}) \subseteq \bigoplus\limits_{\beta=1}^r
    H^0\bigl(E_\beta,\cO_{E_\beta}(mL)\bigr) \Bigr)
    \le C_2 \cdot r^2m^{n-2}
  \]
  for some $C_2$ and for all $m \gg 0$.
  Now by \cref{prop:dfkl31}, there are positive integers $q$ and
  $c$ such that $\fb(\lvert mqL \rvert) \subseteq \fa^{m-c}$
  for all $m > c$, where $\fa$ is the ideal sheaf of $Z$.
  By replacing $L$ by $qL$, we can assume that this
  inclusion holds for $q = 1$.
  The morphism $v_{m,r}$ therefore fits into the following commutative diagram:
  \[
    \begin{tikzcd}
      H^0\bigl(X,\cO_{X}(mL) \otimes
      \fa^{m-c}\bigr) \rar{v'_{m,r}} \dar[equal]
      & \operatorname{\smash{\displaystyle\bigoplus\limits_{\beta=1}^r}}
      H^0\bigl(E_\beta,\cO_{E_\beta}(m\otimes
      \fa^{m-c}\bigr) \dar[hook]\\
      H^0\bigl(X,\cO_{X}(mL)\bigr)
      \rar{v_{m,r}} &
      \operatorname{\smash{\displaystyle\bigoplus\limits_{\beta=1}^r}}
      H^0\bigl(E_\beta,\cO_{E_\beta}(mL)\bigr)\\[-1em]
    \end{tikzcd}
  \]
  \par We claim that there exists a constant $C_3 > 0$ such that for all $m \gg
  0$,
  \begin{equation}\label{eq:dfkl15}
    \codim\Bigl(H^0\bigl(E_\beta,\cO_{E_\beta}(mL) \otimes
      \fa^{m-c}\bigr) \subseteq H^0\bigl(E_\beta,\cO_{E_\beta}(mL)\bigr)
    \Bigr) \ge C_3 \cdot m^{n-1}.
  \end{equation}
  Granted this, we have
  \[
    \dim\biggl( \frac{\ker(u_{m,r})}{\im(v_{m,r})} \biggr) \ge C_4 \cdot \bigl(
    rm^{n-1} - r^2m^{n-2}\bigr)
  \]
  for some constant $C_4 > 0$ and for all $m \gg 0$.
  Fixing a rational number $0 < \delta \ll 1$ and setting $r = m\delta$ for an
  integer $m > 0$ such that $m\delta$ is an integer, we then see that
  there exists a constant $C_5 > 0$ such that
  \[
    h^1\bigl(X,\cO_{X}\bigl(m(L - \delta
    A) \bigr)\bigr) \ge C_5 \cdot \delta m^n
  \]
  for all $m \gg 0$, contradicting \cref{thm:dfkl41cond}.
  \par It remains to show \cref{eq:dfkl15}.
  Since the vanishing locus of $\fa$ may have no $k$-rational points, we will
  pass to the algebraic closure of $k$ to bound the codimension on
  the left-hand side of \eqref{eq:dfkl15} from below.
  Let $\overline{E}_\beta \coloneqq E_\beta \times_k \overline{k}$, and denote
  by $\pi\colon \overline{E}_\beta \to E_\beta$ the projection
  morphism.
  Note that
  \begin{align*}
    \MoveEqLeft[3] \codim\Bigl(H^0\bigl(E_\beta,\cO_{E_\beta}(mL)
    \otimes \fa^{m-c}\bigr) \subseteq
  H^0\bigl(E_\beta,\cO_{E_\beta}(mL)\bigr) \Bigr)\\
    &= \codim\Bigl(H^0\bigl(\overline{E}_\beta,\cO_{\overline{E}_\beta}(m\,
      \pi^*L) \otimes \pi^{-1}\fa^{m-c} \cdot \cO_{\overline{E}_\beta}
      \bigr) \subseteq H^0\bigl(\overline{E}_\beta,\cO_{\overline{E}_\beta}(m\,
      \pi^*L)\bigr)\Bigr)
  \intertext{by the flatness of $k \subseteq \overline{k}$.
  Since $\cO_{\overline{E}_\beta}(\pi^*L)$ is very ample by base change, we
  can choose a closed point $x \in
  Z(\pi^{-1}\fa \cdot \cO_{\overline{E}_\beta}) \cap \overline{E}_\beta$, in
  which case $\cO_{\overline{E}_\beta}(m\,\pi^*L)$ separates $(m-c)$-jets at $x$
  by \cite[Proof of Lem.\ 3.7]{Ito13} (see also \cref{lem:ito37}).
  Finally, the dimension of the space of $(m-c)$-jets at $x$ is at least that
  for a regular point of a variety of dimension $n$, hence}
    \MoveEqLeft[3] \codim\Bigl(H^0\bigl(\overline{E}_\beta,
    \cO_{\overline{E}_\beta}(m\,\pi^*L) \otimes \pi^{-1}\fa^{m-c} \cdot
    \cO_{\overline{E}_\beta} \bigr) \subseteq
    H^0\bigl(\overline{E}_\beta,\cO_{\overline{E}_\beta}(m\,\pi^*L)\bigr)
    \Bigr)\\
    &\ge \codim\Bigl(H^0\bigl(\overline{E}_\beta,\cO_{\overline{E}_\beta}(m\,
      \pi^*L) \otimes \fm_x^{m-c} \cdot \cO_{\overline{E}_\beta} \bigr)
      \subseteq H^0\bigl(\overline{E}_\beta,\cO_{\overline{E}_\beta}(m\,
      \pi^*L)\bigr)\Bigr)\\
      &\ge \binom{m-c+n}{n-1} \ge C_3 \cdot m^{n-1}
  \end{align*}
  for some constant $C_3 > 0$ and all $m \gg 0$, as required.
\end{proof}

\chapter{Moving Seshadri constants}\label{chapter:movingseshadri}
Moving Seshadri constants were defined by
Nakamaye\index{Nakamaye, Michael} \cite{Nak03} as a generalization of the
Seshadri constant introduced in
\cref{sect:seshadriconstants} to arbitrary $\RR$-Cartier divisors.
In this chapter, we extend basic results on moving Seshadri constants from
\cites{Nak03}[\S6]{ELMNP09} to the setting of possibly singular
varieties over arbitrary fields.
These results are new even for complex projective varieties that are not smooth.
Some of this material will appear in joint work with Mihai
Fulger\index{Fulger, Mihai} \cite{FM2}.
\section{Definition and basic properties}
\label{sect:movingseshdef}
Following \cite{ELMNP09}, we define the moving Seshadri constant as follows:
\begin{definition}[cf.\ {\cite[Def.\ 6.1]{ELMNP09}}]\label{def:movingsesh}
  Let $X$ be a normal projective variety over a field $k$ and let $D$ be an
  $\RR$-Cartier divisor on $X$.
  Consider a $k$-rational point $x \in X$.
  If $x \notin \Bplus(D)$, then the \textsl{moving Seshadri constant}
  \index{Seshadri constant, moving, $\varepsilon(\lVert D \rVert;x)$|textbf} of
  $D$ at $x$ is
  \begin{equation}\label{eq:movingseshdef}
    \varepsilon\bigl(\lVert D \rVert;x\bigr) \coloneqq \sup_{f^*D \equiv_\RR A +
    E} \varepsilon(A;x)\glsadd{movingseshadri}
  \end{equation}
  where the supremum runs over all birational morphisms $f\colon X' \to X$ from
  normal projective varieties $X'$ that are isomorphisms over a neighborhood of
  $x$, and $\RR$-numerical equivalences $f^*D \equiv_\RR A + E$ where $A$ is an
  ample $\RR$-Cartier divisor and $E$ is an effective $\RR$-Cartier divisor such
  that $x \notin f(\Supp(E))$.
  If $x \in \Bplus(D)$, then we set $\varepsilon(\lVert D \rVert;x) = 0$.
  By definition, $x \in \Bplus(D)$ if and only if $\varepsilon(\lVert D
  \rVert;x) = 0$.\index{Seshadri constant, moving, $\varepsilon(\lVert D \rVert;x)$!detects $\Bplus(D)$}
\end{definition}
Note that $\RR$-numerical equivalences of the form in \cref{eq:movingseshdef}
exist since $f^{-1}(x) \notin \Bplus(f^*D)$ by \cref{prop:bbp23}.
\medskip
\par We collect some elementary properties of moving Seshadri constants.
Recall that if $x \in X$ is a closed point, then \gls*{bigconex}
denotes the open convex subcone of the big
cone consisting of big $\RR$-Cartier divisor classes $D \in N^1_\RR(X)$ such
that $x \notin \Bplus(D)$ \cite[Def.\ 5.1]{ELMNP09}.
Here $N^1_\RR(X)$ is the \textsl{N\'eron--Severi space} associated to $X$
defined in \cref{eq:neronseveri}.
\begin{proposition}[cf.\ {\cite[Prop.\ 6.3 and Rem.\ 6.5]{ELMNP09}}]
  \label{prop:movingseshelem}
  Let $X$ be a normal projective variety over a field $k$ and let $x \in X$ be a
  $k$-rational point.
  Then, the function
  \[
    \begin{tikzcd}[row sep=0,column sep=1.6em]
      \Bigcone^{\{x\}}_\RR(X) \rar & \RR_{>0}\\
      D \rar[mapsto] & \varepsilon\bigl(\lVert D \rVert;x\bigr)
    \end{tikzcd}
  \]
  is continuous.\index{Seshadri constant, moving, $\varepsilon(\lVert D \rVert;x)$!is continuous}
  Moreover, if $D$ is an $\RR$-Cartier divisor, then we have the following:
  \begin{enumerate}[label=$(\roman*)$,ref=\roman*]
    \item\label{prop:movingseshvol}
      \index{Seshadri constant, moving, $\varepsilon(\lVert D \rVert;x)$!relationship to volume and Hilbert--Samuel multiplicity}
      $\varepsilon(\lVert D \rVert;x) \le (\vol_X(D)/e(\cO_{X,x}))^{1/\dim X}$.
    \item\label{prop:movingseshnumequiv}
      \index{Seshadri constant, moving, $\varepsilon(\lVert D \rVert;x)$!is a numerical invariant}
      If $D$ and $E$ are numerically equivalent $\RR$-Cartier divisors, then
      $\varepsilon(\lVert D \rVert;x) = \varepsilon(\lVert E \rVert;x)$.
    \item\label{prop:movingseshhomog}
      \index{Seshadri constant, moving, $\varepsilon(\lVert D \rVert;x)$!is homogeneous}
      $\varepsilon(\lVert \lambda D \rVert;x) = \lambda \cdot
      \varepsilon(\lVert D \rVert;x)$ for every positive real number $\lambda$.
    \item\label{prop:movingseshadd}
      \index{Seshadri constant, moving, $\varepsilon(\lVert D \rVert;x)$!is concave}
      If $D'$ is another $\RR$-Cartier divisor such that $x \notin \Bplus(D)
      \cup \Bplus(D')$, then
      \[
        \varepsilon\bigl(\lVert D + D' \rVert;x) \ge \varepsilon\bigl(\lVert D
        \rVert;x\bigr) + \varepsilon\bigl(\lVert D' \rVert;x\bigr).
      \]
    \item\label{prop:movingseshbignef}
      \index{Seshadri constant, moving, $\varepsilon(\lVert D \rVert;x)$!is e(D;x) when D is nef@is $\varepsilon(D;x)$ when $D$ is nef}
      \index{Seshadri constant, $\varepsilon(D;x)$!is e(AbsD;x) when D is nef@is $\varepsilon(\lVert D \rVert;x)$ when $D$ is nef}
      If $D$ is a nef $\RR$-Cartier divisor, then
      $\varepsilon(\lVert D \rVert;x) = \varepsilon(D;x)$.
  \end{enumerate}
\end{proposition}
\begin{proof}
  \cref{prop:movingseshvol}--\cref{prop:movingseshadd} follow by
  definition and from the
  analogous properties for usual Seshadri constants;
  for
  \cref{prop:movingseshvol}, the analogous property is \cite[Prop.\
  5.1.9]{Laz04a}.
  The continuity property follows from \cref{prop:movingseshhomog} and
  \cref{prop:movingseshadd} by \cite[Rem.\ 5.4]{ELMNP09}.
  \par We now prove \cref{prop:movingseshbignef}, following \cite[Rem.\
  6.5]{ELMNP09}.
  If $x \in \Bplus(D)$, then $\varepsilon(\lVert D \rVert;x) = 0$ by definition,
  while $\varepsilon(D;x) = 0$ by combining \cref{thm:nakamaye} and \cite[Prop.\
  5.1.9]{Laz04a}.
  It therefore suffices to consider the case when $x \notin \Bplus(D)$.
  As in \cref{def:movingsesh}, choose a birational morphism
  $f\colon X' \to X$ with a decomposition $f^*D \equiv_\RR A +
  E$.
  We have
  \[
    \varepsilon(D;x) = \varepsilon\bigl(f^*D;f^{-1}(x)\bigr) \ge
    \varepsilon\bigl(A;f^{-1}(x)\bigr)
  \]
  where the first equality holds since $D$ is nef and $f$ is an isomorphism at
  $x$.
  The second inequality holds by combining \cref{eq:seshinfcurves} and the fact
  that $x \notin f(\Supp(E))$, hence $E \cdot C \ge 0$ for every integral curve
  $C \subseteq X'$ passing through $f^{-1}(x)$.
  Taking the supremum over all $f$ as in \cref{def:movingsesh}, we have the
  inequality $\varepsilon(D;x) \ge \varepsilon(\lVert D \rVert;x)$.
  For the opposite inequality, write $D \equiv_\RR A +
  E$ with $A$ an ample $\RR$-Cartier divisor and $E$ an
  effective $\RR$-Cartier divisor not containing $x$ in its support.
  For every integer $n \ge 1$, we can write
  \[
    D \equiv_\RR \frac{1}{n}A + \frac{n-1}{n}D + \frac{1}{n}E,
  \]
  hence setting $A_n \coloneqq \frac{1}{n}A + \frac{n-1}{n}D$, we have $D
  \equiv_\RR A_n + \frac{1}{n}E$ for an ample $\RR$-Cartier divisor $A_n$ and a
  fixed effective $\RR$-Cartier divisor $E$.
  We therefore have
  \begin{equation}\label{eq:damineq}
    \varepsilon\bigl(\lVert D \rVert;x\bigr) \ge \varepsilon(A_n;x)
  \end{equation}
  for all $n$.
  Now we note that using the characterization in \cref{eq:seshinfcurves}, we
  have
  \begin{equation}\label{eq:decompseshapprox}
    \varepsilon(A_n;x) = 
    \inf_{C \ni x} \biggl\{ \frac{(A_n \cdot C)}{e(\cO_{C,x})} \biggr\}
    \ge \inf_{C \ni x} \biggl\{ \frac{(D \cdot C)}{e(\cO_{C,x})} \biggr\} +
    \frac{1}{n} \inf_{C \ni x} \biggl\{\frac{(-E \cdot
    C)}{e(\cO_{C,x})} \biggr\}.
  \end{equation}
  Taking the limit $n \to \infty$ in \cref{eq:damineq}, we obtain the
  inequality $\varepsilon(\lVert D \rVert;x) \ge \varepsilon(D;x)$.
\end{proof}
\begin{remark}
  We note that the continuity statement in
  \cref{prop:movingseshelem} is \emph{not} the
  analogue of \cite[Thm.\ 6.2]{ELMNP09}, which states that if $X$ is a smooth
  complex projective variety and $x \in X$ is a closed point, then the function
  $D \mapsto \varepsilon(\lVert D \rVert;x)$ is continuous on the entire
  N\'eron--Severi space $N^1_\RR(X)$.
  A proof of this statement would require extending the main results about
  restricted volume functions in \cite{ELMNP09} to our setting.
\end{remark}
\section{Alternative descriptions}
\label{sect:movingseshaltdefs}
We now give alternative characterizations of the moving Seshadri constant
defined in \cref{def:movingsesh}.
\subsection{Nakamaye's description}
\index{Seshadri constant, moving, $\varepsilon(\lVert D \rVert;x)$!Nakamaye's
description of|(}
Moving Seshadri constants were first defined by Nakamaye by decomposing the
complete linear system $\lvert D \rvert$ into its moving and fixed parts on a
birational model of $X$.
The following is a version of his definition that works over arbitrary fields.
\begin{definition}[cf.\ {\cite[Def.\ 0.4]{Nak03}}]\label{def:movingseshnakamaye}
  Let $X$ be a normal projective variety over a field $k$ and let $D$ be a
  $\QQ$-Cartier divisor on $X$.
  Let $x \notin \SB(D)$ be a $k$-rational point.
  For every integer $m \ge 1$ such that $mD$ is Cartier and $x \notin \Bs(\lvert
  mD \rvert)$, let $\pi_m\colon X_m
  \to X$ be a morphism from a normal projective variety $X_m$ that is an
  isomorphism over a neighborhood of $x$, and such that $\pi_m^{-1}\fb(\lvert mD
  \rvert)\cdot\cO_{X_m} = \cO_{X_m}(-F_m)$ for an effective Cartier divisor
  $F_m$.
  We can then write
  \[
    \bigl\lvert \pi_m^*(mD) \bigr\rvert = \lvert M_m \rvert + F_m,
  \]
  where $\lvert M_m \rvert$ is the moving part and $F_m$ is the fixed part of
  the linear system $\lvert \pi_m^*(mD) \rvert$, respectively.
  We then set
  \[
    \varepsilon_N\bigl(\lVert D \rVert;x\bigr) \coloneqq
    \limsup_{m \to \infty} \frac{\varepsilon\bigl(M_m;\pi_m^{-1}(x)\bigr)}{m}
    \glsadd{movingseshadrinakamaye}
  \]
  where the limit supremum is taken over all $m$ such that $mD$ is
  integral.
  \par To make sure that $\varepsilon_N(\lVert D \rVert;x)$ is well-defined, we
  show that $\varepsilon(M_m;\pi_m^{-1}(x))$ does not depend on the
  choice of morphism $\pi_m$.
  First, any two morphisms $\pi_m\colon X_m \to X$
  and $\pi_m' \colon X_m' \to X$ as above can be dominated by a morphism
  $\pi_m''\colon X_m'' \to X$ satisfying the same properties, and the normality
  of the varieties $X,X_m,X_m',X_m''$ imply that the moving parts on $X_m$ and
  $X_m'$ pullback to the moving part on $X_m''$.
  Since $\pi_m,\pi_m',\pi_m''$ are all isomorphisms in a neighborhood of $x$, we
  see that the Seshadri constants of $M_m,M_m',M_m''$ are equal, hence
  $\varepsilon(M_m;\pi_m^{-1}(x))$ does not depend on the choice of
  $\pi_m$.
\end{definition}
We now show that the limit supremum used to define $\varepsilon_N(\lVert D
\rVert;x)$ is equal to a limit.
\begin{lemma}\label{lem:movingseshnakamayelimit}
  With notation as in \cref{def:movingseshnakamaye}, we have
  \[
    \varepsilon_N\bigl(\lVert D \rVert;x\bigr) =
    \lim_{m \to \infty} \frac{\varepsilon\bigl(M_m;\pi_m^{-1}(x)\bigr)}{m}
    = \sup_m \frac{\varepsilon\bigl(M_m;\pi_m^{-1}(x)\bigr)}{m}.
  \]
\end{lemma}
\begin{proof}
  Let $m$ and $n$ be positive integers such that $mD$ and $nD$ are Cartier
  divisors.
  Choose $\pi\colon X' \to X$ that satisfies the properties in
  \cref{def:movingseshnakamaye} for $\lvert
  mD \rvert$, $\lvert nD \rvert$, and $\lvert (m+n)D \rvert$, for example by
  blowup up the base loci for all three linear systems, and then taking a
  normalization.
  Since we have $M_{m+n} = M_m + M_n + E$ for some effective divisor $E$ with
  $\pi^{-1}(x) \notin \Supp(E)$, we deduce that
  \[
    \varepsilon\bigl(M_{m+n};\pi^{-1}(x)\bigr) \ge
    \varepsilon\bigl(M_{m};\pi^{-1}(x)\bigr) +
    \varepsilon\bigl(M_{n};\pi^{-1}(x)\bigr).
  \]
  Thus, the sequence $\{\varepsilon(M_m;\pi_m^{-1}(x))\}_m$ is superadditive,
  and Fekete's lemma\index{Fekete, Michael, lemma of} \cite[Pt.\ I,
  n\textsuperscript{o} 98]{PS98} implies that the limit supremum is equal to the
  limit and the supremum.
\end{proof}
Nakamaye's definition coincides with the one in \cref{def:movingsesh} for $x
\notin \SB(D)$.
\begin{proposition}[cf.\ {\cite[Prop.\ 6.4]{ELMNP09}}]
  \label{prop:movingseshnakamaye}
  Let $X$ be a normal projective variety over a field $k$ and let $D$ be a
  $\QQ$-Cartier divisor.
  If $x \notin \SB(D)$ is a $k$-rational point, then $\varepsilon(\lVert D
  \rVert;x) = \varepsilon_N(\lVert D \rVert;x)$.
\end{proposition}
\begin{proof}
  By \cref{prop:movingseshelem}\cref{prop:movingseshhomog} and
  \cref{lem:movingseshnakamayelimit}, both invariants are homogeneous with
  respect to taking integer multiples of $D$.
  It therefore suffices to consider the case when $D$ is integral, $\SB(D) =
  \Bs(\lvert D \rvert)_\red$, and $\lvert D \rvert$ induces a rational map
  birational onto its image.
  Let $m$ be a positive integer, and let $\pi_m\colon X_m \to X$ be as in
  \cref{def:movingseshnakamaye}.
  Writing $\lvert \pi_m^*(mD) \rvert = \lvert M_m \rvert + F_m$, we note that by
  assumption on $x$, we have that $\pi_m^{-1}(x) \notin \Supp(F_m)$.
  \par First, suppose that $x \in \Bplus(D)$.
  By assumption, $x \notin \SB(D) = \Bs(\lvert D \rvert)_\red$, hence
  $\pi_m^{-1}(x) \notin
  \Bs(\lvert \pi_m^*D \rvert)_\red = \Supp(F_m)$ by the normality of $X$.
  Thus, \cref{prop:bbp23} implies $\pi_m^{-1}(x) \in
  \Bplus(M_m)$, and \cref{thm:nakamaye} implies there exists a
  subvariety $V \subseteq X_m$ such that $\pi_m^{-1}(x) \in V$ and
  $(M_m^{\dim V} \cdot V) = 0$.
  We therefore see that $\varepsilon(M_m;\pi_m^{-1}(x)) = 0$ by \cite[Prop.\
  5.1.9]{Laz04a}.
  Since this is true for every $m$, we have that $\varepsilon_N(\lVert D
  \rVert;x) = 0$, hence $\varepsilon(\lVert D \rVert;x) =
  \varepsilon_N(\lVert D \rVert;x)$ when $x \in \Bplus(D)$.
  \par In the rest of the proof, we therefore assume that $x \notin \Bplus(D)$.
  Since $\lvert M_m \rvert$ induces a map birational onto the image of $X'$ by
  the assumption that $\lvert D \rvert$ induces a rational map birational onto
  its image, we
  may write $M_m \equiv_\QQ A + E$, where $A$ and $E$ are ample
  and effective $\QQ$-Cartier divisors, respectively, such that
  $\pi_m^{-1}(x) \notin \Supp(E)$ \cite[Cor.\ 2.2.7]{Laz04a}.
  For every integer $n \ge 1$, we can write
  \[
    M_m \equiv_\QQ \frac{1}{n}A + \frac{n-1}{n}M_m + \frac{1}{n}E,
  \]
  hence setting $A_n \coloneqq \frac{1}{n}A + \frac{n-1}{n}M_m$, we have $M_m
  \equiv_\QQ A_n + \frac{1}{n}E$ for an ample $\QQ$-Cartier divisor $A_n$ and a
  fixed effective $\QQ$-Cartier divisor $E$.
  By \cref{def:movingsesh}, we see that
  \begin{align*}
    \varepsilon\bigl(\lVert D \rVert;x\bigr) &\ge
    \frac{\varepsilon\bigl(A_n;\pi_m^{-1}(x)\bigr)}{m}
    \intertext{for every $n$.
    Using \cref{eq:decompseshapprox} and taking the limit as $n \to \infty$, we
    see that}
    \varepsilon\bigl(\lVert D \rVert;x\bigr) &\ge
    \frac{\varepsilon\bigl(M_m;\pi_m^{-1}(x)\bigr)}{m},
  \end{align*}
  hence $\varepsilon(\lVert D \rVert;x) \ge
  \varepsilon_N(\lVert D \rVert;x)$.
  \par We now show the reverse inequality.
  Let $f\colon X' \to X$ and $f^*D \equiv_\RR A+E$ be as in
  \cref{def:movingsesh}.
  Fix $m$ such that $mA$ is a very ample Cartier divisor.
  By taking a normalized blowup of the base locus of $f^*(mD)$, which by
  assumption is an isomorphism in a neighborhood of $f^{-1}(x)$, we can write
  \[
    \lvert f^*(mD) \rvert = \lvert M_m \rvert + F_m
  \]
  as in \cref{def:movingseshnakamaye}.
  Since $mA$ is free, we have $M_m \sim mA + F'_m$ for an effective Cartier
  divisor $F'_m$ such that $F'_m \le mE$, hence $f^{-1}(x) \notin
  \Supp(F'_m)$.
  We therefore have
  \[
    \frac{\varepsilon\bigl(M_m;\pi_m^{-1}(x)\bigr)}{m} \ge
    \varepsilon\bigl(A;f^{-1}(x) \bigr),
  \]
  hence $\varepsilon_N(\lVert D \rVert;x) \ge
  \varepsilon(\lVert D \rVert;x)$.
\end{proof}
\index{Seshadri constant, moving, $\varepsilon(\lVert D \rVert;x)$!Nakamaye's
description of|)}
\subsection{A description in terms of jet separation}
We now show that just as for usual Seshadri constants of ample Cartier divisors
at regular closed points on projective varieties \cite[Thm.\ 5.1.17]{Laz04a},
moving Seshadri constants can be described using separation of jets.
Note that this description (Proposition \ref{prop:jetsep}) is new even in
characteristic zero for singular points.
\par Recall from \cref{def:sepljets} that if $X$ is a scheme, $x \in X$ is a
closed point, and $\ell \ge -1$ is an integer, a coherent sheaf $\sF$
\textsl{separates $\ell$-jets} at $x$ if the restriction morphism
\[
    H^0(X,\sF) \longrightarrow H^0(X,\sF \otimes
    \cO_X/\fm_x^{\ell+1})
\]
is surjective, and that we denote by $s(\sF;x)$ the largest integer $\ell \ge
-1$ such that $\sF$ separates $\ell$-jets at $x$.
We can then define moving Seshadri constants using jet separation.
\begin{definition}\label{def:jetsesh}
  Let $X$ be a projective variety over a field $k$ and let $D$ be a
  $\QQ$-Cartier divisor on $X$.
  Consider a $k$-rational point $x \in X$ with defining ideal $\fm_x \subseteq
  \cO_X$.
  We set
  \[
    \varepsilon_{\textup{jet}}\bigl(\lVert D \rVert;x \bigr) \coloneqq \limsup_{m \to \infty}
    \frac{s(mD;x)}{m},\glsadd{movingseshadrijets}
  \]
  where the limit supremum runs over all integers $m \ge 1$ such that $mD$ is
  integral.
\end{definition}
\par We now prove that the limit supremum used to define $\varepsilon_{\textup{jet}}(\lVert D
\rVert;x )$ is equal to a limit.
\begin{lemma}[{\cite[Lem.\ 6.4]{FM}; cf.\ \cite[Proof of Lem.\ 3.7]{Ito13}}]
  \label{lem:ito37}
  Let $X$ be a scheme, and let $\sF$ and $\sG$ be coherent sheaves on $X$.
  Then, for every closed point $x \in X$ such that $s(\sF;x) \ge 0$ and
  $s(\sG;x) \ge 0$, we have
  \[
    s(\sF;x) + s(\sG;x) \le s(\sF \otimes \sG;x).
  \]
  With notation in \cref{def:jetsesh}, we therefore have
  \[
    \varepsilon_{\textup{jet}}\bigl(\lVert D \rVert;x \bigr) = \lim_{m \to \infty}
    \frac{s(mD;x)}{m} = \sup_{m\ge1} \frac{s(mD;x)}{m}.
  \]
\end{lemma}
\begin{proof}
  We first show that a coherent sheaf $\sF$ separates $\ell$-jets if and
  only if
  \begin{equation}\label{eq:mimiplus1}
    H^0(X,\fm_x^i\sF) \longrightarrow H^0(X,\fm_x^i\sF/\fm_x^{i+1}\sF)
  \end{equation}
  is surjective for every $i \in \{0,1,\ldots,\ell\}$.
  We proceed by induction on $\ell$.
  If $\ell = 0$, then there is nothing to show.
  Now suppose $\ell > 0$.
  By induction and the fact that a coherent sheaf separating $\ell$-jets also
  separates all lower order jets, it suffices to show that if $\sF$
  separates $(\ell-1)$-jets, then $\sF$ separates $\ell$-jets if and only if
  \cref{eq:mimiplus1} is surjective for $i = \ell$.
  Consider the commutative diagram
  \[
    \begin{tikzcd}
      0 \rar & \fm_x^\ell\sF \rar\dar & \sF \rar\dar & \sF/\fm_x^\ell\sF
      \rar\dar[equal] & 0\\
      0 \rar & \fm_x^\ell\sF/\fm_x^{\ell+1}\sF \rar & \sF/\fm_x^{\ell+1}\sF \rar
      & \sF/\fm_x^\ell\sF \rar & 0
    \end{tikzcd}
  \]
  Taking global sections, we obtain the diagram
  \[
    \begin{tikzcd}[column sep=1.475em]
      0 \rar & H^0(X,\fm_x^\ell\sF) \rar\dar & H^0(X,\sF) \rar\dar &
      H^0(X,\sF/\fm_x^\ell\sF) \dar[equal] \rar & 0\\
      0 \rar & H^0(X,\fm_x^\ell\sF/\fm_x^{\ell+1}\sF) \rar &
      H^0(X,\sF/\fm_x^{\ell+1}\sF) \rar & H^0(X,\sF/\fm_x^\ell\sF)
    \end{tikzcd}
  \]
  where the top row remains exact by the assumption that $\sF$ separates
  $(\ell-1)$-jets.
  By the snake lemma, we see that the left vertical arrow is surjective if and
  only if the middle vertical arrow is surjective, as desired.
  \par We now prove the lemma.
  Suppose $\sF$ separates $i$-jets and $\sG$ separates $j$-jets.
  We then have the commutative diagram
  \[
    \begin{tikzcd}
      H^0(X,\fm_x^i\sF) \otimes H^0(X,\fm_x^j\sG) \rar \dar &
      H^0(X,\fm_x^i\sF/\fm_x^{i+1}\sF \otimes
      \fm_x^j\sG/\fm_x^{j+1}\sG)\dar\\
      H^0\bigl(X,\fm_x^{i+j}(\sF \otimes \sG)\bigr) \rar &
      H^0\bigl(X, \fm_x^{i+j}(\sF \otimes
      \sG)/\fm_x^{i+j+1}(\sF \otimes \sG)\bigr)
    \end{tikzcd}
  \]
  Since the top horizontal arrow is surjective by assumption, and the right
  vertical arrow is surjective, essentially by the surjectivity of 
  \[
    \fm_x^i/\fm_x^{i+1} \otimes
    \fm_x^j/\fm_x^{j+1} \simeq (\fm_x^{i}\otimes\fm_x^j) \otimes
    \cO_X/\fm_x \longtwoheadrightarrow
    \fm_x^{i+j}/\fm_x^{i+j+1},
  \]
  we see that the composition from the top left corner to the bottom right
  corner is surjective, hence the bottom horizontal arrow is surjective.
  By running through all combinations of integers $i \le s(\sF;x)$ and
  $j \le s(\sG;x)$, we see that $s(\sF;x) + s(\sG;x) \le
  s(\sF \otimes \sG;x)$ by the argument in the previous
  paragraph.
  \par The last statement about $\varepsilon_{\textup{jet}}(\lVert D \rVert;x)$ follows from
  Fekete's lemma \cite[Pt.\ I, n\textsuperscript{o} 98]{PS98}, since we have
  shown the superadditivity of the
  sequence $\{s( mD ;x)\}_{m \ge 1}$ provided that $x \notin
  \SB(D)$.
  If $x \in \SB(D)$, then $s(mD;x) = -1$ for all $m \ge 1$ such that $mD$ is
  integral, hence the limit, limit supremum, and supremum are all equal to zero.
\end{proof}
The constant $\varepsilon_{\textup{jet}}(\lVert D \rVert;x)$ detects $\Bplus(D)$.
\begin{lemma}\label{lem:jetseshbplus}
  Let $X$ be a normal projective variety over a field $k$ and let $D$ be a
  $\QQ$-Cartier divisor on $X$.
  Consider a $k$-rational point $x \in X$.
  Then, $x \in \Bplus(D)$ if and only if $\varepsilon_{\textup{jet}}(\lVert D
  \rVert;x) = 0$.
\end{lemma}
\begin{proof}
  For $\Leftarrow$, note that $\varepsilon_{\textup{jet}}(\lVert D
  \rVert;x) = 0$ implies $s(mD;x) \le 0$ for all $m$ such that $mD$ is integral,
  since \cref{lem:ito37} implies $\varepsilon_{\textup{jet}}(\lVert D \rVert;x)$ is a
  supremum.
  If $x \in \SB(D)$, then $x \in \Bplus(D)$ as well, hence it suffices to
  consider the case when $x \notin \SB(D)$.
  Suppose $x \notin \Bplus(D)$, and let $A$ be a very ample Cartier divisor on
  $X$.
  By \cite[Prop.\ 1.5]{ELMNP06} and \cite[Prop.\ 2.1.21]{Laz04a}, there exist
  positive integers $q,r$ such that
  \[
    \Bplus(D) = \SB(rD-A) = \Bs\bigl(\bigl\lvert q(rD-A)\bigr\rvert\bigr)_\red.
  \]
  Since $x \notin \Bplus(D)$, we see that $\lvert q(rD-A) \rvert$ is
  basepoint-free at $x$.
  Moreover, since $\cO_X(qA)$ separates $1$-jets at $x$ by the very ampleness of
  $A$, we see that $\cO_X(qrD)$ separates $1$-jets at $x$ by \cref{lem:ito37}, a
  contradiction.
  \par For $\Rightarrow$, we note that if $x \in \SB(D)$, then $s(mD;x) = -1$
  for all $m$ such that $mD$ is integral,
  hence it suffices to consider the case when $x \in \Bplus(D) \smallsetminus
  \SB(D)$.
  Suppose $s(mD;x) > 0$ for some integer $m > 0$.
  By possibly replacing $m$ with a large and divisible enough multiple, we may
  assume that $\SB(mD) = \Bs(\lvert mD \rvert)_\red$ by \cref{lem:ito37}
  and \cite[Prop.\ 2.1.21]{Laz04a}.
  Then, for every subvariety $V \subseteq X$ containing $x$, we have $V
  \not\subseteq \Bs(\lvert mD \rvert)_\red$.
  Moreover, since $\cO_X(mD)$ separates tangent
  directions at $x$, there exists $E \in \lvert mD \rvert$ not containing $V$,
  in which case $(E^{\dim V} \cdot V) > 0$.
  \par Now let $\pi_m\colon X_m \to X$ be the normalized blowup of $\fb(\lvert
  mD \rvert)$, and write $\lvert \pi_m^*(mD)\rvert = \lvert M_m \rvert + F_m$
  where $\cO_{X_m}(-F_m) = \fb(\lvert mD \rvert) \cdot \cO_{X_m}$.
  By \cref{prop:bbp23} and the definition of the augmented base
  locus, we have
  \[
    \pi_m^{-1}(x) \in \Bplus(\pi_m^*D) \subseteq \Bplus(M_m) \cup \Supp(F_m).
  \]
  The fact that $x \notin \SB(D)$ implies $\pi_m^{-1}(x) \in \Bplus(M_m)$.
  By \cref{thm:nakamaye}, there therefore exists a subvariety $W \subseteq
  X_m$ such that $\pi_m^{-1}(x) \in W$ and $(M_m^{\dim W} \cdot W) = 0$.
  Now choose $E \in \lvert mD \rvert$ as in the previous paragraph for $V =
  \pi_{m*}W$.
  Since $F_m = (\pi_m)^{-1}\Bs(\lvert mD \rvert)$, we have that $\pi_m^*E -
  F_m$ is an effective Cartier divisor that contains $\pi_m^{-1}(x)$ but does
  not contain $W$, hence
  \[
    (M_m^{\dim W} \cdot W) = \bigl( (\pi_m^*E - F_m)^{\dim W} \cdot W \bigr) >
    0,
  \]
  a contradiction.
  We therefore have $\varepsilon_{\textup{jet}}(\lVert D \rVert;x) = 0$ if $x \in \Bplus(D)$.
\end{proof}
This lemma has the following consequence:
\begin{corollary}\label{cor:jetseshample}
  Let $X$ be a normal projective variety over an algebraically closed field $k$,
  and let $D$ be a $\QQ$-Cartier divisor on $X$.
  Then, $D$ is ample if and only if $\varepsilon_{\textup{jet}}(\lVert D \rVert;x) > 0$ for
  every closed point $x \in X$.
\end{corollary}
\begin{proof}
  By \cref{lem:jetseshbplus}, we have that $\varepsilon_{\textup{jet}}(\lVert D
  \rVert;x) > 0$ for every closed point $x \in X$ if and only if $\Bplus(D) =
  \emptyset$.
  This condition is equivalent to the ampleness of $D$ by \cite[Ex.\
  1.7]{ELMNP06}.
\end{proof}
We now collect some properties of $\varepsilon_{\textup{jet}}(\lVert D
\rVert;x)$ analogous to those in \cref{prop:movingseshelem}.
Below, $\vol_{X \vert V}(D)$ denotes the restricted volume of $D$ along a
subvariety $V$, as defined in \cref{def:restrictedvolume}.
\begin{proposition}\label{prop:jetseshelem}
  Let $X$ be a projective variety over a field $k$ and let $x \in X$ be a
  $k$-rational point.
  Then, the function
  \begin{equation}\label{eq:jetseshfunction}
    \begin{tikzcd}[row sep=0,column sep=1.6em]
      \Bigcone^{\{x\}}_\QQ(X) \rar & \RR_{>0}\\
      D \rar[mapsto] & \varepsilon_{\textup{jet}}\bigl(\lVert D \rVert;x\bigr)
    \end{tikzcd}
  \end{equation}
  is continuous, and extends to a continuous function
  $\Bigcone^{\{x\}}_\RR(X) \to \RR_{>0}$.
  Moreover, if $D$ is a $\RR$-Cartier divisor, then we have the following:
  \begin{enumerate}[label=$(\roman*)$,ref=\roman*]
    \item\label{prop:jetseshvol}
      $\varepsilon_{\textup{jet}}(\lVert D \rVert;x) \le
      (\vol_{X \vert V}(D)/e(\cO_{V,x}))^{1/\dim V}$ for every
      positive-dimensional subvariety $V \subseteq X$ containing $x$;
    \item\label{prop:jetseshnumequiv}
      If $D$ and $E$ are numerically equivalent $\RR$-Cartier divisors, then
      $\varepsilon_{\textup{jet}}(\lVert D \rVert;x) = \varepsilon_{\textup{jet}}(\lVert E \rVert;x)$;
    \item\label{prop:jetseshhomog}
      $\varepsilon_{\textup{jet}}(\lVert \lambda D \rVert;x) = \lambda \cdot
      \varepsilon_{\textup{jet}}(\lVert D \rVert;x)$ for every positive real number
      $\lambda$;
    \item\label{prop:jetseshadd}
      If $D'$ is another $\RR$-Cartier divisor such that $x \notin \Bplus(D)
      \cup \Bplus(D')$, then
      \[
        \varepsilon_{\textup{jet}}\bigl(\lVert D + D' \rVert;x) \ge
        \varepsilon_{\textup{jet}}\bigl(\lVert D \rVert;x\bigr) +
        \varepsilon_{\textup{jet}}\bigl(\lVert D' \rVert;x\bigr).
      \]
  \end{enumerate}
\end{proposition}
\begin{proof}
  We will prove \cref{prop:jetseshvol}--\cref{prop:jetseshadd} for
  $\QQ$-Cartier divisors $D,D',E$ and $\lambda \in \QQ_{>0}$.
  Then, \cref{prop:jetseshnumequiv} will imply that the function
  \cref{eq:jetseshfunction} is well-defined, and the fact that it extends to a
  continuous function on $\Bigcone^{\{x\}}_\RR(X)$ follows from
  \cref{prop:jetseshhomog} and
  \cref{prop:jetseshadd} by \cite[Rem.\ 5.4]{ELMNP09}, since
  $\varepsilon_{\textup{jet}}(\lVert A \rVert;x) > 0$ for ample $A$
  (\cref{cor:jetseshample}).
  Finally, the 
  general case for \cref{prop:jetseshvol}--\cref{prop:jetseshadd}
  will follow by continuity.
  \par We first prove \cref{prop:jetseshhomog} when $D$ is a
  $\QQ$-Cartier divisor and $\lambda \in \QQ_{>0}$.
  We have
  \begin{alignat*}{3}
    \lambda \cdot \varepsilon_{\textup{jet}}\bigl(\lVert D \rVert;x\bigr)
    &={}& \lambda \cdot \lim_{m \to \infty} \frac{s(mD;x)}{m}
    &{}= \lim_{m \to \infty} \frac{s(mD;x)}{m/\lambda}\\
    &={}& \omit\hfil$\displaystyle\lim_{m \to \infty}
    \frac{s(m\lambda D;x)}{m}$\hfil\ignorespaces
    &{}= \varepsilon_{\textup{jet}}\bigl(\lVert \lambda D \rVert;x\bigr)
  \end{alignat*}
  where the third equality holds since both sides are equal to the limits
  running over all $m$ sufficiently divisible.
  To show the remaining properties, then, it suffices to consider the case when
  $D,D',E$ are Cartier divisors.
  \par Next, we prove \cref{prop:jetseshvol} when $D$ is a Cartier divisor.
  Since the inequality trivially holds when $\varepsilon_{\textup{jet}}(\lVert D \rVert;x) =
  0$, we may assume that $\varepsilon_{\textup{jet}}(\lVert D \rVert;x) > 0$.
  In this case, we have
  \begin{align*}
    \frac{\vol_{X \vert V}(D)}{\mult_xV} &= \lim_{m \to \infty}
    \frac{h^0\bigl(X \vert V,\cO_X(mD)\bigr)}{m^{\dim V}/(\dim V)!}
    \cdot \lim_{\ell \to \infty} \frac{(\ell+1)^{\dim V}/(\dim
    V)!}{h^0\bigl(V,\cO_V(mD) \otimes \cO_V/\fm_x^{\ell+1}\bigr)}\\
    &= \lim_{m \to \infty}
    \frac{h^0\bigl(X \vert V,\cO_X(mD)\bigr)}{h^0\bigl(V,\cO_V(mD) \otimes
    \cO_V/\fm_x^{s(mD;x)+1}\bigr)} \cdot
    \biggl(\frac{s(mD;x)+1}{m}\biggr)^{\dim X},
  \end{align*}
  where the second equality follows from setting $\ell = s(mD;x)$, and then from
  the fact that $s(mD;x) \to \infty$ as $m \to \infty$.
  By definition of $s(mD;x)$ and the commutativity of the diagram
  \[
    \begin{tikzcd}
      H^0\bigl(X,\cO_X(mD)\bigr) \rar\dar & H^0\bigl(X,\cO_X(mD) \otimes
      \cO_X/\fm_x^{\ell+1}\bigr)\dar[twoheadrightarrow]\\
      H^0\bigl(V,\cO_V(mD)\bigr) \rar & H^0\bigl(V,\cO_V(mD) \otimes
      \cO_V/\fm_x^{\ell+1}\bigr)
    \end{tikzcd}
  \]
  we have $h^0(X \vert V,\cO_X(mD)) \ge h^0(V,\cO_V(mD) \otimes
  \cO_V/\fm_x^{s(mD;x)+1})$.
  Thus,
  \[
    \frac{\vol_{X \vert V}(D)}{\mult_xV} \ge \lim_{m \to \infty}
    \biggl(\frac{s(mD;x)+1}{m}\biggr)^{\dim V}
    = \varepsilon_{\textup{jet}}\bigl(\lVert D \rVert;x)^{\dim V}.
  \]
  \par We now prove \cref{prop:jetseshnumequiv}.
  First, recall that $\Bplus(D)$ only depends on the numerical class of $D$, and
  that if $x \in \Bplus(D)$, then $\varepsilon_{\textup{jet}}(\lVert D \rVert;x) = 0$ by
  \cref{lem:jetseshbplus}.
  We can therefore assume that $x \notin \Bplus(D)$.
  By assumption, there exists a numerically trivial Cartier divisor $P$ such
  that $D \sim E + P$, and \cref{prop:kur13prop27} implies that there
  exists a positive integer $j$ such that $\cO_X(jD+iP)$ is globally generated
  away from $\Bplus(D)$ for all integers $i$.
  For every $m$, we therefore see that
  \[
    s(mD;x) \le s\bigl( (m+j)D+(m+j)P;x\bigr) = s(iE;x)
  \]
  by setting $i = m+j$, where the inequality follows from \cref{lem:ito37}
  since $\cO_X(jD+(m+j)P)$ separates $0$-jets at $x$.
  Dividing by $m$ and taking limits, we see that $\varepsilon_{\textup{jet}}(\lVert D
  \rVert;x) \le \varepsilon_{\textup{jet}}(\lVert E \rVert;x)$.
  Repeating the argument above after switching the roles of $D$ and $E$, we have
  $\varepsilon_{\textup{jet}}(\lVert D \rVert;x) = \varepsilon_{\textup{jet}}(\lVert E
  \rVert;x)$.
  \par Finally, \cref{prop:jetseshadd} follows from \cref{lem:ito37}.
\end{proof}
\begin{remark}
  In \cref{prop:jetseshelem}\cref{prop:jetseshvol}, one can ask
  whether
  \[
    \varepsilon_{\textup{jet}}\bigl(\lVert D\rVert;x\bigr) =
    \inf_{V \ni x} \biggl\{\frac{\vol_{X \vert
    V}(D)}{e(\cO_{V,x})}\biggr\}^{1/\dim V},
  \]
  where the infimum runs over all subvarieties $V \subseteq X$ containing $x$.
  This holds for smooth varieties over the complex numbers \cite[Prop.\
  6.7]{ELMNP09}, or when $D$ is nef \cite[Prop.\ 5.1.9]{Laz04a}.
  A proof of this statement in the generality of Proposition
  \ref{prop:jetseshelem} would require extending the main results about
  restricted volume functions in \cite{ELMNP09} to our setting.
\end{remark}
We can now prove our comparison result.
Note that we do not know of any examples where the equalities in
the statement below do not hold without the additional assumptions on $D$ and
$X$.
\begin{proposition}[cf.\ {\cite[Prop.\ 6.6]{ELMNP09}}]\label{prop:jetsep}
  Let $X$ be a projective variety over a field $k$, and let $D$ be a
  $\QQ$-Cartier divisor.
  For every $k$-rational point $x \in X$, we have
  \begin{enumerate}[label=$(\roman*)$,ref=\roman*]
    \item $\varepsilon(D;x) = \varepsilon_{\textup{jet}}(\lVert D \rVert;x)$ if
      $D$ is nef and $x \notin \Bplus(D)$, and\label{prop:jetsepbignef}
    \item $\varepsilon(\lVert D \rVert;x) = \varepsilon_{\textup{jet}}(\lVert D \rVert;x)$ if
      $X$ is normal.\label{prop:jetseps2}
  \end{enumerate}
\end{proposition}
We first show that the case when $D$ is nef implies the general case,
under the assumption that $X$ is normal.
\begin{proof}[Proof that \cref{prop:jetsepbignef} implies
  \cref{prop:jetseps2}]
  Since both sides are zero if $x \in \Bplus(D)$
  (\cref{def:movingsesh,lem:jetseshbplus}), we may assume that $x \notin
  \Bplus(D)$.
  By homogeneity (\cref{prop:movingseshelem,prop:jetseshelem}),
  it suffices to consider the case when $D$ is a Cartier divisor and both sides
  are positive, in which case we still have $x \notin \Bplus(D)$.
  Note that in particular, we have $x \notin \SB(D)$.
  \par For every integer $m \ge 1$ such that $x \notin \Bs(\lvert
  mD \rvert)$, let $\pi_m\colon X_m \to X$ be as in
  \cref{def:movingseshnakamaye}, and write
  \[
    \bigl\lvert\pi_m^*(mD)\bigr\rvert = \lvert M_m \rvert + F_m,
  \]
  where $\lvert M_m \rvert$ is the moving part and $F_m$ is the fixed part of
  the linear system $\lvert\pi_m^*(mD)\rvert$.
  Since $X$ is normal, we have $\pi_{m*}\cO_{X_m} \simeq \cO_X$ 
  \cite[Proof of Cor.\ III.11.4]{Har77}.
  Note that the base ideal of $\lvert\pi_m^*(mD)\rvert$ is
  $\cO_{X_m}(-F_m)$, and in particular, we have $x_m \coloneqq \pi_m^{-1}(x)
  \notin \Supp F_m$.
  Thus, the inclusion
  \[
    H^0\bigl(X_m,\cO_{X_m}(M_m)\bigr) \subseteq
    H^0\bigl(X_m,\pi_m^*\cO_X(mD)\bigr)
  \]
  induced by multiplication by $F_m$ is a bijection,
  and $M_m$ is a free Cartier divisor.
  We then have the commutative diagram
  \begin{equation}\label{eq:jetsepkeydiag}
    \begin{tikzcd}
      H^0\bigl(X_m,\cO_{X_m}(nM_m)\bigr) \rar\arrow[hook]{d}[left]{\cdot
      nF_m} & H^0\bigl(X_m,\cO_{X_m}(nM_m)
      \otimes \cO_X/\fm_{x_m}^{\ell+1}\bigr)\arrow{d}[left]{\cdot
      nF_m}[above,sloped]{\sim}\\
      H^0\bigl(X_m,\pi_m^*\cO_X(mnD)\bigr) \rar & H^0\bigl(X_m,\pi_m^*\cO_X(mnD)
      \otimes \cO_X/\fm_{x_m}^{\ell+1}\bigr)\\
      H^0\bigl(X,\cO_X(mnD)\bigr) \rar\arrow{u}[below,sloped]{\sim}
      & H^0\bigl(X,\cO_X(mnD) \otimes \cO_X/\fm_x^{\ell+1}\bigr)
      \arrow{u}[below,sloped]{\sim}
    \end{tikzcd}
  \end{equation}
  for all integers $n \ge 1$ and $\ell \ge -1$, where the top left vertical
  arrow is an isomorphism for $n = 1$ by the discussion above,
  the bottom left vertical arrow is an isomorphism by the fact that
  $\pi_{m*}\cO_{X_m} \simeq \cO_X$, and the right vertical arrows are
  isomorphisms by the fact that $x \notin \Supp F_m$ and $\pi_m$ is an
  isomorphism in a neighborhood of $x$, respectively.
  \par To show the inequality $\ge$ in \cref{prop:jetseps2}, let $m$ be
  such that $\varepsilon(M_m;x_m) > 0$, in which case $x_m \notin \Bplus(M_m)$
  by \cref{thm:nakamaye} and \cite[Prop.\ 5.1.9]{Laz04a}.
  Note that this property holds for all sufficiently large $m$ by
  \cref{prop:movingseshnakamaye}.
  We then have the chain of inequalities
  \[
    \varepsilon\bigl(\lVert D \rVert;x\bigr)
    \ge \frac{\varepsilon(M_m;x_m)}{m}
    \ge \frac{s(M_m;x_m)}{m}
    = \frac{s(mD;x)}{m}
  \]
  for all such $m$, where the second inequality follows from
  \cref{prop:jetsepbignef}, and the equality follows from the commutativity
  of the diagram \cref{eq:jetsepkeydiag}.
  Taking the limit as $m \to \infty$, we have the inequality $\ge$ in
  \cref{prop:jetseps2}.
  \par To show the inequality $\le$ in \cref{prop:jetseps2}, let $\delta > 0$ be
  arbitrary.
  For $m \gg 0$ and $n \gg 0$, we have the following chain of inequalities:
  \[
    \varepsilon\bigl(\lVert D \rVert;x\bigr) \le \frac{\varepsilon(M_m;x_m)}{m}
    + \frac{\delta}{2} \le \frac{s(nM_m;x_m)}{mn} + \delta \le
    \frac{s(mnD;x_m)}{mn} + \delta.
  \]
  For the middle inequality, we need $m$ to be sufficiently large such that
  $\varepsilon(M_m;x_m) > 0$ as in the previous paragraph, in which case the
  inequality follows from \cref{prop:jetsepbignef} for $n \gg 0$.
  The last inequality follows from the commutativity of the diagram
  \cref{eq:jetsepkeydiag}.
  Taking the limit as $m \to \infty$, and using the fact that $\delta$ was
  arbitrary, we have the inequality $\le$ in \cref{prop:jetseps2}.
\end{proof}
To prove \cref{prop:jetsepbignef}, we need the following elementary lemma:
\begin{lemma}\label{lem:jetsepelem}
  Let $X$ be a projective variety of dimension $n$, and let $x \in X$ be a
  closed point with defining ideal $\fm_x \subseteq \cO_X$.
  Let $L$ be a Cartier divisor on $X$.
  \begin{enumerate}[label=$(\roman*)$,ref=\roman*]
    \item\label{lem:jetsepelem1}
      If $L$ is ample, then for $m$ sufficiently large, we have
      $H^i(X,\cO_X(mL) \otimes \mathfrak{m}_x^a) = 0$
      for all $i > 1$ and $a \ge 0$.
    \item\label{lem:jetsepelem2}
      If $H^1(X,\cO_X(mL) \otimes \mathfrak{m}_x^a) \ne 0$ for
      some $a,m > 0$, then $H^1(X,\cO_X(mL) \otimes
      \mathfrak{m}_x^{a+1}) \ne 0$.
  \end{enumerate}
\end{lemma}
\begin{proof}
  For \cref{lem:jetsepelem1}, consider the exact sequence
  \[
    H^{i-1}\bigl(X,\cO_X(mL) \otimes
    \mathcal{O}_X/\mathfrak{m}_x^a\bigr) \longrightarrow
    H^{i}\bigl(X,\cO_X(mL) \otimes
    \mathfrak{m}_x^a\bigr) \longrightarrow
    H^{i}\bigl(X,\cO_X(mL)\bigr).
  \]
  Note that the left-hand term vanishes if $i > 1$ since
  $\mathcal{O}_X/\mathfrak{m}_x^a$ has zero-dimensional support.
  We have that $H^i(X,\cO_X(mL)) = 0$ for all $m$ sufficiently large by Serre
  vanishing, hence the exact sequence implies $H^i(X,\cO_X(mL) \otimes
  \mathfrak{m}_x^a) = 0$ as well.
  \par For \cref{lem:jetsepelem2}, consider the exact sequence
  \[
    H^{1}\bigl(X,\cO_X(mL) \otimes \mathfrak{m}_x^{a+1}\bigr)
    \longrightarrow 
    H^{1}\bigl(X,\cO_X(mL) \otimes \mathfrak{m}_x^a\bigr)
    \longrightarrow H^{1}\bigl(X,\cO_X(mL) \otimes
    \mathfrak{m}_x^a/\mathfrak{m}_x^{a+1}\bigr).
  \]
  The sheaf $\mathfrak{m}_x^a/\mathfrak{m}_x^{a+1}$ has zero-dimensional
  support, hence the right-hand term vanishes, and the desired non-vanishing
  follows.
\end{proof}
We can now prove \cref{prop:jetsep}\cref{prop:jetsepbignef}.
Part of the proof below was suggested by Harold Blum\index{Blum, Harold},
following the strategy in \cite[Thm.\ 2.3]{Fuj18}.
\begin{proof}[Proof of \cref{prop:jetsep}\,\cref{prop:jetsepbignef}]
  By continuity and homogeneity (\cref{prop:jetseshelem}),
  it suffices to consider the case when $D$ is an ample Cartier divisor.
  \par We prove the inequality $\ge$ in \cref{prop:jetsepbignef}.
  Let $0 < \delta \ll 1$ be arbitrary, and fix positive integers $p_0,q_0$ such
  that
  \[
    0 < \frac{p_0}{2q_0} < \varepsilon(D;x) <
    \frac{p_0}{q_0} < \varepsilon(D;x) + \delta.
  \]
  Then, denoting by $\mu\colon \widetilde{X} \to X$ the blowup at $x$
  with exceptional divisor $E$, we have that $A \coloneqq 2q_0\mu^*D - p_0E$ is
  ample while $B \coloneqq q_0\mu^*D - p_0E$ is not ample by
  \cref{lem:seshample}.
  By \cref{thm:dfkl41} and the homogeneity of asymptotic cohomological functions
  (\cref{prop:nahomogcont}), for some integer $r \gg 2$ and for
  some $i \ge 1$, we have that
  \[
    H^i\bigl(\widetilde{X},\cO_{\widetilde{X}}\bigl(mr\bigl(B -
    (1/r)A\bigr)\bigr)\bigr) \ne 0
  \]
  for infinitely many $m$.
  Now
  \begin{align*}
    mr\bigl(B - (1/r)A\bigr) &= mrq_0\bigl(1-(2/r)\bigr)\mu^*D -
    mrp_0\bigl(1-(1/r)\bigr)E\\
    &= m\bigl( (r-2)q_0\mu^*D - (r-1)p_0E\bigr),
  \end{align*}
  and defining $q_1 = (r-2)q_0$ and $p_1 = (r-1)p_0$, the Leray spectral
  sequence applied to the blowup morphism $\mu$ \cite[Lem.\ 5.4.24]{Laz04a}
  implies
  \[
    H^i\bigl(X,\cO_{X}(mq_1D) \otimes \fm_{x}^{mp_1}\bigr) \simeq
    H^i\bigl(\widetilde{X},\cO_{\widetilde{X}}\bigl(m(q_1\mu^*D -
    p_1E)\bigr)\bigr) \ne 0
  \]
  for infinitely many $m$.
  By \cref{lem:jetsepelem}\cref{lem:jetsepelem1}, we must have $i = 1$.
  Since
  \[
    H^1\bigl(X,\cO_{X}(mq_1D)\bigr) = 0
  \]
  for all $m \gg 0$ by Serre vanishing,
  this implies that $mq_1D$ does not separate $(mp_1-1)$-jets at $x$, hence
  $mp_1 - 1 > s(mq_1D;x)$ for infinitely many $m$.
  Dividing the inequality by $mq_1$ and taking limits as $m \to \infty$, we have
  \[
    \varepsilon(D;x) + \delta > \frac{p_0}{q_0} > \frac{(r-1)p_0}{(r-2)q_0} =
    \frac{p_1}{q_1} \ge \lim_{m \to \infty} \frac{s(mq_1D;x)}{mq_1} =
    \varepsilon_{\textup{jet}}\bigl(\lVert D \rVert;x\bigr),
  \]
  where the limit runs over all $m$ sufficiently large and divisible, and the
  last equality holds by the fact that $\varepsilon_{\textup{jet}}(\lVert D
  \rVert;x)$ is computed by a limit (\cref{lem:ito37}).
  Finally, since $\delta$ was arbitrary, the inequality $\ge$ in
  \cref{prop:jetsepbignef} follows.
  \par We now prove the inequality $\le$ in \eqref{prop:jetsepbignef}.
  Let $0 < \delta \ll 1$ be arbitrary, and fix positive integers $p_0,q_0$ such
  that
  \[
    \varepsilon(D;x) - \delta < \frac{p_0}{q_0} <
    \varepsilon(D;x).
  \]
  Then, denoting by $\mu\colon \widetilde{X} \to X$ the blowup at $x$ with
  exceptional divisor $E$, we have that $q_0\mu^*D - p_0E$ is ample, hence by
  Fujita's vanishing theorem \cite[Thm.\ 5.1]{Fuj83} there exists a natural
  number $n_0$ such that
  \[
    H^1\bigl(\widetilde{X},\cO_{\widetilde{X}}\bigl(n(q_0\mu^*D -
    p_0E) + P\bigr) \bigr) = 0
  \]
  for every integer $n \ge n_0$ and all nef Cartier divisors $P$ on
  $\widetilde{X}$.
  Now let $m \ge n_0q_0$ be an integer, and write $m = nq_0 + q_1$ with $0 \le
  q_1 < q_0$ and $n \ge n_0$.
  Applying the vanishing above for $P = q_1\mu^*D$, we have that
  \[
    H^1\bigl(\widetilde{X},\cO_{\widetilde{X}}(m\mu^*D - np_0E )
    \bigr) = 0.
  \]
  By the Leray spectral sequence applied to the blowup morphism $\mu$
  \cite[Lem.\ 5.4.24]{Laz04a}, we have an isomorphism
  \[
    H^1\bigl(\widetilde{X},\cO_{\widetilde{X}}\bigl(m\mu^*D - np_0E \bigr)
    \bigr) \simeq H^1\bigl(X,\cO_{X}(mD) \otimes \fm_{x}^{np_0} \bigr)
    \bigr) = 0
  \]
  for $m \gg 0$ (which implies $n \gg 0$).
  Thus, we see that $\cO_X(mD)$ separates $(np_0-1)$-jets at $x$.
  Now consider the following chain of inequalities:
  \begin{align*}
    \frac{s(mD;x)}{m} \ge \frac{np_0-1}{m}
    &\ge \frac{np_0-1}{(n+1)q_0}
    = \frac{n}{n+1} \cdot \frac{p_0}{q_0} - \frac{1}{(n+1)q_0}\\
    &> \frac{n}{n+1}\bigl(\varepsilon(D;x) - \delta \bigr) -
    \frac{1}{(n+1)q_0}.
  \end{align*}
  Taking the limit as $m \to \infty$, we have that $n \to \infty$ as well, hence
  $\varepsilon_{\textup{jet}}(\lVert D \rVert;x) \ge \varepsilon(D;x) - \delta$.
  Finally, since $\delta$ was arbitrary, the inequality $\le$ in
  \cref{prop:jetsepbignef} follows.
\end{proof}
\begin{remark}
  We give an alternative proof of \cref{prop:jetsep}\cref{prop:jetsepbignef} in
  \cite[Thm.\ 6.3]{FM}.
\end{remark}
\section{A generalization of Theorem \ref{thm:poscharseshsm}}
\label{sect:proofofms31finj}
Our goal in this section is to prove the following generalization of
\cref{thm:poscharseshsm}.
\begin{theorem}\label{thm:demseshsingular}
  Let\index{Seshadri constant, $\varepsilon(D;x)$!criterion for separation of jets|(}
  $X$ be a projective variety of dimension $n$ over a field $k$,
  and let $L$ be a line bundle on $X$.
  Let $x \in X$ be a $k$-rational point such that either $X$ has singularities
  of dense $F$-injective type\index{F-injective@$F$-injective!type} at $x$ in
  characteristic zero, or $X$ has $F$-injective\index{F-injective@$F$-injective}
  singularities at $x$ in positive characteristic.
  Suppose that for some integer $\ell \ge 0$, one of the following holds:
  \begin{enumerate}[label=$(\roman*)$,ref=\roman*]
    \item $L$ is nef and $\varepsilon(L;x) > n + \ell$;
      or\label{thm:demseshsingularample}
    \item $X$ is normal and $\varepsilon(\lVert L \rVert;x) > n + \ell$.%
      \label{thm:demseshsingularmoving}
  \end{enumerate}
  Then, the sheaf $\omega_X \otimes L$ separates $\ell$-jets at $x$.%
  \index{Seshadri constant, $\varepsilon(D;x)$!criterion for separation of jets|)}
\end{theorem}
We will first show the statement in positive characteristic, from which we will
deduce the characteristic zero case via reduction modulo $p$.
\subsection{Proof in positive characteristic}
We state the main technical result that will imply
\cref{thm:demseshsingular}.
\begin{theorem}\label{thm:murcimproved}
  Let $X$ be a projective variety of dimension $n$ over a field $k$ of
  characteristic $p > 0$, and let $L$ be a Cartier divisor on $X$.
  Consider a $k$-rational point $x \in X \smallsetminus \Bplus(L)$ and consider
  a coherent sheaf $\sF$ on $X$ together with a morphism $\tau\colon F_*^g\sF
  \to \sF$ that is surjective at $x$.
  If $\varepsilon_{\textup{jet}}(\lVert L \rVert;x) > n+\ell$ for an integer $\ell \ge 0$,
  then $\sF \otimes \cO_X(L)$ separates $\ell$-jets at $x$.
\end{theorem}
We note that a coherent sheaf $\sF$ on $X$ together with a morphism $\tau\colon
F_*^g\sF \to \sF$ is an example of a \textsl{Cartier module}%
\index{Cartier, Pierre!module|textbf} as defined in
\cite{BB11}. 
\begin{proof}
  We proceed in a sequence of steps, following the outline of the proof of
  \cite[Thm.\ 3.1]{MS14} and \cite[Thm.\ C]{Mur18}.
  \par We first claim that for every integer $t > 0$, there exist a positive
  integer $m_0$ and a sequence $\{d_e\}$ such that $\cO_X(m_0d_eL)$ separates
  $(\ell p^{ge}+n(p^{ge}-1)+t)$-jets at $x$ for all $e > 0$, and such that
  $p^{ge} - m_0d_e \to \infty$ as $e \to \infty$.
  Let $0 < \delta \ll 1$.
  By \cref{lem:ito37}, there exists an $m_0$ such that
  \[
    \frac{s(m_0L;x)}{m_0} > (1+\delta)(n+\ell).
  \]
  Now for every integer $e > 0$, let
  \[
    d_e = \biggl\lceil \frac{\ell p^{ge}+n(p^{ge}-1)+t}{s(m_0L;x)}
    \biggr\rceil.
  \]
  By the superadditivity property (\cref{lem:ito37}), we have
  \[
    s(m_0d_eL;x) \ge d_e \cdot s(m_0L;x) \ge \ell p^{ge}+n(p^{ge}-1)+t,
  \]
  hence $\cO_X(m_0d_eL)$ separates $(\ell p^{ge}+n(p^{ge}-1)+t)$-jets at $x$.
  We now claim that $p^{ge}-m_0d_e \to \infty$ as $e \to \infty$.
  Note that
  \begin{align*}
    p^{ge} - m_0d_e &= p^{ge} - m_0 \cdot \biggl\lceil \frac{\ell
    p^{ge}+n(p^{ge}-1)+t}{s(m_0H;x)} \biggr\rceil\\
    &\ge p^{ge} - \bigl(\ell p^{ge}+n(p^{ge}-1)+t\bigr) \cdot
    \frac{m_0}{s(m_0H;x)} - m_0\\
    &\ge p^{ge} - \bigl(\ell p^{ge}+n(p^{ge}-1)+t\bigr) \cdot
    \frac{1}{(1+\delta)(n+\ell)} - m_0
  \intertext{and as $e \to \infty$, we have}
    \lim_{e \to \infty} (p^{ge} - m_0d_e) &\ge \lim_{e \to \infty} \biggl( p^{ge}
    - (\ell p^{ge}+n(p^{ge}-1)+t) \cdot \frac{1}{(1+\delta)(n+\ell)} -
    m_0 \biggr)\\
    &= \lim_{e \to \infty} p^{ge}\biggl( 1 - \frac{1}{1+\delta} \biggr) - m_0 =
    \infty.
  \end{align*}
  We therefore see that $\cO_X(m_0d_eL)$ separates $(\ell
  p^{ge}+n(p^{ge}-1)+t)$-jets at $x$, and that $p^{ge}-m_0d_e \to \infty$ as $e
  \to \infty$.
  \par We now show that there exist a positive integer $e$ such that the
  restriction morphism
  \begin{equation}\label{eq:peljetscartierversion}
    H^0\bigl(X,\sF \otimes \cO_X(p^{ge}L)\bigr) \longrightarrow
    H^0\biggl(X,\frac{\sF \otimes
      \cO_X(p^{ge}L)}{(\fm_x^{\ell+1})^{[p^{ge}]}(\sF \otimes
    \cO_X(p^{ge}L))}\biggr)
  \end{equation}
  is surjective.
  By \cref{lem:mur23improved}, there exists an integer $t \ge 0$ such that
  \begin{equation}\label{eq:mur23improvedapp}
    \fm_x^{\ell p^{ge}+n(p^{ge}-1)+1+t} \subseteq (\fm_x^{\ell+1})^{[p^{ge}]},
  \end{equation}
  for all $e > 0$.
  Now let $m_0$ and $\{d_e\}$ as in the previous paragraph.
  Since $x \notin \Bplus(L)$ and since $p^{ge}-m_0d_e \to \infty$,
  \cref{prop:kur13prop27} implies $\sF \otimes \cO_X( (p^{ge}-m_0d_e)L)$ is
  globally generated at $x$ for some $e \gg 0$.
  Since $\cO_X(m_0d_eL)$ separates $(\ell p^{ge}+n(p^{ge}-1)+t)$-jets at $x$,
  \cref{lem:ito37} implies
  \[
    \sF \otimes \cO_X(p^{ge}L) \simeq \sF \otimes
    \cO_X\bigl((p^{ge}-m_0d_e)L\bigr) \otimes \cO_X(m_0d_eL)
  \]
  separates $(\ell p^{ge}+n(p^{ge}-1)+t)$-jets at $x$.
  The inclusion \cref{eq:mur23improvedapp} then implies the surjectivity of
  \cref{eq:peljetscartierversion}.
  \par We now use the $e$th iterate $\tau^e$ of the morphism $\tau$ defined as
  the composition
  \[
    F_*^{ge}\sF \xrightarrow{F_*^{g(e-1)}\tau} F_*^{g(e-1)}\sF
    \xrightarrow{F_*^{g(e-2)}\tau} \cdots
    \overset{\tau}{\longrightarrow} \sF
  \]
  to take out the factors of $p^{ge}$.
  Note that $\tau^e$ is surjective at $x$ by assumption, since the Frobenius and
  its iterates are affine morphisms.
  Twisting $\tau^e$ by $\cO_X(L)$, we have a morphism
  \begin{align*}
    F_*^{ge}(\sF \otimes \cO_X(p^{ge}L)) &\longrightarrow \sF \otimes \cO_X(L)
    \intertext{that is surjective at $x$,
    and by considering the $\cO_X$-module structure on $F_*^{ge}(\sF
    \otimes \cO_X(p^{ge}L))$, we obtain a morphism}
    F_*^{ge}\bigl((\fm_x^{\ell+1})^{[p^{ge}]}\bigl( \sF \otimes \cO_X(p^{ge}L)
    \bigr) \bigr) &\longrightarrow \fm_x^{\ell+1} \bigl( \sF \otimes
    \cO_X(L)\bigr).
  \end{align*}
  We therefore have the commutative diagram
  \[
    \begin{tikzcd}
      0\dar & 0\dar\\
      F_*^{ge}\bigl((\fm_x^{\ell+1})^{[p^{ge}]}\bigl( \sF \otimes \cO_X(p^{ge}L)\bigr)
      \bigr) \rar \dar & \fm_x^{\ell+1}\bigl(\sF \otimes \cO_X(L)\bigr)\dar\\
      F_*^{ge}\bigl(\sF \otimes \cO_X(p^{ge}L)\bigr) \rar \dar
      & \sF \otimes \cO_X(L) \dar\\
      F_*^{ge}\biggl(\dfrac{\sF \otimes \cO_X(p^{ge}L)}{(\fm_x^{\ell+1})^{[p^{ge}]}(\sF
      \otimes \cO_X(p^{ge}L))}\biggr) \rar \dar
      & \dfrac{\sF \otimes \cO_X(L)}{\fm_x^{\ell+1}(\sF
      \otimes \cO_X(L))} \dar\\
      0 & 0
    \end{tikzcd}
  \]
  where the horizontal arrows are induced by $\tau^e$, and are therefore
  surjective at $x$.
  Note that the left column is exact since the Frobenius morphism $F$ is affine.
  Taking global sections in the bottom square, we obtain the
  following commutative square:
  \[
    \begin{tikzcd}
      H^0\bigl(X,\sF \otimes \cO_X(p^{ge}L)\bigr) \rar
      \dar[twoheadrightarrow,swap]{\varphi}
      & H^0\bigl(X,\sF \otimes \cO_X(L)\bigr) \dar\\
      H^0\biggl(X,\dfrac{\sF \otimes
      \cO_X(p^{ge}L)}{(\fm_x^{\ell+1})^{[p^{ge}]}(\sF \otimes
      \cO_X(p^{ge}L))}\biggr) \rar[twoheadrightarrow]{\psi}
      & H^0\biggl(X,\dfrac{\sF \otimes \cO_X(L)}{\fm_x^{\ell+1}(\sF
      \otimes \cO_X(L))}\biggr)
    \end{tikzcd}
  \]
  where $\psi$ is surjective since the corresponding morphism of sheaves is
  a surjective morphism of skyscraper sheaves supported at $x$.
  Since the restriction map $\varphi$ is surjective by the previous paragraph,
  the right vertical map is necessarily surjective by the commutativity of the
  diagram.
  Thus, the sheaf $\sF \otimes \cO_X(L)$ indeed separates $\ell$-jets at
  $x$.
\end{proof}
To prove \cref{thm:demseshsingular}, we need the following elementary lemma.
\begin{lemma}\label{lem:jetseshfieldext}
  Let $X$ be a projective variety over a field $k$, let $\sF$ be a coherent
  sheaf on $X$, and let $x \in X$ be a $k$-rational point.
  Consider a field extension $k \subseteq k'$ such that $X \times_k k'$ is a
  variety, and such that denoting by
  \[
    \pi\colon X \times_k k' \longrightarrow X
  \]
  the first projection morphism, the inverse image $\pi^{-1}(x)$ of $x$ consists
  of a single $k'$-rational point.
  Then, for every integer $\ell\ge 0$, the sheaf $\sF$ separates
  $\ell$-jets at $x$ if and only if $\pi^*\sF$ separates $\ell$-jets at
  $\pi^{-1}(x)$.
  In particular, $\varepsilon_{\textup{jet}}(\lVert D \rVert;x) = \varepsilon_{\textup{jet}}(\lVert
  \pi^*D \rVert;\pi^{-1}(x))$ for every $\QQ$-Cartier divisor $D$ on $X$.
\end{lemma}
\begin{proof}
  The first statement follows from faithfully flat base change, which also
  implies $\varepsilon_{\textup{jet}}(\lVert D \rVert;x) = \varepsilon_{\textup{jet}}(\lVert
  \pi^*D \rVert;\pi^{-1}(x))$ for Cartier divisors.
  We then obtain the same equality for $\QQ$-Cartier divisors by homogeneity
  (\cref{prop:jetseshelem}\cref{prop:jetseshhomog}).
\end{proof}
We now prove \cref{thm:demseshsingular}.
\begin{proof}[Proof of \cref{thm:demseshsingular} in positive characteristic]
  In either setting, we note that
  \[
    \varepsilon_{\textup{jet}}\bigl(\lVert L \rVert;x\bigr) > n+\ell
  \]
  by \cref{prop:jetsep}, where for \cref{thm:demseshsingularample}, we note that
  $\varepsilon(L;x) > n+\ell$ implies $x \notin \Bplus(L)$ by
  \cref{thm:nakamaye} and \cite[Prop.\ 5.1.9]{Laz04a}.
  We therefore need to check that the rest of the hypotheses in
  \cref{thm:murcimproved} can be satisfied.
  \par We first claim that we can reduce to the case when $k$ is $F$-finite.
  By \cref{thm:gammaconst} applied simultaneously to $X$, $\Spec \cO_{X,x}$, and
  $\{x\}$, there exists
  a field extension $k \subseteq k^\Gamma$ such that denoting the projection
  morphism by $\pi^\Gamma \colon X^\Gamma \to X$, the scheme $X^\Gamma$ is a
  variety and $x^\Gamma \coloneqq (\pi^\Gamma)^{-1}(x)$ is a closed
  $k^\Gamma$-rational point such that $\cO_{X^\Gamma,x^\Gamma}$ is
  $F$-injective.
  The formation of $\omega_X$ is compatible with ground field extensions
  \cite[Cor.\ V.3.4$(a)$]{Har66} as is the nefness of $L$ \cite[Prop.\
  B.17]{Kle05}, and
  $\varepsilon_{\textup{jet}}(\lVert L \rVert;x)$ is invariant under the ground field
  extension $k \subseteq k^\Gamma$ by \cref{lem:jetseshfieldext}.
  Since the condition that $\omega_X \otimes \cO_X(L)$ separates $\ell$-jets can
  also be checked after base change to $k^\Gamma$ by \cref{lem:jetseshfieldext},
  it therefore suffices to consider the case when $k$ is $F$-finite.
  In this case, we can apply \cref{thm:murcimproved} for $\sF = \omega_X$,
  since the trace morphism $\Tr_X\colon F_*\omega_X \to \omega_X$ is surjective
  at $x$ by the $F$-injectivity of $\cO_{X,x}$ (\cref{lem:finjffin}).
\end{proof}
\begin{remark}\label{rem:frobseshconst}
  The original proof of \cref{thm:poscharseshsm} in \cite[Thm.\ A]{Mur18}
  inspired the proof of \cref{thm:murcimproved} given above.
  The idea is that the surjectivity of restriction maps of the form in
  \cref{eq:mur23improvedapp} can be detected by \textsl{Frobenius--Seshadri
  constants,} which are a positive-characteristic version of Seshadri
  constants introduced in \cite{MS14} and \cite{Mur18}.
  These constants are defined as follows: Let $L$ be a Cartier divisor on a
  complete variety $X$ over a field $k$, and let $x \in X$ be a $k$-rational
  point.
  Denote by $s_F^\ell(mL;x)$ the largest integer $e \ge 0$ such that
  the restriction map
  \[
    H^0\bigl(X,\cO_X(mL)\bigr) \longrightarrow H^0\bigl(X,\cO_X(mL) \otimes
    \cO_X/(\fm_x^{\ell+1})^{[p^e]}\bigr)
  \]
  is surjective.
  Then, the \textsl{$\ell$th Frobenius--Seshadri constant} of $L$ at $x$ is
  \[
    \varepsilon_F^\ell(L;x) \coloneqq \limsup_{m \to \infty}
    \frac{p^{s^\ell_F(mL;x)}-1}{m/(\ell+1)}.\glsadd{frobsesh}
  \]
  See \cites[\S2]{MS14}[\S2]{Mur18} for basic properties of these constants.
  In particular, lower bounds of the form $\varepsilon_F^\ell(L;x) > \ell + 1$
  imply $\omega_X \otimes L$ separates $\ell$-jets at $x$ \cite[Thm.\ C]{Mur18}.
  One can then deduce \cref{thm:demseshsingular} since the pigeonhole principle
  (\cref{lem:mur23improved}) implies
  \[
    \frac{\ell+1}{\ell+n} \cdot \varepsilon_{\textup{jet}}\bigl(\lVert L
    \rVert;x\bigr) \le \varepsilon^\ell_F(L;x) \le
    \varepsilon_{\textup{jet}}\bigl(\lVert L \rVert;x\bigr),
  \]
  where $n = \dim X$.
  See the proof of \cite[Prop.\ 2.9]{Mur18}.
\end{remark}
\subsection{Proof in characteristic zero}
\label{sect:proofofdemfinjtype}
To prove \cref{thm:demseshsingular} in characteristic zero, we fix the following
notation.
\begin{setup}\label{setup:demfinjtype}
  Let $X$ be a projective variety over a field $k$ of characteristic zero, and
  consider a $k$-rational point $x \in X$ with defining ideal $\fm_x \subseteq
  \cO_X$.
  We then have a commutative diagram
  \[
    \begin{tikzcd}
      \Spec k \rar[hook]{i_x}\arrow[equal]{dr} & X \dar{\pi}\\
      & \Spec k
    \end{tikzcd}
  \]
  where $i_x$ is the closed embedding corresponding to point $x$.
  By spreading out the entire diagram as in \cref{thm:spreadingout}, there
  exists a domain $A_\lambda \subseteq k$ that is of finite type over $\ZZ$ and
  a commutative diagram
  \[
    \begin{tikzcd}
      \Spec A_\lambda \rar[hook]{i_{x,\lambda}}\arrow[equal]{dr} & X_\lambda
      \dar{\pi_\lambda}\\
      & \Spec A_\lambda
    \end{tikzcd}
  \]
  that base changes to the commutative diagram above, where $\pi_\lambda$ is of
  finite type.
  After possibly enlarging $A_\lambda$ by inverting finitely many elements, and
  with notation as in \cref{def:reductionmodulop}, we can assume the following
  properties by \cref{table:spreadingout,table:reductionmodulop}:
  \begin{enumerate}[label=$(\alph*)$]
    \item $i_{x,\lambda}$ is a closed embedding;
    \item $\pi_\lambda$ is flat and projective; and
    \item $X_\fp$ is integral for every closed point $\fp \in \Spec A_\lambda$.
  \end{enumerate}
  We will denote the ideal sheaf defining the image of $i_{x,\lambda}$ as
  $\fm_{x_\lambda}$, and the corresponding subscheme by $x_\lambda$.
  By \cref{rem:spreadingoutothers}, we can also spread out coherent sheaves,
  Cartier divisors, and $\QQ$-Cartier divisors from $X$ to $X_\lambda$.
\end{setup}
We will also need the following result, which describes how separation of jets
and how $\varepsilon_{\textup{jet}}(\lVert D \rVert;x)$ behave under reduction
modulo $\fp$.
We note that in the description of different loci below, we allow $\fq$ to be
non-closed points in $\Spec A_\lambda$.
\begin{lemma}\label{lem:jetsepmodp}
  Let $X$ and $\fm_x$ be as in \cref{setup:demfinjtype}, with models
  $X_\lambda$ and $\fm_{x\lambda}$ over $\Spec A_\lambda$,
  respectively.
  \begin{enumerate}[label=$(\roman*)$,ref=\roman*]
    \item\label{lem:jetsepmodpsheaf}
      Let $\sF$ be a coherent sheaf on $X$ together with a model
      $\sF_\lambda$ over $\Spec A_\lambda$.
      Let $\ell \ge -1$ be an integer, and 
      suppose that $\sF_\lambda$ and $\sF_\lambda/
      \fm_{x_\lambda}^{\ell+1}\sF_\lambda$ are flat and cohomologically flat in
      degree zero over $\Spec A_\lambda$.
      Then, the locus
      \[
        \bigl\{ \fq \in \Spec A_\lambda \bigm\vert \sF_\fq\ \text{separates
        $\ell$-jets at $x_\fq$} \bigr\}
      \]
      is open in $\Spec A_\lambda$.
    \item\label{lem:jetsepmodpdiv}
      Let $D$ be a $\QQ$-Cartier divisor on $X$ together with a model
      $D_\lambda$ over $\Spec A_\lambda$.
      Then, for every integer $m > 0$ such that $mD_\lambda$ is Cartier, the
      locus
      \[
        \bigl\{ \fq \in \Spec A_\lambda \bigm\vert mD_\fq\ \text{separates
        $\ell$-jets at $x_\fq$} \bigr\}
      \]
      contains an open set in $\Spec A_\lambda$ for every integer $\ell \ge -1$.
    \item\label{lem:jetsepmodpsesh}
      Let $D$ be a $\QQ$-Cartier divisor on $X$ together with a model
      $D_\lambda$ over $\Spec A_\lambda$, and let $\delta > 0$ be a real number
      such that $\varepsilon_{\textup{jet}}( \lVert D \rVert;x) > \delta$.
      Then, the locus
      \begin{equation}\label{eq:jetsepseshlocus}
        \bigl\{ \fq \in \Spec A_\lambda \bigm\vert
          \varepsilon_{\textup{jet}}\bigl( \lVert D_\fq \rVert;x_\fq\bigr) >
        \delta \bigr\}
      \end{equation}
      contains a non-empty open set in $\Spec A_\lambda$.
  \end{enumerate}
\end{lemma}
\begin{proof}
  We first prove \cref{lem:jetsepmodpsheaf}.
  By cohomology and base change \cite[Cor.\ 8.3.11]{Ill05},
  the locus where $\sF_\fq$ does not separate $\ell$-jets at $x_\fq$ is
  \[
    \Supp\Bigl(\coker\bigl(
    \pi_{\lambda*}\sF_\lambda \longrightarrow
    \pi_{\lambda*}(\sF_\lambda/
    \fm_{x_\lambda}^{\ell+1}\sF_\lambda)
    \bigr)\Bigr),
  \]
  where $\pi_\lambda \colon X_\lambda \to \Spec A_\lambda$ is as in
  \cref{setup:demfinjtype}.
  Since $\pi_\lambda$ is proper, both direct image sheaves are coherent, and the
  cokernel above is also coherent.
  Thus, the support of this cokernel is closed in $\Spec A_\lambda$, which
  implies \cref{lem:jetsepmodpsheaf}.
  \cref{lem:jetsepmodpdiv} then follows from \cref{lem:jetsepmodpsheaf} by
  setting $\sF = \cO_X(mD)$,
  after possibly enlarging $A_\lambda$ by inverting finitely many elements to
  assume that $\cO_{X_\lambda}(mD_\lambda)$ and $\cO_{X_\lambda}(mD_\lambda)
  \otimes \cO_{X_\lambda}/\fm_{x_\lambda}^{\ell+1}$ are flat and cohomologically
  flat in degree zero over $\Spec A_\lambda$ by generic flatness \cite[Thm.\
  6.9.1]{EGAIV2} and by \cite[Cor.\ 8.3.11]{Ill05}.
  \par We now show \cref{lem:jetsepmodpsesh}.
  By \cref{lem:ito37}, $\varepsilon_{\textup{jet}}(\lVert D \rVert;x)$
  is a limit, hence there exists an integer $m > 0$ such that $mD$ is a
  Cartier divisor, and such that
  \begin{equation}\label{eq:smdgtdelta}
    \frac{s(mD;x)}{m} > \delta.
  \end{equation}
  Since $s(mD;x)$ is an integer, this inequality is equivalent to $s(mD;x) \ge
  \lfloor m\delta \rfloor + 1$.
  By \cref{lem:jetsepmodpdiv}, the locus
  \begin{equation}\label{eq:locuswfloors}
    \bigl\{ \fq \in \Spec A_\lambda \bigm\vert mD_\fq\ \text{separates
    $\bigl(\lfloor m\delta \rfloor+1\bigr)$-jets at $x_\fq$} \bigr\}
  \end{equation}
  contains an open set, which is nonempty by \cref{eq:smdgtdelta} since the
  generic point of $\Spec A_\lambda$ is contained in this set by flat base
  change, and since $A_\lambda \subseteq k$ is flat.
  Now if $mD_\fq$ separates $(\lfloor m\delta \rfloor+1)$-jets at $x_\fq$, then
  we have the inequality
  \[
    \varepsilon\bigl(\lVert D_\fq \rVert;x_\fq\bigr) \ge \frac{s(mD_\fq;x)}{m} >
    \delta
  \]
  by the fact that $s(mD_\fq;x)$ is an integer.
  The locus \cref{eq:locuswfloors} is therefore contained in the locus
  \cref{eq:jetsepseshlocus}, and \cref{lem:jetsepmodpsesh} follows.
\end{proof}
We will also use the following:
\begin{lemma}\label{lem:dualizingcomplexmodel}
  Let $X$ be as in \cref{setup:demfinjtype}, with a model $X_\lambda$ over
  $\Spec A_\lambda$.
  If $\omega^i_{X_\lambda/A_\lambda}$ is flat over $A_\lambda$ for every $i$,
  then the base change isomorphism
  \begin{align*}
    \omega_{X_\lambda/A_\lambda}^\bullet\rvert^{\LL}_{X_\fp} &\longisoto
    \omega_{X_\fp}^\bullet
    \intertext{from \emph{\cite[Cor.\ 4.4.3]{Lip09}} induces an isomorphism}
    \omega_{X_\lambda/A_\lambda}\rvert_{X_\fp} &\longisoto \omega_{X_\fp}.
  \end{align*}
\end{lemma}
\begin{proof}
  This statement follows from the flatness of $\omega^i_{X_\lambda/A_\lambda}$
  and a spectral sequence for sheaf $\TTor$; see \cite[Cor.\ 6.5.9]{EGAIII2}.
\end{proof}
We can now prove \cref{thm:demseshsingular} in characteristic zero via reduction
modulo $\fp$.
\begin{proof}[Proof of \cref{thm:demseshsingular} in characteristic zero]
  As in the proof in positive characteristic, we note that
  \[
    \varepsilon_{\textup{jet}}\bigl(\lVert L \rVert;x\bigr) > n+\ell
  \]
  by \cref{prop:jetsep}, where for \cref{thm:demseshsingularample}, we note that
  $\varepsilon(L;x) > n+\ell$ implies $x \notin \Bplus(L)$ by
  \cref{thm:nakamaye} and \cite[Prop.\ 5.1.9]{Laz04a}.
  We use the notation in \cref{setup:demfinjtype}.
  After possibly further enlarging $A_\lambda$ by inverting finitely many
  elements, we may assume that
  \begin{enumerate}[label=$(\alph*)$,ref=\alph*]
    \item the Cartier divisor $L$ spreads out to a
      Cartier divisor $L_\lambda$ on $X_\lambda$
      such that $\cO_{X_\lambda}(L_\lambda)$ is flat and cohomologically flat in
      degree zero over $A_\lambda$;
    \item\label{item:everythingflat} the sheaves
      $\omega^i_{X_\lambda/A_\lambda}$ and
      $\omega_{X_\lambda/A_\lambda}$
      are flat and cohomologically
      flat in degree zero over $A_\lambda$ for every $i$; and
    \item the sheaf $\omega_{X_\lambda/A_\lambda} \otimes
      \cO_{X_\lambda}(L_\lambda) \otimes
      \cO_{X_\lambda}/\fm_{x_\lambda}^{\ell+1}$ is flat and cohomologically flat
      in degree zero over $A_\lambda$.
  \end{enumerate}
  Here, we have used \cref{rem:spreadingoutothers}, generic flatness \cite[Thm.\
  6.9.1]{EGAIV2}, and cohomology and base change \cite[Cor.\ 8.3.11]{Ill05}.
  \par We can now prove \cref{thm:demseshsingular} in characteristic zero.
  First, we have a base change isomorphism
  \[
    \omega_{X_\lambda/A_\lambda}\rvert_{X_\fp} \longisoto \omega_{X_\fp}
  \]
  by \cref{lem:dualizingcomplexmodel} and the assumption
  \cref{item:everythingflat}.
  Since $\cO_X(L_\lambda)$ is also flat and cohomologically flat in degree zero
  over $\Spec A_\lambda$, the sheaf
  $\omega_{X_\lambda/A_\lambda} \otimes \cO_{X_\lambda}(L_\lambda)$ is a model
  of $\omega_X \otimes \cO_X(L)$ over $\Spec A_\lambda$, and is flat and
  cohomologically flat in degree zero over $\Spec A_\lambda$.
  \par We now claim that $\omega_{X_\fp} \otimes
  \cO_{X_\fp}(L_\fp)$ separates $\ell$-jets at $x_\fp$ for some closed point
  $\fp \in \Spec A_\lambda$.
  Note that $\varepsilon_{\textup{jet}}(\lVert
  L_\fp \rVert;x_\fp) > n + \ell$ holds for all $\fp$ in an open dense subset of
  $\Spec A_\lambda$ by \cref{lem:jetsepmodp}\cref{lem:jetsepmodpsesh}.
  We now claim we can apply \cref{thm:murcimproved}
  to show that $\omega_{X_\fp}
  \otimes \cO_{X_\fp}(L_\fp)$ separates $\ell$-jets at $x_\fp$ for all $\fp$
  such that $\varepsilon_{\textup{jet}}(\lVert L_\fp \rVert;x_\fp) > n + \ell$
  and such that $X_\fp$ is $F$-injective.
  First, we note that $\kappa(\fp)$ is $F$-finite for all $\fp \in \Spec
  A_\lambda$, since it is a finite field by \cref{lem:reductionmodulopischarp}.
  Thus, \cref{lem:finjffin} implies that the Frobenius trace
  $\Tr_{X_\fp}\colon F_*\omega_{X_\fp} \to \omega_{X_\fp}$ on $X_\fp$ exists and
  is surjective.
  We can therefore apply \cref{thm:murcimproved} to show that $\omega_{X_\fp}
  \otimes \cO_{X_\fp}(L_\fp)$ separates $\ell$-jets at $x_\fp$.
  \par Finally, we show that $\omega_X \otimes \cO_X(L)$ separates $\ell$-jets
  at $x$.
  Since the extension $\Frac A_\lambda \subseteq k$ is flat, it suffices by
  flat base change to show that $\omega_{X_\eta} \otimes \cO_{X_\eta}(L_\eta)$
  separates $\ell$-jets at $x_\eta$, where $\eta \in \Spec A_\lambda$ is the
  generic point.
  But this follows from the previous paragraph from
  \cref{lem:jetsepmodp}\cref{lem:jetsepmodpsheaf}, since $\omega_{X_\fp} \otimes
  \cO_{X_\fp}(L_\fp)$ separates $\ell$-jets at $x_\fp$ for some closed point
  $\fp \in \Spec A_\lambda$.
\end{proof}

\chapter{The Angehrn--Siu theorem}\label{ch:angehrnsiu}
The goal of this chapter is to prove the following version of the Angehrn--Siu
theorem \cite[Thm.\ 0.1]{AS95}.
\begin{customthm}{D}\label{thm:myangehrnsiu}
  Let $(X,\Delta)$ be an effective log pair, where $X$ is a normal
  projective variety over an algebraically closed field $k$ of characteristic
  zero, $\Delta$ is a $\QQ$-Weil divisor, and $K_X+\Delta$ is
  $\QQ$-Cartier.
  Let $x \in X$ be a closed
  point such that $(X,\Delta)$ is klt at $x$, and
  let $D$ be a Cartier divisor on $X$ such that setting $H \coloneqq D -
  (K_X+\Delta)$, there exist positive numbers $c(m)$ with the following
  properties:
  \begin{enumerate}[label=$(\roman*)$,ref=\roman*]
    \item For every positive dimensional variety $Z \subseteq X$ containing $x$,
      we have
      \[
        \vol_{X\mid Z}(H) > c(\dim Z)^{\dim Z}.
      \]
    \item The numbers $c(m)$ satisfy the inequality
      \[
        \sum_{m=1}^{\dim X} \frac{m}{c(m)} \le 1.
      \]
  \end{enumerate}
  Then, $\cO_X(D)$ has a global section not vanishing at $x$.
\end{customthm}
Before we prove this statement, we will need to prove a replacement for the
Nadel vanishing theorem \cite[Thm.\ 9.4.17]{Laz04b}.
\section{The lifting theorem}
A major obstacle in proving the Angehrn--Siu theorem in positive characteristic
is that Kodaira-type vanishing theorems are false; see \cref{ex:raynaud}.
While the result below is not yet strong enough to prove their theorem in
positive characteristic, it does give a replacement for the Nadel vanishing
theorem in characteristic zero, after reduction modulo $\fp$.
\par We start by stating the characteristic zero version of the result.
\begin{customthm}{C}\label{thm:gettingsectionscharzero}
  Let $(X,\Delta)$ be an effective log pair such that $X$ is a projective normal
  variety over a field $k$ of characteristic zero, and such that
  $K_X+\Delta$ is $\QQ$-Cartier.
  Consider a $k$-rational point $x \in X$ such that $(X,\Delta)$ is of dense
  $F$-pure type at $x$.
  Suppose that $D$ is a Cartier divisor on $X$ such that $H = D - (K_X+\Delta)$
  satisfies
  \[
    \varepsilon\bigl(\lVert H \rVert;x\bigr) >
    \lct_x\bigl((X,\Delta);\fm_x\bigr).
  \]
  Then, $\cO_X(D)$ has a global section not vanishing at $x$.
\end{customthm}
Note that in \cref{s:intro}, we stated \cref{thm:gettingsectionscharzero} with
``dense $F$-pure type'' replaced by ``klt.''
The formulation in \cref{s:intro} follows from this one since klt pairs
are of dense strongly $F$-regular type by \cref{thm:takagiredmodp}, hence of
dense $F$-pure type by \cref{prop:fsingbasic}\cref{prop:sfrimpliesfpure}.
\par \cref{thm:gettingsectionscharzero} follows from the following result via
reduction modulo $\fp$.
\begin{theorem}\label{thm:gettingsections}
  Let $(X,\Delta)$ be an effective log pair such that $X$ is a projective normal
  variety over an $F$-finite field $k$ of characteristic $p > 0$, and such that
  $K_X+\Delta$ is $\QQ$-Cartier.
  Consider a $k$-rational point $x \in X$ such that $(X,\Delta)$ is
  $F$-pure at $x$.
  Suppose that $D$ is a Cartier divisor on $X$ such that $H = D - (K_X+\Delta)$
  satisfies
  \[
    \varepsilon_{\textup{jet}}\bigl(\lVert H \rVert;x\bigr) >
    \fpt_x\bigl((X,\Delta);\fm_x\bigr).
  \]
  Then, $\cO_X(D)$ has a global section not vanishing at $x$.
\end{theorem}
We first prove \cref{thm:gettingsectionscharzero}, assuming
\cref{thm:gettingsections}.
\begin{proof}[Proof of \cref{thm:gettingsectionscharzero}]
  Let $X_\lambda$, $\Delta_\lambda$, $\fm_{x\lambda}$, and $D_\lambda$ be models
  of $X$, $\Delta$, $\fm_x$, and $D$ over a finitely generated $\ZZ$-algebra
  $A_\lambda \subseteq k$ as in \cref{thm:spreadingout,rem:spreadingoutothers};
  cf.\ \cref{setup:demfinjtype}.
  After possibly enlarging $A_\lambda$ by inverting finitely many elements, we
  may assume that $X_\fp$ is normal for every $\fp \in A_\lambda$ (by
  \cref{table:reductionmodulop}), and moreover, we may assume that
  $\omega_{X_\lambda/A_\lambda}$ is a model for $\omega_{X} \simeq \cO_X(K_X)$
  (by \cref{lem:dualizingcomplexmodel}) that is flat and cohomologically flat
  over $\Spec A_\lambda$ (by generic flatness \cite[Thm.\
  6.9.1]{EGAIV2} and cohomology and base change \cite[Cor.\ 8.3.11]{Ill05}).
  By assumption, $(X_\fp,\Delta_\fp)$ is $F$-pure for a
  dense set of $\fp \in \Spec A_\lambda$, and
  we can also assume that for these $\fp$, the $F$-pure threshold of
  $(X_\fp,\Delta_\fp)$ with respect to $\fm_{x_\fp}$ is strictly less than
  $\varepsilon_{\textup{jet}}(\lVert H \rVert;x)$ by \cref{thm:takagiredmodp}.
  Note that here we have used \cref{prop:jetsep} to say that
  $\varepsilon(\lVert H \rVert;x) = \varepsilon_{\textup{jet}}(\lVert H
  \rVert;x)$.
  By \cref{lem:jetsepmodp}\cref{lem:jetsepmodpsesh}, we have
  \[
    \varepsilon_{\textup{jet}}\bigl(\lVert H_\fp \rVert;x_\fp\bigr) >
    \fpt_x\bigl((X_\fp,\Delta_\fp);\fm_{x_\fp}\bigr)
  \]
  for all but finitely many $\fp \in \Spec A_\lambda$, where $H_\fp = D_\fp -
  (K_{X_\fp} + \Delta_\fp)$.
  \cref{thm:gettingsections} therefore implies that $\cO_X(D_\fp)$ has a global
  section not vanishing at $x_\fp$ after reduction modulo $\fp$ for infinitely
  many $\fp \in \Spec A$, where
  we note that $\kappa(\fp)$ is $F$-finite for every $\fp \in \Spec A_\lambda$,
  since $\kappa(\fp)$ is a finite field by \cref{lem:reductionmodulopischarp}.
  Thus, $\cO_{X_\eta}(D_\eta)$ has a global section not vanishing at $x_\eta$ by
  \cref{lem:jetsepmodp}\cref{lem:jetsepmodpsheaf}, where $\eta \in \Spec
  A_\lambda$ is the generic point.
  Finally, since $\Frac A_\lambda \subseteq k$ is flat, we see that $\cO_X(D)$
  also has a global section not vanishing at $x$ by flat base change.
\end{proof}
We now prove \cref{thm:gettingsections}.
\begin{proof}[Proof of \cref{thm:gettingsections}]
  We proceed in a sequence of steps, following the outline of the proof of
  \cite[Thm.\ 3.1]{MS14} and \cite[Thm.\ C]{Mur18}.
  \par Denoting by $c$ the $F$-pure threshold $\fpt_x((X,\Delta);\fm_x)$ of
  $(X,\Delta)$ with respect to $\fm_x$,
  fix $\delta > 0$ such that
  \begin{align*}
    \varepsilon_{\textup{jet}}(\lVert H \rVert;x) &>
    (1+2\delta)c.
    \intertext{We first claim that there exists a positive
    integer $m_0$ and a sequence $\{d_e\}$ such that $m_0H$ is Cartier,
    $\cO_X(m_0d_eH)$ separates
    $(\lfloor (p^e-1)(1+\delta)c \rfloor-1)$-jets at $x$
    for all $e > 0$, and $p^e - m_0d_e \to \infty$ as $e \to \infty$.
    By \cref{lem:ito37}, there exists a positive integer $m_0$ such that $m_0H$
    is Cartier and}
    \frac{s(m_0H;x)}{m_0} &> (1+2\delta)c.
  \end{align*}
  Now for every integer $e > 0$, let
  \[
    d_e = \biggl\lceil \frac{\lfloor (p^e-1)(1+\delta)c
    \rfloor-1}{s(m_0H;x)} \biggr\rceil.
  \]
  By the superadditivity property (\cref{lem:ito37}), we have
  \[
    s(m_0d_eH;x) \ge d_e \cdot s(m_0H;x) \ge \lfloor
    (p^e-1)(1+\delta)c \rfloor-1,
  \]
  hence $\cO_X(m_0d_eH)$ separates $(\lfloor
  (p^e-1)(1+\delta)c \rfloor-1)$-jets at $x$.
  We now claim that $p^{e}-m_0d_e \to \infty$ as $e \to \infty$.
  Note that
  \begin{align*}
    p^{e} - m_0d_e &= p^{e} - m_0 \cdot \biggl\lceil \frac{\lfloor
    (p^e-1)(1+\delta)c
    \rfloor-1}{s(m_0H;x)} \biggr\rceil\\
    &\ge p^{e} - \bigl(\lfloor
    (p^e-1)(1+\delta)c \rfloor-1\bigr)
    \cdot \frac{m_0}{s(m_0H;x)} - m_0\\
    &\ge p^{e} - \frac{\lfloor
    (p^e-1)(1+\delta)c
    \rfloor-1}{(1+2\delta)c} - m_0
  \intertext{and as $e \to \infty$, we have}
    \lim_{e \to \infty} (p^{e} - m_0d_e) &\ge \lim_{e \to \infty} \biggl(
    p^{e} - \frac{\lfloor (p^e-1)(1+\delta)c
    \rfloor-1}{(1+2\delta)c} - m_0\biggr)\\
    &= \lim_{e \to \infty} p^{e}\biggl( 1 - \frac{1+\delta}{1+2\delta} \biggr) -
    m_0 = \infty.
  \end{align*}
  We therefore see that $\cO_X(m_0d_eH)$ separates $(\lfloor
  (p^e-1)(1+\delta) c \rfloor-1)$-jets at $x$, and that
  $p^{e}-m_0d_e \to \infty$ as $e \to \infty$.
  \par We now show that there exists a positive integer $e$ such that
  $\cO_X(\lceil K_X+\Delta + p^eH\rceil)$ separates $(\lfloor
  (p^e-1)(1+\delta)c \rfloor-1)$-jets at $x$.
  Let $m_0$ and $\{d_e\}$ as in the previous paragraph.
  Since $x \notin \Bplus(H)$ and since $p^e-m_0d_e \to \infty$,
  \cref{prop:kur13prop27} implies that the sheaf
  \[
    \cO_X\bigl( \lceil K_X+\Delta + (p^e-m_0d_e) H \rceil \bigr)
  \]
  is globally generated at $x$ for all $e \gg 0$.
  Since $\cO_X(m_0d_eH)$ separates $(\lfloor
  (p^e-1)(1+\delta) c \rfloor-1)$-jets at $x$, \cref{lem:ito37} implies
  \[
    \cO_X\bigl( \lceil K_X+\Delta + (p^e-m_0d_e) H \rceil \bigr) \otimes
    \cO_X(m_0d_eH) \simeq 
    \cO_X\bigl( \lceil K_X+\Delta + p^eH \rceil \bigr)
  \]
  separates $(\lfloor
  (p^e-1)(1+\delta)c \rfloor-1)$-jets at $x$.
  \par We now use the trace morphism $\Tr^e_{X,\lfloor(p^e-1)\Delta\rfloor}$
  to take out the factors of $p^{e}$.
  Note that $\Tr^e_{X,\lfloor(p^e-1)\Delta\rfloor}$ is surjective at $x$ by
  assumption, since $(X,\Delta)$ is $F$-pure at $x$
  (\cref{cor:fsingstrace}\cref{cor:fsingstracefpure}).
  Twisting $\Tr^e_{X,\lfloor(p^e-1)\Delta\rfloor}$ by $\cO_X(D)$, we have a
  morphism
  \begin{equation}\label{eq:tracemorphismfpt}
    F^e_*\bigl(\cO_X\bigl((1-p^e)K_X-\lfloor
    (p^e-1)\Delta\rfloor+p^eD\bigr)\bigr) \xrightarrow{\Tr_{X,\lfloor
    (p^e-1)\Delta\rfloor}^e(D)} \cO_X(D)
  \end{equation}
  that is surjective at $x$, where the source can be identified with
  \[
    F^e_*\bigl(\cO_X\bigl(\lceil(1-p^e)(K_X+\Delta) + p^eD \rceil\bigr)\bigr)
    = F^e_*\bigl(\cO_X\bigl(\lceil K_X+\Delta + p^eH \rceil\bigr)\bigr).
  \]
  The triple $(X,\Delta,\fm_x^{(1+\delta)c})$ is not $F$-pure,
  since $c$ is the $F$-pure threshold $\fpt_x((X,\Delta);\fm_x)$.
  Thus, the morphism
  \begin{align}\label{eq:tracemorphismfptwincl}
    F^e_*\bigl(\fm_x^{\lfloor (p^e-1)(1+\delta)c
    \rfloor}\cdot\cO_X\bigl(\lceil K_X+\Delta + p^eH \rceil\bigr)\bigr)
    &\xrightarrow{\Tr_{X,\lfloor
    (p^e-1)\Delta\rfloor}^e(D)} \cO_X(D)
  \intertext{induced by the trace morphism \cref{eq:tracemorphismfpt} is not surjective at
  $x$ by \cref{cor:fsingstrace}.
  We therefore see that the morphism \cref{eq:tracemorphismfptwincl} induces a
  morphism}
    F^e_*\bigl(\fm_x^{\lfloor (p^e-1)(1+\delta)c
    \rfloor}\cdot\cO_X\bigl(\lceil K_X+\Delta + p^eH \rceil\bigr)\bigr)
    &\xrightarrow{\Tr_{X,\lfloor (p^e-1)\Delta\rfloor}^e(D)} \fm_x\cdot\cO_X(D)
    \nonumber
  \end{align}
  since the target is $\cO_{X,y}(D)$ after localizing at every point $y \ne x$,
  and at $x$, the non-surjectivity of the localization of
  \cref{eq:tracemorphismfptwincl} at $x$ is
  equivalent to having image in $\fm_x\cdot\cO_{X,x}(D)$, by the fact that
  $\cO_{X,x}$ is local.
  We therefore have the commutative diagram
  \[
    \begin{tikzcd}
      0\dar & 0\dar\\
      F^e_*\bigl(\fm_x^{\lfloor (p^e-1)(1+\delta)c
    \rfloor}\cdot\cO_X\bigl(\lceil K_X+\Delta + p^eH
      \rceil\bigr)\bigr)
      \rar \dar & \fm_x\cdot\cO_X(D)\dar\\
      F^e_*\bigl(\cO_X\bigl(\lceil K_X+\Delta + p^eH
      \rceil\bigr)\bigr) \rar \dar
      & \cO_X(D) \dar\\
      F_*^{e}\Biggl(
        \dfrac{\cO_X\bigl( \lceil K_X+\Delta+p^eH \rceil\bigr)}{\fm_x^{\lfloor
        (p^e-1)(1+\delta)c \rfloor} \cdot \cO_X\bigl(
        \lceil K_X+\Delta+p^eH \rceil\bigr)}
      \Biggr) \rar \dar
      & \dfrac{\cO_X(D)}{\fm_x \cdot \cO_X(D)} \dar\\
      0 & 0
    \end{tikzcd}
  \]
  where the bottom two horizontal arrows are induced by
  $\Tr^e_{X,\lfloor(p^e-1)\Delta\rfloor}$, and are therefore surjective at $x$.
  Note that the left column is exact since the Frobenius morphism $F$ is affine.
  Taking global sections in the bottom square, we obtain the
  following commutative square:
  \[
    \begin{tikzcd}
      H^0\bigl(X,\cO_X\bigl(\lceil K_X+\Delta + p^eH
      \rceil\bigr)\bigr) \rar
      \dar[twoheadrightarrow,swap]{\varphi}
      & H^0\bigl(X,\cO_X(D)\bigr) \dar\\
      H^0\Biggl(X,
        \dfrac{\cO_X\bigl( \lceil K_X+\Delta+p^eH \rceil\bigr)}{\fm_x^{\lfloor
        (p^e-1)(1+\delta)c \rfloor} \cdot \cO_X\bigl(
        \lceil K_X+\Delta+p^eH \rceil\bigr)}
      \Biggr) \rar[twoheadrightarrow]{\psi}
      & H^0\biggl(X,\dfrac{\cO_X(D)}{\fm_x \cdot \cO_X(D)}\biggr)
    \end{tikzcd}
  \]
  where $\psi$ is surjective since the corresponding morphism of sheaves is
  a surjective morphism of skyscraper sheaves supported at $x$.
  Since the restriction map $\varphi$ is surjective by the previous paragraph,
  the right vertical map is necessarily surjective by the commutativity of the
  diagram.
  Thus, the sheaf $\cO_X(D)$ has a global section not vanishing at $x$.
\end{proof}
\begin{remark}
  One can also prove a weaker version of \cref{thm:gettingsections}
  using another variant of Frobenius--Seshadri constants (cf.\
  \cref{rem:frobseshconst}).
  The relevant version of the Seshadri constant is defined using the
  \textsl{Frobenius degeneracy ideals} first introduced by Yao \cite[Rem.\
  2.3(1)]{Yao06} and Aberbach--Enescu \cite[Def.\ 3.1]{AE05}.
  If $(R,\Delta)$ is a sharply $F$-pure pair where $R$ is an $F$-finite local
  ring of characteristic $p > 0$ with maximal ideal $\fm \subseteq R$, then
  following \cites[Def.\ 4.3]{Tuc12}[Def.\ 3.3]{BST12}, the
  \textsl{$e$th Frobenius degeneracy ideal} is
  \[
    I_e^\Delta(\fm) \coloneqq \Bigl\{ f \in R \Bigm\vert \varphi(f) \in \fm\
    \text{for all}\ \varphi \in \Hom_R\bigl(F^e_*R(\lceil (p^e-1)\Delta
    \rceil),R\bigr) \Bigr\}.\glsadd{frobdegen}
  \]
  Note that we have followed the terminology from \cite[Def.\ 2.6]{CRST}.
  Following \cites[Lem.\ 3.9 and Prop.\ 3.10]{DSNB18}, one can show that
  \begin{equation}\label{eq:iedeltaincl}
    \fm^{\lfloor p^e\fpt( (R,\Delta);\fm) \rfloor + 1} \subseteq
    I_e^\Delta(\fm).
  \end{equation}
  Now let $(X,\Delta)$ be a sharply $F$-pure pair, where $X$ is a complete
  variety over an $F$-finite field of characteristic $p > 0$.
  For a $\QQ$-Cartier divisor $H$ on $X$ and for every integer $m \ge 1$ such
  that $mH$ is a Cartier divisor, denote by $s^\Delta_{\textup{$F$-sig}}(mH;x)$
  the largest integer such that the restriction map
  \[
    H^0\bigl(X,\cO_X(mH)\bigr) \longrightarrow
    H^0\bigl(X,\cO_X(mH) \otimes \cO_X/I_e^\Delta(\fm_x)\bigr)
  \]
  is surjective.
  Then, the \textsl{$F$-signature Seshadri constant} of $H$ at $x$ is
  \[
    \varepsilon^\Delta_{\textup{$F$-sig}}(H;x) \coloneqq \limsup_{m \to \infty}
    \frac{p^{s^\Delta_{\textup{$F$-sig}}(mH;x)} - 1}{m},\glsadd{fsigfrobsesh}
  \]
  where the limit supremum runs over all $m$ such that $mH$ is integral.
  The inclusion \cref{eq:iedeltaincl} then implies
  \[
    \varepsilon_{\textup{jet}}\bigl(\lVert H \rVert;x\bigr) \le
    \fpt_x\bigl((X,\Delta);\fm_x\bigr) \cdot
    \varepsilon^\Delta_{\textup{$F$-sig}}(H;x).
  \]
  Using the strategy in \cref{thm:gettingsections}, one can show that a lower
  bound of the form $\varepsilon^\Delta_{\textup{$F$-sig}}(H;x) > 1$ implies the
  existence of global sections of $\cO_X(D)$ as in \cref{thm:gettingsections}.
  This version of the Frobenius--Seshadri constant is difficult to work with
  since we do not know if the analogues of \cite[Lem.\ 2.5 and Prop.\ 2.6]{MS14}
  or \cite[Lem.\ 2.4 and Prop.\ 2.5]{Mur18} hold.
  The core issue is that the sequence $\{I_e^\Delta(\fm_x)\}_{e \in \NN}$ does
  not necessarily form a $p$-family of ideals in the sense of \cite[Def.\
  1.1]{HJ18}.
\end{remark}
\section{Constructing singular divisors and proof of Theorem
\ref{thm:myangehrnsiu}}\label{sect:constructsingulardiv}
The goal in this section is to prove the following result, which will be the
other crucial ingredient
in proving \cref{thm:myangehrnsiu}.
\begin{theorem}[cf.\ {\cite[Thm.\ 6.4]{Kol97}}]\label{thm:kol64}
  Let $(X,\Delta)$ be an effective log pair, where $X$ is a normal projective
  variety over an algebraically closed field $k$ of characteristic zero,
  $\Delta$ is a $\QQ$-Weil divisor, and $K_X+\Delta$ is $\QQ$-Cartier.
  Let $x \in X$ be a
  closed point such that $(X,\Delta)$ is klt at $x$.
  Let $D$ be a Cartier divisor on $X$ such that setting $N \coloneqq D -
  (K_X+\Delta)$, there exist positive numbers $c(m)$ with the following
  properties:
  \begin{enumerate}[label=$(\roman*)$,ref=\roman*]
    \item For every positive dimensional variety $Z \subseteq X$ containing $x$,
      we have
      \[
        \vol_{X\mid Z}(N) > c(\dim Z)^{\dim Z}.
      \]
    \item The numbers $c(m)$ satisfy the inequality
      \[
        \sum_{m=1}^{\dim X} \frac{m}{c(m)} \le 1.
      \]
  \end{enumerate}
  Then, there exist an effective $\QQ$-Cartier $\QQ$-Weil divisor $E \sim_\QQ
  bN$ for some $b \in (0,1)$,
  and an open neighborhood $X^0 \subseteq X$ of $x$ such that
  $(X^0,\Delta+E)$ is log canonical, $(X^0,\Delta+E)$ is klt on $X^0
  \smallsetminus \{x\}$, and $(X,\Delta+E)$ is not klt at $x$.
\end{theorem}
Assuming this, \cref{thm:myangehrnsiu} is not difficult.
\begin{proof}[Proof of \cref{thm:myangehrnsiu}]
  Fix a closed point $x \in X$.
  By \cref{thm:kol64}, there exists a boundary divisor $E \sim_\QQ
  bN$ for some $b \in (0,1)$ such that $(X,\Delta+E)$ is strictly log canonical
  at $x$, but is klt in a punctured neighborhood of $x$.
  Now let $f \colon Y \to X$ be a log resolution for $(X,\Delta+E,\fm_x)$.
  Then, there is divisor $F \subseteq Y$ over $x$ such that $a(F,X,\Delta+E) =
  -1$ by \cref{thm:singpairsvialogres}.
  For $0 < \delta \ll 1$, we claim that $(X,\Delta+(1-\delta)E)$ is klt in a
  neighborhood of $x$, and that
  \[
    \lct_x\bigl( (X,\Delta+(1-\delta)E);\fm_x\bigr) < (1-b)\cdot
    \varepsilon\bigl(\lVert N \rVert;x\bigr).
  \]
  Note that the right-hand side is positive since $b \in (0,1)$.
  The property that $(X,\Delta+(1-\delta)E)$ is klt in a neighborhood of $x$
  follows from \cref{thm:singpairsvialogres}, and the inequality above follows
  for $0 < \delta \ll 1$ from the computation of the log canonical threshold in
  \cref{prop:lctvialogres}, since the Cartier divisor defined by $f^{-1}\fm_x
  \cdot \cO_Y$ contains $F$ as a component.
  We therefore have
  \[
    D - \bigl(K_X+\Delta+(1-\delta)E\bigr) \sim_\QQ N - (1-\delta)bN =
    \bigl(1-(1-\delta)b\bigr)N,
  \]
  hence the conditions of \cref{thm:gettingsectionscharzero} are satisfied for
  $H = D - (K_X+\Delta+(1-\delta)E)$, since
  \begin{align*}
    \varepsilon\bigl(\lVert H \rVert;x\bigr) = \bigl(1-(1-\delta)b\bigr) \cdot
    \varepsilon\bigl(\lVert N \rVert;x\bigr) &> \bigl(1-b\bigr) \cdot
    \varepsilon\bigl(\lVert N \rVert;x\bigr)\\
    &> \lct_x\bigl( (X,\Delta+(1-\delta)E);\fm_x\bigr).\qedhere
  \end{align*}
\end{proof}
It therefore remains to show \cref{thm:kol64}.
The idea of the proof is to first produce a divisor that is
highly singular at a point, and then cut down the dimension of the non-klt locus
at the point until the non-klt locus is isolated.
We mostly follow the proofs in \cite[\S6]{Kol97} and \cite[\S3]{Liu}, with
suitable changes to deal with the weaker positivity condition on $N$.
\par We start with the following result.
Recall that if $D$ is a $\QQ$-Cartier divisor on a variety $X$ over a
field $k$, then a \textsl{graded linear system} \gls*{gradedlinearsystem}
associated to
$D$ is a sequence of subspaces $V_m \subseteq H^0(X,\cO_X(mD))$ for $m$ such
that $mD$ is a Cartier divisor, which satisfies the property that the
multiplication map $V_m \otimes V_n \to H^0(X,\cO_X( (m+n)D))$ has
image in $V_{m+n}$ \cite[Def.\ 2.4.1]{Laz04a}.
The \textsl{volume} of $V_\bullet$ is
\[
  \gls*{volvbullet} \coloneqq \limsup_{m \to \infty} \frac{\dim_k V_m}{m^n/n!},
\]
where $n = \dim X$ \cite[Def.\ 2.4.12]{Laz04a}.
If $f\colon Y \to X$ is a morphism, and $V_\bullet$ is a graded linear
system associated to a $\QQ$-Cartier $D$ on $X$, then the graded linear series
$f^*V_\bullet$ is given by setting
\[
  f^*V_m \coloneqq \im\Bigl( V_m \subseteq
  H^0\bigl(X,\cO_X(mD)\bigr) \longrightarrow
  H^0\bigl(Y,\cO_Y(m\,f^*D)\bigr) \Bigr),
\]
where the morphism is induced by the pullback morphism $\cO_X(mD) \to
f_*\cO_Y(m\,f^*D)$.
In particular, if $f$ is an inclusion $Y \subseteq X$ of a closed subvariety,
then we set $V_\bullet\rvert_Y \coloneqq f^*V_\bullet$.
\begin{lemma}[cf.\ {\cites[Lem.\ 6.1]{Kol97}[Lem.\ 12.2]{Fuj11}}]
  \label{lem:kol61}
  Let $f\colon Y \to Z$ be a surjective projective morphism from a normal
  variety $Y$ to an affine variety $Z$ over an algebraically closed field
  $k$, and
  let $W$ be a general closed subvariety of $Y$ such that $f\rvert_W\colon W \to
  Z$ is generically finite and generically regular.
  Consider a $\QQ$-Cartier divisor $M$ on $Y$, and let $V_\bullet$ be a graded
  linear system associated to $M$.
  Then, for every $\varepsilon > 0$, there exists a positive integer $t > 0$
  such that $tM$ is a Cartier divisor, and such that there exists an effective
  Cartier divisor $D_t = D_t(\varepsilon) \in \lvert V_t \rvert$ on $Y$ such
  that setting $D \coloneqq \frac{1}{t}D_t$, we have
  \begin{equation}\label{eq:kol61concl}
    \mult_W D \ge \biggl(\frac{\vol_F(V_\bullet\rvert_F)}{\deg
    (f\rvert_W)}\biggr)^{1/n} - \varepsilon,
  \end{equation}
  where $F$ is a general fiber of $f$ and $n = \dim F$.
\end{lemma}
Here, \gls*{multiplicity} is the maximum integer $s$ such that $D$ vanishes to
order $s$ everywhere along $W$, and $\deg(f\rvert_W)$\glsadd{degreemap}
is the degree of the generically finite morphism $f\rvert_W$.
\begin{proof}
  Let $t>0$ be an integer such that $tM$ is Cartier, and let $\cI_W \subseteq
  \cO_Y$\glsadd{idealw} be the ideal sheaf defining $W$.
  Then, for every integer $s > 0$, we have the short exact sequence
  \[
    0 \longrightarrow \cI_W^s \otimes \cO_Y(tM) \longrightarrow \cO_Y(tM)
    \longrightarrow \cO_Y(tM) \otimes \cO_Y/\cI_W \longrightarrow 0,
  \]
  and pushing forward by $f$, we obtain the left-exact sequence
  \begin{equation}\label{eq:fpushforwardles}
    0 \longrightarrow f_*\bigl(\cI_W^s \otimes \cO_Y(tM)\bigr)
    \longrightarrow f_*\bigl(\cO_Y(tM)\bigr)
    \longrightarrow f_*\bigl(\cO_Y(tM) \otimes \cO_Y/\cI_W^s\bigr).
  \end{equation}
  \par Now choose integers $s,t > 0$ such that
  \[
    \biggl(\frac{\vol_F(V_\bullet\rvert_F)}{\deg
    (f\rvert_W)}\biggr)^{1/n} > \frac{s}{t} >
    \biggl(\frac{\vol_F(V_\bullet\rvert_F)}{\deg
    (f\rvert_W)}\biggr)^{1/n} - \varepsilon,
  \]
  and recall that for every regular point $x \in F$, we have
  \[
    h^0(F,\cO_F/\fm_x^s) = \binom{n+s-1}{n} =
    \frac{s^n}{n!} + O(s^{n-1}).
  \]
  Thus, after possibly replacing $s$ and $t$ by multiples, we may assume without
  loss of generality that $tM$ is Cartier, and that
  \begin{equation}\label{eq:kol61estimate}
    \begin{aligned}
      \MoveEqLeft[7]\dim_k V_t\rvert_F - h^0\bigl(F,\cO_F(tM\rvert_F) \otimes
      \cO_F/\cI_{W \cap F}^s \bigr)\\
      &= \dim_k V_t\rvert_F - \deg(f\rvert_W) \cdot \binom{n+s-1}{n}
      > 0
    \end{aligned}
  \end{equation}
  by the definition of volume, which implies that $V_t\rvert_F$ has sections
  vanishing to order $s$ everywhere along $W \cap F$.
  Now by generic flatness \cite[Thm.\ 6.9.1]{EGAIV2} and by \cite[Cor.\
  8.3.11]{Ill05}, the sheaf $\cO_Y(tM) \otimes \cO_Y/\cI_W^s$ is generically
  flat and generically cohomologically flat in degree zero over $Z$.
  By cohomology and base change \cite[Cor.\ 8.3.11]{Ill05},
  the estimate \cref{eq:kol61estimate} together with the
  exact sequence \cref{eq:fpushforwardles} therefore implies that the sheaf
  $f_*(\cO_Y(tM) \otimes \cO_Y/\cI_W^s)$ is nonzero, hence has a global section
  by the fact that $Z$ is affine.
  We then let $D_t = D_t(\varepsilon)$ be a Cartier divisor
  corresponding to a section in
  \[
    H^0\bigl(Z,f_*\bigl(\cO_Y(tM) \otimes \cO_Y/\cI_W^s\bigr)\bigr) =
    H^0\bigl(Y,\cO_Y(tM) \otimes \cO_Y/\cI_W^s\bigr),
  \]
  in which case \cref{eq:kol61concl} holds for $D \coloneqq \frac{1}{t}D_t$.
\end{proof}
When $Z = \Spec k$ is a point and $W$ is a closed point $x \in X$, we see that
\cref{lem:kol61} gives a way to construct an effective $\QQ$-Cartier divisor
$D(\varepsilon)$ that is singular at $x$.
However, this divisor may have very bad singularities in a neighborhood of $x$.
The proof of \cref{thm:kol64} is devoted to ensuring that one can replace
$D(\varepsilon)$ with a divisor with mild singularities in a neighborhood of
$x$.
\par In the course of the proof, we will need the following:
\begin{lemma}[{cf.\ \cite[Cor.\ 7.8]{Kol97}}]\label{lem:kol78}
  Let $(Y,\Delta)$ be a klt pair over an algebraically closed field $k$ of
  characteristic zero, and let $y \in Y$ be a closed point.
  Let $C$ be a smooth affine curve,and let $B$ be a $\QQ$-Cartier divisor on $Y
  \times C$ such that $\{y\} \times C \subseteq \Supp B$.
  Let $0 \in C$ be a closed point such that
  $(Y \times \{c\},\Delta \times \{c\} + B\rvert_{Y \times \{c\}})$ is not log
  canonical at $y$ for
  all closed points $c \in C$ in a punctured neighborhood of $0$.
  Then, $(Y \times \{0\},\Delta \times \{0\} + B\rvert_{Y \times \{0\}})$ is not
  log canonical at $y$.
\end{lemma}
Koll\'ar's proof of \cite[Cor.\ 7.8]{Kol97} uses the Koll\'ar--Shokurov
connectedness theorem \cite[Thm.\ 7.4]{Kol97}, among other results.
The proof of this connectedness theorem uses vanishing theorems.
We therefore give a proof of \cref{lem:kol78} that uses reduction modulo $\fp$
instead of vanishing theorems.
\begin{proof}
  Set $X = Y \times C$, and 
  suppose that $(Y \times \{0\},\Delta \times \{0\}+B\rvert_{Y \times \{0\}})$
  is log canonical.
  Then, by inversion of adjunction for log canonical pairs \cite[Thm.\
  4.2]{Tak04inversion}, we see that $(X,\Delta \times C+B+Y \times \{0\})$ is
  log canonical in a neighborhood $U \subseteq X$ of $(y,0)$.
  Letting $W = U \cap (\{y\} \times C)$, we see that
  \[
    \bigl(Y \times \{c\},\Delta \times \{c\} + (B+Y \times \{0\})\rvert_{Y
    \times \{c\}}\bigr) = \bigl(Y \times \{c\},\Delta \times \{c\} + B\rvert_{Y
    \times \{c\}}\bigr)
  \]
  is log canonical at $y$ for general closed points $c \in W$ by Reid's
  Bertini-type theorem \cite[Prop.\ 7.7]{Kol97}, which is a contradiction.
\end{proof}
We can now prove \cref{thm:kol64}.
\begin{proof}[Proof of \cref{thm:kol64}]
  We prove \cref{thm:kol64} in a sequence of steps.
\begin{step}
  Finding a singular divisor at $x$.
\end{step}
\begin{theorem}[cf.\ {\cites[Thm.\ 6.7.1]{Kol97}[Prop.\ 3.3]{Liu}}]
  \label{thm:kol671}
  Let $(X,\Delta)$ be an effective log pair, where $X$ is a normal projective
  variety over an algebraically closed field $k$ of characteristic zero,
  $\Delta$ is a $\QQ$-Weil divisor, and $K_X+\Delta$ is
  $\QQ$-Cartier.
  Consider a closed point $x \in X$ such that $(X,\Delta)$ is klt at $x$.
  Let $H$ be a $\QQ$-Cartier divisor on $X$, and let $V_\bullet$ be a graded
  linear system associated to $H$ such that $\vol_X(V_\bullet) > n^n$.
  Then, there exists an effective $\QQ$-divisor $B_x \sim_\QQ H$ that is a
  multiple of a divisor in $\lvert V_t \rvert$ for some $t$ such that
  $(X,\Delta+B_x)$ is not log canonical at $x$.
\end{theorem}
\begin{proof}
  If $x \in X$ is a regular point, then this immediately follows from
  \cref{lem:kol61} by setting $Z = \Spec k$ and $W = \{x\}$.
  Otherwise, consider the second projection morphism $X \times \AA^1_k \to
  \AA^1_k$, and let $C' \subseteq X \times \AA^1_k$ be a general curve passing
  through $(x,0)$ that is finite over $\AA^1_k$.
  Let $\nu \colon C \to C'$ be the normalization of $C'$.
  We then have the commutative diagram
  \[
    \begin{tikzcd}
      C\arrow[dashed]{dr}{\sigma}\arrow[equal,bend right]{ddr}\arrow[out=0,in=100]{ddrr}\\
      & X \times C \arrow[crossing over]{rr}\dar[swap]{p_2} & & X \times \AA^1_k
      \dar{p_2}\\
      & C \rar{\nu} & C' \rar \arrow[hook]{ur} & \AA^1_k
    \end{tikzcd}
  \]
  where the outer rectangle is cartesian.
  By the universal property of fiber products, this cartesian rectangle
  induces a section $\sigma\colon C \to X \times C$ of $p_2\colon X
  \times C \to C$, such that
  $\sigma(C)$ passes through $(x,0) \in X \times C$ for some closed point $0
  \in C$.
  By applying \cref{lem:kol61} to the graded linear system $p_1^*V_\bullet$ on
  $X \times C$, the surjective morphism $p_2\colon X \times C \to C$, and the
  subvariety $\sigma(C) \subseteq X \times C$, we see that for some $t > 0$,
  there exists an effective $\QQ$-Cartier divisor $B \in p_1^*V_t$ such that
  $\frac{1}{t} B\rvert_{X \times \{c\}}$
  has multiplicity greater than $n$ at $(x,c)$ for every $c \in C$ in a
  punctured neighborhood of $0 \in C$.
  By taking the normalized blowup at $(x,c)$, we see that the pair
  \[
    \biggl(X \times \{c\},\Delta \times \{c\}+\frac{1}{t}B\rvert_{X \times
    \{c\}}\biggr)
  \]
  is not log canonical at $(x,c)$.
  We then take $B_x \coloneqq \frac{1}{t}B\rvert_{X \times \{0\}}$, which we
  identify with its image in $X$ under the isomorphism $X \times \{0\} \simeq
  X$.
  By \cref{lem:kol78}, the pair
  \[
    \bigl(X \times \{0\},\Delta \times \{0\}+B_x\bigr) \simeq (X,\Delta +
    B_x)
  \]
  is not log canonical at $x$.
\end{proof}
\setcounter{step}{1}
\begin{step}
  Inductive step.
\end{step}
The following result is the main part of the proof of \cref{thm:kol64}.
Below, \gls*{nonkltlocus} is the \textsl{non-klt locus} of $(X,\Delta)$, which
is the vanishing locus of the multiplier ideal $\cJ(X,\Delta)$.
\begin{theorem}[{cf.\ \cites[Thm.\ 6.8.1]{Kol97}[Prop.\ 3.4]{Liu}}]
  \label{thm:kol681}
  Let $(X,\Delta)$ be an effective log pair, where $X$ is a normal projective
  variety over an algebraically closed field $k$ of characteristic zero,
  $\Delta$ is a $\QQ$-Weil divisor, and $K_X+\Delta$ is
  $\QQ$-Cartier.
  Consider a closed point $x \in X$ such that
  $(X,\Delta)$ is klt at $x$.
  Let $D$ be an effective $\QQ$-Cartier divisor such that $(X,\Delta+D)$ is log
  canonical on a neighborhood $X^0$ of $x$, and
  suppose that $\Nklt(X,\Delta+D) =
  Z \cup Z'$, where $Z$ is irreducible, $x \in Z$, and $x \notin Z'$.
  Set $m = \dim Z$.
  Let $H$ be a $\QQ$-Cartier divisor such that $\vol_{X\mid Z}(H) > m^m$.
  Then, there exists an effective $\QQ$-Cartier divisor $B \sim_\QQ H$ and
  rational numbers $0 < \delta \ll 1$ and $0 < c < 1$ such that
  \begin{enumerate}[label=$(\roman*)$,ref=\roman*]
    \item $(X,\Delta+(1-\delta)D+cB)$ is log canonical in a
      neighborhood of $x$, and
    \item $\Nklt(X,\Delta+(1-\delta)D+cB) = Z_1 \cup Z_1'$, where $x \in Z_1$,
      $x \notin Z_1'$, and $\dim Z_1 < \dim Z$.
  \end{enumerate}
\end{theorem}
\begin{proof}
  By assumption, there is a proper birational morphism $f\colon Y \to X$ from a
  normal variety $Y$, and a divisor $E \subseteq Y$ such that $a(X,\Delta+D,E) =
  -1$ and $f(E) = Z$.
  Write
  \begin{equation}\label{eq:kol681discrep}
    K_Y = f^*(K_X+\Delta+D) + \sum_i e_iE_i,
  \end{equation}
  where $E = E_1$ and $e_1 = -1$.
  Let $Z^0 \subseteq Z$ be an open subset such that $f\rvert_E \colon
  E \to Z$ is smooth over $Z^0$, and such that if $z \in Z^0$, then
  $(f\rvert_E)^{-1}(z) \not\subseteq E_i$ for $i \ne 1$.
  \par Now let $t \gg 0$ such that $tH$ is Cartier, and such that $\cO_X(tH)
  \otimes I_Z$ is globally generated away from $\Bplus(H)$.
  We then make the following:
  \begin{claim}[cf.\ {\cite[Clms.\ 6.8.3 and 6.8.4]{Kol97}}]
    \label{claim:kol6834}
    We can construct a divisor $F_x \sim tH\rvert_Z$ such that
    \begin{enumerate}[label=$(\roman*)$,ref=\roman*]
      \item $\mult_x F_x > tm$,\label{claim:kol683mult}
      \item\label{claim:kol683rest}
        $F_x$ is the image of a Cartier divisor $F^X_x$ on $X$ under the
        restriction morphism
        \begin{equation}\label{eq:restrictedlinys}
          H^0\bigl(X,\cO_X(tH)\bigr) \longrightarrow
          H^0\bigl(Z,\cO_Z(tH\rvert_Z)\bigr),
        \end{equation}
      \item $(X,\Delta+D+\frac{1}{t}F_x^X)$ is klt on $X^0 \smallsetminus (Z
        \cup Z' \cup \Bplus(D))$,\label{claim:kol683kltcond}
      \item $(X,\Delta+D+\frac{1}{t}F_x^X)$ is log canonical at the generic
        point of $Z$, and\label{claim:kol684lccond}
      \item $(X,\Delta+D+\frac{1}{t}F_x^X)$ is not log canonical at $z$.%
        \label{claim:kol683notlccond}
    \end{enumerate}
  \end{claim}
  \begin{proof}
    As in the proof of \cref{thm:kol671}, we first construct a regular affine
    curve $C$ such that the projection $p_2\colon Z \times C \to C$ has a
    section $\sigma\colon C \to Z \times C$ for which a closed point $0 \in C$
    maps to $x$.
    Now let $p_1^*V_\bullet$ be the graded linear system obtained by pulling
    back the graded linear system arising as the image of the restriction maps
    \cref{eq:restrictedlinys} via the first projection morphism $p_1\colon Z
    \times C \to Z$.
    By \cref{lem:kol61}, there exists an effective Cartier divisor $F \sim
    t\,p_1^*H\rvert_Z$ on $Z \times C$ such that $\mult_{\sigma(C)} F\rvert_{Z
    \times C} > tm$, and by construction, $F = F^X\rvert_{Z \times C}$ for an
    effective Cartier divisor $F^X \in \lvert t\,q_1^*H \rvert$ on $X \times C$,
    where $q_1\colon X \times C \to X$ is the first projection.
    The restriction $F_x \coloneqq F\rvert_{Z \times \{\sigma(0)\}}$ then
    satisfies \cref{claim:kol683mult} and \cref{claim:kol683rest}.
    Note that \cref{claim:kol684lccond} also follows from construction, since
    $F_x$ does not vanish everywhere along $Z$.
    \par We now show that $F_x$ satisfies \cref{claim:kol683kltcond}.
    First, we note that the sublinear system
    $\lvert B \rvert \subseteq \lvert t\,q_1^*H \rvert$ on $X \times C$
    spanned by those
    effective Cartier divisors $B'$ such that either $Z \times C \subseteq B'$
    or $B'\rvert_{Z \times C} = F$ is basepoint-free on $(X \smallsetminus (Z
    \cup \Bplus(H))) \times C$ by the assumption that $\cO_X(tH) \otimes I_Z$ is
    globally generated away from $\Bplus(H)$.
    Thus, \cref{claim:kol683kltcond} follows from the Koll\'ar--Bertini theorem
    \cite[Thm.\ 4.8.2]{Kol97} by choosing $F^X$ generally in $\lvert B \rvert$.
    \par It remains to show \cref{claim:kol683notlccond}.
    By \cref{lem:kol78}, it suffices to show that for 
    every $0 \ne c \in C$ such that $\sigma(c) \in Z^0$,
    the pair $(X,\Delta+D+\frac{1}{t}F_{\sigma(c)}^X)$ is not log canonical at
    $\sigma(c)$, where $F_{\sigma(c)}^X \coloneqq F\rvert_{X \times
    \{\sigma(c)\}}$.
    Let $y$ be the generic point of $(f\rvert_E)^{-1}(\sigma(c))$.
    Writing \cref{eq:kol681discrep} as before, we also write
    \[
      f^*F_{\sigma(c)}^X = F_{\sigma(c)}^Y + \sum_i t f_iE_i
    \]
    where $F_{\sigma(c)}^Y$ is the strict transform $f_*^{-1}F_{\sigma(c)}^X$ of
    $F_{\sigma(c)}^X$.
    We then have
    \[
      K_Y + \frac{1}{t} F_{\sigma(c)}^Y + \sum (f_i - e_i)E_i \sim_\QQ
      f^*\Bigl(K_X+\Delta+D+\frac{1}{t}F_{\sigma(c)}^X\Bigr).
    \]
    Now $(X,\Delta+D+\frac{1}{t}F_{\sigma(c)}^X)$ is not log canonical at
    $\sigma(c)$ if $(Y,\frac{1}{t}F_{\sigma(c)}^Y+\sum(f_i-e_i)E_i)$ is not
    sub-log canonical at $y$.
    Since $Z \not\subseteq F_{\sigma(c)}^X$, we know that $f_1 = 0$.
    Thus, $\sum (f_i-e_i)E_i = E + \sum_{i \ne 1} (f_i-e_i)E_i$, and by
    assumption none of the $E_i$ contain $y$ when $i \ne 1$.
    Moreover, $(Y,\frac{1}{t}F_{\sigma(c)}^Y+\sum(f_i-e_i)E_i)$ is not sub-log
    canonical at $y$ if and only if $(Y,\frac{1}{t}F_{\sigma(c)}^Y+E)$ is not
    log canonical at $y$.
    By inversion of adjunction \cite[Thm.\ 4.2]{Tak04inversion}, the latter
    holds if and only if $(E,\frac{1}{t}f^*F_{\sigma(c)}^X\rvert_E) =
    (E,\frac{1}{t}(f\rvert_E)^*F_{\sigma(c)})$ is not log canonical at $y$.
    Now $E$ is smooth at $y$, and $y$ has codimension $m$ in $E$ and
    $\frac{1}{t}(f\rvert_E)^*F_{\sigma(c)}$ has multiplicity $>m$.
    We then see that $(E,\frac{1}{t}(f\rvert_E)^*F_{\sigma(c)})$ is not log
    canonical at $y$ by taking the normalized blowup at $y$.
    This concludes the proof of \cref{claim:kol6834}.
  \end{proof}
  To finish the proof of \cref{thm:kol681}, we apply \cref{claim:kol6834} and
  set $B = \frac{1}{t} F_x^X$.
  Note that $(X,\Delta+(1-\delta)D)$ is klt at the generic point of $Z$ for
  every $\delta > 0$ by the assumptions that $(X,\Delta)$ is klt, that
  $(X,\Delta+D)$ is log canonical in a neighborhood of $x$, and on the non-klt
  locus of $(X,\Delta+D)$.
  Now choose $0 < \delta \ll 1$ such that $(X,\Delta+(1-\delta)D+B)$ is not log
  canonical at $x$.
  Letting $c$ be the log canonical threshold of $(X,\Delta+(1-\delta)D)$ with
  respect to $B$, we then see that $(X,\Delta+(1-\delta)D+cB)$
  is log canonical but not klt at $x$, and that $\Nklt(X,\Delta+(1-\delta)D+cB)
  = Z_1 \cup Z_1'$, where $x \in Z_1$, $x \notin Z_1'$, and $\dim Z_1 < \dim Z$.
\end{proof}
We can almost show \cref{thm:kol64} using \cref{thm:kol681} and induction.
However, the resulting pair in \cref{thm:kol681} may be such that
$\Nklt(X,\Delta+(1-\delta)D+cB)$ has many irreducible components passing through
$x$.
We take care of this using the following:
\setcounter{step}{2}
\begin{step}
  Tie breaking.
\end{step}
\begin{lemma}[cf.\ {\cite[Lem.\ 6.9.1]{Kol97}}]\label{lem:tiebreaking}
  Let $(X,\Delta)$ be an effective log pair, where $X$ is a normal projective
  variety over a field $k$ of characteristic zero, $\Delta$ is a
  $\QQ$-Weil divisor, and $K_X+\Delta$ is
  $\QQ$-Cartier.
  Consider a $k$-rational point $x \in X$ such that
  $(X,\Delta)$ is klt at $x$.
  Let $D$ be an effective $\QQ$-Cartier divisor on $X$ such that $(X,\Delta+D)$
  is log canonical in a neighborhood of $x$.
  Let $\Nklt(X,\Delta+D) = \bigcup_i Z_i$ be the irreducible decomposition of
  $\Nklt(X,\Delta+D)$, where we label $Z_1$ such that $x \in Z_1$.
  Let $H$ be a $\QQ$-Cartier divisor on $X$ such that $x \notin \Bplus(H)$.
  Then, for every $0 < \delta \ll 1$, there is an effective $\QQ$-Cartier
  divisor $B \sim_\QQ H$ and $0 < c < 1$ such that
  \begin{enumerate}[label=$(\roman*)$,ref=\roman*]
    \item $(X,\Delta+(1-\delta)D+cB)$ is log canonical in a
      neighborhood of $x$, and
    \item $\Nklt(X,\Delta+(1-\delta)D+cB) = W \cup W'$ where $x \in W$, $x
      \notin W'$, and $W \subseteq Z_1$.
  \end{enumerate}
\end{lemma}
\begin{proof}
  Let $t \gg 1$ such that $tH$ is Cartier and such that $\cO_X(tH) \otimes
  I_{Z_1}$ is globally generated away from $\Bplus(H)$
  (\cref{prop:kur13prop27}).
  Let $B'$ correspond to a general section in $H^0(X,\cO_X(tH) \otimes
  I_{Z_1})$.
  By the Koll\'ar--Bertini theorem \cite[Thm.\ 4.8.2]{Kol97}, we see that
  $(X,\Delta+(1-\delta)D+bB')$ is klt outside $Z_1$ in a neighborhood of $x$ for
  $b < 1$.
  However, it is not log canonical along $Z_1$ for $1 > b \gg \delta > 0$.
  Now choose $b = 1/t$ and $1/t \gg \delta > 0$.
  Then, by letting $c \in (0,1)$ be the log canonical threshold of
  $(X,\Delta+(1-\delta)D)$ with respect to $(1/t)B'$, we see that
  $(X,\Delta+(1-\delta)D+(c/t)B')$ is log
  canonical but not klt at $x$.
  We can then set $B = \frac{1}{t}B'$.
\end{proof}
\setcounter{step}{3}
\begin{step}
  Proof of \cref{thm:kol64}.
\end{step}
We prove the following theorem by induction on $j$.
\begin{theorem}\label{thm:kol6101}
  With notation as in \cref{thm:kol64}, let $j \in \{1,2,\ldots,n\}$.
  Then, for every
  \begin{equation}\label{eq:kol6101cmcond}
    d_j \ge \sum_{m=n-j}^n \frac{m}{c(m)},
  \end{equation}
  there exists an effective $\QQ$-Cartier divisor $D_j \sim_\QQ d_jN$ and an
  open neighborhood $X^0 \subseteq X$ of $x$ such that for some $b_j \in (0,1)$,
  we have that
  \begin{enumerate}[label=$(\roman*)$,ref=\roman*]
    \item\label{thm:kol6101lccond} $(X^0,\Delta+b_jD_j)$ is log canonical,
    \item\label{thm:kol6101codcond} $\codim(\Nklt(X^0,\Delta+b_jD_j),X^0) \ge
      j$, and
    \item\label{thm:kol6101nkltcond} $(X,\Delta+b_jD_j)$ is not klt at $x$.
  \end{enumerate}
\end{theorem}
\begin{proof}
  Set $D_0 = \emptyset$.
  By induction, we will assume that $D_j$ has already been constructed, and we
  are trying to construct $D_{j+1}$.
  \par For $j+1 = 1$, we construct $D_1$ by applying \cref{thm:kol671}, and set
  $b_1$ to be the log canonical threshold $\lct_x((X,\Delta);D_1)$.
  Now consider the case when $j+1 > 1$.
  First choose a positive real number $\varepsilon < (j+1)\cdot c(j+1)^{-1}$.
  By \cref{lem:tiebreaking} and by inductive hypothesis, there exists a
  $\QQ$-Cartier divisor $B_j \sim_\QQ\varepsilon N$ such that for some $\delta >
  0$, the $\QQ$-Cartier divisor
  \[
    D_j' \coloneqq (1-\delta)b_jD_j+B_j \sim_\QQ \bigl((1-\delta)b_jd_j +
    \varepsilon\bigr)N
  \]
  satisfies conditions \cref{thm:kol6101lccond}--\cref{thm:kol6101nkltcond} for
  $b_j$ replaced by $1$,
  and in addition, either
  $Z \coloneqq \Nklt(X^0,\Delta+D_j')$ is irreducible of codimension at
  least $j$ at $x$, or it has codimension at least $j+1$ at $x$.
  In the latter case, let $M$ be a general member of $\lvert tN \rvert$ for $t
  \gg 1$.
  By assumption in \cref{eq:kol6101cmcond}, for all rational numbers $0 < \gamma
  \ll 1$, we have $d_{j+1} \ge (1+\gamma)((1-\delta)b_jd_j +
  \varepsilon)$.
  Thus, we can set
  \[
    D_{j+1} \coloneqq (1+\gamma)D_j' + \frac{1}{t}\bigl( d_{j+1} -
    (1+\gamma)\bigl((1-\delta)b_jd_j + \varepsilon\bigr)\bigr) M \sim_\QQ
    d_{j+1}N,
  \]
  and this $\QQ$-divisor $D_{j+1}$ satisfies conditions
  \cref{thm:kol6101lccond}--\cref{thm:kol6101nkltcond} for $b_{j+1} =
  1/(1+\gamma)$.
  \par It remains to consider the case when $Z$ is irreducible of codimension at
  least $j$ at $x$.
  Set $H = ( (j+1)\cdot c(j+1)^{-1}-\varepsilon)N$.
  For $0 < \varepsilon \ll 1$, we have that $(H^{j} \cdot Z) > j^j$, hence we
  can apply \cref{thm:kol681} to obtain a $\QQ$-Cartier divisor
  \[
    D_{j+1} \sim_\QQ \bigl((j+1)\cdot c(j+1)^{-1}-\varepsilon\bigr)N
  \]
  satisfying conditions \cref{thm:kol6101lccond}--\cref{thm:kol6101nkltcond} for
  $b_{j+1}$ replaced by the rational number $c$ in the statement of
  \cref{thm:kol681}.
\end{proof}
Finally, the case $j = \dim X$ in \cref{thm:kol6101} is \cref{thm:kol64},
concluding the proof of \cref{thm:kol64}.
\end{proof}

\appendix

\chapter{\emph{F}-singularities for non-\emph{F}-finite rings}
\label{app:nonffinfsings}
In this appendix, we review classes of singularities defined using the
Frobenius morphism, taking care to avoid $F$-finiteness assumptions.
Most of this material is well-known, but some of the implications in Figure
\ref{fig:fsingsdiagram} are new, at least for non-$F$-finite rings.
We recommend \cite{TW18} for a survey of $F$-singularities (mostly in the
$F$-finite setting), and \cite[\S6]{DS16} and
\cite[\S3]{Has10} as references for the material on strong $F$-regularity in the
non-$F$-finite setting.
Some of this material appears in \cite[Apps.\ A and B]{Mur} and \cite{DaM}.
\medskip
\par To define different versions of $F$-rationality, we will need the
following:
\begin{citeddef}[{\cite[Def.\ 2.1]{HH90}}]
  Let $R$ be a noetherian ring.
  A sequence of elements $x_1,x_2,\ldots,x_n \in R$ is a \textsl{sequence of
  parameters}\index{parameter(s)!sequence of|textbf} if, for every prime ideal
  $\fp$ containing $(x_1,x_2,\ldots,x_n)$, the images of $x_1,x_2,\ldots,x_n$ in
  $R_\fp$ are part of a system of parameters in $R_\fp$.
  An ideal $I \subseteq R$ is a \textsl{parameter
  ideal}\index{parameter(s)!ideal|textbf} if $I$ can
  be generated by a sequence of parameters in $R$.
\end{citeddef}
We now begin defining different classes of singularities.
We start with $F$-singularities defined using tight closure.
Recall that if $R$ is a ring, then $R^\circ$ is the
complement of the union of the minimal primes of $R$.
\begin{citeddef}[{\cite[Def.\ 8.2]{HH90}}]\label{def:tightclosure}
  Let $R$ be a ring of characteristic $p > 0$, and let $\iota\colon N
  \hookrightarrow M$ be an inclusion of $R$-modules.
  The \textsl{tight closure}\index{tight closure, $N^*_M$|textbf} of $N$ in $M$ is the
  $R$-module
  \[
    N^*_M \coloneqq \Biggl\{x \in M \Biggm\vert
      \begin{tabular}{@{}c@{}}
        there exists $c\in R^\circ$ such that for all $e \gg0$,\\
        $x \otimes c \in \im\bigl(\id \otimes \iota\colon N \otimes_R F^e_*R \to
        M \otimes_R F^e_*R\bigr)$
      \end{tabular}
    \Biggr\}.\glsadd{generalizedtightclosure}
  \]
  We say that $N$ is \textsl{tightly closed} in $M$ if $N^*_M = N$.
\end{citeddef}
\begin{definition}[$F$-singularities via tight closure]\label{def:fsingstc}
  Let $R$ be a noetherian ring of characteristic $p > 0$.
  We say that
  \begin{enumerate}[label=$(\alph*)$,ref=\alph*]
    \item $R$ is \textsl{strongly
      $F$-regular}\index{F-regular@$F$-regular!strongly|textbf}
      if $N^*_M = N$
      for every inclusion $N \hookrightarrow M$ of $R$-modules \cite[Def.\ on
      p.\ 166]{Hoc07};\label{def:strongfreg}
    \item $R$ is \textsl{weakly
      $F$-regular}\index{F-regular@$F$-regular!weakly|textbf} if $I^*_R = I$ 
      for every ideal $I \subseteq R$ \cite[Def.\ 4.5]{HH90};
    \item $R$ is \textsl{$F$-regular}\index{F-regular@$F$-regular|textbf} if
      $R_\fp$ is weakly $F$-regular for
      every prime ideal $\fp \subseteq R$ \cite[Def.\ 4.5]{HH90};\label{def:freg}
    \item $R$ is \textsl{$F$-rational}\index{F-rational@$F$-rational|textbf} if
      $I^*_R = I$ for every parameter ideal $I \subseteq R$ \cite[Def.\
      1.10]{FW89}.
  \end{enumerate}
  The original definition of $F$-regularity asserted that
  localizations at every multiplicative set are weakly $F$-regular, but
  $(\ref{def:freg})$ is equivalent to this definition by \cite[Cor.\ 4.15]{HH90}.
\end{definition}
\begin{remark}\label{rem:sfrwhenffin}
  Note that $(\ref{def:strongfreg})$ is not the usual definition of strong
  $F$-regularity, although it coincides with the usual definition (Definition
  \ref{def:fsingssplit}$(\ref{def:fsingssplitreg})$) for $F$-finite rings; see
  Figure \ref{fig:fsingsdiagram}.\index{F-regular@$F$-regular!strongly!for
  $F$-finite rings}
\end{remark}
\par Next, we define $F$-singularities via purity of homomorphisms involving the
Frobenius.
We recall that a ring homomorphism $\varphi\colon R \to S$ is
\textsl{pure}\index{pure ring homomorphism|textbf} if
the induced homomorphism
$\id_M \otimes_R \varphi \colon M \otimes_R R \to M \otimes_R S$
is injective for every $R$-module $M$.
To simplify notation, we fix the following:
\begin{notation}\label{notation:lambdaec}
  Let $R$ be a ring of characteristic $p > 0$.
  For every $c \in R$ and every integer $e > 0$, we denote by $\lambda^e_c$ the
  composition
  \[
    R \overset{F^e}{\longrightarrow} F_*^eR \xrightarrow{F_*^e(-\cdot c)}
    F_*^eR.
  \]
\end{notation}
\begin{definition}[$F$-singularities via purity]\label{def:fsingspure}
  Let $R$ be a noetherian ring of characteristic $p > 0$.
  For every $c \in R$, we say that $R$ is \textsl{$F$-pure along
  $c$}\index{F-pure@$F$-pure!along $c$|textbf} if
  $\lambda^e_c$ is pure for some $e > 0$.
  Moreover, we say that
  \begin{enumerate}[label=$(\alph*)$,ref=\alph*]
    \item $R$ is \textsl{$F$-pure regular}\index{F-pure@$F$-pure!regular|textbf}
      if $R$ is $F$-pure along
      every $c \in R^\circ$ \cite[Rem.\ 5.3]{HH94};\label{def:fpurereg}
    \item $R$ is \textsl{$F$-pure}\index{F-pure@$F$-pure|textbf} if $R$ is
      $F$-pure along $1 \in R$ \cite[p.\ 121]{HR76};\label{def:fpure}
    \item $R$ is \textsl{strongly
      $F$-rational}\index{F-rational@$F$-rational!strongly|textbf} if
      for every $c \in R^\circ$, there exists $e_0 > 0$ such that for all $e \ge
      e_0$, the induced
      homomorphism $\id_{R/I} \otimes_R \lambda^e_c$ is injective for every
      parameter ideal $I \subseteq R$ \cite[Def.\
      1.2]{Vel95}.\label{def:strongfrat}
  \end{enumerate}
  The terminology \textsl{$F$-pure regular} is from \cite[Def.\ 6.1.1]{DS16} to
  distinguish $(\ref{def:fpurereg})$ from
  Definition \ref{def:fsingstc}$(\ref{def:strongfreg})$.
  $F$-pure regular rings are also called \textsl{very strongly
  $F$-regular}\index{F-regular@$F$-regular!very strongly|see
  {$F$-pure regular}} \cite[Def.\ 3.4]{Has10}.
\end{definition}
\par Next, we define $F$-singularities via splitting of homomorphisms involving
the Frobenius.
We use the same notation as for $F$-singularities defined using purity
(Notation \ref{notation:lambdaec}).
\begin{definition}[$F$-singularities via splitting]\label{def:fsingssplit}
  Let $R$ be a noetherian ring of characteristic $p > 0$.
  For every $c \in R$, we say that $R$ is \textsl{$F$-split along
  $c$}\index{F-split@$F$-split!along $c$|textbf} if $\lambda^e_c$
  splits as an $R$-module homomorphism for some $e > 0$.
  Moreover, we say that
  \begin{enumerate}[label=$(\alph*)$,ref=\alph*]
    \item $R$ is \textsl{split
      $F$-regular}\index{F-regular@$F$-regular!split|textbf} if $R$ is $F$-split
      along every $c \in R^\circ$ \cite[Def.\ 5.1]{HH94};\label{def:fsingssplitreg}
    \item $R$ is \textsl{$F$-split}\index{F-split@$F$-split|textbf} if $R$ is
      $F$-split along $1 \in R$ \cite[Def.\ 2]{MR85}.
  \end{enumerate}
  The terminology \textsl{split $F$-regular} is from \cite[Def.\ 6.6.1]{DS16}.
  When $R$ is $F$-finite, split $F$-regularity is usually known as strong
  $F$-regularity in the literature; see Remark \ref{rem:sfrwhenffin}.
\end{definition}
Finally, we define $F$-injective singularities.
\begin{citeddef}[{\cite[Def.\ on p.\ 473]{Fed83}}]\label{def:finjective}
  A noetherian local ring $(R,\fm)$ of characteristic $p > 0$ is
  \textsl{$F$-injective}\index{F-injective@$F$-injective|textbf} if the
  $R$-module homomorphism
  $H^i_\fm(F) \colon H^i_\fm(R) \to H^i_\fm(F_*R)$
  induced by Frobenius is injective for all $i$.
  An arbitrary noetherian ring $R$ of characteristic $p > 0$ is
  \textsl{$F$-injective} if $R_\fm$ is $F$-injective for every maximal ideal
  $\fm \subseteq R$.
\end{citeddef}
We characterize $F$-finite rings that are $F$-injective using
Grothendieck duality.
This characterization is already implicit in \cite[Rem.\ on p.\ 473]{Fed83} and
the proof of \cite[Prop.\ 4.3]{Sch09}.
Note that if $R$ is an $F$-finite ring, then the exceptional pullback $F^!$ from
Grothendieck duality exists by \cref{thm:nayakshriek}, and $R$ has a
normalized dualizing complex $\omega_R^\bullet$ by
\cref{thm:ffiniteaffine}.
\begin{lemma}[cf.\ {\cite[Rem.\ on p.\ 473]{Fed83}}]\label{lem:finjffin}
  Let $R$ be an $F$-finite noetherian ring of characteristic $p > 0$.
  Then, $R$ is $F$-injective if and only if the $R$-module homomorphisms
  \begin{equation}\label{eq:grottracefinj}
    \mathbf{h}^{-i}\Tr_F\colon \mathbf{h}^{-i}F_*F^!\omega_R^\bullet
    \longrightarrow \mathbf{h}^{-i}\omega_R^\bullet
  \end{equation}
  induced by the Grothendieck trace of Frobenius are surjective for all $i$.
\end{lemma}
\cref{lem:finjffin} is most useful when $R$ is essentially of finite type
over an $F$-finite field, in which case
$F^!\omega_R^\bullet \simeq \omega_R^\bullet$ in the derived category
$\mathbf{D}^+_{\mathrm{qc}}(R)$ (\cref{thm:nayakshriek}), hence the homomorphisms in
\eqref{eq:grottracefinj} can be written as $\mathbf{h}^{-i}F_*\omega_R^\bullet
\to \mathbf{h}^{-i}\omega_R^\bullet$.
\begin{proof}
  By Grothendieck local duality \cite[Cor.\ V.6.3]{Har66}, $R$ is $F$-injective
  if and only if
  \begin{align*}
    F^*\colon \Ext^{-i}_R(F_*R,\omega_R^\bullet) &\longrightarrow
    \Ext^{-i}_R(R,\omega_R^\bullet)
    \intertext{is surjective for all $i$.
    By Grothendieck duality for finite morphisms (\cref{thm:nayakshriek}),
    this occurs if and only if}
    F_*\Ext^{-i}_R(R,F^!\omega_R^\bullet) &\longrightarrow
    \Ext^{-i}_R(R,\omega_R^\bullet)
  \end{align*}
  is surjective for all $i$.
  Since $\Ext_R^{-i}(R,-) = \mathbf{h}^{-i}(-)$ and by the description
  of the Grothendieck duality isomorphism \cite[Thm.\ III.6.7]{Har66}, this
  is equivalent to the surjectivity of \eqref{eq:grottracefinj} for all $i$.
\end{proof}
\par The relationship between these classes of singularities is summarized
in Figure \ref{fig:fsingsdiagram}.
Most of the implications therein appear in the literature; see Table
\ref{tab:fsingsdiagram}.
We now show the remaining implications in \cref{fig:fsingsdiagram}, for which we
could not find a suitable reference.
\begin{proof}[Proofs of implications not appearing in Table
  \ref{tab:fsingsdiagram}]
  \par \textit{Weakly $F$-regular $+$ Gorenstein away from
  isolated points $+$ Cohen--Macaulay $\Rightarrow$ strongly
  $F$-regular}.
  Let $R$ be a ring satisfying these properties.
  To show that $R$ is strongly
  $F$-regular, it suffices to show that $0$ is tightly closed in
  $E_\fm \coloneqq E_{R_\fm}(R/\fm)$
  for every maximal ideal $\fm \subseteq R$ \cite[Lem.\ 3.6]{Has10}.
  Since $R_\fm$ is weakly $R$-regular \cite[Cor.\ 4.15]{HH90}, every submodule of a finitely generated
  module is tightly closed \cite[Prop.\ 8.7]{HH90}, hence the \textsl{finitistic
  tight closure}\index{tight closure, $N^*_M$!finitistic} $0^{*\mathrm{fg}}_{E_\fm}$ as
  defined in \cite[Def.\ 8.19]{HH90} is zero.
  Since $0^{*\mathrm{fg}}_{E_\fm} = 0^*_{E_\fm}$ under the
  hypotheses on $R$ \cite[Thm.\ 8.8]{LS01}, we see that $0$ is tightly closed
  in $E_\fm$ for every maximal ideal $\fm \subseteq R$, hence $R$ is strongly
  $F$-regular.
  \par \textit{Weakly $F$-regular $+$ $\NN$-graded $\Rightarrow$ split
  $F$-regular}.
  We adapt the proof of \cite[Cor.\ 4.4]{LS99}.
  Let $R$ be the $\NN$-graded ring with irrelevant ideal $\fm$.
  By assumption (see \cite[\S3]{LS99}), the ring $R$ is finitely generated over a
  field $R_0 = k$ of characteristic $p > 0$.
  The localization $R_\fm$ of $R$ is weakly $F$-regular by \cite[Cor.\
  4.15]{HH90}.
  Now let $L$ be the perfect closure of $k$, and let $\fm'$ be the expansion of
  $\fm$ in $R \otimes_k L$; since $R$ is graded, $\fm'$ is the irrelevant ideal
  in $R \otimes_k L$.
  The ring homomorphism
  $R_\fm \to R_\fm \otimes_k L \simeq (R \otimes_k L)_{\fm'}$
  is purely inseparable and $\fm$ expands to $\fm'$, hence $(R \otimes_k
  L)_{\fm'}$ is weakly $F$-regular by \cite[Thm.\ 6.17$(b)$]{HH94}.
  By the proof of \cite[Cor.\ 4.3]{LS99}, the ring $R \otimes_k L$ is split
  $F$-regular.
  Finally, $R$ is a direct summand of $R \otimes_k L$ as an $R$-module, hence
  $R$ is split $F$-regular as well \cite[Thm.\ 5.5$(e)$]{HH94}.
  \par \textit{$F$-rational $+$ $F$-finite $\Rightarrow$ strongly
  $F$-rational}.
  The hypotheses of \cite[Thm.\ 1.12]{Vel95} are satisfied when the ring is
  $F$-finite since an $F$-finite ring is excellent and is isomorphic to a
  quotient of a regular
  ring of finite Krull dimension by \cref{thm:ffiniteaffine}.
  \par \textit{$F$-rational $\Rightarrow$ $F$-injective.}
  We adapt the proof of \cite[Prop.\ 6.9]{QS17}.
  Let $R$ be the $F$-rational ring, and consider a maximal ideal $\fm \subseteq
  R$.
  By \cite[Thm.\ 3.7]{QS17}, it suffices to show that every ideal $I
  \subseteq R_\fm$ generated by a system of parameters in $R_\fm$ is
  \textsl{Frobenius closed}\index{Frobenius closure} in the
  sense of \cite[(10.2)]{HH94}.
  Write $I = (a_1,a_2,\ldots,a_t)$, where $t$ is the height of $\fm$ and
  $a_i \in R$ for every $i$.
  Note that $\fm$ is minimal over $(a_1,a_2,\ldots,a_t)$.
  Let $J$ be the $\fm$-primary component of $(a_1,a_2,\ldots,a_t)$ in $R$.
  Then, we have $I = JR_\fp$, $\height J = t$, and $\dim R/J \le d - t$, where
  $d = \dim R$.
  We claim there exist elements $b_1,b_2,\ldots,b_t \in J^2$ such that setting
  $x_i = a_i + b_i$, the sequence $x_1,x_2,\ldots,x_t$ is a sequence of
  parameters.
  For $i = 1$, we have
  \begin{align*}
    (a_1) + J^2 \not\subseteq{}& \bigcup_{\substack{\fp \in \Ass R\\\dim R/\fp =
    d}} \fp.
    \intertext{Thus, by a theorem of Davis \cite[Thm.\ 124]{Kap74}, there exists
    $b_1 \in J^2$ such that}
    x_1 \coloneqq a_1 + b_1 \notin{}& \bigcup_{\substack{\fp \in \Ass R\\\dim
    R/\fp = d}} \fp.
  \end{align*}
  For every $1 < i \le t$, the same method implies there exist $b_i \in J^2$
  such that
  \[
    x_i \coloneqq a_i + b_i \notin \bigcup_{\substack{\fp \in
    \Ass(R/(x_1,x_2,\ldots,x_{i-1}))\\\dim R/\fp = d-i+1}} \fp.
  \]
  We then see that $x_1,x_2,\ldots,x_t$ form a sequence of parameters in $R$,
  since they form a sequence of parameters after localizing to $R_\fm$, and are
  not all contained in any other prime ideal by construction.
  Now $(x_1,x_2,\ldots,x_t)R_\fm \subseteq I$ and $I = (x_1,x_2,\ldots,x_t)R_\fm
  + I^2$, hence Nakayama's lemma implies $I = (x_1,x_2,\ldots,x_t)R_\fm$; see
  \cite[Cor.\ to Thm.\ 2.2]{Mat89}.
  By assumption, the ideal $(x_1,x_2,\ldots,x_t)$ is tightly closed in $R$,
  hence Frobenius closed in $R$.
  Since Frobenius closure localizes \cite[Lem.\ 3.3]{QS17}, we therefore see that
  $I = (x_1,x_2,\ldots,x_t)R_\fm$ is Frobenius closed in $R_\fm$.
\end{proof}
\begin{remark}
  The condition that $R$ is the image of a Cohen--Macaulay ring is not
  too restrictive in practice.
  For instance, it suffices for $R$ to be local and excellent \cite[Cor.\
  1.2]{Kaw02} or for $R$ to have a dualizing complex
  \cite[Cor.\ 1.4]{Kaw02}.
  The latter property holds when $R$ is $F$-finite; see \cref{thm:ffiniteaffine}.
\end{remark}
\begin{remark}
  In the implication \textit{Weakly $F$-regular $+$ Gorenstein away from
  isolated points $+$ Cohen--Macaulay $\Rightarrow$ strongly $F$-regular,}
  MacCrimmon\index{MacCrimmon, Brian|(} \cite[Thm.\ 3.3.2]{Mac96} showed that for
  $F$-finite rings, the
  Gorenstein condition can be weakened to the condition of being
  $\QQ$-Gorenstein away from isolated points.
  The implication \textit{weakly $F$-regular $+$ $F$-finite $\Rightarrow$
  split $F$-regular} is a famous open problem, which
  was solved in dimensions at
  most three by Williams\index{Williams, Lori J.} \cite[\S4]{Wil95}.
  See \cite{Abe02} for other situations in which this implication is
  known and for a proof of MacCrimmon's theorem\index{MacCrimmon, Brian|)} (see
  \cite[(2.2.4)]{Abe02}).
\end{remark}
\begin{remark}
  By using the gamma construction (see Appendix \ref{app:gamma}), one can weaken
  the $F$-finiteness hypotheses appearing in Figure \ref{fig:fsingsdiagram}.
  For strong $F$-regularity and
  $F$-purity, see Theorem
  \ref{thm:splittingwithgamma}, and for
  $F$-rationality, see \cite[Thm.\ 3.8]{Vel95}.
\end{remark}
\begin{sidewaysfigure}
  \centering
  \begin{tikzcd}[row sep=huge,column sep=1.8em]
      & & & \text{regular}\\[-2.5em]
      \text{split $F$-regular} \arrow[Rightarrow]{rr} \arrow[Rightarrow]{ddd}
      & & \text{$F$-pure regular} \dar[Rightarrow] \arrow[Rightarrow]{rr}
      & & \text{strongly $F$-rational}\arrow[Rightarrow]{dd}\\
      & & \text{strongly $F$-regular} \rar[Rightarrow]
      \arrow[bend right=25,Leftarrow]{uur}
      \uar[bend right=30,Rightarrow,swap]{\text{local}}
      \arrow[Rightarrow,bend left=10]{ull}[description]{\text{$F$-finite}}
      & \text{$F$-regular} \dar[Rightarrow]
      & & \text{C--M}\\
      & & & \text{weakly $F$-regular} \rar[Rightarrow]
      \arrow[bend left=27,Rightarrow,start anchor=west]{uulll}{\text{$\NN$-graded}}
      \arrow[bend left=10,Rightarrow]{ul}[description,pos=0.6]{\substack{\text{Gor.\ away from
      isolated points $+$ C--M}}}
      \arrow[bend right=30,Rightarrow]{u}[swap,pos=0.45]{\substack{\text{f.t.\ 
      over $k$ s.t.}\\\text{$\trdeg_{\FF_p}k = \infty$}}}
      \arrow[Rightarrow]{d}
      & \text{$F$-rational}
      \arrow[bend right=10,Rightarrow,end anchor=east]{ul}[description]{\text{Gor.}}
      \arrow[bend right=20,Rightarrow,xshift=2pt]{uu}[swap,pos=0.6]{\text{$F$-finite}}
      \arrow[Rightarrow,end anchor=west,bend left=10]{ur}[swap,description]{\substack{%
      \text{image of}\\\text{C--M ring}}}
      \arrow[Rightarrow]{d}
      \rar[Rightarrow]
      & \text{normal}\dar[Rightarrow]\\
      \text{$F$-split} \arrow[Rightarrow]{rrr} & & 
      & \text{$F$-pure}
      \arrow[bend
      left=12,Rightarrow]{lll}{\substack{\text{$F$-finite}\\\text{or}\\\text{complete
      local}}}
      \rar[Rightarrow]
      & \text{$F$-injective}
      \lar[bend left=15,Rightarrow]{\text{quasi-Gor.}}
      \rar[Rightarrow] & \text{weakly normal} \dar[Rightarrow]\\[-1.5em]
      & & & & & \text{reduced}
    \end{tikzcd}
    \caption{Relationships between different classes of
    \emph{F}-singularities}
    \parbox{6in}{Here, ``C--M'' (resp.\ ``Gor.'') is an abbreviation for
    ``Cohen--Macaulay'' (resp.\ ``Gorenstein''), and ``f.t.\ over $k$ s.t.\ 
    $\trdeg_{\FF_p} k = \infty$'' means that the ring is of finite type over a
    field $k$ that has infinite transcendence degree over its prime subfield.}
    \index{F-finite@$F$-finite|ff{}}
    \index{F-regular@$F$-regular|ff{}}
    \index{F-regular@$F$-regular!split|ff{}}
    \index{F-regular@$F$-regular!strongly|ff{}}
    \index{F-regular@$F$-regular!weakly|ff{}}
    \index{F-pure@$F$-pure|ff{}}
    \index{F-pure@$F$-pure!regular|ff{}}
    \index{F-split@$F$-split|ff{}}
    \index{F-injective@$F$-injective|ff{}}
    \index{F-rational@$F$-rational|ff{}}
    \index{F-rational@$F$-rational!strongly|ff{}}
    \index{weak normality|ff{}}
    \label{fig:fsingsdiagram}
\end{sidewaysfigure}
\begin{table}[p]
  \centering
  \renewcommand{\arraystretch}{1.25}
  \begin{tabular}[h]{c@{\ \ $\implies$\ \ }cc}
    \toprule
    \multicolumn{2}{c}{Implication} & Proof\\
    \cmidrule(lr){1-2}\cmidrule(lr){3-3}
    split $F$-regular & $F$-split & Definition\\
    $F$-regular & weakly $F$-regular & Definition\\
    weakly $F$-regular & $F$-rational & Definition\\
    split $F$-regular & $F$-pure regular & split maps are pure\\
    $F$-split & $F$-pure & split maps are pure\\
    regular & strongly $F$-regular & \cite[Thm.\ 6.2.1]{DS16}\\
    $F$-pure regular & strongly $F$-regular & \cite[Lem.\ 3.8]{Has10}\\
    $F$-pure regular & strongly $F$-rational & \cite[Rem.\ 6.1.5]{DS16}\\
    strongly $F$-regular & $F$-regular & \cite[Cor.\ 3.7]{Has10}\\
    weakly $F$-regular & $F$-pure & \cite[Rem.\ 1.6]{FW89}\\
    $F$-pure & $F$-injective & \cite[Cor.\ 6.8]{HR74}\\
    strongly $F$-rational & $F$-rational & \cite[Prop.\ 1.4]{Vel95}\\
    $F$-rational & normal & \cite[Thm.\ 4.2$(b)$]{HH94}\\
    $F$-rational $+$ image of C--M ring
    & Cohen--Macaulay & \cite[Thm.\ 4.2$(c)$]{HH94}\\
    $F$-injective & weakly normal & \cite[Cor.\ 3.5]{DaM}\\
    \midrule
    strongly $F$-regular $+$ $F$-finite
    & split $F$-regular & \cite[Lem.\ 3.9]{Has10}\\
    strongly $F$-regular $+$ local &
    $F$-pure regular & \cite[Lem.\ 3.6]{Has10}\\
    \begin{tabular}{@{}c@{}}
      weakly $F$-regular\\[-0.3em]
      $+$ f.t.\ over $k$ s.t.\ $\trdeg_{\FF_p} k = \infty$
    \end{tabular} & $F$-regular & \cite[Thm.\ 8.1]{HH94}\\
    $F$-pure $+$ $F$-finite & $F$-split & \cite[Cor.\ 5.3]{HR76}\\
    $F$-pure $+$
    complete local & $F$-split & \cite[Lem.\ 1.2]{Fed83}\\
    $F$-rational $+$ Gor. & $F$-regular & \cite[Cor.\ 4.7$(a)$]{HH94}\\
    $F$-injective $+$ quasi-Gor. & $F$-pure & \cite[Rem.\ 3.8]{EH08}\\
    \bottomrule
  \end{tabular}
  \caption{Proofs of relationships between different classes of
  \emph{F}-singularities}
  \label{tab:fsingsdiagram}
\end{table}

\chapter{The gamma construction of Hochster--Huneke}
\label{app:gamma}
We prove a scheme-theoretic version of the gamma construction of Hochster--Huneke
\index{Hochster, Melvin}\index{Huneke, Craig}
\cite{HH94}, which we use to systematically reduce questions about varieties
over an arbitrary imperfect field to the same questions over an
$F$-finite field
(that is still imperfect).
In commutative algebra, the construction was first introduced in order to prove
that test elements (in the sense of tight closure) exist for rings that are
essentially of finite type over an excellent local ring of characteristic $p >
0$.
The material below is from \cite{Mur}.
\section{Construction and main result}
The following is the main consequence of the gamma construction:
\begin{theorem}\label{thm:gammaconstintro}
  Let\index{gamma construction!and singularities|(} $X$ be a scheme essentially of finite type over a field
  $k$ of characteristic $p > 0$, and let $\cQ$ be a set
  of properties in the following list:
  local complete intersection, Gorenstein, Cohen--Macaulay, $S_n$,
  $R_n$, normal, weakly normal\index{weak normality!and the gamma construction}, reduced,
  strongly $F$-regular\index{F-regular@$F$-regular!strongly!and the gamma construction},
  $F$-pure\index{F-pure@$F$-pure!and the gamma construction},
  $F$-rational\index{F-rational@$F$-rational!and the gamma construction},
  $F$-injective\index{F-injective@$F$-injective!and the gamma construction}.
  Then, there exists a purely inseparable field extension $k \subseteq k^\Gamma$
  such that $k^\Gamma$ is $F$-finite and such that the projection morphism
  \[
    \pi^\Gamma\colon X \times_k k^\Gamma \longrightarrow X
  \]
  is a homeomorphism that identifies $\cP$ loci for every $\cP \in \cQ$.%
  \index{gamma construction!and singularities|)}
\end{theorem}
Here, we recall that for a scheme $X$ and a property $\cP$ of local rings on $X$,
the \textsl{$\cP$ locus}\index{P locus@$\cP$ locus|textbf} of $X$ is
\[
  \cP(X) \coloneqq \bigl\{x \in X \bigm\vert \cO_{X,x}\ \text{is $\cP$}\bigr\}.
\]
We will in fact show a more general result (Theorem \ref{thm:gammaconst}), which
allows for $k$ to be replaced by a complete local ring, and allows finitely many
schemes instead of just one.
Note that \cref{thm:gammaconstintro} for weak normality, $F$-purity, and
$F$-injectivity are new even in the affine setting.
\medskip
\par Before describing the construction, we motivate the idea behind the
construction with the following:
\begin{example}
  Let $k$ be a non-$F$-finite field of characteristic $p > 2$, and
  let $a \in k \smallsetminus k^p$.
  For example, we can let $k = \FF_p(x_i)_{i \in \NN}$ and let $a = x_0$.
  Let $S = k[x,y]$ and $f = y^2+x^p-a \in S$, and consider
  Chevalley's\index{Chevalley, Claude} example \cite[Ex.\ 3]{Zar47}
  \[
    R = S/(f) = \frac{k[x,y]}{y^2+x^p-a}.
  \]
  We claim that $R$ is regular.
  Note that $R$ is smooth everywhere except at the maximal ideal
  $(x^p-a,y)$, since the Jacobian for $R$ is $(0,\ 2y)$.
  It therefore suffices to show that $R$ is regular at $\fm_R \coloneqq
  (x^p-a,y)R \subseteq R$.
  To avoid confusion, we denote by $\fm_S$ the ideal $(x^p-a,y)S \subseteq S$.
  We have
  \[
    \dim_{S/\fm_S}\biggl(\frac{\fm_S}{\fm_S^2}\biggr) = 2,
  \]
  since $S$ is regular.
  On the other hand, the defining equation $f = y^2 + x^p - a$ for $R$ is nonzero
  modulo $\fm_S^2$, hence
  \[
    \dim_{R/\fm_R}\biggl(\frac{\fm_R}{\fm_R^2}\biggr) =
    \dim_{S/\fm_S}\biggl(\frac{\fm_S}{\fm_S^2+(f)}\biggr) = 1.
  \]
  Thus, $R_{\fm_R}$ is regular, and $R$ is regular everywhere.
  \par We would now like to find a field extension $k \subseteq k'$ such that $R
  \otimes_k k'$ is $F$-finite and regular.
  First, we claim that setting $k' = k_\perf$ will result in an $F$-finite ring
  that is not regular.
  Set
  \[
    R' \coloneqq R \otimes_k k(a^{1/p}) \simeq \frac{k(a^{1/p})[x,y]}{y^2+x^p-a}
    \simeq \frac{k(a^{1/p})[x,y]}{y^2+(x-a^{1/p})^p},
  \]
  and denote $\fm_{R'} = (x-a^{1/p},y)R'$.
  We have that
  \[
    y^2+x^p-a = y^2 + (x-a^{1/p})^2\cdot(x-a^{1/p})^{2-p} \in (x-a^{1/p},y)^2,
  \]
  hence
  \[
    \dim_{R'/\fm_{R'}}\biggl(\frac{\fm_{R'}}{\fm_{R'}^2}\biggr)
    = \dim_{S/\fm_{S}}\biggl(\frac{\fm_{S}}{\fm_{S}^2+(f)}\biggr) = 2.
  \]
  Thus, we see that $R'$ is not regular at the maximal ideal $\fm_{R'}$.
  We therefore want to find a field extension $k \subseteq k'$ that avoids
  adjoining $a^{1/p}$, such that $k'$ is still $F$-finite.
  The gamma construction (\cref{thm:gammaconstintro}) ensures the existence of
  such an extension, although we
  note that in the specific case where $k = \FF_p(x_i)_{i \in \NN}$ and $a =
  x_0$ above, we can set $k' = \FF_p(x_0)(x_i)_{i \in \NN \smallsetminus
  \{0\}}$.
\end{example}
\par We now give an account of Hochster\index{Hochster, Melvin} and
Huneke's\index{Huneke, Craig} construction.
\begin{citedconstr}[{\cite[(6.7) and (6.11)]{HH94}}]\label{constr:gamma}
  Let\index{gamma construction|textbf}
  $(A,\fm,k)$ be a noetherian complete local ring of characteristic $p > 0$.
  By the Cohen structure theorem, we
  may identify $k$ with a coefficient field $k \subseteq A$.
  Moreover, by Zorn's lemma (see \cite[p.\ 202]{Mat89}), we
  may choose a \textsl{$p$-basis}\index{p -basis@$p$-basis|textbf} $\Lambda$ for
  $k$, which is a subset $\Lambda
  \subseteq k$ such that $k = k^p(\Lambda)$, and such that for every finite
  subset $\Sigma \subseteq \Lambda$ with $s$ elements, we have $[k^p(\Sigma) :
  k^p] = p^s$.
  \par Now let $\Gamma \subseteq \Lambda$ be a cofinite
  subset\index{cofinite subset|textbf}, i.e., a subset
  $\Gamma$ of $\Lambda$ such that $\Lambda \smallsetminus \Gamma$ is a finite
  set.
  For each integer $e \ge 0$, consider the subfield
  $k_e^\Gamma = k[\lambda^{1/p^e}]_{\lambda \in \Gamma}
  \subseteq k_{\perf}$
  of some perfect closure $k_{\perf}$ of $k$.
  These form an ascending chain, and we then set
  \[
    A^\Gamma \coloneqq \varinjlim_e k_e^\Gamma\llbracket A \rrbracket,
  \]
  where $k_e^\Gamma\llbracket A \rrbracket$ is the completion of $k_e^\Gamma
  \otimes_k A$ at the extended ideal $\fm\cdot(k_e^\Gamma \otimes_k A)$.
  Note that if $A = k$ is a field, then $A^\Gamma = k^\Gamma$ is a field
  by construction.
  \par Finally, let $X$ be a scheme essentially of finite type over $A$, and
  consider two cofinite subsets $\Gamma \subseteq \Lambda$ and
  $\Gamma' \subseteq \Lambda$ such that $\Gamma \subseteq \Gamma'$.
  We then have the following commutative diagram whose vertical faces are
  cartesian:
  \[
    \begin{tikzcd}[column sep=-0.6em]
      X^{\Gamma'} \arrow{dr}{\pi^{\Gamma'}}\arrow{rr}{\pi^{\Gamma\Gamma'}}
      \arrow{dd} & & X^\Gamma \arrow{dl}[swap]{\pi^\Gamma}\arrow{dd}\\
      & X\\[-1.35em]
      \Spec A^{\Gamma'}\arrow{rr}\arrow{dr} & & \Spec A^\Gamma \arrow{dl}\\
      & \Spec A\arrow[leftarrow,crossing over]{uu}
    \end{tikzcd}
  \]
\end{citedconstr}
We list some elementary properties of the gamma construction.
\begin{lemma}
  \label{lem:gammaconstbasic}
  Fix notation as in Construction \ref{constr:gamma}, and let $\Gamma \subseteq
  \Lambda$ be a cofinite subset.
  \begin{enumerate}[label=$(\roman*)$,ref=\roman*]
    \item The ring $A^\Gamma$ and the scheme $X^\Gamma$ are noetherian and
      $F$-finite.
      \label{lem:gammaconstbasicffin}
    \item The morphism $\pi^\Gamma$ is a faithfully flat universal homeomorphism
      with local complete intersection fibers.\label{lem:gammaconstbasicfflat}
    \item Given a cofinite subset $\Gamma \subseteq \Gamma'$, the morphism
      $\pi^{\Gamma\Gamma'}$ is a faithfully flat universal
      homeomorphism.\label{lem:gammaconstbasictransitionfflat}
  \end{enumerate}
\end{lemma}
\begin{proof}
  The ring $A^\Gamma$ is noetherian and $F$-finite
  \cite[(6.11)]{HH94}, hence
  $X^\Gamma$ is also by Example \ref{ex:eftoverffin} and the fact that morphisms
  essentially of finite type are preserved under base change
  (\cref{lem:nayak22}).
  The ring extensions $A \subseteq A^\Gamma$ and $A^\Gamma \subseteq
  A^{\Gamma'}$ are purely inseparable and faithfully flat \cite[(6.11)]{HH94},
  hence induce faithfully flat universal homeomorphisms on spectra \cite[Prop.\
  2.4.5$(i)$]{EGAIV2}.
  Thus, the morphisms $\pi^\Gamma$ and $\pi^{\Gamma\Gamma'}$ are faithfully flat
  universal homeomorphisms by base change.
  Finally, the ring extension $A \subseteq A^\Gamma$ is flat with local complete
  intersection fibers \cite[Lem.\ 3.19]{Has10}, hence $\pi^\Gamma$ is also by
  base change \cite[Cor.\ 4]{Avr75}.
\end{proof}
Our goal now is to prove that if a local property of schemes satisfies certain
conditions, then the property is preserved when passing from $X$ to $X^\Gamma$
for ``small enough'' $\Gamma$.
\begin{proposition}\label{prop:gammaconstconds}
  Fix notation as in Construction \ref{constr:gamma}, and let $\cP$ be a
  property of local rings of characteristic $p > 0$.
  \begin{enumerate}[label=$(\roman*)$,ref=\roman*]
    \item Suppose that for every flat local homomorphism $B \to C$ of noetherian
      local rings with local complete intersection fibers, if $B$ is $\cP$, then
      $C$ is $\cP$.
      Then, $\pi^\Gamma(\cP(X^\Gamma)) = \cP(X)$ for every cofinite subset
      $\Gamma \subseteq \Lambda$.
      \label{prop:gammaconstcm}
    \item Consider the following conditions:
      \begin{enumerate}[label=$(\Gamma\arabic*)$,ref=\ensuremath{\Gamma}\arabic*]
        \item If $B$ is a noetherian $F$-finite ring of characteristic $p > 0$,
          then $\cP(\Spec B)$ is open.\label{axiom:gammaopen}
        \item For every flat local homomorphism $B \to C$ of noetherian
          local rings of characteristic $p > 0$ with zero-dimensional
          fibers, if $C$ is $\cP$, then $B$ is
          $\cP$.\label{axiom:gammadescent}
        \item For every local ring $B$ essentially of finite type over $A$, if
          $B$ is $\cP$, then there exists a cofinite subset $\Gamma_1 \subseteq
          \Lambda$ such that $B^\Gamma$ is $\cP$ for every cofinite subset
          $\Gamma \subseteq \Gamma_1$.\label{axiom:gammaascent}
        \item[$(\ref*{axiom:gammaascent}')$]
          \refstepcounter{enumi}
          \makeatletter
          \def\@currentlabel{\ref*{axiom:gammaascent}\ensuremath{'}}
          \makeatother
          For every flat local homomorphism $B \to C$ of noetherian local rings
          of characteristic $p > 0$ such that the closed fiber is a field, if
          $B$ is $\cP$, then $C$ is $\cP$.\label{axiom:gammaascentprime}
      \end{enumerate}
      If $\cP$ satisfies $(\ref{axiom:gammaopen})$,
      $(\ref{axiom:gammadescent})$, and one of either
      $(\ref{axiom:gammaascent})$ or
      $(\ref{axiom:gammaascentprime})$, then there exists a cofinite subset
      $\Gamma_0 \subseteq \Lambda$ such that $\pi^\Gamma(\cP(X^\Gamma)) =
      \cP(X)$ for every cofinite subset $\Gamma \subseteq \Gamma_0$.
      \label{prop:gammaconstaxiomatic}
  \end{enumerate}
\end{proposition}
\begin{proof}
  For $(\ref{prop:gammaconstcm})$, it suffices to note that $\pi^\Gamma$ is
  faithfully flat with local complete intersection fibers by Lemma
  \ref{lem:gammaconstbasic}$(\ref{lem:gammaconstbasicfflat})$.
  \par For $(\ref{prop:gammaconstaxiomatic})$, we first note that
  $(\ref{axiom:gammaascentprime})$ implies $(\ref{axiom:gammaascent})$, since
  there exists a cofinite subset $\Gamma_1 \subseteq \Lambda$ such that the
  closed fiber is a field for every cofinite subset $\Gamma \subseteq \Gamma_1$
  by \cite[Lem.\ 6.13$(b)$]{HH94}.
  From now on, we therefore assume that $\cP$ satisfies
  $(\ref{axiom:gammaopen})$, $(\ref{axiom:gammadescent})$, and
  $(\ref{axiom:gammaascent})$.
  \par For every cofinite subset $\Gamma \subseteq \Lambda$, the set
  $\cP(X^\Gamma)$ is open by $(\ref{axiom:gammaopen})$ since $X^\Gamma$ is
  noetherian and $F$-finite by Lemma
  \ref{lem:gammaconstbasic}$(\ref{lem:gammaconstbasicffin})$.
  Moreover, the morphisms $\pi^\Gamma$ and $\pi^{\Gamma\Gamma'}$ are faithfully
  flat universal homeomorphisms for every cofinite subset $\Gamma' \subseteq
  \Lambda$ such that $\Gamma \subseteq \Gamma'$ by Lemmas
  \ref{lem:gammaconstbasic}$(\ref{lem:gammaconstbasicfflat})$ and 
  \ref{lem:gammaconstbasic}$(\ref{lem:gammaconstbasictransitionfflat})$, hence
  by $(\ref{axiom:gammadescent})$, we have the inclusions
  \begin{equation}\label{eq:ugammaincl}
    \cP(X) \supseteq \pi^\Gamma\bigl(\cP(X^\Gamma)\bigr) \supseteq
    \pi^{\Gamma'}\bigl(\cP(X^{\Gamma'})\bigr)
  \end{equation}
  in $X$, where $\pi^\Gamma(\cP(X^\Gamma))$ and
  $\pi^{\Gamma'}(\cP(X^{\Gamma'}))$ are open.
  Since $X$ is noetherian, it satisfies the ascending chain condition on the
  open sets $\pi^\Gamma(\cP(X^\Gamma))$, hence we can choose a cofinite subset
  $\Gamma_0 \subseteq \Lambda$ such that $\pi^{\Gamma_0}(\cP(X^{\Gamma_0}))$ is
  maximal with respect to inclusion.
  \par We claim that $\cP(X) = \pi^{\Gamma_0}(\cP(X^{\Gamma_0}))$ for every
  cofinite subset $\Gamma \subseteq \Gamma_0$.
  By \eqref{eq:ugammaincl}, it suffices to show the inclusion $\subseteq$.
  Suppose there exists $x \in \cP(X) \smallsetminus
  \pi^{\Gamma_0}(\cP(X^{\Gamma_0}))$.
  By $(\ref{axiom:gammaascent})$, there exists a cofinite subset $\Gamma_1
  \subseteq \Lambda$ such that $(\pi^\Gamma)^{-1}(x) \in \cP(X^\Gamma)$ for
  every cofinite subset $\Gamma \subseteq \Gamma_1$.
  Choosing $\Gamma = \Gamma_0 \cap \Gamma_1$,
  we have $x \in \pi^\Gamma(\cP(X^\Gamma)) \smallsetminus
  \pi^{\Gamma_0}(\cP(X^{\Gamma_0}))$, contradicting the maximality of
  $\pi^{\Gamma_0}(\cP(X^{\Gamma_0}))$.
\end{proof}
We now prove that the properties in Theorem \ref{thm:gammaconstintro} are
preserved when passing to $X^\Gamma$.
Special cases of the following  result appear in \cite[Lem.\ 6.13]{HH94},
\cite[Thm.\ 2.2]{Vel95}, \cite[Lem.\ 2.9]{EH08}, \cite[Lems.\ 3.23 and
3.30]{Has10}, and \cite[Prop.\ 5.6]{Ma14}.
\begin{theorem}\label{thm:gammaconst}
  Fix\index{gamma construction!and singularities|(} notation as in Construction \ref{constr:gamma}.
  \begin{enumerate}[label=$(\roman*)$,ref=\roman*]
    \item For every cofinite subset $\Gamma \subseteq \Lambda$, the map
      $\pi^\Gamma$ identifies local complete intersection, Gorenstein,
      Cohen--Macaulay, and $S_n$ loci.\label{thm:gammaconstcm}
    \item There exists a cofinite subset $\Gamma_0 \subseteq
      \Lambda$ such that $\pi^\Gamma$ identifies $R_n$ (resp.\ normal,
      weakly normal, reduced, strongly $F$-regular, $F$-pure, $F$-rational,
      $F$-injective) loci for every cofinite subset $\Gamma \subseteq \Gamma_0$.%
      \label{thm:gammaconstreg}\index{gamma construction!and singularities|)}
  \end{enumerate}
\end{theorem}
Note that Theorem \ref{thm:gammaconst} implies Theorem \ref{thm:gammaconstintro}
since if $A$ is a
field, then $A^\Gamma$ is also by Construction \ref{constr:gamma}, and moreover
if one wants to preserve more than one property at once, then it suffices to
intersect the various $\Gamma_0$ for the different properties.
\begin{proof}
  For $(\ref{thm:gammaconstcm})$, it suffices to note that these properties
  satisfy the condition in Proposition
  \ref{prop:gammaconstconds}$(\ref{prop:gammaconstcm})$ by \cite[Cor.\
  2]{Avr75} and \cite[Thm.\ 23.4, Cor.\ to Thm.\ 23.3, and Thm.\
  23.9$(iii)$]{Mat89}, respectively.
  \par We now prove $(\ref{thm:gammaconstreg})$.
  We first note that $(\ref{thm:gammaconstreg})$ holds for regularity since
  $(\ref{axiom:gammaopen})$ holds by the excellence of $X^\Gamma$, and
  $(\ref{axiom:gammadescent})$ and $(\ref{axiom:gammaascentprime})$ hold by
  \cite[Thm.\ 23.7]{Mat89}.
  Since $\pi^\Gamma$ preserves the dimension of local rings, we therefore see
  that $(\ref{thm:gammaconstreg})$ holds for $R_n$.
  $(\ref{thm:gammaconstreg})$ for normality and reducedness then follows from
  $(\ref{thm:gammaconstcm})$ since they are equivalent to $R_1+S_2$ and
  $R_0+S_1$, respectively.
  \par To prove $(\ref{thm:gammaconstreg})$ holds in the remaining cases, we
  check the conditions in Proposition
  \ref{prop:gammaconstconds}$(\ref{prop:gammaconstaxiomatic})$.
  For weak normality, $(\ref{axiom:gammaopen})$ holds by \cite[Thm.\
  7.1.3]{BF93}, and $(\ref{axiom:gammadescent})$ holds by \cite[Cor.\
  II.2]{Man80}.
  To show that $(\ref{axiom:gammaascent})$ holds, recall by \cite[Thm.\
  I.6]{Man80} that a reduced ring $B$ is weakly
  normal\index{weak normality|textbf} if and only if
  \begin{equation}\label{eq:manaresiwnchar}
    \begin{tikzcd}[column sep=large]
      B \rar &[-2.125em] B^\nu \rar[shift left=3pt]{b \mapsto b \otimes 1}
      \rar[shift right=3pt,swap]{b \mapsto 1 \otimes b} & (B^\nu \otimes_B
      B^\nu)_\red
    \end{tikzcd}
  \end{equation}
  is an equalizer diagram, where $B^\nu$ is the normalization of $B$.
  Now suppose $B$ is weakly normal, and let $\Gamma_1 \subseteq \Lambda$ be a
  cofinite subset such that $B^\Gamma$ is reduced, $(B^\nu)^\Gamma$ is normal,
  and $((B^\nu \otimes_B B^\nu)_\red)^\Gamma$ is reduced for every cofinite
  subset $\Gamma \subseteq \Gamma_1$; such a $\Gamma_1$ exists by the
  previous paragraph.
  We claim that $B^\Gamma$ is weakly normal for every $\Gamma \subseteq
  \Gamma_1$ cofinite in $\Lambda$.
  Since \eqref{eq:manaresiwnchar} is an equalizer diagram and $A \subseteq
  A^\Gamma$ is flat, the diagram
  \[
    \begin{tikzcd}[column sep=large]
      B^\Gamma \rar &[-2.125em] (B^\nu)^\Gamma \rar[shift left=3pt]{b \mapsto b
      \otimes 1} \rar[shift right=3pt,swap]{b \mapsto 1 \otimes b} &
      \bigl((B^\nu \otimes_B B^\nu)_\red\bigr)^\Gamma
    \end{tikzcd}
  \]
  is an equalizer diagram.
  Moreover, since $B^\Gamma \subseteq (B^\nu)^\Gamma$ is an integral
  extension of rings with the same total ring of fractions, and $(B^\nu)^\Gamma$
  is normal, we see that $(B^\nu)^\Gamma = (B^\Gamma)^\nu$.
  Finally, $((B^\nu \otimes_B B^\nu)_\red)^\Gamma$ is reduced, hence we have the
  natural isomorphism
  \[
    \bigl((B^\nu \otimes_B B^\nu)_\red\bigr)^\Gamma \simeq \bigl((B^\Gamma)^\nu
    \otimes_{B^\Gamma} (B^\Gamma)^\nu\bigr)_\red.
  \]
  Thus, since the analogue of \eqref{eq:manaresiwnchar} with $B$ replaced
  by $B^\Gamma$ is an equalizer diagram, we see that $B^\Gamma$ is weakly normal
  for every $\Gamma \subseteq \Gamma_1$ cofinite in $\Lambda$, hence
  $(\ref{axiom:gammaascent})$ holds for weak normality.
  \par We now prove $(\ref{thm:gammaconstreg})$ for strong $F$-regularity,
  $F$-purity, and $F$-rationality.
  First, $(\ref{axiom:gammaopen})$ holds for strong $F$-regularity by
  \cite[Lem.\ 3.29]{Has10}, and the same argument shows that
  $(\ref{axiom:gammaopen})$ holds for $F$-purity since the $F$-pure and
  $F$-split loci coincide for $F$-finite rings \cite[Cor.\ 5.3]{HR76}.
  Next, $(\ref{axiom:gammaopen})$ for $F$-rationality holds by \cite[Thm.\
  1.11]{Vel95} since the reduced locus is open and reduced $F$-finite rings are
  admissible in the sense of \cite[Def.\ 1.5]{Vel95} by Theorem
  \ref{thm:ffiniteaffine}.
  It then suffices to note that $(\ref{axiom:gammadescent})$ holds by
  \cite[Lem.\ 3.17]{Has10}, \cite[Prop.\ 5.13]{HR76}, and \cite[(6) on p.\
  440]{Vel95}, respectively, and $(\ref{axiom:gammaascent})$ holds by
  \cite[Cor.\ 3.31]{Has10}, \cite[Prop.\ 5.4]{Ma14}, and \cite[Lem.\
  2.3]{Vel95}, respectively.
  \par Finally, we prove $(\ref{thm:gammaconstreg})$ for $F$-injectivity.
  First, $(\ref{axiom:gammaopen})$ and $(\ref{axiom:gammadescent})$ hold by
  \cite[Lem.\ A.2]{Mur} and \cite[Lem.\ A.3]{Mur}, respectively.
  The proof of \cite[Lem.\ 2.9$(b)$]{EH08} implies
  $(\ref{axiom:gammaascent})$, since the residue field of $B$ is a finite
  extension of $k$, hence socles of artinian $B$-modules are finite-dimensional
  $k$-vector spaces.
\end{proof}
\section{Applications}
We now give some applications of the gamma construction (Theorem
\ref{thm:gammaconst}).
See also \cite[\S3.2]{Mur} for applications to the minimal model program over
imperfect fields.
\subsection{Openness of \emph{F}-singularities}
We have the following consequence of Theorem \ref{thm:gammaconst}, which was
first attributed to Hoshi in \cite[Thm.\ 3.2]{Has10proc}.
Note that the analogous statements for strong $F$-regularity and $F$-rationality
appear in \cite[Prop.\ 3.33]{Has10} and \cite[Thm.\ 3.5]{Vel95}, respectively.
\begin{corollary}\label{cor:fpurelocusopen}
  Let $X$ be a scheme essentially of finite type over a local $G$-ring
  $(A,\fm)$ of characteristic $p > 0$.
  Then, the $F$-pure locus is open in $X$.%
  \index{F-pure@$F$-pure!is often an open condition}
\end{corollary}
Recall that a noetherian ring $R$ is a
\textsl{$G$-ring}\index{G-ring@$G$-ring|textbf} if, for every prime
ideal $\fp \subseteq R$, the completion homomorphisms $R_\fp \to
\widehat{R_\fp}$ are regular in the sense of \cite[Def.\ 6.8.1]{EGAIV2}.
\begin{proof}
  Let $A \to \widehat{A}$ be the completion of $A$ at $\fm$, and let $\Lambda$
  be a $p$-basis for $\widehat{A}/\fm\widehat{A}$ as in Construction
  \ref{constr:gamma}.
  For every cofinite subset $\Gamma \subseteq \Lambda$, consider the
  commutative diagram
  \[
    \begin{tikzcd}
      X \times_A \widehat{A}^\Gamma \rar{\pi^\Gamma}\dar & X \times_A
      \widehat{A} \rar{\pi}\dar & X\dar\\
      \Spec \widehat{A}^\Gamma \rar & \Spec \widehat{A} \rar & \Spec A
    \end{tikzcd}
  \]
  where the squares are cartesian.
  By Theorem \ref{thm:gammaconst}, there exists a cofinite subset $\Gamma
  \subseteq \Lambda$ such that $\pi^\Gamma$ is a homeomorphism identifying
  $F$-pure loci.
  Since $X \times_A \widehat{A}^\Gamma$ is $F$-finite, the $F$-pure locus in
  $X \times_A \widehat{A}$ is therefore open by the fact that
  $(\ref{axiom:gammaopen})$ holds for $F$-purity; see the proof of Theorem
  \ref{thm:gammaconst}$(\ref{thm:gammaconstreg})$.
  \par Now let $x \in X \times_A \widehat{A}$.
  Since $A \to \widehat{A}$ is a regular homomorphism, the morphism $\pi$ is
  also regular by base change \cite[Prop.\ 6.8.3$(iii)$]{EGAIV2}.
  Thus, $\cO_{X \times_A \widehat{A},x}$ is $F$-pure if and only if
  $\cO_{X,\pi(x)}$ is $F$-pure by \cite[Prop.\ 5.13]{HR76} and \cite[Props.\
  2.4(4) and 2.4(6)]{Has10}.
  Denoting the $F$-pure locus in $X$ by $W$, we see that $\pi^{-1}(W)$ is
  the $F$-pure locus in $X \times_A \widehat{A}$.
  Since $\pi^{-1}(W)$ is open and $\pi$ is quasi-compact and faithfully flat by
  base change, the $F$-pure locus $W \subseteq X$ is open by \cite[Cor.\
  2.3.12]{EGAIV2}.
\end{proof}
\begin{remark}
  Although \cite[Lem.\ A.2]{Mur} shows that the $F$-injective locus is open
  under $F$-finiteness hypotheses, and the gamma construction (Theorem
  \ref{thm:gammaconst}) implies that the $F$-injective locus is open for schemes
  essentially of finite type over \emph{complete} local rings, the fact that the
  $F$-injective locus is open under the hypotheses of Corollary
  \ref{cor:fpurelocusopen} is a recent result due to
  Rankeya Datta\index{Datta, Rankeya} and the author \cite[Thm.\ B]{DaM}.%
  \index{F-injective@$F$-injective!is often an open condition}
\end{remark}
\subsection{\emph{F}-singularities for rings essentially of finite type} 
We finally show that for rings to which the gamma construction applies, the
notions of strong $F$-regularity and split $F$-regularity coincide, as do the
notions of $F$-purity and $F$-splitting.
This result is unpublished work of Rankeya Datta\index{Datta, Rankeya} and the
author.
\begin{theorem}\label{thm:splittingwithgamma}
  Let $R$ be a ring essentially of finite type over a noetherian complete local
  ring $(A,\fm,k)$ of characteristic $p > 0$.
  If $R$ is strongly $F$-regular (resp.\ $F$-pure), then $R$ is split
  $F$-regular (resp.\ $F$-split).%
  \index{F-regular@$F$-regular!strongly}%
  \index{F-pure@$F$-pure}
\end{theorem}
We first show the following preliminary result, which was communicated to
Hochster\index{Hochster, Melvin} by Auslander\index{Auslander, Maurice}
(although it may be older).
\begin{lemma}[cf.\ {\cite[Lem.\ 1.2]{Fed83}}]\label{lem:fed12}
  Let $(A,\fm,k)$ be a noetherian complete local ring.
  Then, every pure ring homomorphism $A \to B$ splits as an $A$-module
  homomorphism.
\end{lemma}
\begin{proof}
  Let $f\colon A \to B$ be a pure ring homomorphism.
  We claim we have the following commutative diagram with exact rows, where
  the vertical homomorphisms are isomorphisms:
  \[
    \begin{tikzcd}[column sep=3em]
      \Hom_A\bigl(B\otimes_AE_A(k),E_A(k)\bigr)
      \rar{(f \otimes \id_{E_A(k)})^*}\arrow{d}[sloped,above]{\sim} &
      \Hom_A\bigl(A\otimes_AE_A(k),E_A(k)\bigr)
      \rar\arrow{d}[sloped,above]{\sim} & 0\\
      \Hom_A\bigl(B,\Hom_A\bigl(E_A(k),E_A(k)\bigr)\bigr) \rar{f^*} &
      \Hom_A\bigl(A,\Hom_A\bigl(E_A(k),E_A(k)\bigr)\bigr) \rar & 0\\
      \Hom_A(B,A) \rar{f^*}\arrow{u}[sloped,below]{\sim} &
      \Hom_A(A,A) \rar\arrow{u}[sloped,below]{\sim} & 0
    \end{tikzcd}
  \]
  The top row is the Matlis dual
  of the map $f \otimes \id_{E_A(k)} \colon A
  \otimes_A E_A(k) \to B \otimes_A E_A(k)$, and the second row is obtained from
  the first by tensor-hom adjunction.
  The last row is obtained from the isomorphism $\Hom_A(E_A(k),E_A(k)) \simeq A$,
  which holds by the completeness of $A$ \cite[Thm.\ 18.6$(iv)$]{Mat89}.
  Since the last row is surjective, we can choose $g
  \in \Hom_A(B,A)$ such that $f^*(g) = g \circ f = \id_A$.
\end{proof}
We can now show Theorem \ref{thm:splittingwithgamma}.
\begin{proof}[Proof of Theorem \ref{thm:splittingwithgamma}]
  By the gamma construction (Theorem \ref{thm:gammaconst}), there exists a
  faithfully flat ring extension $A \hookrightarrow A^\Gamma$ such that
  $R^\Gamma \coloneqq R \otimes_A A^\Gamma$ is strongly $F$-regular (resp.\
  $F$-pure) and $F$-finite.
  By $F$-finiteness, the ring $R^\Gamma$ is split $F$-regular by
  \cite[Lem.\ 3.9]{Has10} (resp.\ $F$-split by \cite[Cor.\ 5.3]{HR76}).
  Now consider the commutative diagram
  \[
    \begin{tikzcd}[column sep=huge]
      A \rar\dar[hook] & R
      \rar{F^e_R}\dar[hook] & F^e_*R \rar{F^e_*(-\cdot c)}\dar[hook] &
      F^e_*R\dar[hook]\\
      A^\Gamma \rar & R^\Gamma \rar{F^e_{R^\Gamma}} & F^e_*R^\Gamma
      \rar{F^e_*(-\cdot (c\otimes 1))} & F^e_*R^\Gamma
    \end{tikzcd}
  \]
  for every $c \in R^\circ$ and every integer $e > 0$, where the left square
  is cocartesian.
  Note that if $c \in R^\circ$, then $c \otimes 1 \in (R^\Gamma)^\circ$, since
  $R \to R^\Gamma$ satisfies going-down \cite[Thm.\
  9.5]{Mat89}.
  \par Since the inclusion $A \hookrightarrow A^\Gamma$ is faithfully flat, it is
  pure, hence splits as an $A$-module homomorphism by Lemma \ref{lem:fed12}.
  By base change, this implies the inclusion $R \hookrightarrow R^\Gamma$ splits
  as an $R$-module homomorphism.
  For both split $F$-regularity and $F$-splitting, it then suffices to note that
  if $F^e_*(-\cdot (c\otimes 1)) \circ F^e_{R^\Gamma}$ splits for some $c \in
  R^\circ$ and for some $e > 0$, then composing this splitting with a splitting
  of $R \hookrightarrow R^\Gamma$ gives a splitting of $F^e_*(-\cdot c) \circ
  F^e_R$ by the commutativity of the diagram above.
\end{proof}

\backmatter

\printbibliography[heading=bibintoc]

\index{complete linear system|see {linear system, complete}}

\index{Cartan, Henri!zzzzz@\igobble |see {Cartan--Serre--Grothendieck theorem}}
\index{Serre, Jean-Pierre!zzzzz@\igobble |see {Cartan--Serre--Grothendieck theorem; vanishing theorem, Serre}}
\index{Grothendieck, Alexander!zzzzz@\igobble |seealso {Cartan--Serre--Grothendieck theorem}}

\index{ample line bundle or divisor!zzzzz@\igobble |seealso {Cartan--Serre--Grothendieck theorem}}

\index{Riemann, Bernhard!zzzzz@\igobble |see {Riemann--Roch theorem}}
\index{Roch, Gustav!zzzzz@\igobble |see {Riemann--Roch theorem}}

\index{Kawamata, Yujiro!zzzzz@\igobble |seealso {Kawamata--Reid--Shokurov, cohomological method of; klt; vanishing theorem, Kawamata--Viehweg}}
\index{Reid, Miles!zzzzz@\igobble |see {Kawamata--Reid--Shokurov, cohomological method of}}
\index{Shokurov, V. V.!zzzzz@\igobble |see {Kawamata--Reid--Shokurov, cohomological method of}}

\index{Angehrn, Urban!zzzzz@\igobble |seealso {Angehrn--Siu theorem}}
\index{Siu, Yum Tong!zzzzz@\igobble |seealso {Angehrn--Siu theorem}}

\index{g amma construction@$\Gamma$ construction|see {gamma construction}}

\index{Kawamata log terminal|see {klt}}

\index{Seshadri, C. S.!zzzzz@\igobble |see {Seshadri constant}}

\index{Fujita, Takao!zzzzz@\igobble |seealso {vanishing theorem, Fujita}}
\index{Kodaira, Kunihiko!zzzzz@\igobble |see {vanishing theorem, Kodaira}}
\index{Nadel, Alan Michael!zzzzz@\igobble |seealso {vanishing theorem, Nadel}}
\index{Viehweg, Eckart!zzzzz@\igobble |seealso {vanishing theorem, Kawamata--Viehweg}}

\index{Hilbert, David!zzzzz@\igobble |see {Hilbert--Samuel multiplicity}}
\index{Tango--Raynaud curve!zzzzz@\igobble |see {Tango, Hiroshi, curve}}

\index{Szpiro, Lucien!zzzzz@\igobble |seealso {vanishing theorem, Szpiro--Lewin-M\'en\'egaux}}
\index{Lewin-Menegaux, Renee@Lewin-M\'en\'egaux, Ren\'ee!zzzzz@\igobble |seealso {vanishing theorem, Szpiro--Lewin-M\'en\'egaux}}

\index{Mori, Shigefumi!zzzzz@\igobble |seealso {bend and break; Mori--Mukai conjecture}}
\index{Mukai, Shigeru!zzzzz@\igobble |seealso {Mori--Mukai conjecture}}

\index{normalized dualizing complex, $\omega_X^\bullet$!zzzzz@\igobble |see {dualizing complex, normalized}}

\index{jet bundle, $J_\ell(L)$|see {bundle of principal parts}}

\index{Cartier, Pierre!zzzzz@\igobble |see {Cartier divisor; Cartier operator}}
\index{Weil, Andre@Weil, Andr\'e|see {Weil divisor}}

\end{document}